\documentclass[11pt,reqno]{amsart}

\usepackage[T1]{fontenc}
\usepackage{lmodern}
\usepackage{amsmath,amssymb,mathtools,mathrsfs}
\usepackage{amsthm}
\usepackage{aliascnt}
\usepackage{enumitem}
\usepackage{array,tabularx}
\usepackage{microtype}
\usepackage{xcolor}
\usepackage[colorlinks=true,linkcolor=blue,citecolor=blue,urlcolor=blue]{hyperref}
\allowdisplaybreaks
\numberwithin{equation}{section}

\usepackage[
    a4paper,
    left=29mm,
    right=29mm,
    top=25mm,
    bottom=28mm
]{geometry}

\setlist[itemize]{topsep=3pt,itemsep=2pt,parsep=0pt}
\setlist[enumerate]{topsep=3pt,itemsep=2pt,parsep=0pt}

\newtheorem{theorem}{Theorem}[section]
\newaliascnt{proposition}{theorem}
\newtheorem{proposition}[proposition]{Proposition}
\aliascntresetthe{proposition}
\newaliascnt{lemma}{theorem}
\newtheorem{lemma}[lemma]{Lemma}
\aliascntresetthe{lemma}

\newaliascnt{corollary}{theorem}
\newtheorem{corollary}[corollary]{Corollary}
\aliascntresetthe{corollary}
\theoremstyle{remark}
\newaliascnt{remark}{theorem}
\newtheorem{remark}[remark]{Remark}
\aliascntresetthe{remark}
\usepackage[nameinlink,noabbrev]{cleveref}

\newcommand{\T}{\mathbb T}
\newcommand{\Z}{\mathbb Z}
\newcommand{\R}{\mathbb R}
\newcommand{\N}{\mathbb N}
\newcommand{\cH}{\mathcal H}
\newcommand{\cS}{\mathcal S}
\newcommand{\cL}{\mathcal L}

\newcommand{\PP}{\mathbb P}
\newcommand{\Ran}{\operatorname{Ran}}
\newcommand{\Ker}{\operatorname{Ker}}
\newcommand{\Id}{\operatorname{Id}}
\newcommand{\HS}{\operatorname{HS}}
\newcommand{\supp}{\operatorname{supp}}
\newcommand{\dd}{\,\mathrm d}
\newcommand{\ip}[2]{\left\langle #1,#2\right\rangle}
\newcommand{\norm}[1]{\left\lVert #1\right\rVert}

\newcolumntype{Y}{>{\raggedright\arraybackslash}X}

\title[Unique ergodicity of the projective process]
{Unique Ergodicity for the Projective Process of the 2D Navier--Stokes
Equation with Nondegenerate Noise}
\author{Zeng Lian}
\address{School of Mathematics, Sichuan University, Chengdu,
Sichuan 610016, P. R. China}
\email{lianzeng@scu.edu.cn}
\author{Rongchang Liu}
\address{School of Mathematics, Sichuan University, Chengdu,
Sichuan 610064, P. R. China}
\email{rcliu@scu.edu.cn}
\author{Kening Lu}
\address{School of Mathematics, Sichuan University, Chengdu,
Sichuan 610064, P. R. China}
\email{keninglu@scu.edu.cn}
\hypersetup{
  pdftitle={Unique Ergodicity for the Projective Process of the 2D Navier--Stokes Equation with Nondegenerate Noise},
  pdfauthor={Zeng Lian, Rongchang Liu, and Kening Lu},
  pdfsubject={Unique ergodicity and the Furstenberg--Khasminskii formula for the projectivised linearised 2D stochastic Navier--Stokes equation},
  pdfkeywords={stochastic Navier--Stokes equations; projective process; unique ergodicity; Malliavin calculus; asymptotic coupling}
}

\date{}

\begin{document}

\begin{abstract}
We prove unique ergodicity of the projective process associated with the two dimensional Navier--Stokes equation in
vorticity form, with additive diagonal noise acting on every nonzero real Fourier phase and satisfying two sided power law bounds. Consequently, the exact Furstenberg--Khasminskii formula for the top Lyapunov exponent holds. 

The main new ingredient is a compact dense mechanism for asymptotic generalized coupling in the absence of a Foia\c{s}--Prodi type high-low mode decomposition for the projective dynamics. Using the dense range of the Malliavin derivative and compactness of the state derivative, we construct finite rank perturbations of the Wiener path that compensate, to first order, for perturbations of the initial condition, leaving a residual whose logarithmic growth has negative stationary mean and hence contracts locally at an exponential rate, with a cost controlled by a triangular scheme of blockwise Ramer transformations. 
\end{abstract}

\maketitle

\begingroup
\renewcommand{\thefootnote}{}
\footnotetext{
This work was supported by the Fundamental Research Funds for the Central Universities and the National Natural Science Foundation of China (Grant Nos.~12090010, 12090013 and 12471154). 
}
\endgroup

\setcounter{tocdepth}{1}
\tableofcontents

\section{Introduction}

A central question in the mathematical theory of turbulence is to identify
the probability measure that describes observations over long time scales.
In his article \cite{Ruelle78}, Ruelle proposed that the statistical
properties of a dissipative system, and in particular of a turbulent fluid,
should be described by an invariant measure selected by the dynamics.  For
uniformly hyperbolic attractors, SRB measures provide the canonical model
for such physical measures.  The survey of Eckmann and Ruelle
\cite{EckmannRuelle85} placed this program in the broader ergodic theory of
strange attractors and emphasized that physical measures, Lyapunov
exponents, entropy, and dimension are complementary aspects of a
statistical description of chaotic dynamics including fluid flows.

Within this program, positivity of the top Lyapunov exponent is a basic
signature of dynamical instability and chaos.  Random forcing provides a
particularly tractable setting for this question.  Indeed, for many
stochastic fluid equations the associated Markov dynamics possesses a
unique invariant measure, see, among others,
\cite{FM95,BKL01,EM01,HM06,HM08,HM11,KS12}, thereby providing a canonical
stationary law with respect to which Lyapunov exponents can be studied.
For Galerkin approximations of stochastic fluid equations, positivity of
the top Lyapunov exponent has been established through a detailed study of
the associated projective dynamics and the Furstenberg--Khasminskii formula
in \cite{BBPS22b,BPS}, see also related results on
Lagrangian chaos and scalar mixing in
\cite{BBPS21,BBPS22a,BBPS22}.  For the full stochastic Navier--Stokes equation,
however, the corresponding infinite dimensional projective dynamics is
substantially more delicate, and a comparable theory has remained largely
open.

A major recent advance was made in \cite{HPSRY}.  Building on the
spectral median mechanism introduced in \cite{HR25} for the normalised
linearised flow, the authors proved positive time regularisation of the
projective process, constructed a stationary projective law, and obtained
a Furstenberg--Khasminskii lower bound for the top Lyapunov exponent.
The uniqueness of the stationary projective law, and consequently the
exact Furstenberg--Khasminskii formula, remained open.  The purpose of the
present paper is to resolve these questions.

We consider the stochastic Navier--Stokes equation on $\T^2$ in vorticity
form,
\begin{align}\label{eq:intro-sns}
 \dd w_t
 =\bigl(\nu\Delta w_t-B(w_t,w_t)\bigr)\dd t+Q\dd W_t,
 \qquad \nu>0,
\end{align}
where $B(w,w)=(Kw)\cdot\nabla w$ and
$K=\nabla^\perp(-\Delta)^{-1}$.  The noise is diagonal in the real Fourier
basis, acts on every nonzero Fourier phase, and satisfies uniform two sided
power law bounds.  In contrast with the radial covariance in
\cite{HPSRY}, the amplitudes may depend on both the frequency direction and
the real Fourier phase.  The derivative cocycle of \eqref{eq:intro-sns}
induces a Markov process $(w_t,[\rho_t])$ on
$H^b\times\PP(H)$, where $\PP(H)$ denotes the projective space of $H$
and $[\rho]$ the line generated by $0\ne\rho\in H$.  The precise noise
assumption, the choice of $b$, and the functional setting are given in
\Cref{sec:preliminaries-main}.

Our main theorem states that this projective process has exactly one
invariant probability measure.  Thus, over the stationary law of the
stochastic fluid, the tangent direction admits a unique stationary law.
The uniqueness theorem also identifies the stationary logarithmic growth
rate with the top Lyapunov exponent and gives the exact
Furstenberg--Khasminskii formula.  The precise statements are given in
\Cref{thm:main} and \Cref{cor:exact-FK}.

The main difficulty is that the usual high low mode mechanism behind
asymptotic coupling for the base Navier--Stokes equation does not survive
projectivisation.  This failure is already visible for the heat equation.  Let $F_T[\rho]=[e^{\nu T\Delta}\rho]$ be the induced map on $\PP(H)$.  If $|\boldsymbol n|>|\boldsymbol m|$, then in the affine chart at $[e_{\boldsymbol n}]$,
\begin{align*}
DF_T([e_{\boldsymbol n}])e_{\boldsymbol m}=e^{\nu(|\boldsymbol n|^2-|\boldsymbol m|^2)T}e_{\boldsymbol m}.
\end{align*}
Thus a low frequency perturbation of a high frequency projective direction
may be exponentially amplified even though the raw heat flow is strongly
dissipative.  There is therefore no fixed high mode complement on which the
normalised dynamics contract uniformly.  A second obstruction is that the
joint Malliavin endpoint derivative has dense range but need not be onto,
so no bounded inverse Malliavin covariance is available. This later issue was also encountered in \cite{HM06,HM11}, where a 
Tikhonov regularisation was used to obtain an asymptotic
compensation of the initial perturbation.

The new compact dense mechanism of the paper combines these two apparently
insufficient properties.  Positive time parabolic smoothing makes the state
derivative compact, whereas the joint Malliavin endpoint has dense range.
Consequently a compact state derivative can be approximated in operator
norm by the composition of the Malliavin derivative with a finite rank
Cameron--Martin control.  We implement this compact dense
principle simultaneously in the state, the noise path, and the projective
tangent bundle.  On compact state and noise convolution driver it yields finitely many
controls which are patched into a measurable selector.  The corresponding
linear residual is small on a fixed positive fraction of every stationary
path measure and is controlled by the state derivative elsewhere.  Uniform
occupation estimates then turn this into a negative stationary mean for
the logarithmic residual.  This replaces the deterministic contraction of
a Foia\c{s}--Prodi high mode complement and produces, with positive
probability, a random local stable ball for the compensated projective
dynamics.

The dense range argument builds on the Malliavin nondegeneracy mechanism
for the stochastic Navier--Stokes equation developed in \cite{MP06} and
subsequently used in \cite{HM06,HM11}.  The input needed here, however, is
qualitative rather than quantitative. 
We prove that the Malliavin endpoint has almost surely dense range, with a
single exceptional null set independent of both the terminal time and the
initial state.  Starting from an
exact adjoint identity, a bounded variation quadratic variation argument
propagates annihilation through the Fourier Lie algebra even though the
relevant test paths may depend on the whole Brownian trajectory.  A
separation argument for weighted Fourier shifts then forces every
annihilator to vanish. 

There is a further probabilistic issue.  The finite rank selector depends on
the whole fresh noise block, so the compensating Wiener shift is
anticipative inside that block and an adapted Girsanov argument is not
available.  We use Ramer's theorem \cite{Ramer,UZ97} instead.  The Malliavin derivative
of each block shift is Hilbert--Schmidt and its relative entropy cost is
quadratic in the current distance.  Across blocks the construction is
triangular.  An entropy stopping rule and a countable mixture over Cameron--Martin cost
levels give an absolutely continuous generalized coupling on the infinite
path space while retaining positive probability that the compensating
shifts are never stopped.

The resulting stable ball is local, and a separate mechanism is required to
reach it.  Around the prepared state
$(0,[\cos x_1+\cos x_2])$ we prove an observability estimate for the
deterministic tangent skeleton.  Douglas' factorisation theorem then gives
exact first order cancellation by a Cameron--Martin perturbation and hence a
local bridge with both exact transition law marginals.  In particular, this yields local $d$-smallness, which provides one of the
ingredients for a possible future spectral gap argument.

Accessibility of this neighbourhood is proved separately by a control
construction exploiting the specific linearised and projective structure.
The base path is controlled directly, while the fibre is steered indirectly
through first shell seeding, relaxation, and a four dimensional control
algebra.  Gaussian support then gives positive probability accessibility.
Combining accessibility, the prepared bridge, and the random stable ball
closes the asymptotic generalized coupling criterion of \cite{HMS11}.

\section{Preliminaries and main results}
\label{sec:preliminaries-main}

This section gives the precise formulation of the equation, the projective
cocycle, and the two main results.  The analytic facts needed to define the
process globally are collected in \Cref{sec:foundations}.

Let $\T^2=(\R/2\pi\Z)^2$ carry normalised Lebesgue measure and, for
$s\in\R$, let $H^s=H^s_0(\T^2;\R)$ and $H=H^0$, where the subscript
denotes the mean zero subspace.  Since the zero Fourier mode is absent, we
use throughout the equivalent homogeneous Fourier norm
\begin{align*}
 \|f\|_{H^s}^2
 =\sum_{\boldsymbol k\ne\boldsymbol0}
 |\boldsymbol k|^{2s}|\widehat f(\boldsymbol k)|^2.
\end{align*}
Fix the half lattice
\begin{align*}
 \Z^2_+
 =\{\boldsymbol k=(k_1,k_2)\in\Z^2\setminus\{\boldsymbol0\}:
 k_1>0,\ \text{or }k_1=0\text{ and }k_2>0\},
\end{align*}
and let $e_{\boldsymbol k,\cos}$ and $e_{\boldsymbol k,\sin}$ denote the
normalised real Fourier modes.  For mean zero functions $f,g$, set 
\begin{align*}
 B(f,g)&=(Kf)\mathbin\cdot\nabla g, \text{ where }  Kf=\nabla^\perp(-\Delta)^{-1}f,\\
 \cS(f,g)&=-B(f,g)-B(g,f).
\end{align*}

We consider the stochastic Navier--Stokes equation in vorticity form
\begin{align}\label{eq:sns}
 \dd w_t
 =\bigl(\nu\Delta w_t-B(w_t,w_t)\bigr)\dd t+Q\dd W_t,
 \qquad \nu>0,
\end{align}
where $W$ is the cylindrical Wiener process on $H$  and $Q$ is diagonal:
\begin{align*}
 Qe_{\boldsymbol k,\iota}
 =q_{\boldsymbol k}^\iota e_{\boldsymbol k,\iota}.
\end{align*}
We assume that there are constants $0<c_Q\le C_Q<\infty$ and an exponent
$a>14$ such that
\begin{align}\label{eq:q}
 c_Q|\boldsymbol k|^{-a/2}
 \le |q_{\boldsymbol k}^\iota|
 \le C_Q|\boldsymbol k|^{-a/2},
 \qquad
 \boldsymbol k\in\Z^2_+,\quad
 \iota\in\{\cos,\sin\}.
\end{align}
In particular, every nonzero real Fourier phase is forced.  Unlike the
radial covariance considered in \cite{HPSRY}, the amplitudes may depend on
both the frequency direction and the real Fourier phase: we require only
the uniform two sided power law bounds in \eqref{eq:q}. The Lyapunov structure needed for this more general diagonal covariance is
established in \Cref{sec:foundations}.

Throughout,
\begin{align}\label{eq:b}
 b=\frac{a-3}{2}>\frac{11}{2},
\end{align}
which implies that 
\begin{align*}
 \|Q\|_{\HS( H,H^b)}^2
 \lesssim
 \sum_{\boldsymbol k\in\Z^2_+}
 |\boldsymbol k|^{2b-a}
 =
 \sum_{\boldsymbol k\in\Z^2_+}
 |\boldsymbol k|^{-3}
 <\infty,
\end{align*}
so the stochastic convolution is $H^b$ valued.  The equation
\eqref{eq:sns} is globally well posed in $H^b$ by the standard
semilinear parabolic theory, see also 
\Cref{lem:compatible-driver-global}; continuous dependence gives a Feller
Markov semigroup.  Its invariant probability measure is unique \cite{HM06,KS12} and denoted by $\mu$.

For a base trajectory $w$, the derivative cocycle is the solution operator
$U(t,s)$ of
\begin{align}\label{eq:raw-tangent}
 \partial_t\rho_t
 =\nu\Delta\rho_t+\cS(w_t,\rho_t),
 \qquad
 \rho_s=\rho.
\end{align}
Under our regularity assumptions, the standard backward uniqueness
argument for parabolic equations applies to \eqref{eq:raw-tangent}; see,
for instance, \cite[Chapter~III, Section~6]{Temam97}.   Hence
$U(t,s)\rho\ne0$ whenever $t>s$ and $\rho\ne0$, and the derivative
cocycle can be projectivised.  Let
\begin{align*}
 \PP(H)=(H\setminus\{0\})/\R^\times
\end{align*}
and, for unit representatives, use the metric
\begin{align}\label{eq:projective-metric}
 d_{\PP}([u],[v])
 =\min\{\|u-v\|_2,\|u+v\|_2\}.
\end{align}
The state space of the projective process is $\mathsf X=H^b\times\PP(H)$, 
and we write $(P_t)_{t\ge0}$ for the Markov semigroup of
$(w_t,[\rho_t])$.  For a unit representative $\mathsf p$ of the initial
projective direction, we write
\begin{align*}
 \mathsf p_t=\frac{U(t,0)\mathsf p}{\|U(t,0)\mathsf p\|_2}.
\end{align*}
Positive time smoothing and the compact multiplicative ergodic theorem \cite{Ruelle82,LianLu10}
give a well defined top Lyapunov exponent $\lambda_1$ (may take the value $-\infty$) over the stationary
base law $\mu$.  The compactness and integrability hypotheses used for the
multiplicative ergodic theorem are verified in \Cref{subsec:FK-proof}.

Our main result is the unique ergodicity of the full infinite dimensional
projective process.  

\begin{theorem}
\label{thm:main}
Under \eqref{eq:q}--\eqref{eq:b}, the Markov process
\begin{align*}
 (w_t,[\rho_t])\quad\text{on}\quad H^b\times\PP(H)
\end{align*}
has exactly one invariant probability measure $\widehat\mu$.  Its first
marginal is $\mu$, the unique invariant probability measure of
\eqref{eq:sns}.
\end{theorem}

As a corollary, we have the Furstenberg--Khasminskii formula for 
the top Lyapunov exponent.  Its proof is given after the proof of 
\Cref{thm:main} in \Cref{subsec:FK-proof}.

\begin{corollary}
      \label{cor:exact-FK}
      Let $\lambda_1$ denote the top Lyapunov exponent of the derivative
      cocycle $U(t,s)$ over the stationary base law. Then $\lambda_1>-\infty$
      and
      \begin{align}\label{eq:exact-FK}
       \lambda_1
       =
       \int_{H^b\times\PP(H)}
       \left(
         -\nu\norm{\nabla\mathsf p}_2^2
         +\ip{w}{B(\mathsf p,\mathsf p)}_2
       \right)
       \widehat\mu(\dd w,\dd[\mathsf p]),
      \end{align}
      where $\mathsf p$ denotes either representative of $[\mathsf p]$ satisfying
      $\norm{\mathsf p}_2=1$. Moreover,
      \begin{align}\label{eq:stationary-top-growth}
       \lim_{t\to\infty}
       \frac1t\log\norm{U(t,0)\mathsf p}_2
       =
       \lambda_1
      \end{align}
      for $\widehat\mu\otimes\mathbf P$ almost every stationary initial state
      $(w,[\mathsf p])$ and driving Wiener path.
\end{corollary}

\section{The joint Malliavin derivative and its dense range}
\label{sec:adjoint}

This section identifies the Malliavin derivative of the projective
endpoint and proves that its range is almost surely dense.  The proof first
reduces density to a vertical annihilator problem and then eliminates that
annihilator by a pathwise Fourier--Lie bracket argument.

\subsection{The joint Malliavin derivative and its transpose}
\label{rev:sec-transposition}

Let $T>0$, let $0\ne\pi\in H$, and let
$w\in C([0,T];H^b)$ be a base solution. By homogeneity, we may
assume that $\|\pi\|_2=1$. Put
\begin{align*}
 \rho_t=U_w(t,0)\pi,\qquad d_T=\|\rho_T\|_2,\qquad
 \mathsf p_T=\rho_T/d_T.
\end{align*}
For $v\in\cH_T=L^2(0,T; H)$, where $\cH_T$ is the
Cameron--Martin control space over $[0,T]$, define
\begin{align}
 u_t&=\int_0^tU_w(t,s)Qv_s\,\dd s,
\label{rev:eq-base-control-response}\\
 r_t&=\int_0^tU_w(t,s)\cS(\rho_s,u_s)\,\dd s,
\label{rev:eq-cross-control-response}\\
 \widehat A_Tv&=\left(u_T,d_T^{-1}\Pi_{\mathsf p_T^\perp}r_T\right).\notag
\end{align}
Thus $u_t$ and $r_t$ are the Malliavin derivatives of the base solution
$w_t$ and the raw fibre $\rho_t$, respectively, while
$\widehat A_T$ is the Malliavin derivative of the joint endpoint $(w_T,\mathsf p_T)$.

Define
\begin{align}\label{eq:Qstar-definition}
 Q^*:H^{-b}\to H,\qquad
 (Q^*f)_{\boldsymbol k,\iota}
 =q_{\boldsymbol k}^\iota
 \langle f,e_{\boldsymbol k,\iota}\rangle_{H^{-b},H^b}.
\end{align}
It is bounded because
$|q_{\boldsymbol k}^\iota|^2
\lesssim|\boldsymbol k|^{-a}
\le|\boldsymbol k|^{-2b}$, and injective because every coefficient is
nonzero. Formula \eqref{eq:Qstar-definition} is the limit in $ H$
of finite Fourier sums.

For $\rho\in H$, define $\cS_\rho^*:H^1\to H^{-b}$ by
\begin{align}\label{rev:eq-Srho-transpose-def}
 \langle\cS_\rho^*\beta,h\rangle_{H^{-b},H^b}
 =\langle\beta,\cS(\rho,h)\rangle_{H^1,H^{-1}}.
\end{align}
Then
\begin{align}\label{rev:eq-Srho-transpose-bound}
 \|\cS_\rho^*\beta\|_{H^{-b}}
 \le C_b\|\rho\|_2\|\beta\|_{H^1}.
\end{align}
For smooth functions,
\begin{align*}
 \cS_\rho^*\beta=K\rho\cdot\nabla\beta-K^*(\beta\nabla\rho).
\end{align*}
For $\rho\in H$, this formula means only the transposition above; neither
product is asserted to be $L^2$ valued.

The dense range question is most transparent on the transpose.  The next
proposition computes it exactly and reduces density to the absence of a
vertical annihilator. Note that the continuous dual of the endpoint space is
$H^{-b}\oplus \mathsf p_T^\perp$.
\begin{proposition}
\label{prop:annihilator}
 For
$(f,g)\in H^{-b}\oplus \mathsf p_T^\perp$, set
\begin{align}
 \beta_t&=U_w(T,t)^*(g/d_T),
\label{eq:beta}\\
 \alpha_t&=U_w(T,t)^*f
 +\int_t^TU_w(s,t)^*\cS_{\rho_s}^*\beta_s\,\dd s.
\label{eq:alpha}
\end{align}
In \eqref{eq:beta} the adjoint evolution acts on $H$, whereas the first
term in \eqref{eq:alpha} uses its extension to $H^{-b}$.  Then
\begin{align*}
 \beta\in C([0,T];H)\cap L^2(0,T;H^1),\quad 
 \cS_\rho^*\beta\in L^2(0,T;H^{-b}),
\end{align*}
and
\begin{align*}
 \alpha\in C([0,T];H^{-b}), \quad \widehat A_T'(f,g)(t)=Q^*\alpha_t
 \quad\text{in }L^2(0,T; H),
\end{align*}
where $\widehat A_T'$ denotes the Banach transpose of $\widehat A_T$. Moreover,
\begin{align}\label{eq:kernel}
 \Ker\widehat A_T'
 =\left\{(0,g):
 \begin{array}{l}
 g\in \mathsf p_T^\perp,\quad
 \beta_t=U_w(T,t)^*(g/d_T),\\
 \cS_{\rho_t}^*\beta_t=0
 \text{ in }H^{-b}\text{ for a.e. }t
 \end{array}\right\},
\end{align}
and $\langle\beta_t,\rho_t\rangle_2=0$ for all $t$.
\end{proposition}

\begin{proof}
The adjoint energy inequality gives $\beta\in C H\cap L^2H^1$;
\eqref{rev:eq-Srho-transpose-bound} gives the source in $L^2H^{-b}$.
Thus \eqref{eq:alpha} is a continuous $H^{-b}$ valued Bochner
integral.

Equations \eqref{rev:eq-cross-control-response},
\eqref{eq:beta}, and \eqref{rev:eq-Srho-transpose-def} give
\begin{align*}
 \langle g/d_T,r_T\rangle_2
 =\int_0^T\langle\cS_{\rho_t}^*\beta_t,u_t\rangle\,\dd t.
\end{align*}
Substitute \eqref{rev:eq-base-control-response}. Bochner Fubini is
justified by \eqref{rev:eq-Srho-transpose-bound},
$\beta\in L^2H^1$, boundedness of $U_w(t,s)$ on $H^b$, and
$v\in L^2 H$. Adding the base pairing yields
\begin{align*}
 &\langle f,u_T\rangle_{H^{-b},H^b}
 +\langle g/d_T,r_T\rangle_2\\
 &=\int_0^T\left\langle U_w(T,s)^*f+
 \int_s^TU_w(t,s)^*\cS_{\rho_t}^*\beta_t\,\dd t,Qv_s
 \right\rangle\,\dd s\\
 &=\int_0^T\langle Q^*\alpha_s,v_s\rangle_{ H}\,\dd s.
\end{align*}
The projection is invisible because $g\perp \mathsf p_T$, proving the identity $\widehat A_T'(f,g)(t)=Q^*\alpha_t$ in $L^2(0,T; H)$. 

If the transpose vanishes, injectivity of $Q^*$ gives $\alpha=0$ a.e.
and hence everywhere. Its terminal value gives $f=0$, and the
distributional backward equation gives $\cS_\rho^*\beta=0$. The converse
is immediate hence \eqref{eq:kernel} is proved.  The weak forward and adjoint equations imply that the $L^2$
pairing is constant, and at $T$ it equals
$\langle g/d_T,\rho_T\rangle=\langle g,\mathsf p_T\rangle=0$. Hence $\langle\beta_t,\rho_t\rangle_2=0$ for all $t$. 
\end{proof}

\subsection{Almost sure dense range of the joint endpoint}
\label{sec:dense-joint-range}

The vertical condition in \eqref{eq:kernel} can in fact be eliminated.  The
argument is pathwise after choosing  a set for the driving
Brownian motion with full Wiener measure.   We complexify the mean zero spaces and use the bilinear Fourier pairing
$
 \langle f,g\rangle=\sum_{\boldsymbol j\ne\boldsymbol 0}
 f_{\boldsymbol j}g_{-\boldsymbol j}
$.
For every smooth mean zero $h$, set
\begin{align*}
 H_hf:=\mathcal S(h,f).
\end{align*}
We use the same notation for its Sobolev extensions below.  For
$\boldsymbol m\in\mathbb Z^2\setminus\{\boldsymbol 0\}$, put
\begin{align*}
 e_{\boldsymbol m}(x)=e^{\mathrm i\boldsymbol m\cdot x},
 \qquad H_{\boldsymbol m}:=H_{e_{\boldsymbol m}}.
\end{align*}
With an immaterial common choice of Fourier normalisation,
\begin{align}\label{eq:Hm-shift}
 H_{\boldsymbol m}e_{\boldsymbol j}
 =c_{\boldsymbol m}(\boldsymbol j)e_{\boldsymbol j+\boldsymbol m},
 \qquad
 c_{\boldsymbol m}(\boldsymbol j)
 =\det(\boldsymbol m,\boldsymbol j)
 \bigl(|\boldsymbol m|^{-2}-|\boldsymbol j|^{-2}\bigr).
\end{align}
Thus every finite Lie word in the $H_{\boldsymbol m}$'s is a finite sum of weighted
shifts of order at most one.

\subsection{A discrete Fourier symbol class closed under all Lie brackets}
\label{rev:sec-symbol-class}

For a sequence $a:\mathbb Z^2\to\mathbb C$, write
$\Delta_i a(\boldsymbol j)=a(\boldsymbol j+\boldsymbol e_i)
-a(\boldsymbol j)$ and
$\Delta^{\boldsymbol\alpha}=\Delta_1^{\alpha_1}\Delta_2^{\alpha_2}$. For
$\mu\in\mathbb R$, define the set of Fourier symbols of order $\mu$ by
\begin{align}\label{rev:eq-discrete-symbol-class}
 S_\Delta^\mu=\left\{a:
 [a]_{\mu,N}:=
 \max_{|\boldsymbol\alpha|\le N}\sup_{\boldsymbol j\in\mathbb Z^2}
 (1+|\boldsymbol j|)^{-\mu+|\boldsymbol\alpha|}
 |\Delta^{\boldsymbol\alpha} a(\boldsymbol j)|<\infty
 \text{ for every }N\right\}.
\end{align}
The elementary discrete calculus used below is
\begin{align}
 \tau_{\boldsymbol m}:S_\Delta^\mu&\to S_\Delta^\mu,
 &(\tau_{\boldsymbol m}a)(\boldsymbol j)
 &=a(\boldsymbol j+\boldsymbol m),\label{rev:eq-symbol-translation}\\
 S_\Delta^\mu S_\Delta^\eta&\subset S_\Delta^{\mu+\eta},
\label{rev:eq-symbol-product}\\
 \tau_{\boldsymbol m}a-a&\in S_\Delta^{\mu-1}.
\label{rev:eq-symbol-difference}
\end{align}
For fixed $\boldsymbol m$, every output seminorm is bounded by a polynomial in
$(1+|\boldsymbol m|)$ times finitely many input seminorms. Indeed,
\eqref{rev:eq-symbol-translation} follows from
$1+|\boldsymbol j+\boldsymbol m|
\le(1+|\boldsymbol m|)(1+|\boldsymbol j|)$;
\eqref{rev:eq-symbol-product} is the discrete Leibniz rule; and
\eqref{rev:eq-symbol-difference} follows by telescoping along a coordinate
lattice path from $\boldsymbol 0$ to $\boldsymbol m$, followed by the same argument after each
$\Delta^{\boldsymbol\alpha}$.

Let $\mathcal P_1$ be the affine functions of $\boldsymbol j$, and set
\begin{align*}
 \mathfrak A=\mathcal P_1+S_\Delta^{-1}\subset S_\Delta^1.
\end{align*}
For $\boldsymbol m\in\mathbb Z^2$, define the weighted shift
\begin{align*}
 T_{\boldsymbol m,a}e_{\boldsymbol j}
 =a(\boldsymbol j)e_{\boldsymbol j+\boldsymbol m},
\end{align*}
with value zero when input or output is the excluded zero mode.

To iterate the stochastic bracket argument without losing derivatives, we
need a symbol class stable under commutators.  The next lemma gives the exact
closure and identifies the affine principal part.
\begin{lemma}
\label{rev:lem-symbol-closure}
Let $a=A+r$, $b=B+s$ belong to $\mathfrak A$, where
$A,B\in\mathcal P_1$ and $r,s\in S_\Delta^{-1}$. Then
\begin{align}\label{rev:eq-weighted-shift-commutator}
 [T_{\boldsymbol m,a},T_{\boldsymbol n,b}]
 =T_{\boldsymbol m+\boldsymbol n,c},\qquad
 c(\boldsymbol j)=b(\boldsymbol j)a(\boldsymbol j+\boldsymbol n)
 -a(\boldsymbol j)b(\boldsymbol j+\boldsymbol m),
\end{align}
and $c\in\mathfrak A$. Its affine part is exactly
\begin{align}\label{rev:eq-affine-commutator-part}
 c_{\rm aff}(\boldsymbol j)
 =A_{\rm lin}(\boldsymbol n)B(\boldsymbol j)
 -B_{\rm lin}(\boldsymbol m)A(\boldsymbol j),
\end{align}
where $A_{\rm lin},B_{\rm lin}$ are the linear parts.
\end{lemma}

\begin{proof}
The composition formula gives
\eqref{rev:eq-weighted-shift-commutator}. Since
$A(\boldsymbol j+\boldsymbol n)=A(\boldsymbol j)
+A_{\rm lin}(\boldsymbol n)$, and similarly for $B$, the pure affine
contribution is \eqref{rev:eq-affine-commutator-part}; thus the quadratic
terms cancel as an identity, not merely asymptotically. The remainder is
\begin{align*}
 &B(\boldsymbol j)r(\boldsymbol j+\boldsymbol n)
 -r(\boldsymbol j)B(\boldsymbol j+\boldsymbol m)\\
 &\quad+s(\boldsymbol j)A(\boldsymbol j+\boldsymbol n)
 -A(\boldsymbol j)s(\boldsymbol j+\boldsymbol m)\\
 &\quad+s(\boldsymbol j)r(\boldsymbol j+\boldsymbol n)
 -r(\boldsymbol j)s(\boldsymbol j+\boldsymbol m).
\end{align*}
The first line equals
\begin{align*}
 B(\boldsymbol j)[r(\boldsymbol j+\boldsymbol n)-r(\boldsymbol j)]
 -B_{\rm lin}(\boldsymbol m)r(\boldsymbol j)\in S_\Delta^{-1}
\end{align*}
by \eqref{rev:eq-symbol-product}--\eqref{rev:eq-symbol-difference}. The
second line is identical after interchanging the letters. The last line
belongs to $S_\Delta^{-2}$. This proves closure in every seminorm from
\eqref{rev:eq-discrete-symbol-class}.
\end{proof}

For $\boldsymbol m\ne\boldsymbol 0$, the Fourier generator $H_{\boldsymbol m}$ has symbol
\begin{align*}
 c_{\boldsymbol m}(\boldsymbol j)
 =\det(\boldsymbol m,\boldsymbol j)
 (|\boldsymbol m|^{-2}-|\boldsymbol j|^{-2})
 =|\boldsymbol m|^{-2}\det(\boldsymbol m,\boldsymbol j)
 -|\boldsymbol j|^{-2}\det(\boldsymbol m,\boldsymbol j).
\end{align*}
The first term is linear.  Since the zero Fourier mode is absent, extend
the second term to $\boldsymbol j=\boldsymbol0$ by setting it equal to zero. 
Away from the origin its discrete derivatives have the bounds of a
homogeneous symbol of degree $-1$, while the finite modification near the
origin does not affect the symbol order.  Hence
$c_{\boldsymbol m}\in\mathfrak A$.

The commutator closure gives the Sobolev mapping properties needed in the
bracket induction.  The next corollary collects the regularity of the Lie
words and their commutators with a rough parameter.  The Sobolev indices
below are not optimized; the stated bounds are sufficient for the bounded
variation argument.
\begin{corollary}
\label{rev:cor-lie-word-mapping}
Every nested finite Lie word $\mathcal{C}$ in the $H_{\boldsymbol m}$ is one weighted shift
$T_{\boldsymbol\ell,a}$ with $a\in\mathfrak A$; a finite linear combination
of words is a finite sum of such shifts. For every $r\in\mathbb R$,
\begin{align}\label{rev:eq-lie-mapping}
 \mathcal{C}:H^r\to H^{r-1},\qquad
 [\mathcal{C},\Delta]:H^r\to H^{r-2}.
\end{align}
If $s>5$, $h\in H^{s-2}$, and $u\in H^s$ or $H^{s-2}$, then
\begin{align}\label{rev:eq-rough-commutator-estimate}
 \|[\mathcal{C},H_h]u\|_{H^{s-4}}
 \le C_{\mathcal{C}}\|h\|_{H^{s-2}}
 \begin{cases}
  \|u\|_{H^s},&u\in H^s,\\
  \|u\|_{H^{s-2}},&u\in H^{s-2}.
 \end{cases}
\end{align}
The same mapping statements hold for Fourier adjoints.
\end{corollary}

\begin{proof}
The word assertion follows inductively from
\Cref{rev:lem-symbol-closure}. Since $a\in S_\Delta^1$, fixed
translation of Sobolev weights gives the first map in
\eqref{rev:eq-lie-mapping}. Moreover,
\begin{align*}
 [T_{\boldsymbol\ell,a},\Delta]e_{\boldsymbol j}
 =\bigl(|\boldsymbol j+\boldsymbol\ell|^2-|\boldsymbol j|^2\bigr)
 a(\boldsymbol j)e_{\boldsymbol j+\boldsymbol\ell},
\end{align*}
whose coefficient has order at most two.

For \eqref{rev:eq-rough-commutator-estimate}, use the Boit-Savart smoothing $K:H^r\to H^{r+1}$,
the two dimensional multiplication theorem, and $s-3>2$. They give
\begin{align*}
 \|H_hu\|_{H^{s-3}}
 &\le C\|h\|_{H^{s-2}}\|u\|_{H^s},\\
 \|H_hu\|_{H^{s-3}}
 &\le C\|h\|_{H^{s-2}}\|u\|_{H^{s-2}},\\
 \|H_hr\|_{H^{s-4}}
 &\le C\|h\|_{H^{s-2}}\|r\|_{H^{s-3}}.
\end{align*}
Apply the first two estimates before $\mathcal{C}$, and the third after
$\mathcal{C}:H^{s-2}\to H^{s-3}$ or
$\mathcal{C}:H^s\to H^{s-1}$. The adjoint of a fixed
weighted shift has a translated symbol in the same class.
\end{proof}

The bracket induction also needs interior regularity of the forward and
adjoint fibres.  By \Cref{cor:interior-forward-adjoint}, if
$I\Subset(0,T)$, $5<s<b$, and $E=H^{s-2}$, then
\begin{align*}
 \rho,\beta\in C(I;H^s)\cap AC(I;E),\qquad
 \rho',\beta'\in L^1(I;E).
\end{align*}
For a Lie word $\mathcal{C}$, define
\begin{align*}
 \ell_{\mathcal{C}}(t)(h)
 =\langle\beta_t,[\mathcal{C},H_h]\rho_t\rangle,\qquad h\in E.
\end{align*}
Equation \eqref{rev:eq-rough-commutator-estimate}, first with
$\rho\in H^s$ and then with $\rho'\in H^{s-2}$, proves
\begin{align}\label{rev:eq-ell-AC}
 \ell_{\mathcal{C}}\in AC(I;E^*),\qquad
 \ell_{\mathcal{C}}'(t)(h)=
 \langle\beta_t',[\mathcal{C},H_h]\rho_t\rangle+
 \langle\beta_t,[\mathcal{C},H_h]\rho_t'\rangle
\end{align}
in $L^1(I;E^*)$. Likewise
$t\mapsto\langle\beta_t,[\mathcal{C},\Delta]\rho_t\rangle$ is absolutely
continuous: after differentiation the pairings are
$H^{s-2}\times H^{s-2}$ and $H^s\times H^{s-4}$. Hence the
anticipative quadratic variation lemma applies in precisely the dual space
claimed in the dense range proof below.

The bracket induction therefore requires a pathwise quadratic variation
identity that remains valid for anticipative bounded variation test paths.
The next lemma provides this identity on a universal dyadic full measure
set, so that no adaptedness assumption is needed and the exceptional set
does not depend on the later choice of test path.

\begin{lemma}
\label{lem:anticipative-qv}
Let $E$ be a separable Hilbert space, let $Q: H\to E$ be
Hilbert--Schmidt, and let
$Z_t=Z_0+V_t+QW_t$, where $V$ is continuous and of bounded variation
in $E$.  On each $[0,N]$ use the standard dyadic grid and restrict that
grid to subintervals with dyadic endpoints.  There is a probability one set,
depending only on $Q$, $W$, and these countably many grids, with the following
property.  On every such compact subinterval $I=[t_1,t_2]$, for every continuous
bounded variation path
$\ell:I\to E^*$, even when $\ell$ and $V$ are anticipative, and every
real Brownian coordinate $W^r$,
\begin{align}\label{eq:anticipative-qv}
 [\ell(Z),W^r]_t=\int_{t_1}^t\ell_u(Qe_r)\,\mathrm du .
\end{align}
Here the bracket is the limit along the fixed dyadic partitions.  In
particular, if $\ell(Z)$ is of bounded variation, then
$\ell_t(Qe_r)=0$ for every $t\in I$.
The same assertion holds after complexification by taking real and
imaginary parts.
\end{lemma}

\begin{proof}
Fix a dyadic interval $I=[t_1,t_2]\subset[0,N]$ and a Brownian coordinate
$W^r$.  Let $\pi_n(I)=\{s_i\}$ be the restriction to $I$ of the dyadic
partition at level $n$, and write
$\Delta_iW=W_{s_{i+1}}-W_{s_i}$ and
$\Delta_iW^r=W^r_{s_{i+1}}-W^r_{s_i}$.  Set
\begin{align*}
C_n^r(t)=\sum_{s_i<t}(Q\Delta_iW)\Delta_iW^r,
\qquad
c^r(t)=(t-t_1)Qe_r,
\qquad t\in I.
\end{align*}
The centred increments
\begin{align*}
(Q\Delta_iW)\Delta_iW^r-(s_{i+1}-s_i)Qe_r
\end{align*}
are independent mean zero $E$ valued random variables.  Gaussian moment
bounds and the Hilbert space maximal inequality give
\begin{align*}
\mathbf E\sup_{t\in I}\|C_n^r(t)-c^r(t)\|_E^2
\lesssim_{N,Q}2^{-n}.
\end{align*}
Hence, after Borel--Cantelli,
\begin{align*}
\sup_{t\in I}\|C_n^r(t)-c^r(t)\|_E\longrightarrow0
\end{align*}
almost surely.  On the same probability one set,
$\max_i|\Delta_iW^r|\to0$ as dyadic refinement level $n \rightarrow \infty$.   Taking the countable intersection over $N$,
$r$, and all dyadic endpoint intervals $I$ gives a single probability one
set on which these conclusions hold simultaneously.

Fix a path in this set and let $\ell:I\to E^*$ be continuous and of
bounded variation.  For $X_t=\ell_t(Z_t)$, the discrete covariation on
$\pi_n(I)$ satisfies
\begin{align*}
\sum_{s_i<t}\Delta_iX\,\Delta_iW^r
&=\sum_{s_i<t}\ell_{s_i}(Q\Delta_iW)\Delta_iW^r
 +\sum_{s_i<t}\ell_{s_i}(\Delta_iV)\Delta_iW^r\\
&\quad
 +\sum_{s_i<t}(\ell_{s_{i+1}}-\ell_{s_i})(Z_{s_{i+1}})
 \Delta_iW^r .
\end{align*}
Discrete summation by parts and the uniform convergence $C_n^r\to c^r$
yield
\begin{align*}
\sum_{s_i<t}\ell_{s_i}(Q\Delta_iW)\Delta_iW^r
\longrightarrow
\int_{t_1}^t\ell_u(Qe_r)\,\dd u
\end{align*}
uniformly in $t\in I$.  Indeed, the contribution of $C_n^r-c^r$ is
bounded by
\begin{align*}
\bigl(\|\ell\|_{L^\infty(I;E^*)}
+\operatorname{Var}(\ell;I)\bigr)
\|C_n^r-c^r\|_{L^\infty(I;E)},
\end{align*}
while the sums against $c^r$ converge to the displayed integral.

The remaining terms satisfy
\begin{align*}
\left|\sum_{s_i<t}\ell_{s_i}(\Delta_iV)\Delta_iW^r\right|
&\le
\|\ell\|_{L^\infty(I;E^*)}\operatorname{Var}(V;I)
\max_i|\Delta_iW^r|,\\
\left|\sum_{s_i<t}(\ell_{s_{i+1}}-\ell_{s_i})(Z_{s_{i+1}})
\Delta_iW^r\right|
&\le
\operatorname{Var}(\ell;I)\sup_{u\in I}\|Z_u\|_E
\max_i|\Delta_iW^r|,
\end{align*}
and therefore vanish as $n\to\infty$.  This proves
\eqref{eq:anticipative-qv}.  Since the argument is pathwise on the fixed
probability one set, the exceptional set is independent of the later
choice of the bounded variation paths $\ell$ and $V$.

If $\ell(Z)$ is of bounded variation, then its quadratic covariation with
$W^r$ vanishes.  Applying \eqref{eq:anticipative-qv} on every dyadic
endpoint subinterval $J=[a,b]\subset I$ gives
\begin{align*}
\int_a^b\ell_u(Qe_r)\,\dd u=0.
\end{align*}
Since $u\mapsto\ell_u(Qe_r)$ is continuous and dyadic intervals form a
differentiation basis, $\ell_t(Qe_r)=0$ for every $t\in I$.  The
complexified statement follows by applying the real result to the real and
imaginary parts.
\end{proof}

Later we must pass from Lie words to finite Fourier matrices without
treating the operator closure as an algebra.  Here and below, a finite
Fourier matrix means an operator $S$ for which there is a finite set of
real Fourier modes $\Lambda$ such that $S=\Pi_\Lambda S\Pi_\Lambda$.
The next lemma gives the propagation and one sided limiting steps that are
legitimate.
\begin{lemma}
\label{rev:lem-finite-rank-propagation}
Let $F_S(t)=\langle\beta_t,S\rho_t\rangle$ on $I\Subset(0,T)$.
\begin{enumerate}[label=\textup{(\roman*)}]
\item If $S$ is a finite Fourier matrix and $F_S\equiv0$, then
\begin{align}\label{rev:eq-finite-propagation}
 F_{[S,\mathcal{C}]}\equiv0
 \qquad\text{for every finite Lie word }\mathcal{C}.
\end{align}
\item Suppose $R_1,R_2$ are finite Fourier matrices,
$F_{R_1}\equiv0$, and $\mathcal{C}_n$ are Lie words with
\begin{align*}
 \mathcal{C}_n\to R_2\quad\text{in }\mathcal B(H^1,H).
\end{align*}
If $F_{[R_1,\mathcal{C}_n]}\equiv0$ for every $n$, then
\begin{align*}
 F_{[R_1,R_2]}\equiv0.
\end{align*}
\end{enumerate}
Thus no operator algebra property of the
$\mathcal B(H^1,H)$ closure is assumed.
\end{lemma}

\begin{proof}
If $S$ is a finite Fourier matrix, then $[S,H_{\boldsymbol m}]$ is again finite
Fourier rank: $H_{\boldsymbol m}S$ has finite dimensional range, while
$SH_{\boldsymbol m}$ sees
only the finitely many frequencies mapped into the domain support of $S$.
Differentiate $F_S=0$, use \eqref{rev:eq-ell-AC}, and apply \Cref{lem:anticipative-qv} giving 
$F_{[S,H_{\boldsymbol m}]}=0$ for each generator. Iteration is legitimate because all
intervening operators remain finite Fourier matrices; Jacobi's identity
gives \eqref{rev:eq-finite-propagation}.

For \textup{(ii)}, $\rho_t\in H^1$ implies
$\mathcal{C}_n\rho_t\to R_2\rho_t$ in $H$. Since
$R_1:H\to H^r$ is bounded for every $r$,
\begin{align*}
 R_1\mathcal{C}_n\rho_t\to R_1R_2\rho_t\quad\text{in }H.
\end{align*}
Also $R_1\rho_t\in H^1$, whence
\begin{align*}
 \mathcal{C}_nR_1\rho_t\to R_2R_1\rho_t\quad\text{in }H.
\end{align*}
Pairing with $\beta_t\in H$ proves the limit.
\end{proof}

The next algebraic step extracts finite rank matrix units from the closed
Fourier Lie algebra, replacing a finite H\"ormander matrix.  The reversal
symmetry leaves paired units first; a later commutator separates the pair.
Write $E_{\boldsymbol i,\boldsymbol j}e_{\boldsymbol j}=e_{\boldsymbol i}$ and let
$E_{\boldsymbol i,\boldsymbol j}$ vanish on all other Fourier vectors.

\begin{lemma}
\label{lem:paired-units}
Let $\mathfrak L$ be the complex Lie algebra generated by all
$H_{\boldsymbol m}$ as in \eqref{eq:Hm-shift}.  If
$\rho\in H^1_{\mathbb C}$ and $\beta\in L^2_{\mathbb C}$ satisfy
\begin{align}\label{eq:Lie-annihilation}
\langle\beta,\mathcal C\rho\rangle=0
\qquad\text{for every }\mathcal C\in\mathfrak L,
\end{align}
then, whenever $\boldsymbol i,\boldsymbol j\ne\boldsymbol0$,
$\boldsymbol i\ne\boldsymbol j$, and the coefficients
\begin{align*}
a_{\boldsymbol i\boldsymbol j}
&=\det(\boldsymbol i,\boldsymbol j)
\bigl(|\boldsymbol i-\boldsymbol j|^{-2}-|\boldsymbol j|^{-2}\bigr),\\
b_{\boldsymbol i\boldsymbol j}
&=\det(\boldsymbol i,\boldsymbol j)
\bigl(|\boldsymbol i-\boldsymbol j|^{-2}-|\boldsymbol i|^{-2}\bigr)
\end{align*}
are both nonzero, the operator
\begin{align*}
R_{\boldsymbol i,\boldsymbol j}
=a_{\boldsymbol i\boldsymbol j}E_{\boldsymbol i,\boldsymbol j}
-b_{\boldsymbol i\boldsymbol j}E_{-\boldsymbol j,-\boldsymbol i}
\end{align*}
is a limit of elements of $\mathfrak L$ in
$\mathcal B(H^1,L^2)$.  Consequently,
\begin{align}\label{eq:paired-unit-annihilation}
\langle\beta,R_{\boldsymbol i,\boldsymbol j}\rho\rangle=0.
\end{align}
\end{lemma}

\begin{proof}
For $\boldsymbol q\ne\boldsymbol0$, set
$D^{\boldsymbol q}=[H_{\boldsymbol q},H_{-\boldsymbol q}]$.  A direct
calculation using \eqref{eq:Hm-shift} shows that $D^{\boldsymbol q}$ is
diagonal in the complex Fourier basis:
\begin{align}\label{eq:Dq}
D^{\boldsymbol q}e_{\boldsymbol j}
=d_{\boldsymbol j}^{\boldsymbol q}e_{\boldsymbol j},\qquad
d_{\boldsymbol j}^{\boldsymbol q}
=\det(\boldsymbol q,\boldsymbol j)^2
\bigl(|\boldsymbol q|^{-2}-|\boldsymbol j|^{-2}\bigr)
\bigl(|\boldsymbol j-\boldsymbol q|^{-2}
-|\boldsymbol j+\boldsymbol q|^{-2}\bigr),
\end{align}
for $\boldsymbol j\notin\{\boldsymbol0,\pm\boldsymbol q\}$, while the original commutator gives
$d_{\pm\boldsymbol q}^{\boldsymbol q}=0$. Moreover,
\begin{align*}
d_{-\boldsymbol j}^{\boldsymbol q}
=-d_{\boldsymbol j}^{\boldsymbol q},
\qquad
d_{\boldsymbol j}^{\boldsymbol q}
=O_{\boldsymbol q}(|\boldsymbol j|^{-1}).
\end{align*}
Thus $D^{\boldsymbol q}$ is a compact diagonal operator on $L^2$.  

We first construct one diagonal operator whose spectral gaps separate the
Fourier indices relevant to $H_{\boldsymbol m}$.  For fixed
$\boldsymbol m\ne\boldsymbol0$, the relevant quantities are
\begin{align*}
d_{\boldsymbol j+\boldsymbol m}^{\boldsymbol q}
-d_{\boldsymbol j}^{\boldsymbol q},
\qquad
c_{\boldsymbol m}(\boldsymbol j)\ne0.
\end{align*}
We claim that, as functions of $\boldsymbol q$, two such gaps can agree
identically only for the reversal
$\boldsymbol j\mapsto-\boldsymbol j-\boldsymbol m$, and that no active gap
vanishes identically.

To see this, take $\boldsymbol q=N\boldsymbol v$ in \eqref{eq:Dq}.  For
fixed $\boldsymbol x\ne\boldsymbol0$,
\begin{align*}
N d_{\boldsymbol x}^{N\boldsymbol v}
\longrightarrow
-\frac4{|\boldsymbol v|^4}P_{\boldsymbol x}(\boldsymbol v),
\qquad
P_{\boldsymbol x}(\boldsymbol v)
=\frac{\det(\boldsymbol v,\boldsymbol x)^2
(\boldsymbol x\cdot\boldsymbol v)}{|\boldsymbol x|^2}.
\end{align*}
Suppose that two active gaps, corresponding to $\boldsymbol j$ and
$\boldsymbol k$, agree for every lattice $\boldsymbol q$.  Fix
$\boldsymbol v\in\mathbb Z^2\setminus\{\boldsymbol0\}$ and take
$\boldsymbol q=N\boldsymbol v$.  Multiplying the gap identity by $N$ and
letting $N\to\infty$ gives
\begin{align*}
P_{\boldsymbol j+\boldsymbol m}(\boldsymbol v)
-P_{\boldsymbol j}(\boldsymbol v)
=
P_{\boldsymbol k+\boldsymbol m}(\boldsymbol v)
-P_{\boldsymbol k}(\boldsymbol v).
\end{align*}
This holds for every $\boldsymbol v\in\mathbb Z^2$.  Since
$P_{\boldsymbol x}(\boldsymbol v)$ is a homogeneous cubic polynomial in
$\boldsymbol v$, the difference of the two sides is a polynomial vanishing
on the whole integer lattice and hence vanishes identically on
$\mathbb R^2$.

Identify $\mathbb R^2$ with $\mathbb C$ and write
$z=v_1+\mathrm i v_2$.  A direct calculation gives
\begin{align*}
P_{\boldsymbol x}(\boldsymbol v)
=\frac14|\boldsymbol v|^2(\boldsymbol x\cdot\boldsymbol v)
-\frac14\operatorname{Re}\left(
\overline{\frac{\boldsymbol x^3}{|\boldsymbol x|^2}}\,z^3\right).
\end{align*}
Consequently,
\begin{align*}
P_{\boldsymbol j+\boldsymbol m}(\boldsymbol v)
-P_{\boldsymbol j}(\boldsymbol v)
&=\frac14|\boldsymbol v|^2(\boldsymbol m\cdot\boldsymbol v)\\
&\quad-\frac14\operatorname{Re}\left(
\overline{
\frac{(\boldsymbol j+\boldsymbol m)^3}
{|\boldsymbol j+\boldsymbol m|^2}
-\frac{\boldsymbol j^3}{|\boldsymbol j|^2}}
\,z^3\right).
\end{align*}
The first term, which is $|\boldsymbol v|^2$ times a linear polynomial,
depends only on $\boldsymbol m$.  The second is a harmonic cubic.  Since
this decomposition of a homogeneous cubic is unique, equality of two gaps
with the same $\boldsymbol m$ is equivalent to equality of the corresponding
cubic chords
\begin{align*}
\frac{(\boldsymbol j+\boldsymbol m)^3}
{|\boldsymbol j+\boldsymbol m|^2}
-\frac{\boldsymbol j^3}{|\boldsymbol j|^2}.
\end{align*}
The same decomposition also shows that an active gap cannot vanish
identically, since $\boldsymbol m\ne\boldsymbol0$ makes the first term
nonzero.

It remains to determine when two chords agree.  Since the preceding
identity is now an identity on $\mathbb R^2$, we may rotate and scale
coordinates so that $\boldsymbol m=1$.  Put
$t=\boldsymbol j+\tfrac12=u+\mathrm i y$.  The chord becomes
\begin{align*}
F(t)
=\frac{(t+\tfrac12)^3}{|t+\tfrac12|^2}
-\frac{(t-\tfrac12)^3}{|t-\tfrac12|^2}
=\frac{2|t|^2-t^2-\tfrac14}{\bar t^2-\tfrac14}.
\end{align*}
The active condition $c_{\boldsymbol m}(\boldsymbol j)\ne0$ implies
$y\ne0$.  For such $t$,
\begin{align*}
\frac4{F(t)-1}
=\left(\frac uy\right)^2-1-\frac1{4y^2}
-2\mathrm i\frac uy.
\end{align*}
Since $F(t)$ is known, the imaginary part of
\begin{align*}
\frac4{F(t)-1}
=
\left(\frac uy\right)^2-1-\frac1{4y^2}
-2\mathrm i\frac uy
\end{align*}
determines $u/y$.  The real part then determines $y^2$, and hence
$|y|$.  Thus $F(t)$ determines $t=u+\mathrm i y$ up to the simultaneous
change $(u,y)\mapsto(-u,-y)$, namely up to $t\mapsto-t$.  Since
$t=\boldsymbol j+\tfrac12$ in the coordinates where
$\boldsymbol m=1$, the second possibility is $\boldsymbol j\longmapsto-\boldsymbol j-1$. 
Returning to the original coordinates gives the reversal
\begin{align*}
\boldsymbol j\longmapsto-\boldsymbol j-\boldsymbol m.
\end{align*}
Hence two active gaps can coincide only for this reversal pair.

We now choose one diagonal operator for which all these separations hold
simultaneously.  Enumerate one representative of every pair
$\{\boldsymbol q,-\boldsymbol q\}$ by $\boldsymbol q_n$ and set
\begin{align*}
\varpi_n=2^n\bigl(1+\|D^{\boldsymbol q_n}\|_{\mathcal L(L^2)}\bigr).
\end{align*}
On the Banach space
\begin{align*}
X=\left\{\alpha:\sum_{n\ge1}\varpi_n|\alpha_n|<\infty\right\},
\end{align*}
each condition that an active gap vanish, or that two nonreversed active
gaps coincide, is a proper closed hyperplane.  There are only countably
many such conditions.  Baire's theorem therefore gives
$\alpha\in X$ such that
\begin{align*}
D=\sum_{n\ge1}\alpha_nD^{\boldsymbol q_n},
\qquad
De_{\boldsymbol j}=d_{\boldsymbol j}e_{\boldsymbol j},
\end{align*}
has the following property: for every $\boldsymbol m\ne\boldsymbol0$, the corresponding 
active gaps $\lambda_{\boldsymbol m,\boldsymbol j}
=d_{\boldsymbol j+\boldsymbol m}-d_{\boldsymbol j}$
are nonzero and are distinct except for
$\boldsymbol j\leftrightarrow-\boldsymbol j-\boldsymbol m$.

The series defining $D$ converges in operator norm.  Since every
$D^{\boldsymbol q_n}$ is compact, $D$ is compact and therefore
$d_{\boldsymbol j}\to0$ as $|\boldsymbol j|\to\infty$.  If $D_N$ denotes
the corresponding finite partial sum, then $D_N\in\mathfrak L$ and, for
every fixed $n$ and $\boldsymbol m$,
\begin{align*}
\operatorname{ad}_{D_N}^{\,n}H_{\boldsymbol m}
\longrightarrow
\operatorname{ad}_{D}^{\,n}H_{\boldsymbol m}
\qquad\text{in }\mathcal B(H^1,L^2).
\end{align*}
Hence \eqref{eq:Lie-annihilation}, first applied to
$\operatorname{ad}_{D_N}^{\,n}H_{\boldsymbol m}$ and then passed to the
limit, gives
\begin{align}\label{eq:gap-moments}
0=\sum_{\boldsymbol j\ne\boldsymbol0,-\boldsymbol m}
\lambda_{\boldsymbol m,\boldsymbol j}^{\,n}
c_{\boldsymbol m}(\boldsymbol j)\rho_{\boldsymbol j}
\beta_{-(\boldsymbol j+\boldsymbol m)},
\qquad n\ge0.
\end{align}
For fixed
$\boldsymbol m$, define the finite complex atomic measure
\begin{align*}
\nu_{\boldsymbol m}
=\sum_{\boldsymbol j\ne\boldsymbol0,-\boldsymbol m}
c_{\boldsymbol m}(\boldsymbol j)\rho_{\boldsymbol j}
\beta_{-(\boldsymbol j+\boldsymbol m)}
\,\delta_{\lambda_{\boldsymbol m,\boldsymbol j}}.
\end{align*}
It is finite because $H_{\boldsymbol m}:H^1\to L^2$ and
$\beta\in L^2$, so Cauchy--Schwarz gives
\begin{align*}
\sum_{\boldsymbol j}
\left|c_{\boldsymbol m}(\boldsymbol j)\rho_{\boldsymbol j}
\beta_{-(\boldsymbol j+\boldsymbol m)}\right|
\le
\|H_{\boldsymbol m}\rho\|_2\|\beta\|_2.
\end{align*}
Moreover, compactness of $D$ implies
$\lambda_{\boldsymbol m,\boldsymbol j}\to0$ as
$|\boldsymbol j|\to\infty$.  Thus $\nu_{\boldsymbol m}$ is supported on
the compact set consisting of zero and the active gaps.

Equation \eqref{eq:gap-moments} says exactly that
\begin{align*}
\int x^n\,\nu_{\boldsymbol m}(\dd x)=0
\qquad\text{for every }n\ge0.
\end{align*}
Polynomials are dense in the continuous functions on this compact subset
of $\mathbb R$.  Hence
\begin{align*}
\int f\,\dd\nu_{\boldsymbol m}=0
\qquad\text{for every continuous }f,
\end{align*}
and therefore $\nu_{\boldsymbol m}=0$.

Now fix $\boldsymbol i,\boldsymbol j$ as in the statement and put
$\boldsymbol m=\boldsymbol i-\boldsymbol j$.  The gap
$\lambda_{\boldsymbol m,\boldsymbol j}$ is nonzero and hence isolated.
By the separation proved above, the only other index producing the same
gap is $-\boldsymbol j-\boldsymbol m=-\boldsymbol i.$ 
Therefore the mass of $\nu_{\boldsymbol m}$ at this single gap is
\begin{align*}
0=\nu_{\boldsymbol m}
\bigl(\{\lambda_{\boldsymbol m,\boldsymbol j}\}\bigr)=a_{\boldsymbol i\boldsymbol j}\rho_{\boldsymbol j}\beta_{-\boldsymbol i}
-b_{\boldsymbol i\boldsymbol j}\rho_{-\boldsymbol i}\beta_{\boldsymbol j}
=\langle\beta,R_{\boldsymbol i,\boldsymbol j}\rho\rangle.
\end{align*}
This proves \eqref{eq:paired-unit-annihilation}.  The reason that one obtains
a paired matrix unit rather than a single matrix unit is precisely the
unavoidable equality of the two reversed gaps.

It remains to prove the operator closure assertion.  For fixed
$\boldsymbol m=\boldsymbol i-\boldsymbol j$, let
\begin{align*}
K_{\boldsymbol m}
=\{0\}\cup
\{\lambda_{\boldsymbol m,\boldsymbol k}:
c_{\boldsymbol m}(\boldsymbol k)\ne0\}.
\end{align*}
Every nonzero point of $K_{\boldsymbol m}$ is isolated, and the only
repetition is the reversal pair.  Hence there is
$f\in C(K_{\boldsymbol m})$ such that
\begin{align*}
f(\lambda_{\boldsymbol m,\boldsymbol j})=1,
\qquad
f(\lambda)=0
\quad\text{for every other }\lambda\in K_{\boldsymbol m}.
\end{align*}
Choose polynomials $p_n$ converging uniformly to $f$ on
$K_{\boldsymbol m}$.  Since
\begin{align*}
p_n(\operatorname{ad}_D)H_{\boldsymbol m}e_{\boldsymbol k}
=p_n(\lambda_{\boldsymbol m,\boldsymbol k})
c_{\boldsymbol m}(\boldsymbol k)e_{\boldsymbol k+\boldsymbol m},
\end{align*}
the weighted shift estimate gives
\begin{align*}
\|p_n(\operatorname{ad}_D)H_{\boldsymbol m}
-R_{\boldsymbol i,\boldsymbol j}\|_{\mathcal B(H^1,L^2)}
\le
C_{\boldsymbol m}\|p_n-f\|_{C(K_{\boldsymbol m})}
\longrightarrow0.
\end{align*}
For each fixed polynomial $p_n$, replacing $D$ by $D_N$ gives
\begin{align*}
p_n(\operatorname{ad}_{D_N})H_{\boldsymbol m}
\longrightarrow
p_n(\operatorname{ad}_D)H_{\boldsymbol m}
\qquad\text{in }\mathcal B(H^1,L^2),
\end{align*}
while
$p_n(\operatorname{ad}_{D_N})H_{\boldsymbol m}\in\mathfrak L$ because
$D_N,H_{\boldsymbol m}\in\mathfrak L$.  A diagonal choice of $n$ and $N$
therefore yields
\begin{align*}
R_{\boldsymbol i,\boldsymbol j}
\in\overline{\mathfrak L}^{\,\mathcal B(H^1,L^2)}.
\end{align*}
The proof is complete. 
\end{proof}

The preceding stochastic and algebraic inputs now eliminate the vertical
annihilator in \eqref{eq:kernel}.  We obtain one full measure event on which
the joint Malliavin endpoint has dense range simultaneously for all initial
data and all positive times.
\begin{theorem}
\label{thm:dense-joint-range}
Under \eqref{eq:q}--\eqref{eq:b}, there is one probability one set of
Brownian paths such that, simultaneously for every $T>0$, every initial
base state in $H^b$, and every nonzero initial fibre in $H$,
\begin{align}\label{eq:dense-joint-range}
 \overline{\operatorname{Ran}\widehat A_T}
   =H^b\oplus \mathsf p_T^\perp.
\end{align}
Equivalently, the vertical set in \eqref{eq:kernel} is trivial on that
same probability one set.
\end{theorem}

\begin{proof}
Fix once and for all the probability one event supplied by
\Cref{lem:anticipative-qv}, intersected with the event on which
$QW\in C([0,\infty);H^b)$.  This event is chosen before $T$, the initial
data, and the terminal annihilator are chosen.  Fix these objects
arbitrarily; all subsequent arguments are deterministic on this event.

Let $(f,g)\in\Ker\widehat A_T'$, and let $\rho,\beta$ be the forward and
backward fibre solutions in \Cref{prop:annihilator}.  Since the higher
Sobolev regularity needed below is available only away from the initial and
terminal times, fix an interval $I\Subset(0,T)$ whose endpoints belong to
the global dyadic grid.  Such intervals form a countable cover of $(0,T)$.
Fix $5<s<b$ and set $E=H^{s-2}$.  By
\Cref{cor:interior-forward-adjoint},
\begin{align*}
\rho,\beta\in C(I;H^s)\cap AC(I;E),\qquad
\rho',\beta'\in L^1(I;E).
\end{align*}

Write the base equation as
\begin{align*}
w_t=w_0+V_t+QW_t,\qquad
V_t=\int_0^t\bigl(\nu\Delta w_r-B(w_r,w_r)\bigr)\,\dd r.
\end{align*}
Since $w\in C([0,T];H^b)$, the product estimate and $s<b$ give
\begin{align*}
\nu\Delta w-B(w,w)
\in C(I;H^{b-2})+C(I;H^{b-1})
\hookrightarrow L^1(I;E).
\end{align*}
Hence $V\in AC(I;E)$.  Thus \Cref{lem:anticipative-qv} applies to the
decomposition of $w$ on $I$, with $Q$ viewed as an operator from $H$ to
$E$.

For a finite Lie word $\mathcal C$, set
$F_{\mathcal C}(t)=\langle\beta_t,\mathcal C\rho_t\rangle$.  The kernel
identity \eqref{eq:kernel} and symmetry of $\mathcal S$ give
\begin{align*}
F_{H_{\boldsymbol m}}(t)=0
\qquad\text{for almost every }t\in I
\end{align*}
for every real sine or cosine generator.  Since
$H_{\boldsymbol m}:H^s\to H^{s-1}$ and $\rho,\beta\in C(I;H^s)$, these
functions are continuous and therefore vanish everywhere.  After
complexification,
\begin{align}\label{eq:F-Hzero}
F_{H_{\boldsymbol m}}(t)=0
\qquad(t\in I,\ \boldsymbol m\ne\boldsymbol0).
\end{align}

Suppose now that $F_{\mathcal C}\equiv0$.  Differentiating and using the
forward and backward equations gives in $L^1(I)$
\begin{align}\label{eq:FC-derivative}
0=F_{\mathcal C}'
=\nu\langle\beta,[\mathcal C,\Delta]\rho\rangle+\ell_t(w_t),
\qquad
\ell_t(h)=\langle\beta_t,[\mathcal C,H_h]\rho_t\rangle.
\end{align}
By \eqref{rev:eq-rough-commutator-estimate},
$\ell\in AC(I;E^*)$.  Indeed, its derivative is the sum of the two
pairings
\begin{align*}
\langle\beta_t',[\mathcal C,H_h]\rho_t\rangle,\qquad
\langle\beta_t,[\mathcal C,H_h]\rho_t'\rangle,
\end{align*}
which are bounded in $E^*$ by
\begin{align*}
C_{\mathcal C}
\bigl(
\|\beta_t'\|_{H^{s-2}}\|\rho_t\|_{H^s}
+\|\beta_t\|_{H^s}\|\rho_t'\|_{H^{s-2}}
\bigr).
\end{align*}
Likewise
$G_{\mathcal C}(t)=\langle\beta_t,[\mathcal C,\Delta]\rho_t\rangle$
is absolutely continuous, since after differentiation the two pairings are
between $H^{s-2}$ and $H^{s-2}$, and between $H^s$ and $H^{s-4}$.
Thus \eqref{eq:FC-derivative} makes $\ell(w)=-\nu G_{\mathcal C}$ a
bounded variation path.

Applying \Cref{lem:anticipative-qv} to each real noise coordinate gives
\begin{align*}
\langle\beta_t,[\mathcal C,H_{\boldsymbol m}]\rho_t\rangle=0
\qquad(t\in I,\ \boldsymbol m\ne\boldsymbol0),
\end{align*}
because every Fourier coefficient $q_{\boldsymbol m}^\iota$ is nonzero.
Starting from \eqref{eq:F-Hzero} and iterating, we obtain
\begin{align*}
\langle\beta_t,\mathcal C\rho_t\rangle=0
\qquad\text{for every }\mathcal C\in\mathfrak L,\quad t\in I.
\end{align*}

Apply \Cref{lem:paired-units} at each $t\in I$.  Every active paired
operator therefore satisfies
\begin{align*}
F_{R_{\boldsymbol i,\boldsymbol j}}(t)=0,\qquad t\in I.
\end{align*}
The next step separates the two matrix units in each pair.  This must be
done without treating the $\mathcal B(H^1,L^2)$ closure of
$\mathfrak L$ as an algebra.

Let $R_1,R_2$ be two active paired operators.  By
\Cref{lem:paired-units}, choose $\mathcal C_n\in\mathfrak L$ with
$\mathcal C_n\to R_2$ in $\mathcal B(H^1,L^2)$.  Since
$F_{R_1}\equiv0$, \Cref{rev:lem-finite-rank-propagation}\textup{(i)}
gives $F_{[R_1,\mathcal C_n]}\equiv0$, and
\Cref{rev:lem-finite-rank-propagation}\textup{(ii)} gives
\begin{align}\label{eq:paired-commutator-zero}
F_{[R_1,R_2]}(t)=0,\qquad t\in I.
\end{align}

Fix now $\boldsymbol i,\boldsymbol j\ne\boldsymbol0$ with
$\boldsymbol i\ne\boldsymbol j$.  Call an intermediate lattice point
$\boldsymbol k$ admissible if all coefficients in
$R_{\boldsymbol i,\boldsymbol k}$ and
$R_{\boldsymbol k,\boldsymbol j}$ are nonzero and no unintended Fourier
indices coincide.  The conditions that a coefficient vanish have a simple
geometric form.  For example,
\begin{align*}
a_{\boldsymbol i\boldsymbol k}=0
&\quad\Longleftrightarrow\quad
\det(\boldsymbol i,\boldsymbol k)=0
\quad\text{or}\quad
|\boldsymbol i-\boldsymbol k|=|\boldsymbol k|,\\
b_{\boldsymbol i\boldsymbol k}=0
&\quad\Longleftrightarrow\quad
\det(\boldsymbol i,\boldsymbol k)=0
\quad\text{or}\quad
|\boldsymbol i-\boldsymbol k|=|\boldsymbol i|,
\end{align*}
and the analogous conditions hold for
$R_{\boldsymbol k,\boldsymbol j}$.  Thus the excluded points lie on
finitely many lines and circles, together with finitely many points coming
from index collisions.  Choose an integer direction outside the finitely
many excluded directions.  Along a sufficiently large integer multiple of
that direction only finitely many values can meet the remaining affine
lines or circles.  Hence there are arbitrarily large admissible
$\boldsymbol k\in\mathbb Z^2$.

For such $\boldsymbol k$, the absence of index collisions and the identity
$E_{\boldsymbol a,\boldsymbol b}E_{\boldsymbol c,\boldsymbol d}
=\delta_{\boldsymbol b,\boldsymbol c}E_{\boldsymbol a,\boldsymbol d}$
give
\begin{align*}
[R_{\boldsymbol i,\boldsymbol k},R_{\boldsymbol k,\boldsymbol j}]
=A(\boldsymbol k)E_{\boldsymbol i,\boldsymbol j}
-B(\boldsymbol k)E_{-\boldsymbol j,-\boldsymbol i},
\end{align*}
where $A(\boldsymbol k)
=a_{\boldsymbol i\boldsymbol k}a_{\boldsymbol k\boldsymbol j},\, 
B(\boldsymbol k)
=b_{\boldsymbol i\boldsymbol k}b_{\boldsymbol k\boldsymbol j}$. 
In particular,
\begin{align}\label{eq:AB-ratio}
\frac{A(\boldsymbol k)}{B(\boldsymbol k)}
=\frac{|\boldsymbol i|^2}{|\boldsymbol j|^2}
\frac{(|\boldsymbol i|^2-2\boldsymbol k\cdot\boldsymbol i)
(|\boldsymbol k|^2-2\boldsymbol k\cdot\boldsymbol j)}
{(|\boldsymbol k|^2-2\boldsymbol k\cdot\boldsymbol i)
(|\boldsymbol j|^2-2\boldsymbol k\cdot\boldsymbol j)}.
\end{align}
By \eqref{eq:paired-commutator-zero}, every admissible
$\boldsymbol k$ therefore gives the scalar relation
\begin{align}\label{eq:two-unit-relation}
A(\boldsymbol k)X-B(\boldsymbol k)Y=0,
\end{align}
where $X=\langle\beta_t,E_{\boldsymbol i,\boldsymbol j}\rho_t\rangle,\,
Y=\langle\beta_t,E_{-\boldsymbol j,-\boldsymbol i}\rho_t\rangle$. 
Thus it suffices to find two admissible points
$\boldsymbol k_1,\boldsymbol k_2$ for which
\begin{align*}
\frac{A(\boldsymbol k_1)}{B(\boldsymbol k_1)}
\ne
\frac{A(\boldsymbol k_2)}{B(\boldsymbol k_2)}.
\end{align*}
Indeed, the two relations \eqref{eq:two-unit-relation} then form a
nonsingular $2\times2$ linear system for $(X,Y)$.

Suppose first that $\boldsymbol i$ and $\boldsymbol j$ are not collinear.
Along $\boldsymbol k=t\boldsymbol z$, where
$\boldsymbol z\cdot\boldsymbol i$ and
$\boldsymbol z\cdot\boldsymbol j$ are nonzero, \eqref{eq:AB-ratio} gives
\begin{align*}
\frac{A(t\boldsymbol z)}{B(t\boldsymbol z)}
\longrightarrow
\frac{|\boldsymbol i|^2}{|\boldsymbol j|^2}
\frac{\boldsymbol z\cdot\boldsymbol i}
{\boldsymbol z\cdot\boldsymbol j}.
\end{align*}
Since $\boldsymbol i$ and $\boldsymbol j$ are not collinear, the last
ratio is not constant as a function of $\boldsymbol z$.  Choose two
integer directions with distinct values and then take sufficiently large
admissible multiples along these directions.  Their ratios
$A(\boldsymbol k)/B(\boldsymbol k)$ are distinct, and hence
$X=Y=0$.

Suppose next that $\boldsymbol j=c\boldsymbol i$.  Since
$\boldsymbol i\ne\boldsymbol j$, we have $c\ne1$.  Along an admissible
sequence perpendicular to $\boldsymbol i$ and along a generic admissible
sequence, \eqref{eq:AB-ratio} has distinct limits $c^{-4}, \, c^{-3}$.
If $c\ne-1$, these limits are distinct, so two admissible choices give a
nonsingular system for $(X,Y)$ and hence $X=Y=0$.  If $c=-1$, then
$\boldsymbol j=-\boldsymbol i$ and
$E_{-\boldsymbol j,-\boldsymbol i}=E_{\boldsymbol i,\boldsymbol j}$, so
$X=Y$.  In this case \eqref{eq:two-unit-relation} becomes
\begin{align*}
\bigl(A(\boldsymbol k)-B(\boldsymbol k)\bigr)X=0.
\end{align*}
Along a generic admissible sequence,
$A(\boldsymbol k)/B(\boldsymbol k)\to-1$, so for some admissible
$\boldsymbol k$ one has $A(\boldsymbol k)\ne B(\boldsymbol k)$ and
therefore $X=0$.

We have therefore proved
\begin{align}\label{eq:matrix-unit-annihilation}
\langle\beta_t,E_{\boldsymbol i,\boldsymbol j}\rho_t\rangle=0
\qquad
(\boldsymbol i,\boldsymbol j\ne\boldsymbol0,\ 
\boldsymbol i\ne\boldsymbol j,\ t\in I).
\end{align}

By backward uniqueness, $\rho_t\ne0$ for every $t\in I$.  Fix such a
$t$ and choose $\boldsymbol j_0$ with
$(\rho_t)_{\boldsymbol j_0}\ne0$.  Taking
$\boldsymbol j=\boldsymbol j_0$ in
\eqref{eq:matrix-unit-annihilation} shows that every Fourier coordinate of
$\beta_t$ except possibly the one paired with
$\boldsymbol j_0$ vanishes.  The remaining coordinate vanishes because
$\langle\beta_t,\rho_t\rangle=0$.  Thus $\beta_t=0$ on $I$.

Since the dyadic interior intervals cover $(0,T)$, we have
$\beta_t=0$ for every $t\in(0,T)$.  Continuity as $t\uparrow T$ gives
$g=d_T\beta_T=0$, and \Cref{prop:annihilator} then gives $f=0$.
Therefore
\begin{align*}
\Ker\widehat A_T'=\{0\},
\end{align*}
which is equivalent to \eqref{eq:dense-joint-range}.  The probability one
event was fixed before $T$, the initial state, and the terminal covector
were chosen.  Hence the conclusion holds simultaneously for every
$T>0$ and every initial state.
\end{proof}

\section{A random local stable ball}
\label{sec:random-stable-ball}
This section develops the compact dense compensation mechanism without
requiring a quantitative inverse Malliavin estimate.  We choose one
finite rank selector of fixed accuracy, whose norm may depend strongly on
that accuracy, and repeat it on the sampled blocks.  A logarithmic
contraction estimate gives a random local stable ball, while blockwise Ramer
transformations turn the compensating shifts into an absolutely continuous
generalized coupling on the infinite path space.

\subsection{The block endpoint and a Hilbert driver coordinate}

Fix $p_{\rm occ}\in(1,2]$, and let $t_0=t_*>0$ be the reset time in
\Cref{lem:robust-HPSRY-hierarchy}.  All choices used in this section are
made before constructing the selector.  For every $T\ge t_*$, the global
two jet estimate \Cref{lem:projective-log-moment} gives exponents $d_r,R_r$
independent of the block length.  We use it at the two orders $1$ and
$p_{\rm occ}^2$: the first gives the logarithmic moment
\eqref{eq:uniform-logJ}, while the second supplies the augmented Lyapunov
function and, by Jensen's inequality, the required $p_{\rm occ}$ moments.

Apply \Cref{cor:revision-common-block} simultaneously to these two orders,
enlarging its polynomial member once so that it dominates the associated
base weights, and choose the single block
\begin{align}\label{eq:revision-single-block-choice}
\tau=N_0t_0,\qquad N_0\ge N_*.
\end{align}
Here $N_*$ is any common threshold furnished by that corollary for these
two orders.  This fixes the base Lyapunov weight $V_*$, the common moment
exponent $M_*$, the joint Lyapunov function $W_0$, and all constants
attached to the sampled kernel $P_\tau$. After $\tau$ has been fixed, we construct the endpoint map and the finite
patch selector for this block.  The later choice of selector accuracy $q$
changes only finite constants and the additive constant in the $a_q$
moment bound; it changes neither the two moment orders nor the block length
$\tau$.

With the projective metric \eqref{eq:projective-metric}, set
\begin{align*}
 d_{\mathsf X}\bigl((w,[\mathsf p]),(w',[\mathsf p'])\bigr)
 =\|w-w'\|_{H^b}+d_{\PP}([\mathsf p],[\mathsf p']).
\end{align*}
This is a complete separable metric generating the product Borel
$\sigma$-field.  On every fixed projective chart it is locally
bi-Lipschitz equivalent to the norm on
$H^b\oplus \mathsf p^\perp$;
see \eqref{eq:state-chart-metric-comparison} below.  The sampled path distance used in the generalized
coupling criterion is $d_{\mathsf X}$ (or its harmless bounded version
$1\wedge d_{\mathsf X}$).
The projective tangent space carries its quotient Hilbert metric.  For
$z=(w,[\mathsf p])$ with $\|\mathsf p\|_2=1$, write
\begin{align*}
 E_z=H^b\oplus\mathsf p^\perp.
\end{align*}
This is the Hilbert model used for the tangent space $T_z\mathsf X$; it is
independent of the choice between $\mathsf p$ and $-\mathsf p$.  The
differential of the sphere quotient identifies the projective tangent space
isometrically with $\mathsf p^\perp$.  To compare tangent operators at
different projective points, fix an orthonormal basis
$(f_j)_{j\ge1}$ of $H$.  On
\begin{align*}
 \mathcal V_j=\{[\mathsf p]\in\PP(H):\langle \mathsf p,f_j\rangle\ne0\}
\end{align*}
let $s_j([\mathsf p])$ be the unique unit representative satisfying
$c=\langle s_j([\mathsf p]),f_j\rangle>0$.  With $\mathsf p=s_j([\mathsf p])$ and
$\mathbb K_{\mathsf p}=\mathsf p\otimes f_j-f_j\otimes \mathsf p$, define
\begin{align}\label{eq:chart-rotation}
 \mathcal R_j(\mathsf p)=I+\mathbb K_{\mathsf p}+\frac{\mathbb K_{\mathsf p}^2}{1+c}.
\end{align}
This is an orthogonal operator, depends smoothly on $[\mathsf p]\in\mathcal V_j$,
sends $f_j$ to $\mathsf p$, and maps $f_j^\perp$ onto $\mathsf p^\perp$.  Hence
\begin{align*}
 \mathfrak T_j(w,[\mathsf p]):H^b\oplus f_j^\perp\longrightarrow
 E_{(w,[\mathsf p])},\qquad
 \mathfrak T_j(w,[\mathsf p])(h,\eta)=(h,\mathcal R_j(\mathsf p)\eta),
\end{align*}
is a smooth fibrewise isometric trivialisation.  On overlaps the pointwise transition maps are
orthogonal, so intrinsic operator norms are unchanged.  Their first two
derivatives are not asserted to be orthogonal; they are bounded on the
finitely many subcharts used below, where $|\langle \mathsf p,f_j\rangle|$ is
uniformly separated from zero, and are absorbed into the corresponding
finite chart constants.  
The maps $\mathfrak T_j$ are tangent bundle trivialisations, not nonlinear
state coordinates.  All nonlinear state differences, Taylor expansions, and
feedback displacements below use one affine projective coordinate.  For
$z=(w,[\mathsf p])$ and $z'=(w',[q])$, let $\mathscr U_\Delta$ denote
the common affine chart domain for the pair:
\begin{align*}
 \mathscr U_\Delta
 =\left\{(z,z')\in\mathsf X^2:
 d_{\PP}([\mathsf p],[q])<\sqrt2\right\}.
\end{align*}
If $(z,z')\in\mathscr U_\Delta$, choose an $L^2$ unit representative $\mathsf p$
of $[\mathsf p]$ and any nonzero representative $q$ of $[q]$, and define
\begin{align}\label{eq:state-difference-chart}
 \chi_{\mathsf p}([q])
 &=\frac{\Pi_{\mathsf p^\perp}q}{\langle q,\mathsf p\rangle},\qquad
 \Delta_z(z')
 =\left(w'-w,\chi_{\mathsf p}([q])\right)\in E_z.
\end{align}
The quotient in \eqref{eq:state-difference-chart} is independent of the
nonzero representative of $[q]$.  Replacing $\mathsf p$ by $-\mathsf p$ changes the second
component by $-1$, exactly as does the natural identification of
$T_{[\mathsf p]}\PP(H)$ with $\mathsf p^\perp$.  Thus $\Delta_z(z')$ is a well defined
tangent bundle element.  Its inverse coordinate is
\begin{align*}
 \Delta_z^{-1}(h,\eta)
 =\left(w+h,[\mathsf p+\eta]\right),
 \qquad (h,\eta)\in H^b\oplus \mathsf p^\perp.
\end{align*}
In the natural quotient metric tangent identifications,
\begin{align*}
 D_{z'}\Delta_z(z)=\Id_{E_z},
 \qquad D\Delta_z^{-1}(0)=\Id_{E_z}.
\end{align*}
Moreover, for $\eta\in \mathsf p^\perp$,
\begin{align*}
 d_{\PP}([\mathsf p],[\mathsf p+\eta])^2
 =2-\frac{2}{\sqrt{1+\|\eta\|_2^2}}.
\end{align*}
Consequently, on $\|\Delta_z(z')\|\le1$, there are universal constants
$c,C>0$ such that
\begin{align}\label{eq:state-chart-metric-comparison}
 c\|\Delta_z(z')\|_{E_z}
 \le d_{\mathsf X}(z,z')
 \le C\|\Delta_z(z')\|_{E_z}.
\end{align}
The domain $\mathscr U_\Delta$ is open, and the coordinate map is smooth
there in the orthogonal frames above.  When a total map is needed for a
measurability statement, we extend $\Delta_z(z')$ by zero off
$\mathscr U_\Delta$; this extension is Borel, while every stopping rule below
still tests membership in $\mathscr U_\Delta$ before using the coordinate.

We next introduce the Hilbert driver. Recall that $W$ is the cylindrical Wiener process and 
$\mathcal H=L^2(0,\tau; H)$ is the Cameron--Martin  space.
Fix once and for all $\frac14<\sigma<\frac12$. 
For $g\in\mathcal H$, write
\begin{align*}
 (\mathcal Z_Qg)(t)
 =\int_0^t e^{\nu(t-s)\Delta}Qg_s\,\dd s.
\end{align*}
Since $\sigma<1/2$, the path space $H^\sigma(0,\tau;H^b)$ has no continuous endpoint trace.  We therefore retain the terminal value as a separate coordinate and set
\begin{align}\label{eq:driver-Hilbert}
 \mathsf Y=H^\sigma(0,\tau;H^b)\oplus H^b,
 \qquad
 \Gamma g=(\mathcal Z_Qg,\mathcal Z_Qg(\tau)).
\end{align}
Thus $\Gamma:\mathcal H\to\mathsf Y$ is the linear map that sends a
Cameron--Martin control to the corresponding displacement of the Hilbert
driver.  Let
\begin{align*}
 Z_t(W)=\int_0^t e^{\nu(t-s)\Delta}Q\,\dd W_s,
 \qquad Y(W)=(Z(W),Z_\tau(W)),
\end{align*}
where $Y$ is the one block Gaussian driver.
Standard stochastic convolution estimates give, for every $\sigma<1/2$,
\begin{align*}
\mathbf E\|Z\|_{H^\sigma(0,\tau;H^b)}^2+\mathbf E\|Z_\tau\|_{H^b}^2
\le C_{\sigma,\tau,\nu}\|Q\|_{\HS( H,H^b)}^2<\infty.
\end{align*}
Hence $Y=(Z,Z_\tau)$ is a centred Radon Gaussian variable in $\mathsf Y$.
Denote its law on $\mathsf Y$ by $\mathbf P_Y$.

Fix $\sigma<\vartheta<1/2$ and let $c_0^\vartheta([0,\tau];H^b)$ be the separable little H\"older space. Standard stochastic convolution estimates and Kolmogorov's criterion give a $c_0^\vartheta([0,\tau];H^b)$ valued Gaussian version of $Z$. Hence its law on this separable Banach space is Radon Gaussian.

Define the canonical embedding of compatible H\"older paths into the
Hilbert driver space by
\begin{align*}
\iota_\vartheta:c_0^\vartheta([0,\tau];H^b)\longrightarrow\mathsf Y,\qquad \iota_\vartheta\zeta=(\zeta,\zeta(\tau)).
\end{align*}
Since $\vartheta>\sigma$, $\iota_\vartheta$ is continuous, and
\begin{align*}
Y(W)=\iota_\vartheta Z(W)\quad\text{almost surely}.
\end{align*}
Thus the compact driver cores used below may be chosen in this compatible subspace.

If $y=(y^\circ,y^\tau)\in\mathsf Y$, subtract
the stochastic convolution and solve
\begin{align*}
 v_t=e^{\nu t\Delta}w
 -\int_0^t e^{\nu(t-s)\Delta}
 B(v_s+y_s^\circ,v_s+y_s^\circ)\,\dd s,
 \qquad w_\tau=v_\tau+y^\tau.
\end{align*}
The raw fibre equation is driven by the coefficient
$w_s=v_s+y_s^\circ$.  By \Cref{lem:revision-base-two-jet}, projective normalisation of the locally
$C^2$ raw endpoint defines the one block endpoint map
\begin{align*}
 \Phi(z,y),\qquad z\in\mathsf X,\quad y\in\mathsf Y,
\end{align*}
whenever the raw fibre endpoint is nonzero. Backward uniqueness guarantees this nonvanishing for every compatible driver $y=\iota_\vartheta\zeta$, including $y=Y(W)$ almost surely. At $y=Y(W)$, $\Phi$ agrees with the genuine stochastic endpoint. 

Its derivative with respect to the initial state is denoted by
\begin{align*}
 J(z,y)=D_z\Phi(z,y):E_z\longrightarrow E_{\Phi(z,y)}.
\end{align*}
At a random driver $y=Y(W)$, the Malliavin derivative of the same one block
endpoint is represented by
\begin{align*}
 A(z,Y)=D_y\Phi(z,Y)\Gamma:
 \mathcal H\longrightarrow E_{\Phi(z,Y)}.
\end{align*}
We use the same notation $A(z,y)=D_y\Phi(z,y)\Gamma$ for compatible
deterministic drivers.
For $Y=Y(W)$ and $g\in\mathcal H$, set
$W_t^\varepsilon=W_t+\varepsilon\int_0^t g_s\,\dd s$.  Then
$Y(W^\varepsilon)=Y(W)+\varepsilon\Gamma g$.  Differentiating the base
and normalized fibre endpoints at $\varepsilon=0$ gives
\eqref{rev:eq-base-control-response}--\eqref{rev:eq-cross-control-response};
under the natural identification
$E_{\Phi(z,Y)}=H^b\oplus \mathsf p_\tau^\perp$,
\begin{align}\label{eq:augmented-chain-identification}
 A(z,Y)g=D_y\Phi(z,Y)\Gamma g=\widehat A_\tau g.
\end{align}
The same identity holds at every compatible deterministic driver, since it is the Fr\'echet derivative in the direction $\Gamma g$.

The compact core selector will use smooth Hilbert cutoffs in the driver
variable.  For the later Ramer argument we therefore need the linear map from the
Cameron--Martin space into the Hilbert driver to be Hilbert--Schmidt.
\begin{lemma}
\label{lem:driver-HS}
The map $\Gamma:\mathcal H\to\mathsf Y$ in
\eqref{eq:driver-Hilbert} is Hilbert--Schmidt and
\begin{align*}
 \|\Gamma\|_{\HS}^2
 \le C_{\sigma,\tau,\nu}
       \|Q\|_{\HS( H,H^b)}^2<\infty.
\end{align*}
\end{lemma}

The selector must work on one driver event simultaneously for every initial
state, rather than on state dependent full measure events.  The next lemma
constructs such a good driver set, denoted by $\mathsf Y_*$, and transfers
the dense range theorem to the convolution driver space with exactly this
order of quantifiers.

\begin{lemma}
\label{lem:common-endpoint-event}
There is a Borel set
$\mathsf Y_*\subset\mathsf G_\vartheta$, where
\begin{align*}
 \mathsf G_\vartheta
 =\{(\zeta,\zeta(\tau)):
       \zeta\in c_0^\vartheta([0,\tau];H^b)\}\subset\mathsf Y,
\end{align*}
such that $\mathbf P_Y(\mathsf Y_*)=1$ and the following statements hold
simultaneously for every $Y\in\mathsf Y_*$ and every $z\in\mathsf X$:
the endpoint $\Phi(z,Y)$ is globally defined, and
\begin{align}\label{eq:DR}
 \overline{\Ran A(z,Y)}=E_{\Phi(z,Y)}.
\end{align}
If $z\in H^b\times\mathcal V_j$ and
$\Phi(z,Y)\in H^b\times\mathcal V_k$, the fixed space representatives
\begin{align*}
 J_{kj}(z,Y)&=\mathfrak T_k(\Phi(z,Y))^{-1}
               J(z,Y)\mathfrak T_j(z),\\
 A_k(z,Y)&=\mathfrak T_k(\Phi(z,Y))^{-1}A(z,Y)
\end{align*}
are continuous in operator norm on their common domain.
\end{lemma}

\begin{proof}
Since $q_r\ne0$ for every $r$, the convolution path determines all
Wiener coordinates. Enumerate the real Fourier phases by $r\ge1$ and
write $Qe_r=q_re_r$ and $-\Delta e_r=\lambda_re_r$. Set
\begin{align*}
\mathscr W_\tau=\prod_{r\ge1}C_0([0,\tau];\mathbb R),\qquad
\mathscr C_\tau^0=\{\zeta\in c_0^\vartheta([0,\tau];H^b):\zeta_0=0\}.
\end{align*}
For $\zeta\in\mathscr C_\tau^0$, define
\begin{align}\label{eq:driver-coordinate-reconstruction}
(\mathcal R_\tau\zeta)^r_t
=q_r^{-1}\left(\langle\zeta_t,e_r\rangle
+\nu\lambda_r\int_0^t\langle\zeta_s,e_r\rangle\,\dd s\right).
\end{align}
Then $\mathcal R_\tau:\mathscr C_\tau^0\to\mathscr W_\tau$ is continuous.
If $Z$ is the stochastic convolution, its $r$-th coordinate satisfies
\begin{align*}
\langle Z_t,e_r\rangle
=-\nu\lambda_r\int_0^t\langle Z_s,e_r\rangle\,\dd s+q_rW_t^r,
\end{align*}
and hence $\mathcal R_\tau Z=(W^r)_{r\ge1}$ almost surely. Thus, on the compatible driver space, $\mathcal R_\tau$
is precisely the inverse reconstruction of the Wiener coordinate path
from the convolution path.

The common full measure event from \Cref{thm:dense-joint-range} is
transferred to the driver space as follows. For
$\omega\in\mathscr W_\tau$, set
\begin{align*}
S_N(\omega)_t=\sum_{r\le N}q_r\omega_t^re_r.
\end{align*}
Let $\mathscr B_\tau\subset\mathscr W_\tau$ be the set on which
$S_N(\omega)$ converges in $C([0,\tau];H^b)$, with limit $S(\omega)$,
and on which, for every $r$ and every nondegenerate interval
$I=[a,b]\subset[0,\tau]$ with dyadic endpoints,
\begin{align}\label{eq:explicit-driver-qv-event}
\sup_{t\in I}\left\|
\sum_{\substack{[u,v]\in\pi_n(I)\\ v\le t}}
\bigl(S(\omega)_v-S(\omega)_u\bigr)
\bigl(\omega_v^r-\omega_u^r\bigr)
-(t-a)q_re_r
\right\|_{H^b}\longrightarrow0.
\end{align}
This is a Borel set: convergence of $S_N$ is a countable Cauchy
condition in the separable space $C([0,\tau];H^b)$, and
\eqref{eq:explicit-driver-qv-event} is a countable family of uniform
limit conditions. Thus by  \Cref{lem:anticipative-qv},
\begin{align*}
\mathbf P\{(W^r)_{r\ge1}\in\mathscr B_\tau\}=1.
\end{align*}

Define
\begin{align*}
\mathsf Y_*=\iota_\vartheta\bigl(\mathcal R_\tau^{-1}(\mathscr B_\tau)\bigr).
\end{align*}
Here $\mathcal R_\tau^{-1}(\mathscr B_\tau)$ is understood inside
$\mathscr C_\tau^0$. Since $\mathcal R_\tau$ is continuous,
$\mathcal R_\tau^{-1}(\mathscr B_\tau)$ is Borel. Moreover,
$\iota_\vartheta$ is continuous and injective between Polish spaces, so
the Lusin--Souslin theorem implies that $\mathsf Y_*$ is Borel in
$\mathsf Y$. Since $\mathcal R_\tau Z=(W^r)_{r\ge1}$ almost surely, $\mathbf P_Y(\mathsf Y_*)=1$. 

Fix now $Y=\iota_\vartheta\zeta\in\mathsf Y_*$ and put
$\omega=\mathcal R_\tau\zeta$. By
\eqref{eq:driver-coordinate-reconstruction}, coordinatewise and hence in
$H^{b-2}$,
\begin{align*}
S(\omega)_t=\zeta_t-\nu\int_0^t\Delta\zeta_s\,\dd s.
\end{align*}
Consequently, the compatible solution $w=v+\zeta$ satisfies
\begin{align*}
w_t=w_0+\int_0^t\bigl(\nu\Delta w_s-B(w_s,w_s)\bigr)\,\dd s+S(\omega)_t.
\end{align*}
Since $\omega\in\mathscr B_\tau$, it satisfies exactly the convergence
and dyadic quadratic variation properties used in the proof of
\Cref{thm:dense-joint-range}. That proof is therefore pathwise for this
fixed driver and applies simultaneously to every initial state and every
terminal covector. Using \eqref{eq:augmented-chain-identification}, we
obtain
\begin{align*}
\overline{\Ran A(z,Y)}=E_{\Phi(z,Y)}
\end{align*}
for every $z\in\mathsf X$.

The global existence of $\Phi(z,Y)$ follows from
\Cref{lem:compatible-driver-global}; backward uniqueness gives
nonvanishing of the raw fibre endpoint, and the operator norm continuity
of the fixed space representatives $J_{kj}$ and $A_k$ on their common
chart domains follows from \Cref{lem:revision-base-two-jet}. This
completes the proof.
\end{proof}

Here and below, operator continuity in variable tangent spaces means
continuity of the representatives $J_{kj}$ and $A_k$ in the orthogonal
frames \eqref{eq:chart-rotation}.  Intrinsic operator norms do not depend
on the selected frame.  \Cref{cor:rough-driver-state-compact} and \eqref{eq:DR} imply that,
pointwise in $(z,Y)$, $J$ can be approximated
in operator norm by $AK$ with $K$ finite rank.  Indeed, first
approximate the compact operator $J$ in operator norm by a finite rank
operator, and then approximate each of its finitely many output vectors by
a vector in $\Ran A$.  This is the operator theoretic use of compactness of the state
derivative in the selector construction.

\subsection{A universal compact core selector}
\label{subsec:universal-compact-core-selector}

All objects in this subsection refer to the single block $\tau$ chosen in
\eqref{eq:revision-single-block-choice}.  Let $\mathfrak I$ denote the
family of invariant laws of the sampled kernel $P_\tau$.  This family is
nonempty.  Indeed, \Cref{cor:revision-common-block} gives an inf-compact
function $W_0$ such that
\begin{align*}
 P_\tau W_0\le\varrho_0W_0+C_\tau,
 \qquad
 P_\tau W_0(z)<\infty\quad\text{for every }z\in\mathsf X.
\end{align*}
The projective semigroup is Feller by the continuous dependence in
\Cref{lem:revision-base-two-jet}, together with standard backward
uniqueness.  Hence the Krylov--Bogoliubov averages of the sampled chain are
tight after the first step and have a $P_\tau$ invariant weak limit.  Thus
$\mathfrak I\ne\varnothing$.  Since the base dynamics is uniquely ergodic
by \cite{HM06,KS12}, the first marginal of every $\widehat\mu\in\mathfrak I$
is $\mu$.  By \Cref{cor:revision-common-block}, $W_0$ is inf-compact.  Since
every $\widehat\mu\in\mathfrak I$ has first marginal equal to the base
invariant law $\mu$, invariance and
\eqref{eq:revision-W0-pointwise-reset} give
\begin{align*}
 \sup_{\widehat\mu\in\mathfrak I}
 \int_{\mathsf X}W_0\,\widehat\mu(\dd z)
 \le C_\tau\left(1+\int_{H^b}V_*(w)\,\mu(\dd w)\right)<\infty.
\end{align*}
Hence $\mathfrak I$ is uniformly tight.  Moreover,
\Cref{lem:projective-log-moment} with $r=1$, together with the
positive time projective reset and invariance, gives
\begin{align}\label{eq:uniform-logJ}
 B_J:=\sup_{\widehat\mu\in\mathfrak I}
 \int_{\mathsf X}\mathbf E\log\bigl(1+\|J(z,Y)\|\bigr)
 \,\widehat\mu(\dd z)<\infty.
\end{align}

We can now combine compactness of the state derivative with the dense
range in \eqref{eq:DR}.  The next proposition makes this approximation
uniform on a positive mass compact core and patches finitely many controls
without introducing an uncontrolled $\|A\|\,\|K_i\|$ term off the core.
\begin{proposition}
\label{prop:universal-selector}
On the common full measure driver event $\mathsf Y_*$ of
\Cref{lem:common-endpoint-event}, for every
$q\in(0,1/4)$ there is a measurable
finite rank field
\begin{align*}
 K_q(z,Y):E_z\longrightarrow\mathcal H
\end{align*}
with the following properties.
\begin{enumerate}[label=\textup{(\roman*)}]
\item In finitely many state tangent charts,
$K_q$ has the form
\begin{align*}
 K_q(z,Y)=\sum_{i=1}^{N_q}\chi_i(z,Y)K_i(z),
\end{align*}
where $K_i(z)$ is the smooth transport, in the corresponding tangent
trivialisation, of a fixed deterministic finite rank operator, and the
$\chi_i$ are smooth Hilbert coordinate cutoffs satisfying
$0\le\sum_i\chi_i\le1$.
\item There is a Borel good core $G_q\subset\mathsf X\times\mathsf Y$ and a
number $p_*>0$, independent of $q$, such that
\begin{align}\label{eq:uniform-core-mass}
 \inf_{\widehat\mu\in\mathfrak I}
   (\widehat\mu\otimes\mathbf{P}_Y)(G_q)\ge p_*,
\end{align}
and, with the compensated linear residual
\begin{align*}
 R_q(z,Y)=J(z,Y)-A(z,Y)K_q(z,Y),
\end{align*}
one has
\begin{align}
 \|R_q(z,Y)\|&\le q &&\text{on }G_q,
 \notag\\
 \|R_q(z,Y)\|&\le \|J(z,Y)\|+q
 &&\text{everywhere}.                         \label{eq:R-global}
\end{align}
\item There is a finite constant $C_q\ge1$ (not asserted to be uniform as $q\to0$) such that
\begin{align}
 \|K_q(z,Y)h\|_{\mathcal H}
 &\le C_q\|h\|,\label{eq:Kq-bound}\\
 \|D_{\mathcal H}(K_q(z,Y)h)\|_{\HS(\mathcal H)}
 &\le C_q\|h\|.\label{eq:DKq-bound}
\end{align}
Here $D_{\mathcal H}$ denotes the Malliavin derivative with respect to the
fresh Wiener noise on the block.
\end{enumerate}
\end{proposition}

\begin{proof}
By the uniform tightness of $\mathfrak I$, there is a compact set
$C_0\subset\mathsf X$ such that
\begin{align*}
\inf_{\widehat\mu\in\mathfrak I}\widehat\mu(C_0)\ge\frac34.
\end{align*}
Choose finitely many compact sets
$C_\ell\subset C_0\cap(H^b\times\mathcal V_{j(\ell)})$ whose union is
$C_0$, shrinking the corresponding projective subcharts if necessary so
that their denominators are uniformly separated from zero. In the
$\ell$-th fixed model tangent space choose a sequence
\begin{align*}
K_{\ell i}:H^b\oplus f_{j(\ell)}^\perp\longrightarrow\mathcal H,
\qquad i\ge1,
\end{align*}
of finite rank operators with rational finite matrices in fixed
orthonormal bases which is operator norm dense in
$\mathcal K(H^b\oplus f_{j(\ell)}^\perp,\mathcal H)$. Such a sequence
exists because finite rational matrices form a countable
operator norm dense subset of the compact operators between separable
Hilbert spaces. For $z\in H^b\times\mathcal V_{j(\ell)}$, set
\begin{align*}
K_{\ell i}(z)=K_{\ell i}\mathfrak T_{j(\ell)}(z)^{-1}.
\end{align*}

For $z\in C_\ell$ and compatible $Y$, define
\begin{align*}
r_{\ell,N}(z,Y)
=1\wedge\min_{1\le i\le N}
\|J(z,Y)-A(z,Y)K_{\ell i}(z)\|_{\mathcal L},
\end{align*}
and extend $r_{\ell,N}$ by one off the Borel compatible driver space.
By \Cref{lem:revision-base-two-jet} and \Cref{lem:common-endpoint-event}, the fixed space representatives of $J$ and $A$ are locally continuous in
operator norm in $(z,Y)$, while the orthogonal frame maps and
$K_{\ell i}(z)$ are smooth in $z$. Since the transition maps between
orthogonal frames preserve operator norms, $r_{\ell,N}$ is continuous
on the compatible driver space and its extension is Borel on the full
driver space.

The pointwise compact dense approximation and the density of the
rational candidates imply
\begin{align}\label{eq:selector-pointwise-universal}
r_{\ell,N}(z,Y)\downarrow0
\qquad\text{for every }(z,Y)\in C_\ell\times\mathsf Y_*.
\end{align}
The important point is that the same full measure set $\mathsf Y_*$ works simultaneously for every $z$.

We next make this approximation uniform in probability over $C_\ell$.
For fixed $N$, let $z_n\to z$ in $C_\ell$. Fix $Y\in\mathsf Y_*$. Choose
an output chart containing $\Phi(z,Y)$. By continuity of $\Phi$, the same
output chart contains $\Phi(z_n,Y)$ for all sufficiently large $n$, and
\Cref{lem:common-endpoint-event} gives
\begin{align*}
\|A_k(z_n,Y)-A_k(z,Y)\|_{\mathcal L}
+\|J_{k,j(\ell)}(z_n,Y)-J_{k,j(\ell)}(z,Y)\|_{\mathcal L}
\longrightarrow0.
\end{align*}
The corresponding transports of the finitely many $K_{\ell i}$,
$1\le i\le N$, also converge in operator norm. Hence
\begin{align*}
r_{\ell,N}(z_n,Y)\longrightarrow r_{\ell,N}(z,Y)
\qquad\text{for every }Y\in\mathsf Y_*.
\end{align*}
Since $0\le r_{\ell,N}\le1$ and $\mathbf P_Y(\mathsf Y_*)=1$, dominated
convergence shows that
\begin{align*}
\varphi_{\ell,N}(z)=\mathbf E\,r_{\ell,N}(z,Y)
\end{align*}
is continuous on $C_\ell$. Moreover,
\eqref{eq:selector-pointwise-universal} and dominated convergence give
\begin{align*}
\varphi_{\ell,N}(z)\downarrow0
\qquad\text{for every }z\in C_\ell.
\end{align*}
Thus $\varphi_{\ell,N}$ is a decreasing sequence of continuous functions
on the compact set $C_\ell$ with continuous limit zero. By Dini's theorem, $\sup_{z\in C_\ell}\varphi_{\ell,N}(z)\longrightarrow0$. 
Choose $N_\ell=N_\ell(q)$ so large that
\begin{align*}
\sup_{z\in C_\ell}\mathbf E\,r_{\ell,N_\ell}(z,Y)\le\frac{q}{128}.
\end{align*}
Markov's inequality then gives
\begin{align}\label{eq:Dini-good}
\inf_{z\in C_\ell}
\mathbf P_Y\{r_{\ell,N_\ell}(z,Y)<q/16\}\ge\frac78.
\end{align}

Since the law of $Z$ is Radon on the separable
space $c_0^\vartheta([0,\tau];H^b)$, we can choose a compact set
$\widetilde L_q$ contained in the zero initial compatibility space such
that its probability is at least $7/8$, and set
\begin{align*}
L_q=\iota_\vartheta(\widetilde L_q)\subset\mathsf Y.
\end{align*}
Then $L_q$ is compact in $\mathsf Y$,
$\mathbf P_Y(L_q)\ge7/8$, and every $Y\in L_q$ is compatible. By \Cref{lem:compatible-driver-global}, the raw endpoint is defined on
$C_0\times L_q$. Backward uniqueness makes every raw fibre endpoint on
$C_0\times L_q$ nonzero. Since the nonvanishing set is open, there is an
open neighborhood of $C_0\times L_q$ on which all normalized endpoint
quantities are defined.

Define
\begin{align}\label{eq:compact-selector-core}
\mathcal G_q=
\bigcup_\ell
\{(z,Y)\in C_\ell\times L_q:
r_{\ell,N_\ell}(z,Y)\le q/8\}.
\end{align}
Each member of this finite union is compact, since
$r_{\ell,N_\ell}$ is continuous on $C_\ell\times L_q$. Hence
$\mathcal G_q$ is compact. For every $z\in C_0$, choose $\ell$ with
$z\in C_\ell$. By \eqref{eq:Dini-good},
\begin{align*}
\mathbf P_Y\{r_{\ell,N_\ell}(z,Y)<q/16\}\ge\frac78,
\qquad
\mathbf P_Y(L_q)\ge\frac78.
\end{align*}
If both events occur, then $(z,Y)\in\mathcal G_q$. Therefore the union
bound gives
\begin{align*}
\mathbf P_Y\{Y:(z,Y)\in\mathcal G_q\}
\ge1-\frac18-\frac18=\frac34.
\end{align*}
Consequently, for every $\widehat\mu\in\mathfrak I$,
\begin{align*}
(\widehat\mu\otimes\mathbf P_Y)(\mathcal G_q)
&=\int_{\mathsf X}
\mathbf P_Y\{Y:(z,Y)\in\mathcal G_q\}\,\widehat\mu(\dd z)\\
&\ge\frac34\,\widehat\mu(C_0)\ge\frac9{16}.
\end{align*}
Thus we may fix $p_*=1/2$ and set $G_q=\mathcal G_q$.

The finite patch selector is constructed locally and then patched by a
finite partition.  Fix $x=(z,Y)\in\mathcal G_q$.
By \eqref{eq:compact-selector-core}, there are $\ell$ with $z\in C_\ell$
and $1\le j\le N_\ell$ such that
\begin{align*}
\|J(z,Y)-A(z,Y)K_{\ell j}(z)\|_{\mathcal L}\le\frac q8.
\end{align*}
Choose an output projective chart containing $\Phi(z,Y)$. By continuity
of $\Phi$, $J$, $A$, the orthogonal frames, and $K_{\ell j}(\cdot)$ in
the corresponding fixed Hilbert coordinates, there is an open
neighborhood $\mathcal U_x$ of $x$ such that the same candidate satisfies
\begin{align*}
\|J(z',Y')-A(z',Y')K_{\ell j}(z')\|_{\mathcal L}<q
\end{align*}
for every $(z',Y')\in\mathcal U_x$. Shrinking $\mathcal U_x$ if
necessary, we may also assume that the same input and output projective
charts are valid throughout $\mathcal U_x$ and that their denominators
are bounded below by a positive constant.  Shrink it once more so that,
writing $x=(z,Y)$, every state component $z'$ occurring in
$\mathcal U_x$ satisfies $(z,z')\in\mathscr U_\Delta$.

Use the unified state coordinate together with the
identity driver coordinate,
\begin{align*}
 \kappa_x(z',Y')=\bigl(\Delta_z(z'),Y'-Y\bigr),
\end{align*}
to identify $\mathcal U_x$ with an open subset of a Hilbert space.  Since
$\mathcal U_x$ is open, choose $r_x>0$ such that
\begin{align*}
\kappa_x^{-1}\left(
\overline{B(\kappa_x(x),2r_x)}
\right)\subset\mathcal U_x.
\end{align*}
The corresponding radius $r_x$ coordinate balls form an open cover of
the compact set $\mathcal G_q$. Hence there are finitely many
$x_1,\ldots,x_m\in\mathcal G_q$ and radii $r_1,\ldots,r_m>0$ whose
radius $r_i$ coordinate balls cover $\mathcal G_q$, while the closed
radius $2r_i$ balls remain inside the corresponding neighborhoods
$\mathcal U_{x_i}$.

For each $i$, let $\ell_i$ and $a_i$ be the indices selected at $x_i$,
and write
\begin{align*}
K_i(z)=K_{\ell_i a_i}\mathfrak T_{j(\ell_i)}(z)^{-1}
\end{align*}
on the corresponding input chart. Then
\begin{align*}
\|J(z,Y)-A(z,Y)K_i(z)\|_{\mathcal L}<q
\end{align*}
throughout the closed doubled coordinate ball associated with $x_i$.
The relevant projective denominators are uniformly separated from zero
there.

Let $\kappa_i$ be the Hilbert coordinate map of the $i$-th patch. Choose
$\theta\in C^\infty(\mathbb R;[0,1])$ equal to one on $(-\infty,1]$ and
zero on $[4,\infty)$, and define
\begin{align*}
\psi_i(x)
=\theta\left(
\frac{\|\kappa_i(x)-\kappa_i(x_i)\|^2}{r_i^2}
\right)
\end{align*}
on that chart, extending $\psi_i$ by zero outside it. Since
$\operatorname{supp}\psi_i$ lies in the closed doubled ball, this
extension is smooth. Put
\begin{align*}
S=\sum_{i=1}^m\psi_i.
\end{align*}
The radius $r_i$ balls cover $\mathcal G_q$, and $\psi_i=1$ on the
corresponding radius $r_i$ ball. Hence
\begin{align*}
S\ge1\qquad\text{on }\mathcal G_q.
\end{align*}
Choose $\Theta\in C^\infty([0,\infty);[0,1])$ which is zero on
$[0,1/4]$ and one on $[1/2,\infty)$, and define
\begin{align*}
\chi_i=\Theta(S)\frac{\psi_i}{S}\quad\text{on }\{S>0\},
\qquad
\chi_i=0\quad\text{on }\{S=0\}.
\end{align*}
Since $\Theta(S)=0$ whenever $S\le1/4$, the quotient causes no
singularity at $S=0$, and each $\chi_i$ is globally smooth. Moreover,
\begin{align*}
0\le\sum_{i=1}^m\chi_i=\Theta(S)\le1,\qquad
\sum_{i=1}^m\chi_i=1\quad\text{on }\mathcal G_q,
\end{align*}
and $\operatorname{supp}\chi_i$ is contained in the doubled ball on
which $K_i$ has residual below $q$.

There are finitely many patches and radii. Hence the driver
derivatives of the $\chi_i$ needed below are uniformly bounded. Moreover,
each $K_i(z)$ is obtained by composing a fixed finite rank operator with the
smooth orthogonal frame $\mathfrak T_{j(\ell_i)}(z)^{-1}$, and is therefore
continuous and uniformly bounded on the support of $\chi_i$.

Define
\begin{align*}
K_q(z,Y)=\sum_{i=1}^m\chi_i(z,Y)K_i(z),
\end{align*}
where each summand is extended by zero outside its patch. Then $K_q$ is a finite rank operator field on the union of these
patches, vanishes outside that union, and is continuous there. Its range is
contained in the finite dimensional space
generated by the ranges of the finitely many $K_{\ell_i a_i}$.

Let
\begin{align*}
s=\sum_{i=1}^m\chi_i.
\end{align*}
Since $K_q=\sum_i\chi_iK_i$,
\begin{align*}
J-AK_q=(1-s)J+\sum_{i=1}^m\chi_i(J-AK_i).
\end{align*}
On $\mathcal G_q$ one has $s=1$, and every active candidate satisfies
$\|J-AK_i\|_{\mathcal L}<q$. Hence
\begin{align*}
\|J-AK_q\|_{\mathcal L}<q
\qquad\text{on }\mathcal G_q.
\end{align*}
For arbitrary points in the support of the selector,
\begin{align*}
\|J-AK_q\|_{\mathcal L}
\le(1-s)\|J\|_{\mathcal L}+sq
\le\|J\|_{\mathcal L}+q.
\end{align*}
Thus the large norms of the individual $K_i$ do not enter the residual
estimate.

Finally, $K_i(z)$ is independent of the driver variable and
$D_{\mathcal H}Y=\Gamma$. Therefore, for fixed $h$ and $v\in\mathcal H$,
\begin{align*}
D_{\mathcal H}(K_qh)[v]
=\sum_{i=1}^m
\bigl(D_{\mathsf Y}\chi_i\,\Gamma v\bigr)K_i h.
\end{align*}
This is a finite rank operator on $\mathcal H$. Since $\Gamma$ is
Hilbert--Schmidt and the driver derivatives of the finitely many cutoffs
are bounded,
\begin{align*}
\|D_{\mathcal H}(K_qh)\|_{\mathcal L_2(\mathcal H)}
\le C_q\|h\|.
\end{align*}
This proves the Malliavin derivative bound required in
\eqref{eq:DKq-bound} and completes the construction.
\end{proof}

For a selector $K_q$ supplied by \Cref{prop:universal-selector}, define the
logarithmic residual
\begin{align}\label{eq:adef}
 a_q(z,Y)=\log\bigl(\|R_q(z,Y)\|+q\bigr).
\end{align}

The residual in \eqref{eq:adef} becomes a logarithmic contraction
observable.  Choosing the accuracy once yields a uniform negative stationary
mean over all sampled invariant laws.
\begin{corollary}
\label{cor:uniform-invariant-negative}
There are a fixed $q>0$, a corresponding selector $K_q$, and
$\lambda>0$ such that
\begin{align}\label{eq:all-invariant-negative}
 \sup_{\widehat\mu\in\mathfrak I}
 \int_{\mathsf X}\mathbf{E} a_q(z,Y)\,
       \widehat\mu(\dd z)\le-5\lambda.
\end{align}
\end{corollary}

\begin{proof}
On $G_q$, one has $a_q\leq \log(2q)$. 
On the complement, for $q<1/4$, \eqref{eq:R-global}  gives 
\begin{align*}
a_q\leq \log\bigl(\|J\|+2q\bigr)\leq \log(1+\|J\|).
\end{align*}
Consequently, by \eqref{eq:uniform-core-mass} and
\eqref{eq:uniform-logJ},
\begin{align*}
\sup_{\widehat\mu\in\mathfrak I}
\int_{\mathsf X}\mathbf E a_q(z,Y)\,\widehat\mu(\dd z)
\leq p_*\log(2q)+B_J.
\end{align*}
The right hand side tends to $-\infty$ as $q\downarrow0$.
\end{proof}

\subsection{Uniform occupation negativity}

We first construct a single Lyapunov function that simultaneously controls the occupation and temperedness quantities needed below. Unlike the base Lyapunov hierarchy in \Cref{sec:foundations}, this augmented function also incorporates the logarithmic costs of the endpoint two jet and chart radius, while avoiding inadmissible powers of the exponential base weight.

Keep $p=p_{\rm occ}$, the block $\tau$, and the common weight $W_0$ fixed in \eqref{eq:revision-single-block-choice}.  The common block
choice already includes the two orders $1$ and $p^2$: the first gives
\eqref{eq:uniform-logJ}, while the second supplies the moments needed for
the augmented Lyapunov function below.  No new block is selected in the
following proposition. Define the one block logarithmic cost $L$ and its   moments
$F_s$ by
\begin{align}
 L(z,Y)
 &=\log^+\left(
  1+\|J(z,Y)\|+\mathsf C_\tau(z,Y)+\mathsf r_\tau(z,Y)^{-1}
 \right),                                             \label{eq:block-log-cost}\\
 F_s(z)&=\mathbf{E}_z L(z,Y)^s. \notag
\end{align}
Here $\mathsf C_\tau$ is the two jet bounds in \eqref{eq:revision-C-def} and $\mathsf r_\tau$ is the base tube and chart radius given in  \eqref{eq:revision-chart-radius}.

To upgrade stationary negativity to a pathwise last exit estimate, we
need one Lyapunov function controlling both the residual and the local
two jet cost.  The next proposition constructs such a function without
changing the previously fixed block.
\begin{proposition}
\label{cor:revision-conditional-Lyapunov}
There is a lower semicontinuous inf-compact function
$\mathcal W:\mathsf X\to[1,\infty]$, and
constants $0<\rho<1$, $C<\infty$ such that
\begin{align}
 P_\tau\mathcal W&\le\rho\mathcal W+C,
 \label{eq:joint-V-drift}\\
 \mathbf{E}_z\left(L(z,Y)^p+\left|a_q(z, Y)\right|^p\right)&\le C\mathcal W(z)^{1/p},
 \label{eq:conditional-block-moments}
\end{align}
where the constant $C$ in the second inequality may depend on $q$ through $a_q$. Moreover, 
\begin{align*}
      \sup _{\widehat{\mu} \in \mathfrak{I}} \int_{\mathsf X} \mathbf{E}_z\left[L(z, Y)^p+\left|a_q(z, Y)\right|^p\right] \widehat{\mu}(\dd z)<\infty.
\end{align*}
\end{proposition}

\begin{proof}
Write
\begin{align*}
X_1=\Phi(z,Y_0)=(w_\tau,[\mathsf p_\tau]),
\end{align*}
where $Y_0$ is the driver on the first block and $\mathsf p_\tau$ is the unit representative of the projective component. By \Cref{cor:revision-common-block}, $\mathsf p_\tau\in H^1$ almost surely. Moreover, the common hierarchy was chosen so that $M_*\ge d_{p^2}$ and $V_*$ dominates the base weight associated with the order $p^2$.

Conditionally on $X_1$, applying \Cref{lem:projective-log-moment} with $T=\tau$ and $r=p^2$  gives
\begin{align*}
F_{p^2}(X_1)\le C\left(V_*(w_\tau)+1+\|\mathsf p_\tau\|_{H^1}^{d_{p^2}}\right).
\end{align*}
Taking expectation over the first block yields
\begin{align*}
P_\tau F_{p^2}(z)=\mathbf E_zF_{p^2}(X_1)\le C\mathbf E_z\left[V_*(w_\tau)+1+\|\mathsf p_\tau\|_{H^1}^{d_{p^2}}\right].
\end{align*}
The base Foster estimate and common reset \eqref{eq:revision-common-reset} in
\Cref{cor:revision-common-block} give
\begin{align*}
\mathbf E_zV_*(w_\tau)\le \theta V_*(w)+C_\tau,
\qquad
\mathbf E_z\|\mathsf p_\tau\|_{H^1}^{d_{p^2}}
\le 1+\mathbf E_z\|\mathsf p_\tau\|_{H^1}^{M_*}
\le C_\tau(1+V_*(w)).
\end{align*}
Consequently,
\begin{align}\label{eq:revision-two-block-reset}
P_\tau F_{p^2}(z)\le C\bigl(1+V_*(w)\bigr)<\infty.
\end{align}
The bound is uniform in the initial projective direction.

Set
\begin{align}\label{eq:revision-W-augmentation}
\mathcal W=W_0+\varepsilon_1F_{p^2},
\end{align}
where $\varepsilon_1>0$ will be chosen below. Since $V_*\le W_0$, \eqref{eq:robust-joint-drift} and \eqref{eq:revision-two-block-reset} give
\begin{align*}
P_\tau\mathcal W  \le(\varrho_0+\varepsilon_1C)W_0+C
\le\rho\mathcal W+C
\end{align*}
after choosing $\varepsilon_1$ so that $\varrho_0+\varepsilon_1C<\rho<1$. This proves \eqref{eq:joint-V-drift}.

Let
\begin{align*}
\mathsf D_{\mathcal W}=\{z:\mathcal W(z)<\infty\}.
\end{align*}
By \Cref{lem:projective-log-moment} and $C^\beta \hookrightarrow H$, $F_{p^2}$ is finite on $\mathsf D_0=C^\beta\times\PP(H^1)$, while $W_0=+\infty$ on $\mathsf X\setminus\mathsf D_0$ by  \eqref{eq:robust-joint-drift}. Hence $\mathsf D_{\mathcal W}=\mathsf D_0$. Moreover, \eqref{eq:revision-W0-pointwise-reset} and \eqref{eq:revision-two-block-reset} give
\begin{align}\label{eq:revision-W-pointwise-reset}
P_\tau\mathcal W(w,[\mathsf p])\le C_\tau\bigl(1+V_*(w)\bigr)<\infty,\qquad P_\tau\bigl((w,[\mathsf p]),\mathsf D_{\mathcal W}\bigr)=1.
\end{align}
Thus the sampled chain enters $\mathsf D_{\mathcal W}$ almost surely after one block, independently of the regularity of the initial projective direction.

As $(z,y)\mapsto L(z,y)$ is lower semicontinuous when extended by $+\infty$ at a zero raw endpoint, Fatou's lemma  gives
\begin{align*}
F_{p^2}(z)\le\liminf_{n\to\infty}F_{p^2}(z_n)
\end{align*}
whenever $z_n\to z$. Hence $F_{p^2}$ and $\mathcal W$ are lower semicontinuous. Since $\mathcal W\ge W_0$ and $W_0$ is inf-compact, every sublevel $\{\mathcal W\le M\}$ is a closed subset of the compact set $\{W_0\le M\}$. Thus $\mathcal W$ is inf-compact.

For the residual $R_q=J-AK_q$, \eqref{eq:R-global} and $q<1/4$ give
\begin{align*}
q\le\|R_q\|+q\le\|J\|+2q\le1+\|J\|.
\end{align*}
Consequently,
\begin{align*}
|a_q(z,Y)|^p\le C_{p,q}\left(1+\bigl[\log^+(1+\|J(z,Y)\|)\bigr]^p\right)\le C_{p,q}\bigl(1+L(z,Y)^p\bigr).
\end{align*}
Since $\mathcal W\ge1$, the preceding estimates yield
\begin{align*}
\mathbf E_z\left[L(z,Y)^p+|a_q(z,Y)|^p\right]\le C\mathcal W(z)^{1/p},
\end{align*}
which proves \eqref{eq:conditional-block-moments}.

It remains to deduce the uniform stationary integrability. Let $\widehat\mu\in\mathfrak I$. Since its base marginal is the unique base invariant law $\mu$, invariance and \eqref{eq:revision-W0-pointwise-reset} give
\begin{align*}
\int_{\mathsf X}W_0\,\dd\widehat\mu
=\int_{\mathsf X}P_\tau W_0\,\dd\widehat\mu
\le C_\tau\left(1+\int_{H^b}V_*(w)\,\mu(\dd w)\right)<\infty.
\end{align*}
Likewise, by \eqref{eq:revision-two-block-reset},
\begin{align*}
\int_{\mathsf X}F_{p^2}\,\dd\widehat\mu
=\int_{\mathsf X}P_\tau F_{p^2}\,\dd\widehat\mu
\le C\left(1+\int_{H^b}V_*(w)\,\mu(\dd w)\right)<\infty.
\end{align*}
The bounds are independent of $\widehat\mu$, and hence
\begin{align*}
\sup_{\widehat\mu\in\mathfrak I}\int_{\mathsf X}\mathcal W\,\dd\widehat\mu<\infty.
\end{align*}
Integrating \eqref{eq:conditional-block-moments} and Jensen's inequality therefore prove
\begin{align*}
\sup_{\widehat\mu\in\mathfrak I}\int_{\mathsf X}\mathbf E_z\left[L(z,Y)^p+|a_q(z,Y)|^p\right]\,\widehat\mu(\dd z)<\infty.
\end{align*}
\end{proof}

Let $(Y_n)_{n\geq 0}$ be independent copies of the one block
Gaussian driver $Y$. For $x\in\mathsf X$, set $X_0=x$ and
\begin{align}\label{e.XnYn}
X_{n+1}=\Phi(X_n,Y_n), \qquad n\geq 0.
\end{align}
Then $(X_n)_{n\geq 0}$ is the Markov chain with transition kernel
$P_\tau$; equivalently, $X_n$ has the law of the projective
process sampled at time $n\tau$. Fix $q$, $K_q$, and $\lambda>0$ as in
\Cref{cor:uniform-invariant-negative}, let $a_q$ be defined by
\eqref{eq:adef}, and let $\mathcal W$ be the inf-compact Lyapunov function
from \Cref{cor:revision-conditional-Lyapunov}.  Put
$p=p_{\rm occ}$ and use the concave Lyapunov weight
$\mathcal V=\mathcal W^{1/p}$ on the finite domain
$\mathsf D_{\mathcal W}=\{\mathcal W<\infty\}$.

The negative invariant mean must persist uniformly along long trajectories
started from bounded Lyapunov sets.  The next proposition gives precisely
this last exit form. 
\begin{proposition}
\label{prop:uniform-occupation}
For every set $D$
on which $\mathcal W$ is bounded,
\begin{align}\label{eq:uniform-last-exit}
 \lim_{N\to\infty}\sup_{x\in D}
 \mathbf{P}_x\left\{\exists n\ge N:
   \frac1n\sum_{k=0}^{n-1}a_q(X_k,Y_k)>-4\lambda\right\}=0.
\end{align}
\end{proposition}

\begin{proof}
By \eqref{eq:revision-W-pointwise-reset}, the chain enters
$\mathsf D_{\mathcal W}$ after one block and remains there almost surely.
Equations \eqref{eq:joint-V-drift} and
\eqref{eq:conditional-block-moments} give
\begin{align*}
 P_\tau\mathcal W\le\rho\mathcal W+C,
 \qquad
 \mathbf E|a_q(x,Y)|^p\le C\mathcal V(x)
 \quad (x\in\mathsf D_{\mathcal W}),
\end{align*}
for some $0<\rho<1$.  Set
$\bar a_q(x)=\mathbf E a_q(x,Y)$.  The finite patch construction and
operator norm flow continuity make $(z,Y)\mapsto a_q(z,Y)$ pointwise
continuous at every compatible driver.  On each compact set
$K_M=\{\mathcal V\le M\}$,
\eqref{eq:conditional-block-moments} bounds
$\mathbf E|a_q(z,Y)|^p$ uniformly.  Almost sure continuity at the compatible
random driver, followed by Vitali's theorem, therefore makes
$z\mapsto\bar a_q(z)$ continuous on $K_M$.  Finally,
\eqref{eq:all-invariant-negative} gives
\begin{align*}
 \int\bar a_q\,\dd\widehat\mu\le-5\lambda
\end{align*}
for every invariant law $\widehat\mu$.  For the remainder of the proof write
$a=a_q$ and $\bar a=\bar a_q$.

If $x\in\mathsf D_{\mathcal W}$, the drift gives
$P_\tau\mathcal W(x)<\infty$, and hence
$P_\tau(x,\mathsf D_{\mathcal W})=1$.  Starting from a set on which
$\mathcal W$ is bounded, induction therefore keeps the chain in
$\mathsf D_{\mathcal W}$ almost surely.  Every variable in the following
martingale argument is consequently finite.  On this invariant finite
domain, Jensen's inequality gives
\begin{align}\label{eq:V-concave-drift}
 P_\tau\mathcal V
 \le(P_\tau\mathcal W)^{1/p}
 \le\theta\mathcal V+C_0,\qquad \theta=\rho^{1/p}<1.
\end{align}
Put $\Delta_{k+1} =\mathcal V(X_{k+1})-P_\tau\mathcal V(X_k)$ for $k\geq 0$, which is a martingale difference. Its sum sequence
\begin{align*}
 \qquad M_n=\sum_{k=0}^{n-1}\Delta_{k+1}, \, n\geq 1, \, M_0=0
\end{align*}
forms a Dynkin martingale.  Since $\mathcal V^p=\mathcal W$,
\begin{align*}
 \mathbf{E}\bigl[|\Delta_{k+1}|^p\mid \mathcal F_k\bigr]
 \le C\big(P_\tau\mathcal W(X_k)
        +(P_\tau\mathcal V(X_k))^p\big)\le C(1+\mathcal W(X_k)).
\end{align*}
The drift \eqref{eq:joint-V-drift} and boundedness of $\mathcal W$ on $D$
give
\begin{align*}
 \sup_{x\in D}\sup_{k\ge0}\mathbf{E}_x\mathcal W(X_k)<\infty.
\end{align*}
The maximal inequality for $1<p\le2$ now gives
\begin{align}\label{eq:V-mart-maximal}
 \sup_{x\in D}\mathbf{E}_x\max_{\ell\le n}|M_\ell|^p\le C_Dn.
\end{align}
By the definition of $\Delta_{k+1}$ and \eqref{eq:V-concave-drift},
\begin{align*}
\mathcal V(X_{k+1})\le\theta\mathcal V(X_k)+C_0+\Delta_{k+1}.
\end{align*}
Summing over $0\le k<n$ gives
\begin{align*}
(1-\theta)\sum_{k<n}\mathcal V(X_k)\le\mathcal V(X_0)+C_0n+M_n.
\end{align*}
Since $\mathcal W$, and hence $\mathcal V$, is bounded on $D$, for $K$
sufficiently large depending only on $D$, the event
\begin{align*}
\frac1n\sum_{k<n}\mathcal V(X_k)>K
\end{align*}
implies $M_n\ge cKn$, with $c>0$ independent of $n$ and $x\in D$.

Decompose $\{n\ge N\}$ into the intervals
$[2^jN,2^{j+1}N)$, $j\ge0$.  If the preceding event occurs for some
$n$ in the $j$-th interval, then
\begin{align*}
\max_{\ell\le2^{j+1}N}|M_\ell|\ge cK2^jN.
\end{align*}
Hence \eqref{eq:V-mart-maximal} and Markov's inequality give
\begin{align*}
\sup_{x\in D}\mathbf P_x\left\{\exists\,2^jN\le n<2^{j+1}N:
\frac1n\sum_{k<n}\mathcal V(X_k)>K\right\}
\le C_DK^{-p}(2^jN)^{1-p}.
\end{align*}
Since $p>1$, summing over $j\ge0$ yields
\begin{align}\label{eq:empirical-V-tail}
\sup_{x\in D}\mathbf P_x\left\{\exists n\ge N:
\frac1n\sum_{k<n}\mathcal V(X_k)>K\right\}
\le C_DK^{-p}N^{1-p}.
\end{align}

It remains to exclude long empirical trajectories whose residual average
violates the invariant law gap. We show that any such trajectories would
generate, after passage to a subsequence, an invariant probability measure
with an inadmissibly large $\bar a$ average.

Suppose \eqref{eq:uniform-last-exit} were false. Then there are
$\varepsilon>0$, $N_j\uparrow\infty$, and $x_j\in D$ such that
\begin{align*}
\mathbf P_{x_j}\left\{\exists n\ge N_j:\frac1n\sum_{k<n}a(X_k,Y_k)>-4\lambda\right\}\ge\varepsilon.
\end{align*}
Define
\begin{align*}
\tau_j=\inf\left\{n\ge N_j:\frac1n\sum_{k<n}a(X_k,Y_k)>-4\lambda\right\},
\end{align*}
with $\inf\varnothing=\infty$. Thus
$\mathbf P_{x_j}\{\tau_j<\infty\}\ge\varepsilon$. On
$\{\tau_j<\infty\}$ set
\begin{align*}
L_j=\frac1{\tau_j}\sum_{k<\tau_j}\delta_{X_k},
\end{align*}
and introduce the conditional probability $\mathbf Q_j(A)=\mathbf P_{x_j}(A\mid\tau_j<\infty)$. In particular,
\begin{align}\label{eq:conditional-cost}
\mathbf Q_j(A)\le\varepsilon^{-1}\mathbf P_{x_j}(A)
\end{align}
for every event $A$. Under $\mathbf Q_j$, $\tau_j$ is finite and
$\tau_j\ge N_j$ almost surely.

We first prove tightness of the empirical measures. By
\eqref{eq:empirical-V-tail} and \eqref{eq:conditional-cost}, for all
sufficiently large $K$,
\begin{align*}
\mathbf Q_j\left\{\int\mathcal V\,\dd L_j>K\right\}
 \le\varepsilon^{-1}\mathbf P_{x_j}\left\{\exists n\ge N_j:\frac1n\sum_{k<n}\mathcal V(X_k)>K\right\}\le C_D\varepsilon^{-1}K^{-p}N_j^{1-p}.
\end{align*}
Since $p>1$, the right hand side is uniformly small for large $K$.
Moreover, $\mathcal V$ is lower semicontinuous and inf-compact. Hence, for
every $K<\infty$, the set
\begin{align*}
\left\{\ell\in\mathcal P(\mathsf X):\int\mathcal V\,\dd\ell\le K\right\}
\end{align*}
is compact in $\mathcal P(\mathsf X)$. Indeed,
$\ell\{\mathcal V>R\}\le K/R$ and
$\{\mathcal V\le R\}$ is compact, while lower semicontinuity of
$\mathcal V$ makes the displayed set closed. It follows that the laws of
$L_j$ under $\mathbf Q_j$ are tight in $\mathcal P(\mathcal P(\mathsf X))$.
Passing to a subsequence, write
\begin{align*}
\mathcal L_{\mathbf Q_j}(L_j)\Longrightarrow\mathscr P
\end{align*}
for some probability measure $\mathscr P$ on $\mathcal P(\mathsf X)$.

We next show that $\mathscr P$ is concentrated on $P_\tau$ invariant
probability measures. Let $f$ be bounded and Lipschitz. Since
\begin{align*}
\int(P_\tau f-f)\,\dd L_j
=\frac{f(X_{\tau_j})-f(X_0)}{\tau_j}
-\frac1{\tau_j}\sum_{k<\tau_j}\bigl(f(X_{k+1})-P_\tau f(X_k)\bigr),
\end{align*}
it is enough to control the second term. The summands
$f(X_{k+1})-P_\tau f(X_k)$ are bounded martingale differences. Doob's
inequality followed by the same dyadic decomposition as above gives
\begin{align*}
\mathbf P_x\left\{\sup_{n\ge N}\frac1n\left|\sum_{k<n}\bigl(f(X_{k+1})-P_\tau f(X_k)\bigr)\right|>\eta\right\}
\le C_f\eta^{-2}N^{-1}.
\end{align*}
Therefore, by \eqref{eq:conditional-cost} and $\tau_j\ge N_j$,
\begin{align*}
\int(P_\tau f-f)\,\dd L_j\longrightarrow0
\end{align*}
in $\mathbf Q_j$ probability. Since $P_\tau$ is Feller, the map
\begin{align*}
\ell\longmapsto\int(P_\tau f-f)\,\dd\ell
\end{align*}
is continuous on $\mathcal P(\mathsf X)$. Taking a countable
convergence determining family of bounded Lipschitz functions, we conclude
that
\begin{align}\label{eq:limit-measure-invariant}
\ell P_\tau=\ell
\qquad\text{for $\mathscr P$ almost every }\ell\in\mathcal P(\mathsf X).
\end{align}

The violation inequality is transferred to the limiting empirical
measures. Put
\begin{align*}
\Delta^a_{k+1}=a(X_k,Y_k)-\bar a(X_k).
\end{align*}
Then $(\Delta^a_{k+1})$ is a martingale difference sequence and
\eqref{eq:conditional-block-moments}, Jensen's inequality, and the drift
of $\mathcal W$ give
\begin{align*}
\sup_{x\in D}\mathbf E_x\max_{\ell\le n}\left|\sum_{k<\ell}\Delta^a_{k+1}\right|^p\le C_Dn.
\end{align*}
The same dyadic decomposition therefore yields, for every $\eta>0$,
\begin{align*}
\sup_{x\in D}\mathbf P_x\left\{\sup_{n\ge N}\frac1n\left|\sum_{k<n}\Delta^a_{k+1}\right|>\eta\right\}
\le C_{D,\eta}N^{1-p}.
\end{align*}
Using \eqref{eq:conditional-cost} with $N=N_j$ and recalling that
$\tau_j\ge N_j$, we obtain
\begin{align*}
\frac1{\tau_j}\sum_{k<\tau_j}\bigl(a(X_k,Y_k)-\bar a(X_k)\bigr)
\longrightarrow0
\end{align*}
in $\mathbf Q_j$ probability. On the other hand, by the definition of
$\tau_j$,
\begin{align*}
\frac1{\tau_j}\sum_{k<\tau_j}a(X_k,Y_k)>-4\lambda
\end{align*}
$\mathbf Q_j$ almost surely. Consequently, for every $\eta>0$,
\begin{align}\label{eq:bad-empirical-a}
\mathbf Q_j\left\{\int\bar a\,\dd L_j\ge-4\lambda-\eta\right\}\longrightarrow1.
\end{align}

It remains to pass to the limit in the unbounded observable $\bar a$.
By Jensen's inequality and \eqref{eq:conditional-block-moments},
\begin{align*}
|\bar a(x)|^p\le\mathbf E|a(x,Y)|^p\le C\mathcal V(x).
\end{align*}
For $M<\infty$, let $K_M=\{\mathcal V\le M\}$. Since $\bar a$ is
continuous on $K_M$, the Tietze extension theorem gives
$g_M\in C_b(\mathsf X)$ such that
\begin{align*}
g_M=\bar a\quad\text{on }K_M,\qquad \|g_M\|_\infty\le CM^{1/p}.
\end{align*}
Hence, for every probability measure $\ell$ with
$\int\mathcal V\,\dd\ell<\infty$,
\begin{align}
\left|\int\bar a\,\dd\ell-\int g_M\,\dd\ell\right|
\le CM^{-(p-1)/p}\int\mathcal V\,\dd\ell.
\label{eq:a-sublevel-localization}
\end{align}

Fix $\eta,\delta>0$. By tightness, choose $K<\infty$ so that
\begin{align*}
\sup_j\mathbf Q_j\left\{\int\mathcal V\,\dd L_j>K\right\}<\delta,
\end{align*}
and then choose $M$ so large that
$CKM^{-(p-1)/p}<\eta$. By \eqref{eq:bad-empirical-a}, for all sufficiently
large $j$,
\begin{align*}
\mathbf Q_j\left\{\int\bar a\,\dd L_j\ge-4\lambda-\eta\right\}>1-\delta.
\end{align*}
Thus, with $\mathbf Q_j$ probability at least $1-2\delta$,
\eqref{eq:a-sublevel-localization} gives
\begin{align*}
\int g_M\,\dd L_j\ge-4\lambda-2\eta,\qquad
\int\mathcal V\,\dd L_j\le K.
\end{align*}
The corresponding subset of $\mathcal P(\mathsf X)$ is closed, since
$g_M$ is bounded and continuous and
$\ell\mapsto\int\mathcal V\,\dd\ell$ is lower semicontinuous.
Therefore, by Portmanteau,
\begin{align*}
\mathscr P\left\{\ell:\int g_M\,\dd\ell\ge-4\lambda-2\eta,\
\int\mathcal V\,\dd\ell\le K\right\}\ge1-2\delta.
\end{align*}
Applying \eqref{eq:a-sublevel-localization} once more shows that
\begin{align*}
\mathscr P\left\{\ell:\int\bar a\,\dd\ell\ge-4\lambda-3\eta\right\}\ge1-2\delta.
\end{align*}
Letting first $\delta\downarrow0$ and then $\eta\downarrow0$ gives
\begin{align}\label{eq:bad-invariant-limit}
\int\bar a\,\dd\ell\ge-4\lambda
\qquad\text{for $\mathscr P$ almost every }\ell.
\end{align}

By \eqref{eq:limit-measure-invariant}, $\mathscr P$ is concentrated on
$\mathfrak I$, while \eqref{eq:all-invariant-negative} gives
\begin{align*}
\int\bar a\,\dd\ell\le-5\lambda
\qquad\text{for every }\ell\in\mathfrak I.
\end{align*}
This contradicts \eqref{eq:bad-invariant-limit} and proves
\eqref{eq:uniform-last-exit}.
\end{proof}

Let the local geometry cost on the $n$-th block be
\begin{align}\label{e.cn}
 c_n=\log^+\left(
  1+\mathsf C_\tau(X_n,Y_n)+\mathsf r_\tau(X_n,Y_n)^{-1}\right).
\end{align}
For every $\mathcal W$ bounded initial set $D$, the moment bound \eqref{eq:conditional-block-moments}, the drift
\eqref{eq:joint-V-drift}, and a union bound give
\begin{align}\label{eq:uniform-tempered}
 \sup_{x\in D}
 \mathbf{P}_x\{\exists n\ge N:c_n>\varepsilon n\}
 \le C_{D,\varepsilon}\sum_{n\ge N}n^{-p}
 \xrightarrow[N\to\infty]{}0,
\end{align}
for every $\varepsilon>0$.  Applying Borel--Cantelli to a countable
sequence $\varepsilon\downarrow0$ proves $c_n/n\to0$ almost surely; for
the uniform stable radius quantile below we use
\eqref{eq:uniform-tempered} with $\varepsilon=\lambda$.  The finite
prefixes of the products and local $C^2$ constants are tight uniformly on
such an initial set.

\subsection{A random local stable ball}

Suppose $(z,z')\in\mathscr U_\Delta$ and put
$h=\Delta_z(z')\in E_z$.  For fixed compatible $(z,y)$, define the
one block compensated endpoint difference map by
\begin{align*}
\Psi_{z,y}(h)
=\Delta_{\Phi(z,y)}
\left(
\Phi\bigl(\Delta_z^{-1}(h),
y-\Gamma K_q(z,y)h\bigr)
\right)
\end{align*}
whenever
\begin{align*}
\|(h,-\Gamma K_q(z,y)h)\|_{E_z\oplus\mathsf Y}
<\mathsf r_\tau(z,y),
\end{align*}
where $\mathsf r_\tau$ is the base tube and chart radius given in  \eqref{eq:revision-chart-radius}. 
Here and below $E_z\oplus\mathsf Y$ carries its Hilbert product norm,
$\|h\|$ for $h\in E_z$ means $\|h\|_{E_z}$, and all operator norms
between tangent spaces are taken with respect to the quotient Hilbert
norms.  Since $D\Delta_z^{-1}(0)$ and
$D_{z'}\Delta_{\Phi(z,y)}(\Phi(z,y))$ are the natural identity
identifications, the chain rule gives
\begin{align}\label{eq:compensated-difference-linearisation}
D\Psi_{z,y}(0)
=J(z,y)-A(z,y)K_q(z,y)
=R_q(z,y).
\end{align}

Thus $\Psi_{z,y}$ describes the endpoint difference obtained by
perturbing the initial state by $h$ and compensating the convolution driver
by $-\Gamma K_q(z,y)h$. 

We keep the sampled chain $(X_n)$ and the fresh drivers $(Y_n)$ from
\eqref{e.XnYn}. Set
\begin{align*}
R_n=R_q(X_n,Y_n):E_{X_n}\longrightarrow E_{X_{n+1}},
\end{align*}
and
\begin{align}
C_n&=\mathsf C_\tau(X_n,Y_n)
\bigl(1+\|\Gamma\|_{\mathcal L(\mathcal H,\mathsf Y)}C_q\bigr)^2,
\label{eq:feedback-Taylor-Cn}\\
r_n&=\min\left\{
\frac{\mathsf r_\tau(X_n,Y_n)}
{1+\|\Gamma\|_{\mathcal L(\mathcal H,\mathsf Y)}C_q},
\frac1{2C_q}\right\},
\notag
\end{align}
where $\mathsf C_\tau$ is the two jet bounds in \eqref{eq:revision-C-def}, and  $C_q\ge1$ is the constant in
\eqref{eq:Kq-bound}--\eqref{eq:DKq-bound}.  For later use, put
\begin{align*}
\kappa_q=
\bigl(1+\|\Gamma\|_{\mathcal L(\mathcal H,\mathsf Y)}C_q\bigr)^2
+\bigl(1+\|\Gamma\|_{\mathcal L(\mathcal H,\mathsf Y)}C_q\bigr)+2C_q.
\end{align*}
Then, with $c_n$ defined in \eqref{e.cn},
\begin{align}\label{eq:Cn-rn-by-cn}
C_n+r_n^{-1}
\le\kappa_q\left(1+\mathsf C_\tau(X_n,Y_n)
+\mathsf r_\tau(X_n,Y_n)^{-1}\right)
\le\kappa_q e^{c_n}.
\end{align}

The linear residual contracts only infinitesimally, so we next close the
nonlinear Taylor iteration.  The next lemma packages the negative logarithmic
growth and tempered two jet constants into a positive random stable radius.
\begin{lemma}
\label{lem:random-stable-radius}
There is a measurable random radius $r_*(x,\omega)>0$ with the following
property. If $X'_0=x'$ satisfies $(x,x')\in\mathscr U_\Delta$ and
\begin{align*}
\|\Delta_x(x')\|_{E_x}<r_*(x,\omega),
\end{align*}
define recursively
\begin{align*}
h_n=\Delta_{X_n}(X'_n),\qquad
X'_{n+1}=\Phi\bigl(X'_n,Y_n-\Gamma K_q(X_n,Y_n)h_n\bigr).
\end{align*}
Then $(X_n,X'_n)\in\mathscr U_\Delta$ for every $n$, the recursion is
well defined, and
\begin{align}\label{eq:stable-decay}
\|h_n\|_{E_{X_n}}\le C(x,\omega)e^{-\lambda n}\|h_0\|_{E_x},
\qquad
\sum_{n\ge0}\|h_n\|_{E_{X_n}}^2<\infty.
\end{align}
Consequently $d_{\mathsf X}(X_n,X'_n)\to0$ exponentially.

If $D\subset\mathsf X$ is a set on which $\mathcal W$ is bounded, then
there are deterministic $r_0,p_0>0$ such that
\begin{align}\label{eq:stable-quantile}
\inf_{x\in D}\mathbf P_x\{r_*(x,\omega)\ge r_0\}\ge p_0.
\end{align}
\end{lemma}

\begin{proof}
If $\|h\|_{E_{X_n}}\le r_n$, then
\begin{align*}
\|(h,-\Gamma K_q(X_n,Y_n)h)\|_{E_{X_n}\oplus\mathsf Y}
\le\bigl(1+\|\Gamma\|_{\mathcal L(\mathcal H,\mathsf Y)}C_q\bigr)
\|h\|_{E_{X_n}}
\le \mathsf r_\tau(X_n,Y_n).
\end{align*}
Hence the whole segment
$s(h,-\Gamma K_q(X_n,Y_n)h)$, $0\le s\le1$, remains in the
state and driver $C^2$ tube of \Cref{subsec:global-two-jet-revision}. In particular,
\eqref{eq:projective-chart-denominator} shows that the raw fibre endpoint
along this segment satisfies
\begin{align*}
\langle\rho_\tau^e,\mathsf p_\tau\rangle\ge\frac{7d_\tau}{8}>0,
\end{align*}
so the fixed affine output chart centred at $X_{n+1}=\Phi(X_n,Y_n)$ is
valid throughout the segment. Taylor's theorem,
\eqref{eq:compensated-difference-linearisation}, \eqref{eq:feedback-Taylor-Cn} and
\eqref{eq:projective-chart-two-jet-bound} therefore give
\begin{align}\label{eq:nonlinear-residual}
\|\Psi_{X_n,Y_n}(h)-R_nh\|_{E_{X_{n+1}}}
\le C_n\|h\|_{E_{X_n}}^2,\qquad \|h\|_{E_{X_n}}\le r_n.
\end{align}

Define the accumulated linear residual product by
\begin{align*}
A_0=1,\qquad
A_n=\prod_{k<n}\bigl(\|R_k\|_{\mathcal L(E_{X_k},E_{X_{k+1}})}+q\bigr)
=\exp\left(\sum_{k<n}a_q(X_k,Y_k)\right).
\end{align*}
By \Cref{prop:uniform-occupation}, the last violation time is finite
almost surely.  Hence there is an almost surely finite random integer
$N_*=N_*(x,\omega)$ such that
\begin{align*}
\frac1n\sum_{k<n}a_q(X_k,Y_k)\le-4\lambda
\qquad\text{for every }n\ge N_*,
\end{align*}
and hence $A_n\le e^{-4\lambda n}$
eventually. Moreover, \eqref{eq:uniform-tempered} and
\eqref{eq:Cn-rn-by-cn} imply
\begin{align*}
C_n+r_n^{-1}\le\kappa_qe^{\lambda n}
\end{align*}
for all sufficiently large $n$. Consequently,
\begin{align*}
M(x,\omega):=
\max\left\{1,\kappa_q,\sup_{n\ge0}A_ne^{3\lambda n},
\sup_{n\ge0}(C_n+r_n^{-1})e^{-\lambda n}\right\}
\end{align*}
is finite and measurable almost surely, and
\begin{align}\label{eq:stable-global-bounds}
A_n\le Me^{-3\lambda n},\qquad
C_n+r_n^{-1}\le Me^{\lambda n},\qquad n\ge0.
\end{align}

Define measurable
\begin{align}\label{eq:rstar-formula}
r_*(x,\omega)=\frac12\inf_{n\ge0}
\min\left\{\frac{r_n}{A_n},
\frac{q}{(1+C_n)A_n}\right\}.
\end{align}
Since $M\ge1$, \eqref{eq:stable-global-bounds} gives
\begin{align*}
\frac{r_n}{A_n}\ge M^{-2}e^{2\lambda n},\qquad
\frac{q}{(1+C_n)A_n}\ge\frac{q}{2M^2}e^{2\lambda n},
\end{align*}
and hence $r_*(x,\omega)>0$ almost surely.

Let $\|h_0\|_{E_x}<r_*(x,\omega)$. We prove inductively that
\begin{align}\label{eq:hn-An}
\|h_n\|_{E_{X_n}}\le A_n\|h_0\|_{E_x}.
\end{align}
Suppose \eqref{eq:hn-An} holds at time $n$. By
\eqref{eq:rstar-formula},
\begin{align*}
\|h_n\|_{E_{X_n}}<\frac{r_n}{2},\qquad
C_n\|h_n\|_{E_{X_n}}<\frac q2.
\end{align*}
Thus the perturbed state and driver segment lies inside the $\mathsf r_\tau(X_n,Y_n)$ tube,
so \eqref{eq:nonlinear-residual} applies and the affine output chart at
$X_{n+1}$ is valid. Therefore
\begin{align*}
h_{n+1}
=\Delta_{X_{n+1}}(X'_{n+1})
=\Psi_{X_n,Y_n}(h_n)
\end{align*}
is well defined, and \eqref{eq:nonlinear-residual} followed by \eqref{eq:hn-An} implies
\begin{align*}
\|h_{n+1}\|_{E_{X_{n+1}}}
&\le\|R_n\|_{\mathcal L(E_{X_n},E_{X_{n+1}})}
\|h_n\|_{E_{X_n}}+C_n\|h_n\|_{E_{X_n}}^2\\
&\le\bigl(\|R_n\|_{\mathcal L(E_{X_n},E_{X_{n+1}})}+q\bigr)
\|h_n\|_{E_{X_n}}
\le A_{n+1}\|h_0\|_{E_x}.
\end{align*}
In particular, the validity of the affine output chart implies
$(X_{n+1},X'_{n+1})\in\mathscr U_\Delta$, while
\eqref{eq:rstar-formula} at the index $n+1$ gives
$\|h_{n+1}\|_{E_{X_{n+1}}}<r_{n+1}/2$. Thus the induction continues for
all $n$.

By \eqref{eq:stable-global-bounds},
\begin{align*}
\|h_n\|_{E_{X_n}}\le Me^{-3\lambda n}\|h_0\|_{E_x},
\end{align*}
which proves \eqref{eq:stable-decay}. Since $r_n\le \mathsf r_\tau(X_n,Y_n)\le1/8$,
the metric comparison \eqref{eq:state-chart-metric-comparison} converts
this coordinate decay into exponential decay in $d_{\mathsf X}$.

It remains to prove the uniform positive quantile. Let $D$ be a set on
which $\mathcal W$ is bounded. By \eqref{eq:uniform-last-exit}, choose
$N$ so large that, uniformly for $x\in D$,
\begin{align*}
\mathbf P_x\left\{
\frac1n\sum_{k<n}a_q(X_k,Y_k)\le-4\lambda
\text{ for every }n\ge N\right\}\ge\frac78.
\end{align*}
Increasing $N$ if necessary, \eqref{eq:uniform-tempered} with
$\varepsilon=\lambda$ gives
\begin{align*}
\mathbf P_x\left\{c_n\le\lambda n
\text{ for every }n\ge N\right\}\ge\frac78
\end{align*}
uniformly for $x\in D$.

For the finite prefix, \eqref{eq:conditional-block-moments}, the drift \eqref{eq:joint-V-drift}, and $c_n\le L(X_n,Y_n)$ give
\begin{align*}
\sup_{x\in D}\sup_{n\ge0}
\mathbf E_x\left[L(X_n,Y_n)^p+c_n^p\right]<\infty.
\end{align*}
Hence, since only $0\le n<N$ are involved, we may choose $B<\infty$ so
large that, uniformly for $x\in D$,
\begin{align*}
\mathbf P_x\left\{
\max_{0\le n<N}\bigl(L(X_n,Y_n)+c_n\bigr)\le B\right\}\ge\frac78.
\end{align*}
On this event, by \eqref{eq:R-global}, $a_q(X_n,Y_n)\le L(X_n,Y_n)\le B$ for $n<N$, and
\eqref{eq:Cn-rn-by-cn} gives
\begin{align*}
\max_{0\le n\le N}A_n\le e^{NB},\qquad
\max_{0\le n<N}(C_n+r_n^{-1})\le\kappa_qe^B.
\end{align*}
Set
\begin{align*}
M_0=\max\{1,e^{NB},\kappa_qe^B\}.
\end{align*}
On the intersection of the three preceding events, whose probability is at
least $5/8$ uniformly for $x\in D$, we therefore have
\begin{align*}
A_n\le M_0,\qquad C_n+r_n^{-1}\le M_0,\qquad 0\le n<N,
\end{align*}
while for $n\ge N$,
\begin{align*}
A_n\le e^{-4\lambda n},\qquad
C_n+r_n^{-1}\le\kappa_qe^{\lambda n}.
\end{align*}
For $n<N$ these bounds imply
\begin{align*}
\frac{r_n}{A_n}\ge M_0^{-2},\qquad
\frac{q}{(1+C_n)A_n}\ge\frac{q}{M_0(1+M_0)},
\end{align*}
whereas for $n\ge N$ they give
\begin{align*}
\frac{r_n}{A_n}\ge\kappa_q^{-1}e^{3\lambda n}\ge\kappa_q^{-1},\qquad
\frac{q}{(1+C_n)A_n}\ge\frac{q}{1+\kappa_q}e^{3\lambda n}
\ge\frac{q}{1+\kappa_q}.
\end{align*}
Consequently, on this intersection,
\begin{align*}
r_*(x,\omega)\ge
\frac12\min\left\{M_0^{-2},
\frac{q}{M_0(1+M_0)},\kappa_q^{-1},\frac{q}{1+\kappa_q}\right\}
=:r_0>0.
\end{align*}
Thus
\begin{align*}
\inf_{x\in D}\mathbf P_x\{r_*(x,\omega)\ge r_0\}\ge\frac58,
\end{align*}
and \eqref{eq:stable-quantile} follows, for instance with $p_0=1/2$.
\end{proof}

\subsection{Triangular blockwise Ramer transformations}
To realise the feedback on a genuine Wiener space, fix any $s>1$ and put
\begin{align}\label{eq:canonical-block-AWS}
 \Omega_\tau=C_0([0,\tau];H^{-s}),
 \qquad
 \mathcal H_0
 =\left\{
 \mathbf g(t)=\int_0^t g_r\,\dd r:g\in\mathcal H
 \right\},
\end{align}
equipped with the canonical Wiener measure $\mathbf P$.
Since the embedding $H\hookrightarrow H^{-s}$ is Hilbert--Schmidt,
the cylindrical Wiener process has a genuine realisation on
$\Omega_\tau$, whose Cameron--Martin space is
$\mathcal H_0\simeq\mathcal H$.  Throughout this subsection,
$g\in\mathcal H$ denotes the Cameron--Martin control,
$\mathbf g(t)=\int_0^t g_r\,\dd r$ its Wiener path realization, and
$\Gamma g\in\mathsf Y$ the corresponding displacement of the convolution
driver.

The stochastic convolution admits a Borel realisation
$Y:\Omega_\tau\to\mathsf Y$ such that
\begin{align}\label{eq:canonical-block-translation}
 \mathbf P_Y=Y_\#\mathbf P,
 \qquad
 Y(\omega+\mathbf g)=Y(\omega)+\Gamma g.
\end{align}
More precisely, the second identity holds simultaneously for every
$g\in\mathcal H$ on one Cameron--Martin invariant Borel vector subspace
$\Omega_\tau^*\subset\Omega_\tau$ of full measure.  In what follows we
work on this fixed set and suppress the superscript whenever no ambiguity
can arise.

The selector is anticipative within each block, while the block
transformations are triangular in time.  Thus adapted Girsanov theory is not
available inside a block, and the correct replacement is Ramer's
change of variables theorem.  We use the formulation in
\cite[Theorem~3.3 and equation~(3.10)]{UZ97}, which supplies the strengthened
Wiener space form of \cite{Ramer} and gives the entropy estimate below.

\begin{lemma}
\label{lem:block-Ramer-entropy}
Let $u:\Omega_\tau\to\mathcal H$ be Borel and suppose that
\begin{align*}
 u\in\mathbb D^{1,p}(\mathcal H)
 \qquad\text{for every finite }p,
\end{align*}
where $\mathbb D^{1,p}(\mathcal H)$ is the usual Malliavin Sobolev space \cite{Bogachev10}.
Define
\begin{align*}
 \mathbf u(\omega)(t)=\int_0^t u(\omega)_r\,\dd r,
 \qquad
 T_u(\omega)=\omega+\mathbf u(\omega).
\end{align*}
Suppose moreover that, for some $d\ge0$,
\begin{align}\label{eq:Ramer-small}
 \|u\|_{\mathcal H}\le Cd,\qquad
 \|D_{\mathcal H}u\|_{\HS}\le Cd,\qquad
 \|D_{\mathcal H}u\|_{\rm op}\le\frac12
\end{align}
$\mathbf P$-almost surely. Then $(T_u)_\#\mathbf P\sim\mathbf P$
and
\begin{align*}
 \operatorname{Ent}\bigl((T_u)_\#\mathbf P\mid\mathbf P\bigr)
 \le C_{\rm Ram}d^2,
\end{align*}
where $\operatorname{Ent}(\cdot\mid\cdot)$ denotes relative entropy and
$C_{\rm Ram}$ depends only on the fixed constant $C$ in
\eqref{eq:Ramer-small}.
\end{lemma}

\begin{proof}
By \cite[Theorem~3.3]{UZ97}, the condition
$\|D_{\mathcal H}u\|_{\rm op}\le1/2$ implies that $T_u$ is almost surely
bijective and that $(T_u)_\#\mathbf P$ and $\mathbf P$ are equivalent.
Moreover, \cite[equation~(3.10)]{UZ97} gives
\begin{align*}
 (T_u)_\#(\Lambda_u\mathbf P)=\mathbf P,
 \qquad
 \Lambda_u
 =
 \left|\det_2(I+D_{\mathcal H}u)\right|
 \exp\left(
   -\delta u-\frac12\|u\|_{\mathcal H}^2
 \right),
\end{align*}
where $\delta$ denotes the Malliavin divergence operator and
\begin{align*}
 \det_2(I+B)=\det\bigl((I+B)e^{-B}\bigr)
\end{align*}
is the Carleman--Fredholm determinant. If
$\nu=(T_u)_\#\mathbf P$, then
\begin{align*}
 \frac{\dd\nu}{\dd\mathbf P}(T_u\omega)
 =\Lambda_u(\omega)^{-1},
 \qquad
 \operatorname{Ent}(\nu\mid\mathbf P)
 =\mathbf E[-\log\Lambda_u].
\end{align*}
Since $u\in\mathbb D^{1,2}(\mathcal H)$, one has
$u\in\operatorname{Dom}\delta$ and $\mathbf E\delta u=0$. Furthermore,
when $\|B\|_{\rm op}\le1/2$,
\begin{align*}
 \left|\log\left|\det_2(I+B)\right|\right|
 \le C\|B\|_{\HS}^2.
\end{align*}
Consequently,
\begin{align*}
 \operatorname{Ent}(\nu\mid\mathbf P)
 \le
 C\mathbf E\|D_{\mathcal H}u\|_{\HS}^2
 +\frac12\mathbf E\|u\|_{\mathcal H}^2\le C_{\rm Ram}d^2.
\end{align*}
\end{proof}

The abstract Ramer estimate must now be verified for the actual selector
feedback.  The finite patch construction makes this possible because all
driver dependence is carried by smooth Hilbert cutoffs.
\begin{lemma}
\label{lem:selector-Ramer-domain}
For $z\in\mathsf X$ and $h\in E_z$, define
\begin{align*}
 u_{z,h}(\omega)
 =-K_q(z,Y(\omega))h
 =-\sum_i\chi_i(z,Y(\omega))K_i(z)h.
\end{align*}
Then $(z,h,\omega)\mapsto u_{z,h}(\omega)$ is Borel and, for every fixed
$(z,h)$,
\begin{align*}
 u_{z,h}\in\mathbb D^{1,p}(\mathcal H)
 \qquad\text{for every finite }p.
\end{align*}
In addition,
\begin{align}
 \|u_{z,h}\|_{\mathcal H}
 +\|D_{\mathcal H}u_{z,h}\|_{\HS}
 &\le C_q\|h\|,
 \label{eq:selector-Ramer-HS}\\
 \|D_{\mathcal H}u_{z,h}\|_{\rm op}
 &\le C_q\|h\|.
 \label{eq:selector-Ramer-op}
\end{align}
Furthermore, 
\begin{align}\label{eq:feedback-realizes-driver-shift}
 Y\bigl(\omega+\mathbf u_{z,h}(\omega)\bigr)
 =
 Y(\omega)-\Gamma K_q(z,Y(\omega))h,
\end{align}
and whenever $\|h\|\le(2C_q)^{-1}$,
\begin{align}\label{eq:Ramer-block-application}
 (T_{z,h})_\#\mathbf P&\sim\mathbf P,\notag\\
 \operatorname{Ent}\bigl((T_{z,h})_\#\mathbf P\mid\mathbf P\bigr)
 &\le C_{\rm Ram}\|h\|^2,
\end{align}
where
\begin{align*}
 T_{z,h}(\omega)
 =\omega+\mathbf u_{z,h}(\omega).
\end{align*}
\end{lemma}

\begin{proof}
We verify the Malliavin regularity directly. Let $(e_j)_{j\ge1}$ be an
orthonormal basis of $\mathcal H$, let $P_N$ be the projection onto its
first $N$ elements, and set
\begin{align*}
 Y_N=\sum_{j=1}^N(\delta e_j)\Gamma e_j,
 \qquad
 F_{i,z}^{(N)}=\chi_i(z,Y_N).
\end{align*}
The variables $F_{i,z}^{(N)}$ are smooth cylindrical. Since
$\Gamma:\mathcal H\to\mathsf Y$ is Hilbert--Schmidt by
\Cref{lem:driver-HS},
\begin{align*}
 Y_N\longrightarrow Y
 \qquad\text{in }L^p(\Omega;\mathsf Y)
\end{align*}
for every finite $p$. The cylindrical chain rule gives
\begin{align*}
 D_{\mathcal H}F_{i,z}^{(N)}
 =
 P_N\Gamma^*D_y\chi_i(z,Y_N).
\end{align*}
The first two driver derivatives of the finitely many cutoffs are bounded.
Hence $F_{i,z}^{(N)}\to F_{i,z}:=\chi_i(z,Y)$ in $L^p$, while
\begin{align*}
 P_N\Gamma^*D_y\chi_i(z,Y_N)
 \longrightarrow
 \Gamma^*D_y\chi_i(z,Y)
 \qquad\text{in }L^p(\Omega;\mathcal H).
\end{align*}
Closability of the Malliavin derivative therefore gives
\begin{align}\label{eq:selector-Malliavin-chain}
 F_{i,z}\in\mathbb D^{1,p},
 \qquad
 D_{\mathcal H}F_{i,z}
 =
 \Gamma^*D_y\chi_i(z,Y)
\end{align}
for every finite $p$.

Since $K_i(z)h$ is deterministic, we have 
\begin{align*}
 u_{z,h}
 =-\sum_iF_{i,z}K_i(z)h
 \in\mathbb D^{1,p}(\mathcal H).
\end{align*}
Equations \eqref{eq:Kq-bound}, \eqref{eq:DKq-bound}, and
\eqref{eq:selector-Malliavin-chain} give
\eqref{eq:selector-Ramer-HS}--\eqref{eq:selector-Ramer-op}. Joint Borel
measurability follows from the explicit finite patch construction.

Finally, \eqref{eq:canonical-block-translation} holds simultaneously for every
$g\in\mathcal H$.  We may therefore substitute the random value
$g=u_{z,h}(\omega)$ and obtain
\begin{align*}
 Y\bigl(\omega+\mathbf u_{z,h}(\omega)\bigr)
 &=
 Y(\omega)+\Gamma u_{z,h}(\omega)\\
 &=
 Y(\omega)-\Gamma K_q(z,Y(\omega))h,
\end{align*}
which is \eqref{eq:feedback-realizes-driver-shift}. If
$\|h\|\le(2C_q)^{-1}$, then
\eqref{eq:selector-Ramer-op} gives
$\|D_{\mathcal H}u_{z,h}\|_{\rm op}\le1/2$.
Applying \Cref{lem:block-Ramer-entropy} with $d=\|h\|$ proves
\eqref{eq:Ramer-block-application}.
\end{proof}

A one block quasi invariance statement is not yet enough for the infinite
coupling.  Write $P_{\tau,[\infty]}\delta_x$ for the law on
$\mathsf X^{\mathbb N_0}$ of the sampled chain $(X_n)_{n\ge0}$ started
from $x$.  The next lemma concatenates the anticipative shifts triangularly,
controls the entropy by an entropy threshold, and passes to the full path space.
\begin{lemma}
\label{lem:infinite-Ramer}
Set the one block Ramer threshold $d_{\rm Ram}=(2C_q)^{-1}$. For $R>0$, let $h_n$ denote the proposed feedback displacement, defined by $h_n=\Delta_{X_n}(X'_n)$ on $\{(X_n,X'_n)\in\mathscr U_\Delta\}$ and by $h_n=0$ otherwise, and define
\begin{align*}
 \tau_R=\inf\left\{n\ge0:(X_n,X'_n)\notin\mathscr U_\Delta\ \text{or}\ \|h_n\|>d_{\rm Ram}\ \text{or}\ C_{\rm Ram}\sum_{k=0}^n\|h_k\|^2>R\right\}.
\end{align*}
At the $n$-th step, apply $T_{X_n,h_n}$ to the fresh Wiener input of the second component when $n<\tau_R$, and use the unshifted Wiener input when $n\ge\tau_R$. Denote the resulting generalized coupling by $\mathbf Q^R_{x,x'}$. Then the following conclusions hold.
\begin{itemize}
 \item The first marginal of $\mathbf Q^R_{x,x'}$ is $P_{\tau,[\infty]}\delta_x$, while its second marginal $\mathbf Q^{R,2}_{x,x'}$ satisfies
 \begin{align}\label{eq:triangular-entropy}
  \mathbf Q^{R,2}_{x,x'}\ll P_{\tau,[\infty]}\delta_{x'},\qquad \operatorname{Ent}\bigl(\mathbf Q^{R,2}_{x,x'}\mid P_{\tau,[\infty]}\delta_{x'}\bigr)\le R.
 \end{align}
 Moreover, $(x,x')\mapsto\mathbf Q^R_{x,x'}$ may be chosen Borel.
 \item Define the countable mixture over integer entropy thresholds
 \begin{align}\label{eq:entropy-threshold-mixture}
  \mathbf Q_{x,x'}=\sum_{m=1}^\infty 2^{-m}\mathbf Q^m_{x,x'}.
 \end{align}
 Then $\mathbf Q_{x,x'}$ is a Borel generalized coupling kernel, its first marginal is $P_{\tau,[\infty]}\delta_x$, and its second marginal is absolutely continuous with respect to $P_{\tau,[\infty]}\delta_{x'}$.
 \item For the corresponding unstopped feedback construction, let
 \begin{align*}
  \mathcal E_\infty=\bigcup_{m=1}^\infty\{\tau_m=\infty\}.
 \end{align*}
 be the event that at least one integer entropy threshold is never reached.
 Let $\mathsf D$ be a Borel event of paired sampled state paths. If there is a measurable event $\mathcal E\subset\mathcal E_\infty$ of positive probability for the unstopped construction such that the unstopped paired path belongs to $\mathsf D$ on $\mathcal E$, then $\mathbf Q_{x,x'}(\mathsf D)>0$.
\end{itemize}
\end{lemma}

\begin{proof}
Let $\mathscr G_n$ be the full joint sigma field generated by the coupling before the $n$-th fresh Wiener input is sampled, and let $\mathscr F'_n=\sigma(W'_0,\ldots,W'_{n-1})$ be the past of the second noise coordinate. The random variables $X_n$, $X'_n$, $h_n$, and $\mathbf 1_{\{n<\tau_R\}}$ are all $\mathscr G_n$ measurable.

Conditionally on $\mathscr G_n$, the fresh Wiener input $W_n$ driving the first component has law $\mathbf P$. On $\{n<\tau_R\}$, the corresponding noise input of the second component is $W'_n=T_{X_n,h_n}(W_n)$, while on $\{n\ge\tau_R\}$ it is unshifted and hence again has law $\mathbf P$. Therefore, if $\kappa_n=\mathcal L(W'_n\mid\mathscr G_n)$, then the
selector verification \Cref{lem:selector-Ramer-domain}, together with the
one block estimate \Cref{lem:block-Ramer-entropy}, gives
\begin{align*}
 \operatorname{Ent}(\kappa_n\mid\mathbf P)\le C_{\rm Ram}\|h_n\|^2\mathbf 1_{\{n<\tau_R\}}.
\end{align*}

We next pass from the full joint past to the second coordinate past. Put $\kappa'_n=\mathcal L(W'_n\mid\mathscr F'_n)$. Since $\mathscr F'_n\subset\mathscr G_n$, the tower property for conditional distributions gives, for every bounded Borel function $\varphi$ on $\Omega_\tau$,
\begin{align*}
 \int\varphi\,\dd\kappa'_n=\mathbf E\left[\int\varphi\,\dd\kappa_n\,\middle|\,\mathscr F'_n\right].
\end{align*}
Thus, after the value of the second past is fixed, $\kappa'_n$ is the barycentre of the $\mathscr G_n$ conditional laws $\kappa_n$ over the remaining uncertainty in the joint past. Convexity of relative entropy therefore yields
\begin{align*}
 \operatorname{Ent}(\kappa'_n\mid\mathbf P)\le \mathbf E\left[\operatorname{Ent}(\kappa_n\mid\mathbf P)\,\middle|\,\mathscr F'_n\right]\le \mathbf E\left[C_{\rm Ram}\|h_n\|^2\mathbf 1_{\{n<\tau_R\}}\,\middle|\,\mathscr F'_n\right].
\end{align*}

Let $\nu_N^R=\mathcal L(W'_0,\ldots,W'_N)$. The entropy chain rule, followed by the preceding conditional estimate and the tower property, gives
\begin{align*}
 \operatorname{Ent}(\nu_N^R\mid\mathbf P^{\otimes(N+1)})
 &=\sum_{n=0}^N\mathbf E\,\operatorname{Ent}\bigl(\mathcal L(W'_n\mid\mathscr F'_n)\mid\mathbf P\bigr)\\
 &\le\sum_{n=0}^N\mathbf E\left[C_{\rm Ram}\|h_n\|^2\mathbf 1_{\{n<\tau_R\}}\right]
 =C_{\rm Ram}\mathbf E\sum_{n=0}^N\|h_n\|^2\mathbf 1_{\{n<\tau_R\}}.
\end{align*}
By the prospective definition of $\tau_R$, the shift at index $n$ is executed only when $C_{\rm Ram}\sum_{k=0}^n\|h_k\|^2\le R$. Hence, pathwise,
\begin{align*}
 C_{\rm Ram}\sum_{n=0}^\infty\|h_n\|^2\mathbf 1_{\{n<\tau_R\}}\le R,
\end{align*}
and consequently
\begin{align}\label{eq:path-entropy-uniform}
 \operatorname{Ent}(\nu_N^R\mid\mathbf P^{\otimes(N+1)})\le R
\end{align}
for every $N\ge0$.

To pass to the entire second noise sequence, put $\boldsymbol\Omega=\Omega_\tau^{\mathbb N_0}$ and $\boldsymbol{\mathbf P}=\mathbf P^{\otimes\mathbb N_0}$, let $\pi_n(\omega_0,\omega_1,\ldots)=\omega_n$ be the $n$-th coordinate map, and set $\mathscr A_N=\sigma(\pi_0,\ldots,\pi_N)$. Let $\boldsymbol\nu^R=\mathcal L((W'_n)_{n\ge0})$, which is already defined by the stopped recursion, and let $f_N$ be the density of its restriction to $\mathscr A_N$ with respect to the restriction of $\boldsymbol{\mathbf P}$ to $\mathscr A_N$. Equivalently, $f_N$ is the density of $\nu_N^R$ with respect to $\mathbf P^{\otimes(N+1)}$, viewed as an $\mathscr A_N$ measurable function on $\boldsymbol\Omega$.

For every $A\in\mathscr A_N$, consistency of the finite dimensional marginals gives
\begin{align*}
 \int_A f_{N+1}\,\dd\boldsymbol{\mathbf P}=\boldsymbol\nu^R(A)=\int_A f_N\,\dd\boldsymbol{\mathbf P}.
\end{align*}
By the defining property of conditional expectation,
\begin{align*}
 f_N=\mathbf E_{\boldsymbol{\mathbf P}}[f_{N+1}\mid\mathscr A_N].
\end{align*}
Thus $(f_N)_{N\ge0}$ is a nonnegative $\boldsymbol{\mathbf P}$ martingale with $\mathbf E_{\boldsymbol{\mathbf P}}f_N=1$.

Set $\Psi(r)=r\log r-r+1$, with $0\log0=0$. Since $f_N$ has mean one,
\begin{align*}
 \mathbf E_{\boldsymbol{\mathbf P}}\Psi(f_N)=\mathbf E_{\boldsymbol{\mathbf P}}[f_N\log f_N]=\operatorname{Ent}(\nu_N^R\mid\mathbf P^{\otimes(N+1)})\le R
\end{align*}
by \eqref{eq:path-entropy-uniform}. Since $\Psi(r)/r\to\infty$ as $r\to\infty$, the de la Vall\'ee Poussin criterion shows that $(f_N)$ is uniformly integrable. Hence the martingale convergence theorem gives $f_N\to f$ almost surely and in $L^1(\boldsymbol{\mathbf P})$ for some $f\ge0$ with $\mathbf E_{\boldsymbol{\mathbf P}}f=1$.

Moreover, $\mathbf E_{\boldsymbol{\mathbf P}}[f\mid\mathscr A_N]=f_N$. Therefore, for every $A\in\mathscr A_N$,
\begin{align*}
 \int_A f\,\dd\boldsymbol{\mathbf P}=\int_A f_N\,\dd\boldsymbol{\mathbf P}=\boldsymbol\nu^R(A).
\end{align*}
Since $\bigcup_N\mathscr A_N$ generates the product Borel sigma field, a monotone class argument yields $\boldsymbol\nu^R=f\boldsymbol{\mathbf P}$. Finally, $\Psi\ge0$ and $f_N\to f$ almost surely, so Fatou's lemma gives
\begin{align*}
 \operatorname{Ent}(\boldsymbol\nu^R\mid\boldsymbol{\mathbf P})=\mathbf E_{\boldsymbol{\mathbf P}}\Psi(f)\le\liminf_{N\to\infty}\mathbf E_{\boldsymbol{\mathbf P}}\Psi(f_N)\le R.
\end{align*}

To pass from noise paths to state paths, let $\mathcal S_{x'}:\boldsymbol\Omega\to\mathsf X^{\mathbb N_0}$ be the ordinary sampled solution map: for $\omega'=(\omega'_n)_{n\ge0}$, set $\widetilde X_0=x'$ and recursively
\begin{align*}
 \widetilde X_{n+1}=\Phi(\widetilde X_n,Y(\omega'_n)).
\end{align*}
Thus $\mathcal S_{x'}(\omega')=(\widetilde X_n)_{n\ge0}$. The second component of our coupling is obtained by applying this same unmodified solution map to the transformed noise sequence $(W'_n)$; the feedback has already been encoded in the law of that noise sequence. Consequently,
\begin{align*}
 \mathbf Q^{R,2}_{x,x'}=(\mathcal S_{x'})_\#\boldsymbol\nu^R,\qquad P_{\tau,[\infty]}\delta_{x'}=(\mathcal S_{x'})_\#\boldsymbol{\mathbf P}.
\end{align*}
Relative entropy decreases under measurable pushforward, so the preceding infinite noise entropy estimate proves \eqref{eq:triangular-entropy}. The first noise coordinate is never changed, hence the first marginal is $P_{\tau,[\infty]}\delta_x$.

For each integer $m\ge1$, the endpoint map, the feedback, and the stopping rule are Borel in $(x,x')$, so $(x,x')\mapsto\mathbf Q^m_{x,x'}$ is a Borel kernel. Therefore \eqref{eq:entropy-threshold-mixture} defines a Borel generalized coupling kernel. Its first marginal remains $P_{\tau,[\infty]}\delta_x$. If $A$ is a Borel state path event with $P_{\tau,[\infty]}\delta_{x'}(A)=0$, then $\mathbf Q^{m,2}_{x,x'}(A)=0$ for every $m$ by \eqref{eq:triangular-entropy}, and hence the second marginal of the mixture also assigns zero mass to $A$. This proves its absolute continuity.

It remains to prove the last assertion. Realise the level $m$ stopped constructions and the unstopped construction from the same underlying sequence of fresh Wiener inputs. On $\{\tau_m=\infty\}$ the level $m$ construction never invokes its stopping rule and therefore coincides pathwise with the unstopped construction. If $\mathcal E\subset\mathcal E_\infty$ has positive probability, then
\begin{align*}
 \mathcal E=\bigcup_{m=1}^\infty\bigl(\mathcal E\cap\{\tau_m=\infty\}\bigr),
\end{align*}
so there is some $m\ge1$ such that $\mathbf P(\mathcal E\cap\{\tau_m=\infty\})>0$. If the unstopped paired path belongs to the Borel path event $\mathsf D$ on $\mathcal E$, then on this intersection the level $m$ paired path also belongs to $\mathsf D$. Hence
\begin{align*}
 \mathbf Q^m_{x,x'}(\mathsf D)\ge \mathbf P\bigl(\mathcal E\cap\{\tau_m=\infty\}\bigr)>0,
\end{align*}
and \eqref{eq:entropy-threshold-mixture} gives
\begin{align*}
 \mathbf Q_{x,x'}(\mathsf D)\ge2^{-m}\mathbf Q^m_{x,x'}(\mathsf D)>0.
\end{align*}
This proves the final claim.
\end{proof}

\section{A prepared bridge and accessibility}
\label{sec:prepared}

This section constructs a prepared neighborhood that can both be reached
from arbitrary states and coupled locally.  A deterministic observability
estimate yields exact first order cancellation, while a finite dimensional
shell control argument gives accessibility of the prepared fibre.  This
control construction uses the same basic Fourier interaction and relaxation
mechanism as in \cite{AS05,AS06}, adapted to the linearised projective
dynamics.

We consider the deterministic base fibre skeleton corresponding to zero
driver, with initial state $(0,\phi)$, where
\begin{align*}
\phi(x)=\cos x_1+\cos x_2.
\end{align*}
With normalized Lebesgue measure, $\|\phi\|_2=1$ and $-\Delta\phi=\phi$.
For the zero driver and zero initial base state, the deterministic base
trajectory is $w_t\equiv0$, while the raw fibre solves
$\partial_t\rho_t=\nu\Delta\rho_t$ and hence
$\rho_t=e^{-\nu t}\phi$.  Its projective representative is therefore
constant: $[\rho_t]=[\phi]$.

We linearize the deterministic endpoint map at this zero driver skeleton.
In the fixed chart $\phi^\perp$, the augmented tangent generator on
$E_b=H^b\oplus\phi^\perp$ is
\begin{align}\label{eq:L-prepared}
\mathcal L(\xi,\eta)
=\left(\nu\Delta\xi,\nu(\Delta+1)\eta+C_\phi\xi\right),
\end{align}
where
\begin{align}\label{eq:Cphi}
C_\phi\xi
=-\Pi_{\phi^\perp}\bigl(B(\xi,\phi)+B(\phi,\xi)\bigr).
\end{align}
For $b>1$, $C_\phi:H^b\to H$ is bounded.  Let
$S_T=e^{T\cL}$ denote the uncontrolled linearised endpoint.  If $u$ solves
\begin{align*}
 \partial_tu=\cL u+(Qv,0),\qquad u(0)=0,
\end{align*}
then $A_Tv=u(T)$ denotes the controlled endpoint generated by the control
$v$.  Let $Q_b: H\to H^b$ denote $Q$ viewed as an operator with
values in $H^b$.  Its Hilbert adjoint is
\begin{align*}
 (Q_b^*p)_{\boldsymbol k,\iota}
 =q_{\boldsymbol k}^\iota|\boldsymbol k|^{2b}
 p_{\boldsymbol k,\iota},
 \qquad p\in H^b.
\end{align*}

The Banach transpose of $C_\phi:H^1\to H$ is the bounded operator $ C_\phi^{*,0}:H^r\cap\phi^\perp\longrightarrow H^{r-1}$ for $r\in\mathbb R$ given by 
\begin{align}\label{eq:Cstar0}
      C_\phi^{*,0}q
      =\bigl[\Id-(-\Delta)^{-1}\bigr]J(\phi,q),
\end{align}
where 
\begin{align*}
      J(f,g)=\partial_1f\,\partial_2g-
             \partial_2f\,\partial_1g.
\end{align*}
Here \eqref{eq:Cstar0} is understood distributionally when $q\in H$;
it follows first for trigonometric polynomials by integration by parts and then by density. The Hilbert adjoint of
$C_\phi:H^b\to H$, relative to the $H^b$ inner product in the base
coordinate, is the bounded map $H\to H^b$ given by 
\begin{align*}
 C_{\phi,b}^*=(-\Delta)^{-b}C_\phi^{*,0}.
\end{align*}

\subsection{Observability and Douglas factorization}
For $N\in\N$, let $E_N=\Ker(-\Delta-N)$ be the $N$-th Laplace
shell and let $\Pi_N$ denote the $L^2$ orthogonal projection onto $E_N$.
Define
\begin{align*}
 T_N=(\Id-\Pi_1)J(\phi,\cdot)|_{E_N}.
\end{align*}
Thus $T_N$ is the part of the Jacobian interaction visible away from the
first shell.

The prepared observability estimate begins with coercivity on each Laplace
shell.  The next lemma shows that the coupling operator observes every
shell direction except the distinguished prepared direction $[\phi]$.
\begin{lemma}\label{lem:shell}
For $q_N\in E_N$, with $q_1\perp\phi$ when $N=1$,
\begin{align*}
 \norm{C_\phi^{*,0}q_N}_2
 \ge \frac14\norm{q_N}_2.
\end{align*}
In particular,
\begin{align*}
 \Ker(C_\phi^{*,0}|_{E_N})
 =\begin{cases}
   \{0\},&N\ne1,\\
   \operatorname{span}\{\phi\},&N=1.
  \end{cases}
\end{align*}
\end{lemma}
\begin{proof}
      It suffices to work in the complexification, since all the operators involved
      preserve the real subspace and commute with complex conjugation.  Let $e_{\boldsymbol n}(x)=e^{\mathrm i\boldsymbol n\cdot x}$, 
      \begin{align*}
      \Lambda_N=\{\boldsymbol n\in\mathbb Z^2:|\boldsymbol n|^2=N\},\qquad
      E_N^{\mathbb C}=\operatorname{span}_{\mathbb C}
      \{e_{\boldsymbol n}:\boldsymbol n\in\Lambda_N\}.
      \end{align*}
      Put
      $\mathcal A=\{\pm\boldsymbol e_1,\pm\boldsymbol e_2\}$.  Since
      \begin{align*}
      \phi=\frac12\sum_{\boldsymbol m\in\mathcal A}e_{\boldsymbol m},\qquad
      J(e_{\boldsymbol m},e_{\boldsymbol n})=-\det(\boldsymbol m,\boldsymbol n)e_{\boldsymbol m+\boldsymbol n}.
      \end{align*}
      we have, for every $\boldsymbol n\in\Lambda_N$,
      \begin{align}\label{eq:TN-column}
      T_Ne_{\boldsymbol n}
      =-\frac12\sum_{\substack{\boldsymbol m\in\mathcal A\\
      |\boldsymbol n+\boldsymbol m|^2\neq1}}
      \det(\boldsymbol m,\boldsymbol n)e_{\boldsymbol n+\boldsymbol m}.
      \end{align}
      Thus, before the shell one outputs are removed, the column indexed by
      $\boldsymbol n=(n_1,n_2)$ is
      \begin{align*}
      J(\phi,e_{\boldsymbol n})
      =-\frac{n_2}{2}e_{\boldsymbol n+\boldsymbol e_1}
      +\frac{n_2}{2}e_{\boldsymbol n-\boldsymbol e_1}
      +\frac{n_1}{2}e_{\boldsymbol n+\boldsymbol e_2}
      -\frac{n_1}{2}e_{\boldsymbol n-\boldsymbol e_2}.
      \end{align*}
      
      We first consider $N>5$.  No term in \eqref{eq:TN-column} is then removed: indeed, if $\boldsymbol n+\boldsymbol m\in\mathcal A$ for some
      $\boldsymbol m\in\mathcal A$, then
      $\boldsymbol n\in\mathcal A-\mathcal A$, and hence
      $|\boldsymbol n|^2\le4$.  Consequently every column has squared norm
      \begin{align*}
      \|T_Ne_{\boldsymbol n}\|_2^2
      =\frac14(2n_1^2+2n_2^2)=\frac N2.
      \end{align*}
      
      Suppose that the columns indexed by two distinct points
      $\boldsymbol n,\boldsymbol n'\in\Lambda_N$ have a common output.  Then
      $\boldsymbol n+\boldsymbol m=\boldsymbol n'+\boldsymbol m'$ for some
      $\boldsymbol m,\boldsymbol m'\in\mathcal A$, and therefore
      $\boldsymbol d:=\boldsymbol n'-\boldsymbol n$ belongs to
      \begin{align}\label{e.021101}
            (\mathcal A-\mathcal A)\setminus\{0\} = \left\{(\pm2,0),\,(0,\pm2),\,(\pm1,\pm1)\right\}.
      \end{align}
      Since $\boldsymbol n$ and $\boldsymbol n'$ lie on the same shell,
      $|\boldsymbol n+\boldsymbol d|^2=|\boldsymbol n|^2$.  Substitution of
      the possible values of $\boldsymbol d$ shows that an overlap can occur only
      when
      \begin{align}\label{eq:overlap-lines}
      n_1=\pm1,\qquad n_2=\pm1,\qquad
      n_1+n_2=\pm1,\qquad n_1-n_2=\pm1.
      \end{align}
      
      Every lattice point lying on two distinct lines in
      \eqref{eq:overlap-lines} has squared norm at most $5$.  Hence, for $N>5$,
      each vertex of the column overlap graph has degree at most one.  After
      reordering the Fourier basis, the Gram matrix $T_N^*T_N$ is therefore a
      direct sum of one dimensional blocks and two dimensional blocks.
      
      For an axial displacement from \eqref{e.021101}, symmetry reduces the two indices to
      $\boldsymbol n=(-1,a)$ and $\boldsymbol n'=(1,a)$, with $N=1+a^2$.
      The columns have the single common output $(0,a)$, at which their
      coefficients are $-a/2$ and $a/2$.  Their Gram off diagonal entry is
      therefore $-a^2/4=-(N-1)/4$.
      
      For a diagonal displacement from \eqref{e.021101}, symmetry reduces the indices to
      $\boldsymbol n=(a,a+1)$ and $\boldsymbol n'=(a+1,a)$, with
      $N=a^2+(a+1)^2$.  The two common outputs contribute respectively
      $-a^2/4$ and $-(a+1)^2/4$, so the Gram off diagonal entry is $-N/4$.
      
      Thus every two dimensional Gram block has the form
      \begin{align*}
      \begin{pmatrix}N/2&c\\ \overline c&N/2\end{pmatrix},
      \qquad |c|\le\frac N4.
      \end{align*}
      Its least eigenvalue is $N/2-|c|\ge N/4$.  The same lower bound is
      immediate for each one dimensional block, and hence
      \begin{align}\label{eq:TN-large}
      \|T_Nq_N\|_2\ge\frac{\sqrt N}{2}\|q_N\|_2,
      \qquad q_N\in E_N^{\mathbb C},\quad N>5.
      \end{align}
      
      The nonempty shells with $N\le5$ are $N=1,2,4,5$.  The exact matrices by a direct computation give
      \begin{align*}
      \operatorname{spec}(T_1^*T_1)
      =\left\{0,\frac12,\frac12,1\right\},\,\quad
      T_2^*T_2=\frac12\Id,\,\quad
      T_4^*T_4=2\Id,\,\quad
      T_5^*T_5\ge\frac14\Id.
      \end{align*}
      The kernel of $T_1$ is the equal coefficient vector on $\mathcal A$, namely
      $\phi$.  Therefore, on $E_1^{\mathbb C}\cap\phi^\perp$ and on each of the
      other nonempty low shells,
      $\|T_Nq_N\|_2\ge\frac12\|q_N\|_2$.  Together with
      \eqref{eq:TN-large}, this yields
      \begin{align}\label{eq:TN-uniform}
      \|T_Nq_N\|_2\ge\frac12\|q_N\|_2
      \end{align}
      for every nonempty shell, with the restriction $q_1\perp\phi$ when $N=1$.
      
      Finally,
      \begin{align*}
      C_\phi^{*,0}q_N
      =\bigl[\Id-(-\Delta)^{-1}\bigr]T_Nq_N,
      \end{align*}
      because $\bigl[\Id-(-\Delta)^{-1}\bigr]\Pi_1=0$.  Moreover,
      $T_Nq_N$ has neither a constant component nor a shell one component.
      On every remaining shell $E_\ell$, where $\ell\ge2$, the multiplier
      $\Id-(-\Delta)^{-1}$ equals multiplication by $1-\ell^{-1}$, whose
      absolute value is at least $1/2$.  Orthogonality of the Laplace shells
      therefore gives
      \begin{align*}
      \|C_\phi^{*,0}q_N\|_2^2
      &=\sum_{\ell\ge2}\left(1-\frac1\ell\right)^2
        \|\Pi_\ell T_Nq_N\|_2^2
       \ge\frac14\|T_Nq_N\|_2^2.
      \end{align*}
      Combining this with \eqref{eq:TN-uniform} proves
      \begin{align*}
      \|C_\phi^{*,0}q_N\|_2\ge\frac14\|q_N\|_2.
      \end{align*}
      
      For $N\neq1$, the lower bound implies
      $\Ker(C_\phi^{*,0}|_{E_N})=\{0\}$.  For $N=1$, one has
      $C_\phi^{*,0}\phi=0$ because $J(\phi,\phi)=0$.  If
      $q_1=c\phi+q_1^\perp$, with $q_1^\perp\perp\phi$, the lower bound applied
      to $q_1^\perp$ shows that $C_\phi^{*,0}q_1=0$ only when
      $q_1^\perp=0$.  Hence
      \begin{align*}
      \Ker(C_\phi^{*,0}|_{E_1})=\operatorname{span}\{\phi\}.
      \end{align*}
      The complex estimate restricts to the real sine cosine subspace.
\end{proof}

The second ingredient is a short time observability estimate for the
finitely many exponential rates produced by one Fourier output mode.  This is where the fact that at most four rates occur enters
quantitatively.
\begin{lemma}\label{lem:gram}
Fix $T,\nu>0$.  Let $K\in\N$, let
$d_1,\dots,d_r$ be distinct integers with
$0\le r\le4$ and $|d_j|\le2\sqrt K$, and suppose
\begin{align*}
 f'(t)=-\nu Kf(t)+\sum_{j=1}^rz_je^{-\nu(K+d_j)t},
\end{align*}
where $z_j$ are constants. Then
\begin{align*}
 \int_0^T|f(t)|^2\dd t
 \ge c_{\nu,T}(1+K)^{-15}
 \left(|f(0)|^2+\sum_{j=1}^r|z_j|^2\right).
\end{align*}
The constant is independent of $K,r,d_j,f(0),z_j$.
\end{lemma}

\begin{proof}
Set
\begin{align*}
\delta=\min\left\{T,1,\frac{1}{4\nu(1+K+2\sqrt K)}\right\}.
\end{align*}
Since $\nu K\delta\le1/4$, the factor $e^{-\nu Kt}$ is bounded above and below by positive constants on $[0,\delta]$. Writing $f(t)=e^{-\nu Kt}g(t)$, direct integration gives
\begin{align*}
g(t)=f(0)+\sum_{j=1}^rz_j\frac{1-e^{-\nu d_jt}}{\nu d_j},
\end{align*}
with the quotient interpreted as $t$ when $d_j=0$. After setting $t=\delta x$, $y_j=\nu d_j\delta$, and
\begin{align*}
\psi_y(x)=\frac{1-e^{-yx}}{y},\qquad \psi_0(x)=x,
\end{align*}
we obtain
\begin{align}\label{eq:finite-rate-rescaled}
g(\delta x)=f(0)+\delta\sum_{j=1}^rz_j\psi_{y_j}(x).
\end{align}
By the definition of $\delta$ and $|d_j|\le2\sqrt K$, all $y_j$ belong to a fixed compact interval, for instance $[-1/8,1/8]$. Since the $d_j$ are distinct integers,
\begin{align}\label{eq:rate-separation}
|y_i-y_j|=\nu\delta|d_i-d_j|\ge\nu\delta,\qquad i\neq j.
\end{align}

Let $G(y_1,\dots,y_r)$ denote the $L^2(0,1)$ Gram matrix of $1,\psi_{y_1},\dots,\psi_{y_r}$. We claim that, uniformly for $0\le r\le4$ and distinct $y_j\in[-1/8,1/8]$,
\begin{align}\label{eq:finite-rate-Gram}
\lambda_{\min}G(y_1,\dots,y_r)\ge c\prod_{i<j}|y_i-y_j|^2.
\end{align}
Indeed, successive divided difference operations on the columns show that
\begin{align*}
\frac{\det G(y_1,\dots,y_r)}{\prod_{i<j}(y_i-y_j)^2}
\end{align*}
extends continuously to configurations in which some parameters coalesce. At such a confluent configuration the divided differences become successive $y$ derivatives of $\psi_y$. These functions remain linearly independent: differentiating in $x$ reduces the relevant family to functions of the form $x^ke^{-yx}$, and exponential polynomials with distinct exponents and their confluent multiplicities are linearly independent. Hence the extended quotient is everywhere positive. Compactness of the parameter set gives a uniform positive lower bound. Since all entries of $G$ are uniformly bounded on the same compact set, $\det G=\prod_\ell\lambda_\ell(G)$ then yields \eqref{eq:finite-rate-Gram}.

There are at most $6$ pairs. Thus \eqref{eq:rate-separation} and \eqref{eq:finite-rate-Gram} imply
\begin{align*}
\int_0^1\left|a_0+\sum_{j=1}^ra_j\psi_{y_j}(x)\right|^2\dd x\ge c_\nu\delta^{12}\left(|a_0|^2+\sum_{j=1}^r|a_j|^2\right).
\end{align*}
Applying this with $a_0=f(0)$ and $a_j=\delta z_j$, using \eqref{eq:finite-rate-rescaled}, $\dd t=\delta\,\dd x$, and the lower bound on $e^{-\nu Kt}$ over $[0,\delta]$, gives
\begin{align*}
\int_0^T|f(t)|^2\dd t&\ge c\delta\int_0^1|g(\delta x)|^2\dd x\ge c_\nu\delta^{13}\left(|f(0)|^2+\delta^2\sum_{j=1}^r|z_j|^2\right)\\
&\ge c_\nu\delta^{15}\left(|f(0)|^2+\sum_{j=1}^r|z_j|^2\right),
\end{align*}
where the last step uses $\delta\le1$. Finally, $1+K+2\sqrt K\le2(1+K)$ gives $\delta\ge c_{\nu,T}(1+K)^{-1}$, and therefore
\begin{align*}
\int_0^T|f(t)|^2\dd t\ge c_{\nu,T}(1+K)^{-15}\left(|f(0)|^2+\sum_{j=1}^r|z_j|^2\right).
\end{align*}
\end{proof}

Recall $\cL$, $S_T$, and $A_T$  defined in  \eqref{eq:L-prepared}--\eqref{eq:Cphi}.  
The shell coercivity and finite rate Gram estimate can now be combined for
the prepared linear system.  The resulting observability inequality yields
exact first order null control through Douglas' factorisation theorem.
\begin{theorem}
\label{thm:observability}
For every $T>0$,
\begin{align}\label{eq:observability}
 \norm{S_T^*\xi}_{E_b}^2
 \le C
 \norm{A_T^*\xi}_{L^2(0,T; H)}^2,
 \qquad \xi\in E_b,
\end{align}
where $C=C(T,\nu,a,b,Q)$, and there exists
$K_T\in\mathcal L(E_b,L^2(0,T; H))$ such that
\begin{align}\label{eq:douglas-factor}
 S_T=A_TK_T,
 \qquad \norm{K_T}\le\sqrt{C}.
\end{align}
\end{theorem}
      
\begin{proof}
      We use complex Fourier notation throughout the proof; all estimates restrict
      to the real sine cosine subspace. Fix
      $\xi=(p_T,q_T)\in E_b$. In the reversed time variable $\tau=T-t$, let
      $(p(\tau),q(\tau))$ be the mild solution of the adjoint system
      \begin{align*}
      \partial_\tau p=\nu\Delta p+C_{\phi,b}^*q,\qquad
      \partial_\tau q=\nu(\Delta+1)q,\qquad
      (p(0),q(0))=(p_T,q_T).
      \end{align*}
      The boundedness of $C_{\phi,b}^*:H\to H^b$ makes this triangular system
      well posed in $E_b$, and
      \begin{align*}
      S_T^*\xi=(p(T),q(T)).
      \end{align*}
      
      The controlled endpoint has the representation
      \begin{align*}
      A_Tv=\int_0^T S_{T-t}(Q_bv(t),0)\,\dd t.
      \end{align*}
      Consequently, Hilbert space duality gives
      \begin{align*}
      (A_T^*\xi)(t)=Q_b^*p(T-t),\qquad
      \norm{A_T^*\xi}_{L^2(0,T; H)}^2
      =\int_0^T\norm{Q_b^*p(\tau)}_{ H}^2\,\dd\tau.
      \end{align*}
      
      Put $q_{T,N}=\Pi_Nq_T$. Since the second
      adjoint equation is diagonal in the Laplace decomposition,
      \begin{align}\label{eq:q-shell-evol}
      q_N(\tau)=e^{-\nu(N-1)\tau}q_{T,N}.
      \end{align}
      
      Fix a nonzero base Fourier mode $\boldsymbol k$ and write
      \begin{align*}
      K=|\boldsymbol k|^2,\qquad
      f_{\boldsymbol k}(\tau)
      =|\boldsymbol k|^{2b}p_{\boldsymbol k}(\tau).
      \end{align*}
      Taking the $\boldsymbol k$-th Fourier coefficient in the mild equation for
      $p$, and using
      $C_{\phi,b}^*=(-\Delta)^{-b}C_\phi^{*,0}$, gives
      \begin{align*}
      f_{\boldsymbol k}'(\tau)
      =-\nu Kf_{\boldsymbol k}(\tau)
      +\bigl(C_\phi^{*,0}q(\tau)\bigr)_{\boldsymbol k}.
      \end{align*}
      
      Recall that the Fourier support of $\phi$ is
      $\mathcal A=\{\pm\boldsymbol e_1,\pm\boldsymbol e_2\}$. If an input
      frequency contributes through $C_\phi^{*,0}$ to the output frequency
      $\boldsymbol k$, then it is of the form
      $\boldsymbol k+\boldsymbol m$ for some $\boldsymbol m\in\mathcal A$,
      because $\mathcal A=-\mathcal A$. Its Laplace eigenvalue is
      \begin{align*}
      N=|\boldsymbol k+\boldsymbol m|^2
      =K+2\boldsymbol k\cdot\boldsymbol m+1.
      \end{align*}
      Define the set of distinct rate offsets
      \begin{align*}
      \mathcal R_{\boldsymbol k}
      =\left\{
      2\boldsymbol k\cdot\boldsymbol m:
      \boldsymbol m\in\mathcal A,\ 
      K+2\boldsymbol k\cdot\boldsymbol m+1\ge1
      \right\}.
      \end{align*}
      Repeated values are included only once. For
      $d\in\mathcal R_{\boldsymbol k}$, put
      \begin{align*}
      N_{\boldsymbol k,d}=K+d+1,\qquad
      z_{\boldsymbol k,d}
      =\bigl(C_\phi^{*,0}\Pi_{N_{\boldsymbol k,d}}q_T\bigr)_{\boldsymbol k}.
      \end{align*}
      There are at most four distinct offsets, and
      \begin{align*}
      |\mathcal R_{\boldsymbol k}|\le4,\qquad
      |d|\le2|\boldsymbol k|=2\sqrt K.
      \end{align*}
      Using \eqref{eq:q-shell-evol}, the scalar modal equation becomes
      \begin{align}\label{eq:fk}
      f_{\boldsymbol k}'(\tau)
      =-\nu Kf_{\boldsymbol k}(\tau)
      +\sum_{d\in\mathcal R_{\boldsymbol k}}
      z_{\boldsymbol k,d}e^{-\nu(K+d)\tau}.
      \end{align}
      The exponents in \eqref{eq:fk} are distinct because the offsets in
      $\mathcal R_{\boldsymbol k}$ are distinct.
      
      By the formula for $Q_b^*$ and the lower bound in \eqref{eq:q},
      \begin{align*}
      \norm{A_T^*\xi}_{L^2(0,T; H)}^2
      &=\int_0^T\norm{Q_b^*p(\tau)}_{ H}^2\,\dd\tau\\
      &\ge c_Q^2\sum_{\boldsymbol k\ne\boldsymbol 0}
      |\boldsymbol k|^{-a}
      \int_0^T|f_{\boldsymbol k}(\tau)|^2\,\dd\tau.
      \end{align*}
      The finite rate observability estimate in \Cref{lem:gram} applies to
      \eqref{eq:fk}. Since $K=|\boldsymbol k|^2\ge1$,
      \begin{align*}
       |\boldsymbol k|^{-a}
       =K^{-a/2}
       \ge c(1+K)^{-a/2}.
      \end{align*}
      Hence, setting $L=15+a/2$ gives
      \begin{align}\label{eq:obs-mode}
      \norm{A_T^*\xi}_{L^2(0,T; H)}^2
      \ge c\sum_{\boldsymbol k\ne\boldsymbol 0}(1+K)^{-L}
      \left(
      |f_{\boldsymbol k}(0)|^2
      +\sum_{d\in\mathcal R_{\boldsymbol k}}|z_{\boldsymbol k,d}|^2
      \right).
      \end{align}
      
      Reorganise the second term in \eqref{eq:obs-mode} according to the
      input Laplace shell. If $N=N_{\boldsymbol k,d}$, then
      $N=|\boldsymbol k+\boldsymbol m|^2$ for some
      $\boldsymbol m\in\mathcal A$ with $|\boldsymbol m|=1$. Hence
      \begin{align*}
      1+N\le1+2|\boldsymbol k|^2+2|\boldsymbol m|^2
      \le3(1+K).
      \end{align*}
      Conversely,
      $\boldsymbol k=(\boldsymbol k+\boldsymbol m)-\boldsymbol m$ gives
      \begin{align*}
      1+K\le3(1+N).
      \end{align*}
      Thus
      \begin{align*}
      3^{-L}(1+N)^{-L}
      \le(1+K)^{-L}
      \le3^L(1+N)^{-L}.
      \end{align*}
      
      For fixed $N$ and $\boldsymbol k$, the coefficient
      $\bigl(C_\phi^{*,0}q_{T,N}\bigr)_{\boldsymbol k}$ occurs in the sum in
      \eqref{eq:obs-mode} with the unique offset
      $d=N-K-1$, whenever it is nonzero. Conversely, every
      $z_{\boldsymbol k,d}$ is of this form. Since all summands are nonnegative,
      Tonelli's theorem and Parseval's identity therefore give
      \begin{align*}
      \sum_{\boldsymbol k\ne\boldsymbol 0}(1+K)^{-L}
      \sum_{d\in\mathcal R_{\boldsymbol k}}|z_{\boldsymbol k,d}|^2
      &\ge c\sum_{N\ge1}(1+N)^{-L}
      \sum_{\boldsymbol k}
      \left|
      \bigl(C_\phi^{*,0}q_{T,N}\bigr)_{\boldsymbol k}
      \right|^2\\
      &=c\sum_{N\ge1}(1+N)^{-L}
      \norm{C_\phi^{*,0}q_{T,N}}_2^2.
      \end{align*}
      
      Since $q_T\in\phi^\perp$ and $\phi\in E_1$, one has
      $q_{T,1}\perp\phi$. The shell coercivity estimate in \Cref{lem:shell}
      therefore applies for every $N$ and yields
      \begin{align*}
      \norm{A_T^*\xi}_{L^2(0,T; H)}^2
      \ge c\sum_{N\ge1}(1+N)^{-L}\norm{q_{T,N}}_2^2.
      \end{align*}
      
      By \eqref{eq:q-shell-evol},
      \begin{align*}
      \norm{q(T)}_2^2
      &=\sum_{N\ge1}e^{-2\nu(N-1)T}\norm{q_{T,N}}_2^2\\
      &\le
      \left[
      \sup_{N\ge1}e^{-2\nu(N-1)T}(1+N)^L
      \right]
      \sum_{N\ge1}(1+N)^{-L}\norm{q_{T,N}}_2^2\\
      &\le C\norm{A_T^*\xi}_{L^2(0,T; H)}^2.
      \end{align*}
      The supremum is finite because exponential decay dominates every
      polynomial. This controls the fibre component of $S_T^*\xi$.
      
      It remains to control the base component. Variation of constants in
      \eqref{eq:fk} gives
      \begin{align*}
      f_{\boldsymbol k}(T)
      =e^{-\nu KT}f_{\boldsymbol k}(0)
      +\sum_{d\in\mathcal R_{\boldsymbol k}}z_{\boldsymbol k,d}
      \int_0^T e^{-\nu K(T-r)}
      e^{-\nu(K+d)r}\,\dd r.
      \end{align*}
      If $K\ge16$, then $|d|\le2\sqrt K$ implies
      \begin{align*}
      K+d\ge K-2\sqrt K\ge\frac K2.
      \end{align*}
      Consequently,
      \begin{align*}
      e^{-\nu K(T-r)}e^{-\nu(K+d)r}
      \le e^{-\nu KT/2},
      \qquad 0\le r\le T.
      \end{align*}
      Since $|\mathcal R_{\boldsymbol k}|\le4$, it follows that
      \begin{align}\label{eq:base-pullback}
      |f_{\boldsymbol k}(T)|^2
      \le Ce^{-c\nu KT}
      \left(
      |f_{\boldsymbol k}(0)|^2
      +\sum_{d\in\mathcal R_{\boldsymbol k}}|z_{\boldsymbol k,d}|^2
      \right).
      \end{align}
      There are only finitely many Fourier modes with $K<16$, and
      \eqref{eq:base-pullback} remains valid for those modes after increasing
      $C$.
      
      Because $f_{\boldsymbol k}=|\boldsymbol k|^{2b}
      p_{\boldsymbol k}$,
      \begin{align*}
      \norm{p(T)}_{H^b}^2
      =\sum_{\boldsymbol k\ne\boldsymbol 0}
      K^{-b}|f_{\boldsymbol k}(T)|^2.
      \end{align*}
      Moreover,
      \begin{align*}
      \sup_{K\ge1}e^{-c\nu KT}K^{-b}(1+K)^L<\infty.
      \end{align*}
      Combining this bound with \eqref{eq:base-pullback} and
      \eqref{eq:obs-mode} yields
      \begin{align*}
      \norm{p(T)}_{H^b}^2
      \le C\norm{A_T^*\xi}_{L^2(0,T; H)}^2.
      \end{align*}
      Together with the fibre estimate, this proves
      \begin{align*}
      \norm{S_T^*\xi}_{E_b}^2
      =\norm{p(T)}_{H^b}^2+\norm{q(T)}_2^2
      \le C\norm{A_T^*\xi}_{L^2(0,T; H)}^2,
      \end{align*}
      which is \eqref{eq:observability}.
      
      Finally, \eqref{eq:observability} is equivalent to the positive operator
      inequality
      \begin{align*}
      S_TS_T^*\le CA_TA_T^*
      \qquad\text{on }E_b.
      \end{align*}
      Douglas' range inclusion lemma \cite{Douglas}, applied to
      $S_T:E_b\to E_b$ and
      $A_T:L^2(0,T; H)\to E_b$, therefore yields
      $K_T\in\mathcal L(E_b,L^2(0,T; H))$ such that
      \begin{align*}
      S_T=A_TK_T,\qquad \norm{K_T}\le\sqrt C.
      \end{align*}
      This proves \eqref{eq:douglas-factor}.
\end{proof}

\subsection{A convolution driver bridge and local $d$-smallness}
\label{subsec:convolution-local-bridge}

The deterministic factorisation in \Cref{thm:observability} has a local
stochastic consequence that will be used in the final coupling argument
in \Cref{subsec:coupling-closure}. Use the projective charts
\begin{align}\label{eq:charts}
 \iota(\eta)=\left[\phi+\eta\right],
 \qquad \chi_\phi([\rho])=\frac{\Pi_{\phi^\perp}\rho}{\ip{\rho}{\phi}},
\end{align}
where $\eta\in\phi^\perp$ and the second expression is used when
$\ip{\rho}{\phi}>0$.  This is precisely the unified affine state coordinate
\eqref{eq:state-difference-chart} centred at $(0,[\phi])$: the map
$\iota$ is its inverse projective component and $\chi_\phi$ is its forward
projective component.  Write $z_*=(0,0)\in E_b$ for the coordinate
representative of $(0,[\phi])$.  For $T>0$, recall
\begin{align*}
 \mathsf Y_T&=H^\sigma(0,T;H^b)\oplus H^b,\qquad \frac14<\sigma<\frac12,\notag\\
 \Gamma_Tg&=\left(t\mapsto\int_0^t e^{\nu(t-s)\Delta}Qg_s\,\dd s,\ 
                  \int_0^T e^{\nu(T-s)\Delta}Qg_s\,\dd s\right),\notag\\
 Y_T(W)&=\left(t\mapsto\int_0^t e^{\nu(t-s)\Delta}Q\,\dd W_s,\ 
                  \int_0^T e^{\nu(T-s)\Delta}Q\,\dd W_s\right).
\end{align*}
The operator $\Gamma_T:\cH_T\to\mathsf Y_T$ is Hilbert--Schmidt.  In particular,
\begin{align}\label{eq:c-cm}
 \norm{\Gamma_Tg}_{\mathsf Y_T}\le c_{\Gamma,T}\norm g_{\cH_T}.
\end{align}
The random variable $Y_T(W)$ is a centered Radon Gaussian variable in the
separable Hilbert space $\mathsf Y_T$. For $y=(y^\circ,y^T)\in\mathsf Y_T$, subtract the convolution and solve
\begin{align*}
 v_t&=e^{\nu t\Delta}w_0-\int_0^t e^{\nu(t-s)\Delta}B(v_s+y_s^\circ,v_s+y_s^\circ)\,\dd s,\\
 w^\circ&=v+y^\circ,\qquad w_T=v_T+y^T.
\end{align*}
The raw fibre is driven by $w^\circ$.  Denote the resulting normalized
projective endpoint by $\Phi_T(z,y)$ whenever the raw endpoint is
nonzero, and by $F_T(z,y)$ its representation in the charts
\eqref{eq:charts}.  At the prepared zero driver skeleton,
\begin{align}\label{eq:prepared-convolution-derivatives}
 D_zF_T(z_*,0)=S_T,\qquad D_yF_T(z_*,0)\Gamma_T=A_T.
\end{align}
The second identity follows by differentiating in the direction
$\Gamma_Tg$: the base variation is
$u_t=\int_0^t e^{\nu(t-s)\Delta}Qg_s\,\dd s$, and the fibre variation is
exactly the controlled augmented tangent equation defining $A_Tg$.

We use the canonical Wiener realisation from
\eqref{eq:canonical-block-AWS}--\eqref{eq:canonical-block-translation},
with $\tau$ replaced by $T$.  Thus the convolution driver admits a Borel
realisation $Y_T:\Omega_T\to\mathsf Y_T$ such that, on one
Cameron--Martin invariant Borel set of full measure,
\begin{align}\label{eq:canonical-driver-translation}
 Y_T(\omega+\mathbf g)=Y_T(\omega)+\Gamma_Tg,
 \qquad g\in\cH_T.
\end{align}
Its law is a centred Radon Gaussian measure on $\mathsf Y_T$, with
Cameron--Martin space $\Gamma_T\cH_T$, and therefore
\begin{align}\label{eq:YT-support}
 \supp\mathcal L(Y_T)
 =\overline{\Gamma_T\cH_T}^{\,\mathsf Y_T}.
\end{align}
In particular, every open neighbourhood of zero has positive probability.

The factorisation is linear, whereas the bridge concerns the nonlinear
stochastic endpoint.  The following lemma gives the local consequence of
\Cref{lem:revision-base-two-jet} needed to transfer the cancellation from
the prepared skeleton to nearby state and driver pairs.

\begin{lemma}
\label{lem:local-flow}
There are $R_0>0$ and a modulus $m_T(R)\downarrow0$ as $R\downarrow0$
such that, whenever
$\norm{z-z_*}_{E_b}+\norm y_{\mathsf Y_T}\le R\le R_0$, the endpoint
$F_T(z,y)$ is defined and
\begin{align}\label{eq:local-flow-modulus}
 \norm{D_zF_T(z,y)-S_T}_{\cL(E_b,E_b)}
 +\norm{D_yF_T(z,y)\Gamma_T-A_T}_{\cL(\cH_T,E_b)}
 \le m_T(R).
\end{align}
Moreover, every corresponding raw fibre endpoint satisfies
\begin{align}\label{eq:normalization-tube}
 \norm{\rho_T}_2\ge\frac12e^{-\nu T},\qquad
 \ip{\rho_T}{\phi}\ge\frac12e^{-\nu T}.
\end{align}
\end{lemma}

\begin{proof}
By \Cref{lem:revision-base-two-jet}, the unnormalised endpoint is
operator norm $C^2$ in $(z,y)$ near $(z_*,0)$.  At the prepared skeleton,
$\rho_T=e^{-\nu T}\phi$, so endpoint continuity gives
\eqref{eq:normalization-tube} after shrinking the neighborhood.  Hence
projective normalisation is $C^2$ there.  The identities
\eqref{eq:prepared-convolution-derivatives} and continuity of the first
derivative then give \eqref{eq:local-flow-modulus}.
\end{proof}

A Borel maximal coupling of a Cameron--Martin translate will also be
needed.  Let $\mathbf P_{-g}$ be the law of $W-\mathbf g$, where
$\mathbf g(t)=\int_0^t g_s\,\dd s$, and put
$r_g=\dd\mathbf P_{-g}/\dd\mathbf P$. The bridge also requires a coupling of a Wiener law with a
Cameron--Martin translate.  We use the maximal coupling below, whose Borel
dependence on the translation is needed for the final coupling kernel.
\begin{lemma}
\label{lem:Borel-CM-maximal}
There is a maximal coupling $\kappa_g$ of $(\mathbf P_{-g},\mathbf P)$
which is Borel in $g\in\cH_T$.  With
$m_g=1\wedge r_g$ and $c_g=\int m_g\,\dd\mathbf P$, it is given, when
$c_g<1$, by
\begin{align}\label{eq:Borel-maximal-coupling}
 \kappa_g(\dd x,\dd y)&=m_g(x)\mathbf P(\dd x)\delta_x(\dd y)\notag\\
 &\quad+\frac{(r_g(x)-1)_+\mathbf P(\dd x)(1-r_g(y))_+\mathbf P(\dd y)}{1-c_g},
\end{align}
and by the diagonal coupling when $c_g=1$.
\end{lemma}

\begin{proof}
For deterministic $g\in\cH_T$, the Wiener integral
\begin{align*}
 \delta g(\omega)
 =\int_0^T\langle g_t,\dd W_t(\omega)\rangle
\end{align*}
admits a jointly Borel version in $(g,\omega)$.  Hence
\begin{align*}
 r_g(\omega)
 =\exp\left(-\delta g(\omega)-\frac12\|g\|_{\cH_T}^2\right)
\end{align*}
is jointly Borel, and so is \eqref{eq:Borel-maximal-coupling}.  Direct
integration shows that the two marginals are $\mathbf P_{-g}$ and
$\mathbf P$.  Moreover,
\begin{align*}
 \kappa_g\{x=y\}
 =c_g
 =1-\|\mathbf P_{-g}-\mathbf P\|_{\rm TV},
\end{align*}
so $\kappa_g$ is maximal.
\end{proof}

We can now couple two nearby prepared initial states with exact Wiener
marginals.  The deterministic factorisation cancels their first order
separation on a small driver event, while maximal coupling pays only the
Cameron--Martin translation cost.
\begin{theorem}
\label{thm:eventwise-bridge}
For every $T>0$ and every $\theta>0$, there are $\delta>0$ and $r>0$ such
that the following holds.  Let
\begin{align*}
 \mathcal U_r=\{(w,\iota(\eta)):\norm{(w,\eta)}_{E_b}<r\}
\end{align*}
be the $r$ neighborhood of the prepared state. If $\mathcal A\subset\Omega_T$ is Borel and
\begin{align*}
 \mathcal A\subset\{\norm{Y_T(W)}_{\mathsf Y_T}\le\delta\},\qquad
 \mathbf P(\mathcal A)>0,
\end{align*}
then, for every $z,z'\in\mathcal U_r$, write
$h=\Delta_{(0,[\phi])}(z')-\Delta_{(0,[\phi])}(z)$, equivalently
$h=z'-z$ in the prepared coordinates, and put $g=K_Th$.  There is a coupling law
$\mathbf Q^{\rm br}_{z,z'}$ of two standard Wiener paths
$(W,\widetilde W)$, Borel in $(z,z')$, such that
\begin{align}\label{eq:eventwise-bridge}
 \mathbf Q^{\rm br}_{z,z'}\left\{W\in\mathcal A,\
  \widetilde W=W-\int_0^\cdot g_s\,\dd s\right\}
 \ge\mathbf P(\mathcal A)-\frac12\norm g_{\cH_T}.
\end{align}
The corresponding endpoints have the exact transition laws
$P_T(z,\cdot)$ and $P_T(z',\cdot)$.  On the event in
\eqref{eq:eventwise-bridge},
\begin{align}\label{eq:local-cancel}
 \norm{F_T(z',Y_T(\widetilde W))-F_T(z,Y_T(W))}_{E_b}
 \le\theta\norm h_{E_b}.
\end{align}
\end{theorem}

\begin{proof}
Let $K_T$ be the factor in \eqref{eq:douglas-factor} and put
$\kappa=\norm{K_T}$.  Choose $R>0$ from \Cref{lem:local-flow} so that
\eqref{eq:normalization-tube} holds and $m_T(R)(1+\kappa)\le\theta$.
Choose $\delta>0$ and then $r>0$ such that
\begin{align*}
 \delta+2r+2c_{\Gamma,T}\kappa r<R,
\end{align*}
where $c_{\Gamma,T}$ is the constant from \eqref{eq:c-cm}. 
Fix $z,z'\in\mathcal U_r$, let
$h=\Delta_{(0,[\phi])}(z')-\Delta_{(0,[\phi])}(z)$, and set
$g=K_Th$.  On
$\{\norm{Y_T(W)}_{\mathsf Y_T}\le\delta\}$, the segment in the prepared coordinates 
\begin{align*}
 (z_s,y_s)=(z+sh,Y_T(W)-s\Gamma_Tg),\qquad 0\le s\le1,
\end{align*}
remains in the radius $R$ tube.  The fundamental theorem of calculus and
$S_T=A_TK_T$ give
\begin{align*}
 F_T(z',Y_T(W)-\Gamma_Tg)-F_T(z,Y_T(W))
 =\int_0^1\bigl(D_zF_T(z_s,y_s)h-D_yF_T(z_s,y_s)\Gamma_Tg\bigr)\,\dd s.
\end{align*}
and hence, by \Cref{lem:local-flow},
\begin{align*}
 \norm{F_T(z',Y_T(W)-\Gamma_Tg)-F_T(z,Y_T(W))}_{E_b}
 \le m_T(R)(1+\kappa)\norm h_{E_b}
 \le\theta\norm h_{E_b}.
\end{align*}

By Cameron--Martin and Pinsker,
\begin{align}\label{eq:pinsker}
 \operatorname{Ent}(\mathbf P_{-g}\mid\mathbf P)=\frac12\norm g_{\cH_T}^2,
 \qquad \norm{\mathbf P_{-g}-\mathbf P}_{\rm TV}\le\frac12\norm g_{\cH_T}.
\end{align}
Sample $(X,\widetilde W)$ from the Borel maximal coupling in
\Cref{lem:Borel-CM-maximal}, set $W=X+\mathbf g$, and denote the law of
$(W,\widetilde W)$ by $\mathbf Q^{\rm br}_{z,z'}$.  Both $W$ and
$\widetilde W$ are standard Wiener paths.  On the diagonal part,
$\widetilde W=X=W-\mathbf g$, and
\eqref{eq:canonical-driver-translation} gives
$Y_T(\widetilde W)=Y_T(W)-\Gamma_Tg$.  Since $W$ has law $\mathbf P$, the
union bound and \eqref{eq:pinsker} prove \eqref{eq:eventwise-bridge}; the
preceding pathwise estimate proves \eqref{eq:local-cancel}.  
\end{proof}

For a bounded metric $d$, write $W_d$ for the corresponding
Kantorovich--Wasserstein distance. \Cref{thm:eventwise-bridge} has an immediate transport consequence.  After
truncating the ambient distance, the prepared neighborhood is a genuine
$d$-small set with a uniform gap.
\begin{corollary}
\label{cor:d-small}
For every $T>0$, there are $r>0$, $L<\infty$, and $\varepsilon_*>0$ such
that
\begin{align*}
 \mathcal U_r=\{(w,\iota(\eta)):\norm{(w,\eta)}_{E_b}<r\}
\end{align*}
is $d_L$ small for $P_T$, where
\begin{align*}
 d_L\bigl((w,[\rho]),(\widetilde w,[\widetilde\rho])\bigr)
 =1\wedge L\left(\norm{w-\widetilde w}_{H^b}+d_{\PP}([\rho],[\widetilde\rho])\right).
\end{align*}
Thus, for every $z,z'\in\mathcal U_r$,
\begin{align}\label{eq:d-small}
 W_{d_L}\bigl(P_T(z,\cdot),P_T(z',\cdot)\bigr)\le1-\varepsilon_*.
\end{align}
\end{corollary}

\begin{proof}
Apply \Cref{thm:eventwise-bridge} with fixed $\theta>0$.  Since zero
belongs to the Gaussian support in \eqref{eq:YT-support},
\begin{align*}
 p_\delta:=\mathbf P\{\norm{Y_T(W)}_{\mathsf Y_T}\le\delta\}>0.
\end{align*}
Take $\mathcal A=\{\norm{Y_T(W)}_{\mathsf Y_T}\le\delta\}$.  On the output
tube, the prepared chart and the product metric are locally bi-Lipschitz,
so for some $c_{\rm ch}<\infty$ the matching event satisfies
\begin{align*}
 \norm{w_T-\widetilde w_T}_{H^b}+d_{\PP}([\rho_T],[\widetilde\rho_T])
 \le c_{\rm ch}\theta\norm h_{E_b}.
\end{align*}
Shrink $r$ so that $r\le p_\delta/(4\kappa)$ when
$\kappa=\norm{K_T}>0$.  Then \eqref{eq:eventwise-bridge} gives matching
probability at least $p_\delta/2$.  Choose $L$ so that
$2Lc_{\rm ch}\theta r\le1/2$.  The $d_L$ cost is at most $1/2$ on the
matching event and at most $1$ otherwise, whence
\begin{align*}
 W_{d_L}\bigl(P_T(z,\cdot),P_T(z',\cdot)\bigr)
 \le1-\frac12\mathbf P\{\text{matching event}\}
 \le1-\frac{p_\delta}{4}.
\end{align*}
Thus \eqref{eq:d-small} holds with $\varepsilon_*=p_\delta/4$.
\end{proof}

\begin{remark}
\Cref{cor:d-small} gives only local $d$-smallness.  The additional
$d$ contraction estimate required for a weak Harris spectral gap argument
is currently out of reach and is left for future work.
\end{remark}

\subsection{Accessibility}

It remains to establish the topological accessibility needed for the
coupling argument.  The controls below prescribe the base vorticity itself.  If a smooth path $\bar w$
is prescribed, its forcing is
\begin{align}\label{eq:skeleton-control}
 g=Q^{-1}\bigl(\partial_t\bar w-
       \nu\Delta\bar w+B(\bar w,\bar w)\bigr).
\end{align}
For a spatially smooth path with compact time interval, $g\in\cH_T$ because
$Q^{-1}$ has only polynomial growth. The theorem below is the precise
accessibility statement needed in the uniqueness argument, whose proof is given at the end of this subsection. 
\begin{theorem}
\label{thm:accessibility}
For every $z\in H^b\times\PP(H)$ and every $r>0$, there is
$T=T(z,r)>0$ such that
\begin{align}\label{eq:point-access}
 P_T(z,\mathcal U_r)>0.
\end{align}
Moreover, for the block length $\tau$ fixed in
\eqref{eq:revision-single-block-choice}, there is $n=n(z,r)\ge1$ such that
\begin{align}\label{eq:sampled-point-access}
 P_{n\tau}(z,\mathcal U_r)>0.
\end{align}
Consequently, every $P_\tau$ invariant probability measure $\nu$ satisfies
$\nu(\mathcal U_r)>0$.
\end{theorem}

Recall that $E_1=\Ker(-\Delta-1)$ is a real four dimensional space.
We separate the finite dimensional control geometry from the later
full PDE leakage estimates.  The next lemma seeds the first shell, relaxes
to it, and identifies the shell two control algebra acting on $E_1$.
\begin{lemma}
      \label{lem:first-shell-seeding-control}
      Let $\Pi_1$ be the orthogonal projection onto $E_1$.
      \begin{enumerate}[label=\textup{(\roman*)}]
      \item If $h\in H\setminus\{0\}$ and $\Pi_1h=0$, a smooth,
      zero to zero, finite energy base pulse sends $h$ to a raw fibre with
      nonzero first shell projection.
      \item If $\Pi_1h\ne0$, a zero base relaxation makes the projective fibre
      arbitrarily close to $[\Pi_1h]$.
      \item The four shell two controls on $E_1$ generate
      $\mathfrak{so}(E_1)$.
      \end{enumerate}
\end{lemma}
      
\begin{proof}
      \emph{Seeding the first shell.}
      We first work in the complexified Fourier representation. Write
      \begin{align*}
       h=\sum_{\boldsymbol k\ne\boldsymbol 0}h_{\boldsymbol k}e_{\boldsymbol k},
       \qquad e_{\boldsymbol k}(x)=e^{\mathrm i\boldsymbol k\cdot x},
      \end{align*}
      so that $h_{\boldsymbol k}$ denotes the $\boldsymbol k$-th complex Fourier
      coefficient of $h$; for real valued $h$ one has
      $h_{-\boldsymbol k}=\overline{h_{\boldsymbol k}}$. The Fourier coefficient
      of the symmetric tangent interaction is, up to the fixed Fourier
      normalisation,
      \begin{align}\label{eq:Gamma}
       \widehat{\cS(f,g)}(\boldsymbol k)
       =\sum_{\boldsymbol m+\boldsymbol n=\boldsymbol k}
       \Gamma(\boldsymbol m,\boldsymbol n)f_{\boldsymbol m}g_{\boldsymbol n},
       \quad
       \Gamma(\boldsymbol m,\boldsymbol n)
       =\det(\boldsymbol m,\boldsymbol n)
       \bigl(|\boldsymbol m|^{-2}-|\boldsymbol n|^{-2}\bigr).
      \end{align}
      
      Assume $\Pi_1h=0$. Since $h\ne0$, choose
      $\boldsymbol n\in\mathbb Z^2\setminus\{\boldsymbol 0\}$ such that
      $h_{\boldsymbol n}\ne0$. Necessarily $|\boldsymbol n|^2>1$. Choose a unit
      lattice vector $\boldsymbol\ell$, $|\boldsymbol\ell|=1$, such that
      $\det(\boldsymbol\ell,\boldsymbol n)\ne0$, and put
      $\boldsymbol m=\boldsymbol\ell-\boldsymbol n$. Then
      \begin{align*}
       \det(\boldsymbol m,\boldsymbol n)
       =\det(\boldsymbol\ell,\boldsymbol n)\ne0,
       \qquad
       |\boldsymbol m|^2-|\boldsymbol n|^2
       =1-2\boldsymbol\ell\cdot\boldsymbol n\ne0.
      \end{align*}
      Hence
      $\Gamma(\boldsymbol m,\boldsymbol n)h_{\boldsymbol n}\ne0$.
      
      Choose $\beta\in C_c^\infty(0,\delta)$ with $\beta\ge0$ and
      $\beta\not\equiv0$. For the moment consider the complexified base pulse
      $\bar w_\varepsilon(s)=\varepsilon\beta(s)e_{\boldsymbol m}$ and let
      $\rho^\varepsilon_s=U_{\bar w_\varepsilon}(s,0)h$. Thus
      \begin{align*}
       \partial_s\rho^\varepsilon_s
       =\nu\Delta\rho^\varepsilon_s
       +\cS(\bar w_\varepsilon(s),\rho^\varepsilon_s),
       \qquad \rho^\varepsilon_0=h.
      \end{align*}
      At $\varepsilon=0$, $\rho^0_s=e^{\nu s\Delta}h$. Therefore the first
      variation
      \begin{align*}
       \zeta_s
       =\left.\frac{\mathrm d}{\mathrm d\varepsilon}\right|_{\varepsilon=0}
       \rho^\varepsilon_s
      \end{align*}
      satisfies
      \begin{align*}
       \partial_s\zeta_s
       =\nu\Delta\zeta_s
       +\beta(s)\cS(e_{\boldsymbol m},e^{\nu s\Delta}h),
       \qquad \zeta_0=0,
      \end{align*}
      and hence
      \begin{align*}
       \zeta_\delta
       =\int_0^\delta e^{\nu(\delta-s)\Delta}
       \beta(s)\cS(e_{\boldsymbol m},e^{\nu s\Delta}h)\,\dd s.
      \end{align*}
      Since $\boldsymbol m+\boldsymbol n=\boldsymbol\ell$ and
      $|\boldsymbol\ell|=1$, the $\boldsymbol\ell$-th Fourier coefficient is
      \begin{align*}
       \left[
       \left.\frac{\mathrm d}{\mathrm d\varepsilon}\right|_{\varepsilon=0}
       \Pi_1U_{\bar w_\varepsilon}(\delta,0)h
       \right]_{\boldsymbol\ell}
       =
       \Gamma(\boldsymbol m,\boldsymbol n)h_{\boldsymbol n}e^{-\nu\delta}
       \int_0^\delta\beta(s)
       e^{-\nu(|\boldsymbol n|^2-1)s}\,\dd s\ne0.
      \end{align*}
      
      Returning to real controls, since
      $e_{\boldsymbol m}=\cos(\boldsymbol m\cdot x)
      +\mathrm i\sin(\boldsymbol m\cdot x)$, the preceding complex first
      variation is the complexification of the two real first variations
      corresponding to the base paths
      $\varepsilon\beta(s)\cos(\boldsymbol m\cdot x)$ and
      $\varepsilon\beta(s)\sin(\boldsymbol m\cdot x)$. Hence at least one of
      these two real phase directions has a nonzero first shell derivative.
      Fix such a real Fourier phase $\upsilon$ and set
      $\bar w_{\varepsilon,\upsilon}(s)=\varepsilon\beta(s)\upsilon$.
      
      This prescribed real base path is an exact controlled Navier--Stokes
      skeleton. Namely, by \eqref{eq:skeleton-control}, define
      \begin{align*}
       g_{\varepsilon,\upsilon}(s)
       =Q^{-1}\bigl(
       \varepsilon\beta'(s)\upsilon
       -\nu\varepsilon\beta(s)\Delta\upsilon
       +\varepsilon^2\beta(s)^2B(\upsilon,\upsilon)
       \bigr).
      \end{align*}
      Then
      \begin{align*}
       \partial_s\bar w_{\varepsilon,\upsilon}
       =\nu\Delta\bar w_{\varepsilon,\upsilon}
       -B(\bar w_{\varepsilon,\upsilon},
          \bar w_{\varepsilon,\upsilon})
       +Qg_{\varepsilon,\upsilon}.
      \end{align*}
      Since $\upsilon$ is a single smooth Fourier phase and every corresponding
      noise coefficient is nonzero, $g_{\varepsilon,\upsilon}\in\cH_\delta$.
      Moreover, $\beta\in C_c^\infty(0,\delta)$ implies that
      $\bar w_{\varepsilon,\upsilon}$ is identically zero near both endpoints,
      so it is a smooth zero to zero base loop.
      
      For this real phase, put
      \begin{align*}
       v=
       \left.\frac{\mathrm d}{\mathrm d\varepsilon}\right|_{\varepsilon=0}
       \Pi_1U_{\bar w_{\varepsilon,\upsilon}}(\delta,0)h.
      \end{align*}
      By construction $v\ne0$, whereas
      $\Pi_1U_{\bar w_{0,\upsilon}}(\delta,0)h
      =\Pi_1e^{\nu\delta\Delta}h=0$. The $C^1$ dependence of the fibre endpoint
      on $\varepsilon$ therefore gives
      \begin{align*}
       \Pi_1U_{\bar w_{\varepsilon,\upsilon}}(\delta,0)h
       =\varepsilon v+o(\varepsilon),
       \qquad v\ne0.
      \end{align*}
      Thus the first shell projection is nonzero for every sufficiently small
      nonzero $\varepsilon$, proving \textup{(i)}.
      
      \emph{Relax to $E_1$.}
      Suppose now that $\Pi_1h\ne0$, and write
      $h=\Pi_1h+\Pi_{>1}h$, where $\Pi_{>1}=I-\Pi_1$. With zero base, the
      raw fibre is $e^{\nu t\Delta}h$. Since the first shell has Laplace
      eigenvalue $1$ and every mode in $\Pi_{>1}H$ has eigenvalue at least $2$,
      \begin{align*}
       [e^{\nu t\Delta}h]
       =\left[\Pi_1h+r_t\right],
       \qquad
       r_t=e^{\nu t}e^{\nu t\Delta}\Pi_{>1}h,
       \qquad
       \norm{r_t}_2\le e^{-\nu t}\norm{\Pi_{>1}h}_2.
      \end{align*}
      Since $r_t\perp\Pi_1h$, the definition of the projective metric gives
      \begin{align*}
       d_{\PP}([e^{\nu t\Delta}h],[\Pi_1h])
       \le C\frac{\norm{r_t}_2}{\norm{\Pi_1h}_2}
       \le Ce^{-\nu t}
       \frac{\norm{\Pi_{>1}h}_2}{\norm{\Pi_1h}_2}.
      \end{align*}
      Thus a sufficiently long zero base interval makes the projective fibre
      arbitrarily close to $[\Pi_1h]$, proving \textup{(ii)}.
      
      \emph{Rotate inside $E_1$.}
      Let $\boldsymbol e_1=(1,0)$ and $\boldsymbol e_2=(0,1)$. A real
      $f\in E_1$ has the complex Fourier representation
      \begin{align*}
       f=z_1e_{\boldsymbol e_1}+\bar z_1e_{-\boldsymbol e_1}
         +z_2e_{\boldsymbol e_2}+\bar z_2e_{-\boldsymbol e_2},
       \qquad
       z_1=f_{\boldsymbol e_1},\quad z_2=f_{\boldsymbol e_2}.
      \end{align*}
      Thus $E_1\simeq\mathbb C^2\simeq\mathbb R^4$. Consider the two
      shell two frequencies
      \begin{align*}
       \boldsymbol q=\boldsymbol e_2-\boldsymbol e_1,
       \qquad
       \boldsymbol r=\boldsymbol e_1+\boldsymbol e_2,
      \end{align*}
      and write
      $a=w_{\boldsymbol q}$ and $c=w_{\boldsymbol r}$. Since $w$ is real,
      $w_{-\boldsymbol q}=\bar a$ and
      $w_{-\boldsymbol r}=\bar c$.
      
      We compute the shell one component of the infinitesimal control operator
      $\Pi_1\cS(w,\cdot)|_{E_1}$. For the output frequency
      $\boldsymbol e_1$, the only relevant pairs in \eqref{eq:Gamma} are
      $(-\boldsymbol q,\boldsymbol e_2)$ and
      $(\boldsymbol r,-\boldsymbol e_2)$; for the output frequency
      $\boldsymbol e_2$, they are
      $(\boldsymbol q,\boldsymbol e_1)$ and
      $(\boldsymbol r,-\boldsymbol e_1)$. Directly from
      \eqref{eq:Gamma},
      \begin{align*}
       \Gamma(-\boldsymbol q,\boldsymbol e_2)
       =-\gamma,\qquad
       \Gamma(\boldsymbol r,-\boldsymbol e_2)=\gamma,\qquad
       \Gamma(\boldsymbol q,\boldsymbol e_1)=\gamma,\qquad
       \Gamma(\boldsymbol r,-\boldsymbol e_1)=-\gamma,
      \end{align*}
      where $\gamma\ne0$ is the same fixed constant, equal to $1/2$ under the
      unnormalised complex convention. The scalar heat contribution on $E_1$
      is $-\nu I$ and therefore has no effect on the projective direction.
      After suppressing this common scalar and absorbing $\gamma$ into the
      control amplitude, the projected control equations are
      \begin{align*}
       \dot z_1=-\bar a z_2+c\bar z_2,
       \qquad
       \dot z_2=a z_1-c\bar z_1.
      \end{align*}
      
      Write
      $x=(\Re z_1,\Im z_1,\Re z_2,\Im z_2)$. After an immaterial rescaling
      and sign choice of the four real Fourier phases, the choices
      $(a,c)=(1,0)$, $(\mathrm i,0)$, $(0,1)$, and $(0,\mathrm i)$ give,
      respectively, the skew symmetric matrices
      \begin{align*}
      M_1&=\begin{pmatrix}0&0&-1&0\\0&0&0&-1\\1&0&0&0\\0&1&0&0\end{pmatrix},&
      M_2&=\begin{pmatrix}0&0&0&-1\\0&0&1&0\\0&-1&0&0\\1&0&0&0\end{pmatrix},\\
      M_3&=\begin{pmatrix}0&0&1&0\\0&0&0&-1\\-1&0&0&0\\0&1&0&0\end{pmatrix},&
      M_4&=\begin{pmatrix}0&0&0&1\\0&0&1&0\\0&-1&0&0\\-1&0&0&0\end{pmatrix}.
      \end{align*}
      
      To see directly that these controls generate all of $\mathfrak{so}(4)$,
      let $J_{ij}=E_{ij}-E_{ji}$. Then
      \begin{align*}
 M_1&=-J_{13}-J_{24},\qquad M_2=-J_{14}+J_{23},\\
 M_3&=J_{13}-J_{24},\qquad M_4=J_{14}+J_{23}.
\end{align*}
      while
      \begin{align*}
       [M_1,M_2]=2(J_{12}-J_{34}),
       \qquad
       [M_3,M_4]=-2(J_{12}+J_{34}).
      \end{align*}
      Hence the Lie algebra generated by $M_1,\ldots,M_4$ contains all six
      standard generators $J_{ij}$, $1\le i<j\le4$, and therefore equals
      $\mathfrak{so}(4)$. By the standard analytic subgroup theorem for
      matrix Lie groups \cite[Chapter~II, Theorem~2.1]{Hall15}, the connected
      subgroup generated by the one parameter groups
      $\exp(tM_j)$ has precisely this Lie algebra. Since $SO(4)$ is connected,
      that subgroup is $SO(4)$. It acts transitively on the unit sphere of the
      real four dimensional space $E_1$. This proves \textup{(iii)}.
\end{proof}

Put $\Pi_1^\perp=I-\Pi_1$.  In accordance with the calculation in
\Cref{lem:first-shell-seeding-control}\textup{(iii)}, define the four shell two real phases by
\begin{align*}
 \mathcal A_2=\{e_{\boldsymbol e_2-\boldsymbol e_1,\cos},e_{\boldsymbol e_2-\boldsymbol e_1,\sin},e_{\boldsymbol e_1+\boldsymbol e_2,\cos},e_{\boldsymbol e_1+\boldsymbol e_2,\sin}\}.
\end{align*}
For $\upsilon\in\mathcal A_2$, set
\begin{align*}
 R_\upsilon=\left.\Pi_1\cS(\upsilon,\cdot)\right|_{E_1}.
\end{align*}
For $f,g\in E_1$, one has $\cS(f,g)=0$ and therefore
$\ip{R_\upsilon f}{g}+\ip{f}{R_\upsilon g}=-\ip{\cS(f,g)}{\upsilon}=0$.  Thus $R_\upsilon\in\mathfrak{so}(E_1)$.
Fix $\beta\in C_c^\infty(0,\delta)$ with
$\int_0^\delta\beta(s)\,\dd s=1$.  Let $U_{\varepsilon,\upsilon}$ be the
raw fibre propagator over the base loop
$\bar w(s)=\varepsilon\beta(s)\upsilon$, and remove the common shell one
heat scalar by setting
\begin{align}\label{eq:veup}
V_{\varepsilon,\upsilon}
=e^{\nu\delta}U_{\varepsilon,\upsilon}(\delta,0).
\end{align}

To realise the shell rotations by actual full PDE pulses, we need a
uniform first order expansion on $E_1$ together with bounds for leakage to
and from higher shells.  The next lemma provides these estimates for each
of the four generating phases.
\begin{lemma}
\label{lem:rev-one-pulse-blocks}
There are $\varepsilon_0>0$ and $C<\infty$, uniform over
$\upsilon\in\mathcal A_2$, such that, for
$|\varepsilon|\le\varepsilon_0$,
\begin{align}
\bigl\|\Pi_1V_{\varepsilon,\upsilon}\Pi_1-(I_{E_1}+\varepsilon R_\upsilon)\bigr\|
&\le C\varepsilon^2,
\label{eq:rev-pulse-EE}\\
\bigl\|\Pi_1^\perp V_{\varepsilon,\upsilon}\Pi_1\bigr\|
+\bigl\|\Pi_1V_{\varepsilon,\upsilon}\Pi_1^\perp\bigr\|
&\le C|\varepsilon|,
\notag\\
\bigl\|\Pi_1^\perp V_{\varepsilon,\upsilon}\Pi_1^\perp\bigr\|
&\le e^{-\nu\delta}+C|\varepsilon|,
\qquad
\bigl\|V_{\varepsilon,\upsilon}\bigr\|\le C.
\notag
\end{align}
\end{lemma}

\begin{proof}
      Put $S_t=e^{\nu t\Delta}$ and $H_\upsilon=\mathcal S(\upsilon,\cdot)$.  Since
      $\upsilon$ is one of the four fixed smooth Fourier phases,
      $H_\upsilon:H\to H^{-1}$ is bounded, while
      $\|S_t\|_{\mathcal L(H^{-1},H)}\le Ct^{-1/2}$.  Hence the Duhamel formula
      for the fibre equation over $\bar w(s)=\varepsilon\beta(s)\upsilon$ gives,
      uniformly for $|\varepsilon|$ sufficiently small,
      \begin{align*}
      U_{\varepsilon,\upsilon}(t,0)
      =S_t+\varepsilon\int_0^tS_{t-s}\beta(s)H_\upsilon
      U_{\varepsilon,\upsilon}(s,0)\,\dd s,
      \qquad
      \sup_{t\le\delta}\|U_{\varepsilon,\upsilon}(t,0)\|\le C.
      \end{align*}
      It follows first that
      $\sup_{t\le\delta}\|U_{\varepsilon,\upsilon}(t,0)-S_t\|\le C|\varepsilon|$,
      and substitution of this estimate once more into the Duhamel formula yields
      the operator norm expansion
      \begin{align*}
      U_{\varepsilon,\upsilon}(\delta,0)
      =S_\delta+\varepsilon\int_0^\delta
      S_{\delta-s}\beta(s)H_\upsilon S_s\,\dd s+O(\varepsilon^2)
      \qquad\text{in }\mathcal L(H).
      \end{align*}
      The constants are uniform in $\upsilon\in\mathcal A_2$ because
      $\mathcal A_2$ contains only four elements.
      
      Now let $v\in E_1$.  Since $S_sv=e^{-\nu s}v$ and
      $\Pi_1S_{\delta-s}=e^{-\nu(\delta-s)}\Pi_1$, multiplication by
      $e^{\nu\delta}$ gives
      \begin{align*}
      e^{\nu\delta}\Pi_1
      \int_0^\delta S_{\delta-s}\beta(s)H_\upsilon S_sv\,\dd s
      =\left(\int_0^\delta\beta(s)\,\dd s\right)\Pi_1H_\upsilon v
      =R_\upsilon v.
      \end{align*}
      Therefore
      \begin{align*}
      \Pi_1V_{\varepsilon,\upsilon}\Pi_1
      =I_{E_1}+\varepsilon R_\upsilon+O(\varepsilon^2),
      \end{align*}
      which proves \eqref{eq:rev-pulse-EE}.
      
      At $\varepsilon=0$ one has
      $V_{0,\upsilon}=e^{\nu\delta}e^{\nu\delta\Delta}$, which preserves the
      Laplace shell decomposition.  Hence both cross blocks vanish, whereas on
      $E_1^\perp$ the smallest Laplace eigenvalue is $2$, so
      \begin{align*}
      \|\Pi_1^\perp V_{0,\upsilon}\Pi_1^\perp\|
      =e^{-\nu\delta}.
      \end{align*}
      Finally, the preceding Duhamel expansion gives
      $\|V_{\varepsilon,\upsilon}-V_{0,\upsilon}\|\le C|\varepsilon|$ and
      $\|V_{\varepsilon,\upsilon}\|\le C$.  The remaining three estimates follow
      immediately.
\end{proof}

For a zero base relaxation of length $L$, use the projectively equivalent
scaled heat operator
\begin{align*}
D_L=e^{\nu L}e^{\nu L\Delta},
\qquad
D_L\Pi_1=\Pi_1,
\qquad
\bigl\|D_L\Pi_1^\perp\bigr\|\le e^{-\nu L}.
\end{align*}

The one pulse estimate can now be iterated with zero base relaxation
between pulses.  The next proposition shows that a finite ideal word in
$E_1$ survives in the full PDE while the transverse leakage is made
arbitrarily small.
\begin{proposition}
\label{lem:pulse-relax-concatenation}
Let $\upsilon_1,\dots,\upsilon_N\in\mathcal A_2$ and
$|\varepsilon_j|\le\varepsilon_0$ where $\varepsilon_0$ is from \Cref{lem:rev-one-pulse-blocks}, and set
$G_j=I_{E_1}+\varepsilon_jR_{\upsilon_j}$.
Let $x\in H$ satisfy $\Pi_1x\ne0$, and suppose that
\begin{align}
 d_{\PP(E_1)}
 \bigl([G_N\cdots G_1\Pi_1x],[\phi]\bigr)<\eta/4.
 \label{eq:rev-ideal-word-target}
\end{align}
If $\sum_{j=1}^N\varepsilon_j^2$ is sufficiently small, then there are
finite relaxation times $L_0,\dots,L_N$ such that, with
\begin{align*}
 x_0=D_{L_0}x,\qquad
 x_{j+1}
 =D_{L_{j+1}}V_{\varepsilon_{j+1},\upsilon_{j+1}}x_j,
 \qquad 0\le j<N,
\end{align*}
where $V_{\varepsilon_{j+1},\upsilon_{j+1}}$ is given as in \eqref{eq:veup}, one has
\begin{align*}
 \Pi_1x_j\ne0,\qquad 0\le j\le N,
 \qquad
 d_{\PP(H)}([x_N],[\phi])<\eta.
\end{align*}
The corresponding base pieces form a smooth zero to zero loop of finite
Cameron--Martin energy.
\end{proposition}

\begin{proof}
Write
\begin{align*}
 a_j=\Pi_1x_j,\qquad
 b_j=\Pi_1^\perp x_j,\qquad
 r_j=\frac{\|b_j\|_2}{\|a_j\|_2}.
\end{align*}
Since $D_L$ is the identity on $E_1$ and contracts $E_1^\perp$ by
$e^{-\nu L}$, \Cref{lem:rev-one-pulse-blocks} gives
\begin{align*}
 a_{j+1}&=G_{j+1}a_j+e_{j+1},\\
 \|e_{j+1}\|_2
 &\le C\bigl(\varepsilon_{j+1}^2
       +|\varepsilon_{j+1}|r_j\bigr)\|a_j\|_2,\\
 \|b_{j+1}\|_2
 &\le Ce^{-\nu L_{j+1}}
       (1+r_j)\|a_j\|_2.
\end{align*}
Since $R_{\upsilon_{j+1}}$ is skew adjoint,
$\|G_{j+1}a_j\|_2\ge\|a_j\|_2$.  Hence, whenever
\begin{align*}
 C\varepsilon_{j+1}^2
 +C|\varepsilon_{j+1}|r_j\le\frac14,
\end{align*}
we have
\begin{align*}
 \|a_{j+1}\|_2\ge\frac34\|a_j\|_2,
 \qquad
 r_{j+1}
 \le\frac{4C}{3}e^{-\nu L_{j+1}}(1+r_j).
\end{align*}

Choose small numbers $\theta_0,\dots,\theta_N>0$.  First take $L_0$
large enough that $r_0\le\theta_0$.  Inductively, after $x_j$ has been
constructed, choose $L_{j+1}<\infty$ so that
$r_{j+1}\le\theta_{j+1}$.  Taking the $\theta_j$ sufficiently small,
and using the assumed smallness of $\sum_j\varepsilon_j^2$, ensures the
preceding admissibility condition at every step.  In particular,
$a_j\ne0$ throughout the word.

It remains to compare the first shell motion with the ideal word.  Put
$y_0=a_0$ and $y_{j+1}=G_{j+1}y_j$.  Since the word is finite, iteration
of the preceding recurrence gives
\begin{align*}
 \|a_N-y_N\|_2
 \le C_{\rm w}\|a_0\|_2
 \left(
   \sum_{j=1}^N\varepsilon_j^2
   +\sum_{j=1}^N|\varepsilon_j|\theta_{j-1}
 \right),
\end{align*}
where $C_{\rm w}<\infty$ depends only on the fixed finite word.
Moreover, skew adjointness gives
$\|y_j\|_2\ge\|a_0\|_2$ for every $j$.  Thus, by first taking
$\sum_j\varepsilon_j^2$ sufficiently small and then the
$\theta_j$ sufficiently small, the true first shell endpoint is
arbitrarily close, projectively, to $[y_N]$, while
$r_N\le\theta_N$ makes the $E_1^\perp$ component arbitrarily small.
Together with \eqref{eq:rev-ideal-word-target}, this gives
\begin{align*}
 d_{\PP(H)}([x_N],[\phi])<\eta.
\end{align*}

Finally, each pulse has base path
$\bar w(s)=\varepsilon\beta(s)\upsilon$, with
$\beta\in C_c^\infty(0,\delta)$, and therefore joins smoothly to the zero
base relaxation intervals.  Its skeleton control is
\begin{align*}
 g_{\varepsilon,\upsilon}
 =Q^{-1}\bigl(
   \varepsilon\beta'\upsilon
   -\nu\varepsilon\beta\Delta\upsilon
   +\varepsilon^2\beta^2B(\upsilon,\upsilon)
 \bigr).
\end{align*}
All terms have finite Fourier support and the corresponding noise
coefficients are nonzero, so
\begin{align}
 \|g_{\varepsilon,\upsilon}\|_{L^2(0,\delta;H)}^2
 \le C_\upsilon(\varepsilon^2+\varepsilon^4).
 \label{eq:rev-pulse-energy}
\end{align}
There are only finitely many pulses and the relaxation intervals use zero
control, hence the concatenated loop has finite Cameron--Martin energy.
\end{proof}

The finite dimensional controllability can therefore be lifted to the full
projective fibre.  The following corollary packages seeding, relaxation, and
rotation into a single finite energy zero to zero controlled loop.

\begin{corollary}
\label{cor:rev-accessibility-closed}
For every $h\in H\setminus\{0\}$ and every $\eta>0$, there is a smooth
finite energy zero to zero controlled base path whose tangent propagator
sends $[h]$ into the $\eta$ neighborhood of $[\phi]$.  The control time is
finite and may depend on $h$.
\end{corollary}

\begin{proof}
If $\Pi_1h=0$, \Cref{lem:first-shell-seeding-control}\textup{(i)} supplies one smooth finite energy pulse after which the first shell projection is nonzero. If $\Pi_1h\ne0$, no seeding pulse is needed. In either case, \Cref{lem:first-shell-seeding-control}\textup{(ii)} gives an input $x$ with $\Pi_1x\ne0$ and with $\|\Pi_1^\perp x\|_2/\|\Pi_1x\|_2$ arbitrarily small.

By \Cref{lem:first-shell-seeding-control}\textup{(iii)}, the operators
$R_\upsilon$, $\upsilon\in\mathcal A_2$, generate
$\mathfrak{so}(E_1)$.  Hence finite products of the corresponding one
parameter groups generate $SO(E_1)$, and a finite product of exponentials $\exp(tR_\upsilon)$ sends $[\Pi_1x]$ to $[\phi]$. Approximate each factor by its Euler products
\begin{align*}
 \exp(tR_\upsilon)
 =\lim_{n\to\infty}
 \left(I_{E_1}+\frac{t}{n}R_\upsilon\right)^n.
\end{align*}
Since the product contains only finitely many factors, the subdivisions
may be chosen so fine that the resulting finite word satisfies
\eqref{eq:rev-ideal-word-target}, while
\begin{align*}
 \sum_j\varepsilon_j^2
\end{align*}
is as small as required in
\Cref{lem:pulse-relax-concatenation}, thereby realizes the
word by the full PDE, with sufficiently small transverse leakage, and
therefore sends the projective fibre into the $\eta$ neighborhood of
$[\phi]$.

All relaxation intervals have zero control, while the seeding pulse has
finite energy and the finitely many rotation pulses satisfy
\eqref{eq:rev-pulse-energy}.  Their concatenation is therefore a smooth
zero to zero controlled base path of finite Cameron--Martin energy.
\end{proof}

\begin{proof}[Proof of \Cref{thm:accessibility}]
Fix $z=(w,[\rho])$ and $r>0$.  After a short zero control interval,
parabolic smoothing makes both the base and the raw fibre smooth, while
backward uniqueness keeps the fibre nonzero.  Connect the resulting smooth
base state to zero by a smooth controlled path using
\eqref{eq:skeleton-control}, and then apply
\Cref{cor:rev-accessibility-closed} to steer the projective fibre into an
arbitrarily small neighborhood of $[\phi]$.  Thus there is a finite energy
Cameron--Martin skeleton from $z$ whose endpoint belongs to
$\mathcal U_r$.

In convolution coordinates this skeleton corresponds to a point
$\Gamma_Tg\in\mathsf Y_T$.  Since
\begin{align*}
 \supp\mathcal L(Y_T)
 =\overline{\Gamma_T\mathcal H_T}^{\,\mathsf Y_T}
\end{align*}
by \eqref{eq:YT-support}, and the endpoint map is continuous at the smooth
skeleton, every sufficiently small neighborhood of it has positive
probability.  Hence
\begin{align*}
 P_T(z,\mathcal U_r)>0,
\end{align*}
which proves \eqref{eq:point-access}.

For the sampled assertion, choose $r'>0$ so small that the zero control
flow sends $\mathcal U_{r'}$ into $\mathcal U_r$ for every
$s\in[0,\tau]$.  Apply the preceding construction with $\mathcal U_{r'}$
as target, and append zero control until the next time $n\tau$.
The same support argument then gives
\begin{align*}
 P_{n\tau}(z,\mathcal U_r)>0
\end{align*}
for some $n=n(z,r)\ge1$, proving
\eqref{eq:sampled-point-access}.

Finally, letting $\nu P_\tau=\nu$ and setting
\begin{align*}
 A_n=\{z:P_{n\tau}(z,\mathcal U_r)>0\},\qquad n\ge1, 
\end{align*}
then $\bigcup_{n\ge1}A_n=\mathsf X$, so
$\nu(A_n)>0$ for some $n$ and invariance gives
\begin{align*}
 \nu(\mathcal U_r)
 =\int_{\mathsf X}P_{n\tau}(z,\mathcal U_r)\,\nu(\dd z)>0.
\end{align*}
\end{proof}

\section{Asymptotic generalized coupling and uniqueness}
\label{sec:compact-dense-closure}

This section connects the random stable ball from
\Cref{sec:random-stable-ball} to the prepared neighborhood.  Since the local bridge and the selector are
both written in the convolution coordinate, only one Gaussian driver is
needed throughout this section.

\subsection{A compact outer endpoint}
\label{subsec:strong-convolution-driver}

The bridge in \Cref{sec:prepared} is formulated in the Hilbert driver
$\mathsf Y$, which is the topology needed for differentiability and
Cameron--Martin translation.  At one point we also need compactness of the
endpoint in the stronger state topology entering the Lyapunov function.
For this purpose only a mild additional spatial regularity of the
stochastic convolution near the end of the block is needed.

Retain $b<\beta<a/2-1$ from \Cref{sec:foundations} and choose a noninteger
\begin{align*}
 \beta<\beta_+<\frac a2.
\end{align*}
Set
\begin{align}\label{eq:strong-convolution-space}
 \mathsf Y^+
 =
 \left\{
 y=(y^\circ,y^\tau)\in\mathsf Y:
 y^\circ|_{[\tau/3,\tau]}\in
 C([\tau/3,\tau];c^{\beta_+}),\
 y^\tau\in c^{\beta_+}
 \right\},
\end{align}
with the sum norm.  Here $c^r$ denotes the separable little H\"older space.
The inclusion $\mathsf Y^+\hookrightarrow\mathsf Y$ is continuous.

We need the bridge event to land not merely in the Hilbert state space but
in a set on which the Lyapunov function is bounded.  A small amount of extra
spatial regularity of the convolution near the end of the block gives the
required compact driver event.
\begin{lemma}
\label{lem:compact-strong-driver-event}
The stochastic convolution driver $Y=(Z,Z_\tau)$ has a centred Radon
Gaussian version in $\mathsf Y^+$.  Consequently, for every $\delta>0$
there is a compact set $\mathfrak L\subset\mathsf Y^+$ such that
\begin{align}\label{eq:compact-small-driver}
 p_{\mathfrak L}:=\mathbf P\{Y\in\mathfrak L\}>0,
 \qquad
 \sup_{y\in\mathfrak L}\|y\|_{\mathsf Y}<\delta.
\end{align}
The version in $\mathsf Y^+$ may be chosen from the same Fourier partial
sums as the canonical $\mathsf Y$ valued version, and hence the two agree
almost surely.
\end{lemma}

\begin{proof}
For an inhomogeneous Littlewood--Paley block $\Delta_j$, Gaussian maximal
estimates give, for every finite $r$,
\begin{align}\label{eq:strong-OU-dyadic}
 \mathbf E\|\Delta_j(Z_t-Z_s)\|_\infty^r
 \le C_r(1+j)^{r/2}2^{r(1-a/2)j}
       \min\{|t-s|^{1/2},2^{-j}\}^r.
\end{align}
Choose $\vartheta>0$ so small that
$\beta_++2\vartheta<a/2$, and then choose $r$ so large that the time
Kolmogorov loss is smaller than $\vartheta$.  Using
\begin{align*}
 \min\{h^{1/2},2^{-j}\}
 \le h^\vartheta2^{-j(1-2\vartheta)}
\end{align*}
in \eqref{eq:strong-OU-dyadic}, summing the spatial Fourier tails, and
applying the Banach valued Kolmogorov theorem show that the Fourier partial
sums converge almost surely in
$C([\tau/3,\tau];c^{\beta_+})$.  The same partial sums already converge in
$\mathsf Y$ by the construction in \Cref{sec:random-stable-ball}, while
evaluation at $\tau$ gives the $c^{\beta_+}$ terminal component.  Thus they
converge in the intersection norm \eqref{eq:strong-convolution-space}.
Since $\mathsf Y^+$ is separable, the resulting Gaussian law is Radon.

The centred Gaussian law of $Y$ assigns positive mass to every
$\mathsf Y$ neighbourhood of zero.  Hence
$\mathbf P\{\|Y\|_{\mathsf Y}<\delta/2\}>0$.  This event is Borel in
$\mathsf Y^+$, so inner regularity of the $\mathsf Y^+$ valued law gives a
positive probability compact subset $\mathfrak L$ of it.  This proves
\eqref{eq:compact-small-driver}.
\end{proof}

Apply \Cref{thm:eventwise-bridge} with $T=\tau$ and a fixed $\theta>0$.
Let $\mathcal U_R$ be the resulting outer prepared neighbourhood and let
$\delta$ be its driver radius.  Shrink $R$ and $\delta$ if necessary so that
$\overline{\mathcal U}_R\times\{y:\|y\|_{\mathsf Y}\le\delta\}$
is contained in the neighborhood where
\eqref{eq:local-flow-modulus}--\eqref{eq:normalization-tube} hold.  Choose $\mathfrak L$ from
\Cref{lem:compact-strong-driver-event} with this $\delta$, and on the
canonical Wiener space put
\begin{align}\label{eq:outer-bridge-event}
 \mathcal A_{\mathfrak L}=\{Y(W)\in\mathfrak L\}.
\end{align}
Then $\mathcal A_{\mathfrak L}$ is Borel,
$\mathbf P(\mathcal A_{\mathfrak L})=p_{\mathfrak L}$, and
$\|Y(W)\|_{\mathsf Y}<\delta$ on $\mathcal A_{\mathfrak L}$.  The prepared chart comparison
therefore gives a finite constant $\theta_R$ such that, on the matching
event of the bridge,
\begin{align}\label{eq:outer-bridge-bound}
 d_{\mathsf X}(\mathbb X_1,\mathbb X_1^{\prime})
 \le\theta_R
 \left\|
 \Delta_{(0,[\phi])}(z')-\Delta_{(0,[\phi])}(z)
 \right\|_{E_b},
 \qquad z,z'\in\mathcal U_R.
\end{align}

The compact driver event can now be pushed through the prepared endpoint
map.  The next proposition shows that the resulting outer endpoint set is
compact in the state topology and bounded for the Lyapunov function used by
the stable ball estimate.
\begin{proposition}
\label{prop:compact-outer-endpoint}
Let the outer endpoint set be
\begin{align}\label{eq:Cout-def}
 C_{\rm out}
 =
 \overline{
 \{\Phi(z,y):z\in\overline{\mathcal U}_{R},\ y\in\mathfrak L\}
 }^{\,H^b\times\PP(H)}.
\end{align}
Then $C_{\rm out}$ is compact and the Lyapunov function $\mathcal W$ from
\Cref{cor:revision-conditional-Lyapunov} is bounded on $C_{\rm out}$.
More precisely, for some $\epsilon>0$,
\begin{align}\label{eq:outer-strong-bounds}
 \sup_{z\in\overline{\mathcal U}_{R},\,y\in\mathfrak L}
 \left(
 \|w_\tau(z,y)\|_{c^\beta}
 +
 \left\|
 \frac{U_{w^\circ(z,y)}(\tau,0)\mathsf p}
      {\|U_{w^\circ(z,y)}(\tau,0)\mathsf p\|_2}
 \right\|_{H^{1+\epsilon}}
 \right)<\infty,
\end{align}
where $\mathsf p$ is any unit representative of the projective component
of $z$.
\end{proposition}

\begin{proof}
Since $\mathfrak L$ is compact in $\mathsf Y^+$, it is bounded both in
$\mathsf Y$ and in
\begin{align*}
 C([\tau/3,\tau];c^{\beta_+})\oplus c^{\beta_+}.
\end{align*}
Together with the boundedness of $\overline{\mathcal U}_{R}$,
\Cref{lem:compatible-driver-global} and
\eqref{eq:rough-driver-global-bound} give a uniform
$C([0,\tau];H^b)$ bound for $v$ and a uniform
$L^{p_\sigma}(0,\tau;H^b)$ bound for
$w^\circ=v+y^\circ$.

Choose $\epsilon_0>0$ so small that
$b+1-3\epsilon_0>\beta$.  On $[\tau/3,\tau]$,
\begin{align*}
 H^b\hookrightarrow c^{b-1-\epsilon_0},
 \qquad
 y^\circ\in C([\tau/3,\tau];c^{\beta_+}),
 \qquad
 \beta_+>b.
\end{align*}
Apply the mild equation \eqref{eq:compatible-global-mild} for $v$ first on $[\tau/3,2\tau/3]$ and then on
$[2\tau/3,\tau]$.  Using
\begin{align*}
 B:c^r\times c^r\to c^{r-1},
 \qquad
 \|e^{\nu t\Delta}\|_{c^{r-1}\to c^{r+1-\epsilon_0}}
 \le Ct^{-1+\epsilon_0/2},
\end{align*}
the first step gives a uniform $c^{b-2\epsilon_0}$ bound for $v$ on the
second interval.  Since $\beta_+>b$, the same bound holds for
$w^\circ=v+y^\circ$ there.  Applying the estimate once more gives
\begin{align*}
 \sup_{z\in\overline{\mathcal U}_{R},\,y\in\mathfrak L}
 \|v_\tau(z,y)\|_{c^{b+1-3\epsilon_0}}<\infty.
\end{align*}
Since $b+1-3\epsilon_0>\beta$ and
$y^\tau\in c^{\beta_+}$ with $\beta_+>\beta$, we obtain
\begin{align*}
 \sup_{z\in\overline{\mathcal U}_{R},\,y\in\mathfrak L}
 \|w_\tau(z,y)\|_{c^\beta}<\infty.
\end{align*}

For the fibre, choose $\epsilon>0$ so small that
\begin{align*}
 \frac{1+\epsilon}{2}
 <1-\frac2{p_\sigma}=2\sigma.
\end{align*}
The uniform $L^{p_\sigma}(0,\tau;H^b)$ bound for $w^\circ$ and
\Cref{lem:gain}\textup{(ii)} allow two successive gains of
$(1+\epsilon)/2$ on the intervals
$[\tau/3,2\tau/3]$ and $[2\tau/3,\tau]$.  Hence
\begin{align*}
 U_{w^\circ(z,y)}(\tau,0):H\longrightarrow H^{1+\epsilon}
\end{align*}
with operator norm bounded uniformly for
$z\in\overline{\mathcal U}_{R}$ and $y\in\mathfrak L$.
By \eqref{eq:normalization-tube},
\begin{align*}
 \|U_{w^\circ(z,y)}(\tau,0)\mathsf p\|_2
 \ge\frac12e^{-\nu\tau}
\end{align*}
uniformly on the same set.  This proves the second bound in
\eqref{eq:outer-strong-bounds}.

The compact embeddings
\begin{align*}
 c^\beta\Subset H^b,
 \qquad
 H^{1+\epsilon}\Subset H
\end{align*}
show that the endpoint image in \eqref{eq:Cout-def} is relatively compact.
Hence $C_{\rm out}$ is compact.

It remains to bound $\mathcal W$.  By \eqref{eq:block-log-cost} and
\eqref{eq:revision-W-augmentation}, together with
\Cref{lem:projective-log-moment} at $r=p_{\rm occ}^2$,
the bound \eqref{eq:outer-strong-bounds} implies that
$\mathcal W$ is uniformly bounded on the endpoint image.

Finally, $\mathcal W$ is lower semicontinuous by
\Cref{cor:revision-conditional-Lyapunov}.  Therefore its finite sublevel
sets are closed.  Since the endpoint image is contained in one such
sublevel set, its closure $C_{\rm out}$ is contained in the same sublevel
set.  Hence $\mathcal W$ is bounded on $C_{\rm out}$.
\end{proof}

\subsection{Proof of the main theorem}
\label{subsec:coupling-closure}

The three ingredients above can now be concatenated: accessibility of the
prepared neighborhood, the exact marginal bridge into a compact outer set,
and the random stable ball realised by the triangular Ramer coupling.  This
produces the positive probability asymptotic generalized coupling required
by the uniqueness criterion of \cite{HMS11}.

\begin{proof}[Proof of \Cref{thm:main}]
Existence of a $P_\tau$ invariant probability measure was established at
the beginning of \Cref{subsec:universal-compact-core-selector}.  It remains to prove uniqueness.

Since $\mathcal W$ is bounded on $C_{\rm out}$ by
\Cref{prop:compact-outer-endpoint}, the uniform quantile in
\Cref{lem:random-stable-radius} gives
\begin{align}\label{eq:Cout-stable}
 r_0,p_0>0,\qquad
 \inf_{x\in C_{\rm out}}\mathbf P_x\{r_*(x,\omega)\ge r_0\}\ge p_0.
\end{align}
Decrease and relabel $r_0$ so that $r_0<1$ and, by
\eqref{eq:state-chart-metric-comparison},
$d_{\mathsf X}(x,x')<r_0$ implies $(x,x')\in\mathscr U_\Delta$ and
$\|\Delta_x(x')\|_{E_x}$ is smaller than the stable radius in
\eqref{eq:Cout-stable}.  The constants in this comparison are universal,
so the choice is uniform for $x\in C_{\rm out}$.

Choose an inner prepared neighborhood
$\mathcal U_r\subset\mathcal U_R$ so small that
$2\theta_Rr<r_0$ and that the maximal coupling loss in
\eqref{eq:eventwise-bridge} is at most $p_{\mathfrak L}/2$.  Apply
\Cref{thm:eventwise-bridge} to the fixed Borel event
$\mathcal A_{\mathfrak L}$ from \eqref{eq:outer-bridge-event}, and denote
by $(Z_1,Z_1')$ the endpoint pair under the resulting Borel bridge law
$\mathbf Q^{\rm br}_{z,z'}$.  On its matching event,
$Z_1\in C_{\rm out}$ and \eqref{eq:outer-bridge-bound} gives
$d_{\mathsf X}(Z_1,Z_1')<r_0$.  Therefore, with
$p_b:=p_{\mathfrak L}/2>0$,
\begin{align}\label{eq:bridge-to-stable}
 \inf_{z,z'\in\mathcal U_r}
 \mathbf Q^{\rm br}_{z,z'}
 \{Z_1\in C_{\rm out},\ d_{\mathsf X}(Z_1,Z_1')<r_0\}
 \ge p_b.
\end{align}

Let
\begin{align*}
 E_{\rm br}
 &=
 \{Z_1\in C_{\rm out},\ d_{\mathsf X}(Z_1,Z_1')<r_0\},\\
 \mathsf D_\infty
 &=
 \{((X_n),(X_n')):d_{\mathsf X}(X_n,X_n')\to0\}.
\end{align*}
Conditionally on the bridge sigma field, attach a fresh future Wiener
sequence independent of the bridge block.  On $E_{\rm br}$ run the
mixture over entropy thresholds from \Cref{lem:infinite-Ramer} starting at
$(Z_1,Z_1')$; on $E_{\rm br}^c$ use conditionally independent genuine
Markov evolutions.  Because the first future coordinate on $E_{\rm br}$
has the genuine Markov law from $Z_1\in C_{\rm out}$,
\eqref{eq:Cout-stable} gives conditional probability at least $p_0$ for
the stable radius event.  On this event,
\Cref{lem:random-stable-radius} gives
$(X_n,X_n')\in\mathscr U_\Delta$ for every $n$,
$\sum_n\|h_n\|^2<\infty$, and exponential convergence.  Moreover, the
proof of that lemma gives $\|h_n\|<r_n/2\le d_{\rm Ram}$ for every $n$.
Hence the stable event is contained in $\mathcal E_\infty$ from
\Cref{lem:infinite-Ramer}, and the corresponding unstopped paired path
belongs to $\mathsf D_\infty$.  The last assertion of
\Cref{lem:infinite-Ramer}, applied with
$\mathsf D=\mathsf D_\infty$, therefore shows that the future
mixture over entropy thresholds assigns positive conditional probability
to $\mathsf D_\infty$ on $E_{\rm br}$.

On $E_{\rm br}$ the second conditional path law is absolutely continuous
with respect to the genuine Markov path law from $Z_1'$ by
\Cref{lem:infinite-Ramer}; on $E_{\rm br}^c$ it equals that law.
Disintegration at the bridge endpoint and exactness of both bridge
marginals therefore preserve absolute continuity of the second full path
marginal.  The Borel bridge kernel from \Cref{thm:eventwise-bridge},
concatenated with the Borel kernel obtained by mixing over entropy thresholds in
\Cref{lem:infinite-Ramer}, gives a Borel kernel
$(z,z')\mapsto\mathbf Q_{z,z'}$ on the two sampled path spaces.  Since
\eqref{eq:bridge-to-stable} gives
$\mathbf Q^{\rm br}_{z,z'}(E_{\rm br})\ge p_b>0$, the preceding
conditional success and marginal calculation yield
\begin{align}
 \mathbf Q_{z,z'}(\mathsf D_\infty)&>0,
 \label{eq:positive-asymptotic}\\
 (\operatorname{pr}_1)_\#\mathbf Q_{z,z'}&=P_{\tau,[\infty]}\delta_z,
 \qquad
 (\operatorname{pr}_2)_\#\mathbf Q_{z,z'}
 \ll P_{\tau,[\infty]}\delta_{z'}.
 \label{eq:generalized-marginals}
\end{align}

By the last assertion of \Cref{thm:accessibility}, every
$P_\tau$ invariant probability measure charges $\mathcal U_r$.  We may
therefore apply \cite[Corollary~2.2]{HMS11} with the Markov operator
$P_\tau$, the Borel set $A=\mathcal U_r$, and the Borel convergence event
$\mathsf D_\infty$.  For every $z,z'\in A$, the map
$(z,z')\mapsto\mathbf Q_{z,z'}$ is Borel, and
\eqref{eq:positive-asymptotic}--\eqref{eq:generalized-marginals} give
\begin{align*}
 \mathbf Q_{z,z'}(\mathsf D_\infty)>0,\qquad
 (\operatorname{pr}_1)_\#\mathbf Q_{z,z'}=P_{\tau,[\infty]}\delta_z,\qquad
 (\operatorname{pr}_2)_\#\mathbf Q_{z,z'}
 \ll P_{\tau,[\infty]}\delta_{z'}.
\end{align*}
Thus all the hypotheses of the cited asymptotic coupling criterion are
satisfied, and $P_\tau$ has at most one invariant probability measure.
Existence therefore gives a unique $P_\tau$ invariant probability measure,
denoted by $\widehat\mu$.

For every $t\ge0$,
\begin{align*}
 (\widehat\mu P_t)P_\tau
 =(\widehat\mu P_\tau)P_t
 =\widehat\mu P_t.
\end{align*}
Thus $\widehat\mu P_t$ is $P_\tau$ invariant, so uniqueness gives
$\widehat\mu P_t=\widehat\mu$ for every $t\ge0$.  Hence
$\widehat\mu$ is invariant for the full continuous time semigroup.
Conversely, every continuous time invariant probability measure is
$P_\tau$ invariant, and therefore equals $\widehat\mu$.  This proves
uniqueness for the continuous time projective process and completes
the proof.
\end{proof}

\subsection{Proof of the Furstenberg--Khasminskii formula}
\label{subsec:FK-proof}

\begin{proof}[Proof of \Cref{cor:exact-FK}]
Let $\mu$ be the first marginal of $\widehat\mu$, and let $t_*>0$ be the reset time in \Cref{lem:robust-HPSRY-hierarchy}. For a unit representative $\mathsf p$, set
\begin{align*}
 \mathfrak a(w,[\mathsf p])=-\nu\|\nabla\mathsf p\|_2^2+\langle w,B(\mathsf p,\mathsf p)\rangle_2,
\end{align*}
which is independent of the choice of sign of $\mathsf p$.

We first give the two ingredients used below. By invariance of $\widehat\mu$, \eqref{eq:robust-all-H1-moments}, and the integrability of the base Lyapunov weights,
\begin{align*}
 \int\|\mathsf p\|_{H^1}^2\,\widehat\mu(\dd w,\dd[\mathsf p])<\infty.
\end{align*}
Moreover, since $b>2$, integration by parts and the Biot--Savart gain give
\begin{align*}
 |\langle w,B(\mathsf p,\mathsf p)\rangle_2|=|\langle\mathsf p,B(\mathsf p,w)\rangle_2|\le C\|w\|_{H^b}\|\mathsf p\|_2^2.
\end{align*}
Hence
\begin{align}\label{eq:FK-integrability}
 \mathfrak a\in L^1(\widehat\mu).
\end{align}
If $\rho_t=U(t,0)\mathsf p$ and $\mathsf p_t=\rho_t/\|\rho_t\|_2$, backward uniqueness ensures that $\rho_t\ne0$ for $t>0$, and the energy identity yields
\begin{align}\label{eq:FK-additive-functional}
 \frac{\dd}{\dd t}\log\|\rho_t\|_2=-\nu\|\nabla\mathsf p_t\|_2^2+\langle w_t,B(\mathsf p_t,\mathsf p_t)\rangle_2=\mathfrak a(w_t,[\mathsf p_t]).
\end{align}

It remains to identify the stationary exponent in \eqref{eq:FK-additive-functional} with the top Lyapunov exponent $\lambda_1$. Start the base with $w_0\sim\mu$. The compact multiplicative ergodic theorem \cite{Ruelle82,LianLu10} applies to the cocycle on $H$: positive time smoothing makes $U(1,0)$ compact, backward uniqueness makes it injective, and the energy estimate together with stationarity gives the required integrability of $\log^+\|U(1,0)\|_{\mathcal L(H)}$. Consequently, for almost every $(\omega,w_0)$ there is a proper closed subspace $F(\omega,w_0)\subset H$ such that
\begin{align*}
 \lim_{t\to\infty}\frac1t\log\|U(t,0)\mathsf p\|_2=\lambda_1
\end{align*}
for every $\mathsf p\notin F(\omega,w_0)$.

Choose a covariance injective $H^1$ valued Gaussian measure $\gamma$ on $H$. Since $\gamma(F)=0$ for every proper closed linear subspace $F\subset H$, Fubini's theorem yields a deterministic unit vector $\mathsf p_*\in H^1$ such that
\begin{align}\label{eq:generic-top-direction}
 \lim_{t\to\infty}\frac1t\log\|U(t,0)\mathsf p_*\|_2=\lambda_1
 \qquad\text{almost surely}.
\end{align}

Write
\begin{align*}
 \mathsf p_t=\frac{U(t,0)\mathsf p_*}{\|U(t,0)\mathsf p_*\|_2},\qquad Z_t=(w_t,[\mathsf p_t]),\qquad L_t=\log\|U(t,0)\mathsf p_*\|_2,
\end{align*}
and define the mean occupation laws
\begin{align*}
 \overline\nu_T=\frac1T\int_0^T\operatorname{Law}(Z_s)\,\dd s.
\end{align*}
The positive time reset estimates for the base and projective direction make $(\overline\nu_T)_{T\ge1}$ tight. Every weak limit is invariant for the projective semigroup. Indeed, for $\varphi\in C_b$ and $r>0$,
\begin{align*}
 \int P_r\varphi\,\dd\overline\nu_T-\int\varphi\,\dd\overline\nu_T
 =\frac1T\left(\int_T^{T+r}\mathbf E\varphi(Z_s)\,\dd s-\int_0^r\mathbf E\varphi(Z_s)\,\dd s\right)\longrightarrow0.
\end{align*}
The Feller property permits passage to weak limits. By \Cref{thm:main}, the invariant law is unique, and therefore
\begin{align}\label{eq:FK-occupation-convergence}
 \overline\nu_T\Longrightarrow\widehat\mu.
\end{align}

Fix $h=t_*$ and define
\begin{align*}
 \Psi_h(w,[\mathsf p])=\mathbf E_{w,[\mathsf p]}\log\|U(h,0)\mathsf p\|_2,
 \qquad \|\mathsf p\|_2=1.
\end{align*}
The logarithmic energy identity gives
\begin{align*}
 \left|\log\|U(h,0)\mathsf p\|_2\right|
 \le \nu\int_0^h\|\mathsf p_s\|_{H^1}^2\,\dd s+C\int_0^h(1+\|w_s\|_{H^b})\,\dd s.
\end{align*}
Together with \eqref{eq:robust-block-energy}, \eqref{eq:revision-cond-reset-p}, and \eqref{eq:revision-maximal-base-moment}, this implies, for some $r>1$,
\begin{align}\label{eq:FK-block-UI}
 \sup_{s\ge0}\mathbf E|L_{s+h}-L_s|^r<\infty,\qquad
 \sup_{t\ge1}\mathbf E\left|\frac{L_t}{t}\right|^r<\infty.
\end{align}
By conditional Jensen,
\begin{align*}
 \sup_{T\ge1}\int|\Psi_h|^r\,\dd\overline\nu_T<\infty.
\end{align*}

Fix $b<\beta<a/2-1$ and set
\begin{align*}
 \mathcal K_M=\{(w,[\mathsf p]):\|w\|_{C^\beta}\le M,\ \|\mathsf p\|_2=1,\ \|\mathsf p\|_{H^1}\le M\}.
\end{align*}
The compact embeddings $C^\beta\Subset H^b$ and $H^1\Subset H$ imply that $\mathcal K_M$ is compact in $H^b\times\PP(H)$. The reset estimates give
\begin{align*}
 \lim_{M\to\infty}\limsup_{T\to\infty}\overline\nu_T(\mathcal K_M^c)=0,\qquad
 \widehat\mu(\mathcal K_M^c)\longrightarrow0.
\end{align*}
On each $\mathcal K_M$, continuity of the solution map and backward uniqueness imply continuity of $\Psi_h$. Hence \eqref{eq:FK-occupation-convergence}, the compact core tightness above, and the uniform $L^r$ bound give, by a standard truncation argument,
\begin{align}\label{eq:FK-block-observable-limit}
 \int\Psi_h\,\dd\overline\nu_T\longrightarrow\int\Psi_h\,\dd\widehat\mu.
\end{align}

On the other hand, the Markov property and the cocycle identity give
\begin{align*}
 \int\Psi_h\,\dd\overline\nu_T
 &=\frac1T\int_0^T\mathbf E(L_{s+h}-L_s)\,\dd s\\
 &=\frac1T\mathbf E\left[\int_T^{T+h}L_s\,\dd s-\int_0^hL_s\,\dd s\right].
\end{align*}
By \eqref{eq:FK-block-UI}, $(L_t/t)_{t\ge1}$ is uniformly integrable. Thus \eqref{eq:generic-top-direction} and Vitali's theorem give
\begin{align*}
 \mathbf E\frac{L_t}{t}\longrightarrow\lambda_1,
\end{align*}
and therefore
\begin{align}\label{eq:FK-top-block}
 \lim_{T\to\infty}\int\Psi_h\,\dd\overline\nu_T=h\lambda_1.
\end{align}
Combining \eqref{eq:FK-block-observable-limit} and \eqref{eq:FK-top-block},
\begin{align*}
 h\lambda_1=\int\Psi_h\,\dd\widehat\mu.
\end{align*}

Finally, stationarity, \eqref{eq:FK-integrability}, and \eqref{eq:FK-additive-functional} give
\begin{align*}
 \int\Psi_h\,\dd\widehat\mu
 &=\int\mathbf E_{w,[\mathsf p]}\int_0^h\mathfrak a(w_s,[\mathsf p_s])\,\dd s\,\widehat\mu(\dd w,\dd[\mathsf p])\\
 &=h\int\mathfrak a(w,[\mathsf p])\,\widehat\mu(\dd w,\dd[\mathsf p]).
\end{align*}
Consequently,
\begin{align*}
 \lambda_1=\int_{H^b\times\PP(H)}
 \left(-\nu\|\nabla\mathsf p\|_2^2+\langle w,B(\mathsf p,\mathsf p)\rangle_2\right)
 \widehat\mu(\dd w,\dd[\mathsf p]),
\end{align*}
which proves \eqref{eq:exact-FK}.

Since $\widehat\mu$ is the unique invariant probability measure, Birkhoff's theorem and \eqref{eq:FK-additive-functional} therefore give
\begin{align*}
 \lim_{t\to\infty}\frac1t\log\|U(t,0)\mathsf p\|_2
 =\int\mathfrak a\,\dd\widehat\mu=\lambda_1
\end{align*}
for $\widehat\mu\otimes\mathbf P$ almost every stationary initial state and driving path. This proves \eqref{eq:stationary-top-growth}.
\end{proof}

\appendix

\section{Lyapunov structure for diagonal power law noise}
\label{sec:foundations}

This appendix develops the Lyapunov structure needed for diagonal power
law noise satisfying \eqref{eq:q}.  The spectral median mechanism of
\cite{HPSRY} is adapted to amplitudes
that may depend on the Fourier direction and real phase.  The covariance is
decomposed into an independent radial part and a nonnegative residual part;
a conditional one trial estimate and a fixed ratio spectral descent then
give the positive time reset and the inf-compact Lyapunov function used in
the coupling argument.  The required base estimates are collected in
\Cref{sec:base-linearisation-estimates}.

We also recall the parameter translation.  The equation of \cite{HPSRY} is
written for the velocity.  In the present variables $u=Kw$, the real velocity
noise amplitudes are comparable to
$|\boldsymbol k|^{-(a+2)/2}$.  Thus
\begin{align*}
 \alpha_{\rm H}=a+2,
 \qquad \frac{\alpha_{\rm H}-3}{2}=b+1,
\end{align*}
and $K:H^b\to H^{b+1}$.  The reduction from viscosity $\nu>0$ to unit
viscosity is
\begin{align*}
 s=\nu t,
 \qquad \bar w_s=\nu^{-1}w_{s/\nu},
 \qquad \bar W_s=\sqrt\nu W_{s/\nu}.
\end{align*}
The rescaled velocity noise amplitude is $\nu^{-3/2}$ times the original
one.  All constants below may therefore depend on
$(\nu,c_Q,C_Q)$. 

\subsection{A radial subnoise and a conditional one trial estimate}
\label{subsec:covariance-robust-reset-revision}

After the preceding rescaling, write the real velocity noise as
\begin{align*}
 \overline Q=Q_0D,
 \qquad
 Q_0e_{\boldsymbol k,\iota}
 =|\boldsymbol k|^{-\alpha_{\rm H}/2}e_{\boldsymbol k,\iota},
 \qquad
 De_{\boldsymbol k,\iota}=d_{\boldsymbol k,\iota}e_{\boldsymbol k,\iota}.
\end{align*}
Changing the sign of a real Brownian coordinate if necessary, we may and do assume
\begin{align*}
 0<c_*\le d_{\boldsymbol k,\iota}\le C_*<\infty.
\end{align*}
Put $d_0=c_*/2$ and
\begin{align*}
 r_{\boldsymbol k,\iota}
 =\bigl(d_{\boldsymbol k,\iota}^2-d_0^2\bigr)^{1/2},
 \qquad
 Re_{\boldsymbol k,\iota}=r_{\boldsymbol k,\iota}e_{\boldsymbol k,\iota}.
\end{align*}
On a product Wiener space carrying independent cylindrical Brownian motions
$W^0,W^1$, the Gaussian noise
\begin{align*}
 d_0Q_0\dd W^0+Q_0R\dd W^1
\end{align*}
has the same covariance, and therefore the same law, as
$Q_0D\dd W$.  We use this auxiliary realisation only for transition
estimates.  The Wiener coordinate used in the Malliavin and Ramer arguments
in the body of the paper is unchanged.

We state the one trial input in the unit viscosity velocity variables.  Let
$P_t=e^{t\Delta}$ be the heat semigroup and let
$\Pi_0=\Pi_{\le\sqrt2}$ be the fixed low mode projection.  Below
$\Pi_{\rm L}$ denotes the Leray projection and $\otimes_s$ the symmetric
tensor product.  Set
\begin{align*}
 \mathcal R[v]f&=vf+(\Delta v)(-\Delta)^{-1}f,\\
 \Phi[v,f]_t&=-\int_0^tP_{t-r}\operatorname{div}
                    \bigl(\mathcal R[v_r]f_r\bigr)\,\dd r,\\
 \Psi[v_1,v_2]_t&=-\int_0^tP_{t-r}\Pi_{\rm L}
       \operatorname{div}(v_{1,r}\mathbin\otimes_s v_{2,r})\,\dd r.
\end{align*}
For $\beta_0\ge1$ define the spectral quantile
\begin{align*}
 M^{(\beta_0)}(\rho)
 =\min\left\{M\in\mathbb N:
   \|\Pi_{>M}\rho\|_2\le
   \beta_0\|\Pi_{\le M}\rho\|_2\right\}.
\end{align*}

We begin the projective reset with one radial trial in a frozen compatible
background.  Relative to the radial setting of \cite{HPSRY}, the additional
point needed here is uniformity over the bounded residual background
generated by the nonradial covariance.  In this subsection
$\mathbf P_0$ and $\mathbf E_0$ denote probability and expectation with
respect to the radial noise $W^0$, with the compatible background fixed.
\begin{lemma}
\label{lem:revision-radial-background-trial}
Put
\begin{align*}
 s_*=\frac{\alpha_{\rm H}}2-\frac14,
 \qquad s_*<\gamma<\frac{\alpha_{\rm H}}2,
 \qquad \lambda_0=\frac{\alpha_{\rm H}}2+5,
 \qquad t_M=\lambda_0M^{-2}\log M.
\end{align*}
Fix $R_0<\infty$.  Let
$z\in C_0([0,t_M];C^\gamma)$ be a deterministic, mean zero,
divergence free compatible path with
$\|z\|_{C([0,t_M];C^\gamma)}\le R_0$, let $Z^0$ be the
stochastic convolution generated by the radial noise $d_0Q_0\dd W^0$, and
put
\begin{align*}
 D_t=P_tu_0+z_t,
 \qquad X=D+Z^0,
 \qquad U_0=1+\|u_0\|_{C^\gamma}.
\end{align*}
Let $u$ be the Navier--Stokes solution with linear part $X$, and let $\rho$
solve its linearised vorticity equation from $0\ne\rho_0\in L^2$.  There
are deterministic $M_0,K_{R_0},r_0<\infty$ and
$a_0,\alpha_0>0$, depending only on the fixed equation parameters, such
that, whenever
\begin{align}
 M\ge M_0\vee M^{(2)}(\rho_0)\vee K_{R_0}U_0^{r_0},
 \label{eq:revision-background-trial-threshold}
\end{align}
one has
\begin{align}
 \mathbf P_0\left\{
  \|\Pi_0\rho_{t_M}\|_2
  \ge a_0M^{-3}
       \|\rho_0\|_{H^{-(\alpha_{\rm H}/2+1)}}
 \right\}\ge\alpha_0.
 \label{eq:revision-background-trial-success}
\end{align}
The constants are uniform over all such compatible paths in the fixed
$R_0$ ball.  Here ``compatible'' means a limit, in the displayed topology,
of smooth Fourier truncated velocity paths; this class contains the
residual stochastic convolutions used below.
\end{lemma}

\begin{proof}
The deterministic background enters through the following terms.  With
$Y=P_t\rho_0$, the finite expansion of
\cite[(4.9)--(4.11), Sections~4--5]{HPSRY} gives
\begin{align}
 \rho=Y+\Phi[X,Y]+\mathcal P[X,Y]+\operatorname{Rem}[u,\rho],
 \label{eq:revision-radial-expansion}
\end{align}
where
\begin{align*}
 \mathcal P[X,Y]
 =\Phi[\Psi[X,X],Y]+
   \Phi[X,\Phi[X,Y]]+
   \Phi[X,\Phi[X,\Phi[X,Y]]].
\end{align*}
This is an algebraic mild expansion and remains valid for every fixed
$D\in C([0,t_M];C^\gamma)$.  Denote by $\Pi^{(1)}_0$ the first homogeneous
Wiener chaos projection with respect to $W^0$, and put
\begin{align*}
 A&=\Pi_0\Pi^{(1)}_0\Phi[Z^0,Y]_{t_M},\\
 E_D&=\Pi_0\Pi^{(1)}_0\mathcal P[D+Z^0,Y]_{t_M}.
\end{align*}
The leading term contains no $D$.  The radial Fourier calculation in
\cite[Lemmas~5.1--5.2]{HPSRY}, with the fixed intensity $d_0$, gives
\begin{align}
 \|A\|_{L^2(\mathbf P_0;L^2)}^2
 \ge cM^{-6}
   \|\rho_0\|_{H^{-(\alpha_{\rm H}/2+1)}}^2.
 \label{eq:revision-radial-leading}
\end{align}

For clarity, the complete first chaos correction is the first chaos
projection of
\begin{align*}
 &2\mathcal T_\Psi(D,Z^0)
 +\mathcal T_2(D,Z^0)+\mathcal T_2(Z^0,D)\\
 &\quad+\mathcal T_3(Z^0,D,D)+\mathcal T_3(D,Z^0,D)
       +\mathcal T_3(D,D,Z^0)+\mathcal T_3(Z^0,Z^0,Z^0),
\end{align*}
where
\begin{align*}
 \mathcal T_\Psi(v_1,v_2)&=\Pi_0\Phi[\Psi[v_1,v_2],Y]_{t_M},\\
 \mathcal T_2(v_1,v_2)&=\Pi_0\Phi[v_1,\Phi[v_2,Y]]_{t_M},\\
 \mathcal T_3(v_1,v_2,v_3)&=
  \Pi_0\Phi[v_1,\Phi[v_2,\Phi[v_3,Y]]]_{t_M}.
\end{align*}
The last expression includes the three radial Wick pairings.  No residual
noise is contracted because it has been frozen into $D$.

The deterministic estimates used in the proof of
\cite[Proposition~5.4]{HPSRY} are, for
$f\in C_tH^{-(s_*+1)}$ and $v\in C_tC^\gamma$,
\begin{align*}
 \sup_{r\le t}\|\Phi[v,f]_r\|_{H^{-(s_*+1)}}
 &\le Ct^{1/2}\|v\|_{C_tC^\gamma}
       \sup_{r\le t}\|f_r\|_{H^{-(s_*+1)}},\\
 \|\Pi_0\Phi[v,f]_t\|_2
 &\le Ct\|v\|_{C_tC^\gamma}
       \sup_{r\le t}\|f_r\|_{H^{-(s_*+1)}}.
\end{align*}
The estimates depend on the path norm of $v$ rather than on the
identity $v=P_tu_0$.  We give a mixed term for which the proof
of \cite[Proposition~5.4]{HPSRY} uses the special choice $P_tu_0$.
Writing $\sigma_{\boldsymbol k,\iota}$ for a real velocity Fourier mode,
put
\begin{align*}
 \mathcal U_{r,\tau,\boldsymbol k,\iota}
 =\operatorname{div}P_{r-\tau}\Pi_{\rm L}
   \bigl(D_\tau\mathbin\otimes_s
        \sigma_{\boldsymbol k,\iota}\bigr).
\end{align*}
The same Fourier multiplier estimate gives
\begin{align*}
 &\sum_{\boldsymbol k,\iota}|\boldsymbol k|^{-\alpha_{\rm H}}
  \bigl\|\Pi_0\mathcal R[
       \mathcal U_{r,\tau,\boldsymbol k,\iota}]Y_r\bigr\|_2^2\\
 &\qquad\le C(r-\tau)^{-1}\|D_\tau\|_{H^{s_*}}^2
             \|Y_r\|_{H^{-(s_*+1)}}^2
 \le C(r-\tau)^{-1}\|D\|_{C_tC^\gamma}^2
             \|Y_r\|_{H^{-(s_*+1)}}^2,
\end{align*}
because $C^\gamma\hookrightarrow H^{s_*}$ and $\gamma>s_*$.  Jensen's
inequality and integration of the kernel $(r-\tau)^{-1/2}$ give the same
$t_M^4$ bound for $\Phi[\Psi[D,Z^0],Y]$ as in the radial proof.  Here are
the details of the Fourier estimate and the time power.  In complex
notation, expand $D_\tau=\sum_{\boldsymbol m}D_{\tau,\boldsymbol m}e_{\boldsymbol m}$ and
$Y_r=\sum_{\boldsymbol n}Y_{r,\boldsymbol n}e_{\boldsymbol n}$.  For a fixed output mode
$0<|\boldsymbol\ell|\le\sqrt2$, write $\boldsymbol j=\boldsymbol m+\boldsymbol k$ and
$\boldsymbol n=\boldsymbol\ell-\boldsymbol j$.  Boundedness of the Leray multiplier and the fact
\begin{align*}
 \left|1-\frac{|\boldsymbol j|^2}{|\boldsymbol\ell-\boldsymbol j|^2}\right||\boldsymbol j|\le C_{\boldsymbol\ell}
\end{align*}
away from finitely many modes, which are estimated directly, give
\begin{align*}
 \left|\bigl(\Pi_0\mathcal R[
       \mathcal U_{r,\tau,\boldsymbol k}]Y_r\bigr)_{\boldsymbol\ell}\right|
 \le C\sum_{\boldsymbol j} e^{-(r-\tau)|\boldsymbol j|^2}
          |D_{\tau,\boldsymbol j-\boldsymbol k}|\,|Y_{r,\boldsymbol\ell-\boldsymbol j}|.
\end{align*}
The term $\boldsymbol n=\boldsymbol\ell-\boldsymbol j=\boldsymbol0$ is absent because $Y$ has zero spatial mean, and the
$\boldsymbol j=\boldsymbol0$ term is killed by the divergence; this also disposes of the apparent
singular cases in the preceding multiplier inequality.
Cauchy--Schwarz with the weight
$(1+|\boldsymbol j-\boldsymbol k|)^{s_*}$, followed by summation in the
noise mode, yields
\begin{align*}
 &\sum_{\boldsymbol k}|\boldsymbol k|^{-\alpha_{\rm H}}
 \left|\bigl(\Pi_0\mathcal R[
       \mathcal U_{r,\tau,\boldsymbol k}]Y_r\bigr)_{\boldsymbol\ell}\right|^2\\
 &\quad\le C\|D_\tau\|_{H^{s_*}}^2
 \sum_{\boldsymbol j} e^{-2(r-\tau)|\boldsymbol j|^2}|Y_{r,\boldsymbol\ell-\boldsymbol j}|^2
 \sum_{\boldsymbol k}|\boldsymbol k|^{-\alpha_{\rm H}}
       (1+|\boldsymbol j-\boldsymbol k|)^{-2s_*}\\
 &\quad\le C\|D_\tau\|_{H^{s_*}}^2
 \sum_{\boldsymbol j} e^{-2(r-\tau)|\boldsymbol j|^2}(1+|\boldsymbol j|)^{-2s_*}
       |Y_{r,\boldsymbol\ell-\boldsymbol j}|^2\\
 &\quad\le C(r-\tau)^{-1}\|D_\tau\|_{H^{s_*}}^2
       \|Y_r\|_{H^{-(s_*+1)}}^2.
\end{align*}
The second inequality uses the standard convolution bound for the two
summable weights, and the last uses
$e^{-2h|\boldsymbol j|^2}(1+|\boldsymbol j|)^2\le Ch^{-1}$; translating by one of the
finitely many low modes $\boldsymbol\ell$ changes only the constant.  The real
sine cosine sum is the same estimate after an orthogonal change of basis.

To see the time power, let $a$ be the radial noise time.  It\^o isometry
and Minkowski's inequality bound the $L^2(\mathbf P_0;L^2)$ norm of the
mixed term by
\begin{align*}
 C\int_0^{t_M}\int_\tau^{t_M}
 \left(\sum_{\boldsymbol k}|\boldsymbol k|^{-\alpha_{\rm H}}
  \int_0^\tau e^{-2|\boldsymbol k|^2(\tau-a)}
  \|\Pi_0\mathcal R[\mathcal U_{r,\tau,\boldsymbol k}]Y_r\|_2^2
  \,\dd a\right)^{1/2}\dd r\dd\tau.
\end{align*}
Using $\int_0^\tau e^{-2|\boldsymbol k|^2(\tau-a)}\dd a\le\tau$ and the
preceding bound, this is at most
\begin{align*}
 C\|D\|_{C_tH^{s_*}}\|Y\|_{C_tH^{-(s_*+1)}}
 \int_0^{t_M}\int_\tau^{t_M}
     \tau^{1/2}(r-\tau)^{-1/2}\,\dd r\dd\tau
 \le Ct_M^2\|D\|_{C_tH^{s_*}}
       \|Y\|_{C_tH^{-(s_*+1)}}.
\end{align*}
Squaring proves the asserted $t_M^4$ estimate.

A representative cubic mixed background term can be estimated with the
same constants and without using any special form of $D$.  Put $t=t_M$ and
$G_r=\Phi[Z^0,Y]_r$.  The negative Sobolev Gaussian estimate in
\cite[Lemmas~5.3 and~5.6]{HPSRY} gives
\begin{align*}
 \sup_{0<a\le r}\left(\mathbf E_0
   \|G_a\|_{H^{-(s_*+1)}}^2\right)^{1/2}
 \le Cr\|\rho_0\|_{H^{-(s_*+1)}},
 \qquad 0<r\le t.
\end{align*}
Using the integral forms of the two deterministic estimates above and
Minkowski's inequality, we therefore obtain
\begin{align*}
 &\|\mathcal T_3(D,D,Z^0)\|_{L^2(\mathbf P_0;L^2)}\\
 &\quad\le C\|D\|_{C_tC^\gamma}^2
 \int_0^t\int_0^{r_1}(r_1-r_2)^{-1/2}
 \left(\mathbf E_0\|G_{r_2}\|_{H^{-(s_*+1)}}^2\right)^{1/2}
 \,\dd r_2\dd r_1\\
 &\quad\le C\|D\|_{C_tC^\gamma}^2
 \|\rho_0\|_{H^{-(s_*+1)}}
 \int_0^t\int_0^{r_1}(r_1-r_2)^{-1/2}r_2
 \,\dd r_2\dd r_1\\
 &\quad=\frac{8C}{15}t^{5/2}\|D\|_{C_tC^\gamma}^2
 \|\rho_0\|_{H^{-(s_*+1)}}.
\end{align*}
Thus its squared norm is bounded by the second grouped right hand side.
The placements $\mathcal T_3(D,Z^0,D)$ and
$\mathcal T_3(Z^0,D,D)$ follow from the same time simplex, applying,
respectively, the full space and low output Gaussian estimate in the slot
occupied by $Z^0$.

The preceding calculation, its two slot permutations, the two deterministic
estimates above, and the Gaussian estimates in
\cite[Lemmas~5.3 and~5.6]{HPSRY} yield the grouped bounds
\begin{align*}
 \|E_D^{(1)}\|_{L^2(\mathbf P_0;L^2)}^2
 &\le Ct_M^4\|D\|_{C_tC^\gamma}^2
       \|\rho_0\|_{H^{-(s_*+1)}}^2,\\
 \|E_D^{(2)}\|_{L^2(\mathbf P_0;L^2)}^2
 &\le Ct_M^5\|D\|_{C_tC^\gamma}^4
       \|\rho_0\|_{H^{-(s_*+1)}}^2,\\
 \|E_D^{(3)}\|_{L^2(\mathbf P_0;L^2)}^2
 &\le Ct_M^5\|\rho_0\|_{H^{-(s_*+1)}}^2.
\end{align*}
Here $E_D^{(1)}$ is the sum of the $\mathcal T_\Psi$ and
$\mathcal T_2$ terms, $E_D^{(2)}$ is the sum of the three cubic terms
with one radial noise leg, and $E_D^{(3)}$ consists of the three radial
Wick contractions.  For example, with an independent radial copy
$\overline Z^0$, one of the latter equals
\begin{align*}
 \overline{\mathbf E}_0\,
 \mathcal T_3(Z^0,\overline Z^0,\overline Z^0),
\end{align*}
and the other two follow by permuting the slots.  Jensen's inequality in
the copy variable and the same Gaussian estimates give the last grouped
bound.  Since $\|D\|_{C_tC^\gamma}\le C_{R_0}U_0$, consequently
\begin{align}
 \|E_D\|_{L^2(\mathbf P_0;L^2)}^2
 \le C_{R_0}t_M^4U_0^4
       \|\rho_0\|_{H^{-(s_*+1)}}^2.
 \label{eq:revision-background-correction}
\end{align}
These estimates also justify removal of the spatial cutoffs.  Fix
$s_*<\gamma_0<\gamma$.  Let $S_N$ be a smooth Fourier multiplier with
symbol $\chi(\boldsymbol k/N)$, where
$\chi\in C_c^\infty(\mathbb R^2)$, $0\le\chi\le1$, and $\chi=1$ near the
origin.  Put $u_0^N=S_Nu_0$ and choose, by compatibility, smooth
Fourier truncated $z^N\to z$ in $C_tC^\gamma$.  Then
$u_0^N\to u_0$ in $C^{\gamma_0}$ and
$\sup_N\|u_0^N\|_{C^\gamma}\le C\|u_0\|_{C^\gamma}$.  Set
$D^N=P_tu_0^N+z^N$, $Y^N=P_tS_N\rho_0$, and
$Z^{0,N}=S_NZ^0$.  The constants in the preceding operator and Fourier
estimates are independent of $N$: the $D^N$ norms are uniformly bounded by
$C_{R_0}U_0$, and every noise sum is dominated by the corresponding
untruncated sum.  Moreover, $D^N\to D$ in $C_tC^{\gamma_0}$,
$Y^N\to Y$ in $C_tH^{-(s_*+1)}$, and
$Z^{0,N}\to Z^0$ in $L^q(\mathbf P_0;C_tC^{\gamma_0})$ for every
$q<\infty$.  Multilinearity and the displayed bounds therefore make each
first chaos kernel converge in $L^2(\mathbf P_0;L^2)$.  Hence
\eqref{eq:revision-radial-leading} and
\eqref{eq:revision-background-correction}, first proved with cutoffs, pass
to the stated untruncated quantities with the same constants.

The estimates above supply the additional input needed to run the
radial argument in a bounded compatible background.

The quantile hypothesis gives the elementary low high comparison
\begin{align}
 \|\rho_0\|_{H^{-(s_*+1)}}
 \le CM^{1/4}
   \|\rho_0\|_{H^{-(\alpha_{\rm H}/2+1)}}.
 \label{eq:revision-background-quantile}
\end{align}
Taking square roots in
\eqref{eq:revision-radial-leading}--
\eqref{eq:revision-background-correction}, and using
$(\log M)^2\le CM^{1/4}$, yields
\begin{align*}
 \frac{\|E_D\|_{L^2}}{\|A\|_{L^2}}
 \le C_{R_0}M^{-1/2}U_0^2.
\end{align*}
After increasing $K_{R_0},r_0,M_0$, the right hand side is at most $1/2$.

Let $F_D$ be the projection onto $\Pi_0$ of all the polynomial terms in
\eqref{eq:revision-radial-expansion}.  Its Wiener chaos decomposition
contains only orders $0,1,2,3$, and
\begin{align*}
 \Pi_0^{(1)}F_D=A+E_D.
\end{align*}
By the preceding bound on $E_D$ and
\eqref{eq:revision-radial-leading},
\begin{align*}
 \|\Pi_0^{(1)}F_D\|_{L^2(\mathbf P_0;L^2)}
 \ge\|A\|_{L^2(\mathbf P_0;L^2)}
      -\|E_D\|_{L^2(\mathbf P_0;L^2)}
 \ge\frac12\|A\|_{L^2(\mathbf P_0;L^2)},
\end{align*}
and hence
\begin{align*}
 \mathbf E_0\|F_D\|_2^2
 \ge cM^{-6}
 \|\rho_0\|_{H^{-(\alpha_{\rm H}/2+1)}}^2.
\end{align*}
Hilbert valued hypercontractivity for the finite sum of chaoses gives
\begin{align*}
 \mathbf E_0\|F_D\|_2^4
 \le C\bigl(\mathbf E_0\|F_D\|_2^2\bigr)^2.
\end{align*}
Paley--Zygmund therefore gives constants $c_1,\alpha_1>0$ such that
\begin{align}
 \mathbf P_0\left\{
   \|F_D\|_2\ge c_1M^{-3}
       \|\rho_0\|_{H^{-(\alpha_{\rm H}/2+1)}}
 \right\}\ge2\alpha_1.
 \label{eq:revision-background-PZ}
\end{align}

It remains to compare the polynomial with the true endpoint.  The pathwise
remainder estimate in \cite[Lemma~5.5]{HPSRY} is independent of the
covariance and gives
\begin{align*}
 \|\Pi_0\operatorname{Rem}[u,\rho]_{t_M}\|_2
 &\le Ct_M^2\|\rho_0\|_{H^{-(s_*+1)}}
   \exp\left(Ct_M\|u\|_{C_tC^\gamma}^2\right)\\
 &\qquad\times
   \bigl(\|D+Z^0\|_{C_tC^\gamma}
              +\|u\|_{C_tC^\gamma}\bigr)^4.
\end{align*}
We give the uniform path estimate because the compatible path need not have
a time derivative.  Put
\begin{align*}
 \widetilde Z=z+Z^0,
 \qquad v=u-\widetilde Z,
 \qquad \mathcal Z=\|\widetilde Z\|_{C_tC^\gamma}.
\end{align*}
Directly from the mild equation, without differentiating $z$, one has
\begin{align}
 v_r=P_ru_0-\int_0^rP_{r-a}\Pi_{\rm L}
       \operatorname{div}\bigl((v_a+\widetilde Z_a)
                 \mathbin\otimes(v_a+\widetilde Z_a)\bigr)\,\dd a.
 \label{eq:revision-background-mild-v}
\end{align}
For smooth spatial truncations let
$\omega=\operatorname{curl}v$ and
$\chi=\operatorname{curl}\widetilde Z$.  Then
\begin{align*}
 \partial_r\omega
 =\Delta\omega-(v+\widetilde Z)\mathbin\cdot
          \nabla(\omega+\chi).
\end{align*}
The transport of $\omega$ vanishes in the $L^2$ pairing.  Since
$\gamma>s_*>2$, the remaining terms satisfy
\begin{align*}
 |\langle\omega,v\mathbin\cdot\nabla\chi\rangle|
 &\le C\|\nabla\chi\|_\infty
       (\|\omega\|_2^2+U_0^2),\\
 |\langle\omega,\widetilde Z\mathbin\cdot\nabla\chi\rangle|
 &\le\tfrac14\|\omega\|_2^2
       +C\|\widetilde Z\|_\infty^2\|\nabla\chi\|_2^2.
\end{align*}
Poincar\'e's inequality absorbs the displayed quarter term into the
dissipation.
Consequently, uniformly in all spatial truncations,
\begin{align*}
 \frac{\dd}{\dd r}(1+\|\omega_r\|_2^2)
  +\|\nabla\omega_r\|_2^2
 \le C(1+\mathcal Z)(1+\|\omega_r\|_2^2)
       +CU_0^2(1+\mathcal Z)^4.
\end{align*}
Gronwall and Poincar\'e therefore give
\begin{align}
 \sup_{r\le t}\|v_r\|_{H^1}
 \le CU_0(1+\mathcal Z)^2e^{C(1+\mathcal Z)}.
 \label{eq:revision-background-enstrophy}
\end{align}
This is the standard enstrophy step in
\cite[proof of Lemma~7.5]{HPSRY}; the calculation above shows that it uses
only the spatial norm and not a time derivative of the background.

For completeness, choose $\delta>0$ and set
$\theta_0=1-3\delta>0$ and
$\theta_{j+1}=\min\{\gamma,\theta_j+1-\delta\}$.  The product estimate
$H^1\cdot H^1\hookrightarrow H^{1-\delta}$ and the heat estimate in
\eqref{eq:revision-background-mild-v} give
\begin{align*}
 V_0:=\sup_{r\le t}\|v_r\|_{C^{\theta_0}}
 &\le C\{U_0+(\sup_{r\le t}\|v_r\|_{H^1}+\mathcal Z)^2\},\\
 V_{j+1}&\le C\{U_0+(V_j+\mathcal Z)^2\}.
\end{align*}
After a fixed finite number of steps, depending only on $\gamma$, this and
\eqref{eq:revision-background-enstrophy} imply
\begin{align*}
 1+\|D+Z^0\|_{C_tC^\gamma}+\|u\|_{C_tC^\gamma}
 \le C U_0^R(1+\mathcal Z)^R e^{CR\mathcal Z}.
\end{align*}
Since $\mathcal Z\le R_0+\|Z^0\|_{C_tC^\gamma}$, Fernique's theorem now
gives, for every $q<\infty$,
\begin{align}
 \mathbf E_0\left[
 \left(1+\|D+Z^0\|_{C_tC^\gamma}
          +\|u\|_{C_tC^\gamma}\right)^q
 \right]
 \le C_{q,R_0}U_0^{R_q},
 \qquad 0<t\le1.
 \label{eq:revision-background-path-moment}
\end{align}
To remove the truncations, approximate the compatible $z$ and $Z^0$ in
$C_tC^\gamma$ by their Fourier truncations.  Equation
\eqref{eq:revision-background-mild-v} has continuous coefficients and no
$\dot z$ term; its two dimensional mild solutions therefore converge by
the usual energy compactness and pathwise uniqueness argument.  The bounds
are uniform and pass to the limit by Fatou's lemma.  Thus the deterministic
residual path enters only through $R_0$.

Dividing the remainder bound by
$c_1M^{-3}\|\rho_0\|_{H^{-(\alpha_{\rm H}/2+1)}}$ and using
\eqref{eq:revision-background-quantile} gives
\begin{align*}
 CM^{-3/4}(\log M)^2
 \exp(Ct_M\mathcal A_M^2)\mathcal A_M^4,
 \qquad
 \mathcal A_M=1+\|D+Z^0\|_{C_tC^\gamma}+\|u\|_{C_tC^\gamma}.
\end{align*}
On $\{\mathcal A_M\le M^{1/16}\}$ this tends to zero.  On the complement,
Markov's inequality and
\eqref{eq:revision-background-path-moment} at one fixed sufficiently large
order give
\begin{align*}
 \mathbf P_0\left\{
  \|\Pi_0\operatorname{Rem}_{t_M}\|_2
  >\frac{c_1}{2}M^{-3}
    \|\rho_0\|_{H^{-(\alpha_{\rm H}/2+1)}}
 \right\}
 \le C_{R_0}M^{-4}U_0^{R}.
\end{align*}
The last member is at most $\alpha_1$ under
\eqref{eq:revision-background-trial-threshold}, after one enlargement of
the constants.  Combining this estimate with
\eqref{eq:revision-background-PZ} proves
\eqref{eq:revision-background-trial-success}.
\end{proof}

The radial trial is next converted into a statement valid after arbitrary
stopping times for the original nonradial noise.  Conditioning on the
residual subnoise leaves a fresh radial component to which the preceding
trial applies.
\begin{proposition}
\label{prop:revision-conditional-radial-trial}
Let $S$ be an almost surely finite stopping time, freeze
$(u_S,\rho_S)$ with $\rho_S\ne0$, and let $M$ be an
$\mathcal F_S$ measurable integer.  There are deterministic
$K_{\rm tr},r_{\rm tr}<\infty$ and $a_{\rm tr},\alpha_{\rm tr}>0$ such
that, on
\begin{align*}
 M\ge M^{(2)}(\rho_S)
       \vee K_{\rm tr}(1+\|u_S\|_{C^\gamma})^{r_{\rm tr}},
\end{align*}
one has
\begin{align}
 \mathbf P\left(
  \|\Pi_0\rho_{S+t_M}\|_2
  \ge a_{\rm tr}M^{-3}
      \|\rho_S\|_{H^{-(\alpha_{\rm H}/2+1)}}
  \;\middle|\;\mathcal F_S\right)
 \ge\alpha_{\rm tr}.
 \label{eq:revision-conditional-success}
\end{align}
The constants do not depend on $S,M,u_S$, or $\rho_S$.
\end{proposition}

\begin{proof}
Let $Z^{1,S}$ be the residual stochastic convolution generated by the fresh
post $S$ increments of $W^1$.  The upper bound on $R$, the dyadic
heat kernel estimate, and Gaussian concentration give, for every admissible
$\gamma$, uniformly over $S$, $0<t\le1$, and Fourier truncations,
\begin{align*}
 \mathbf E\left[
  \|Z^{1,S}\|_{C([0,t];C^\gamma)}^q
  \;\middle|\;\mathcal F_S\right]\le C_q,
 \qquad q<\infty.
\end{align*}
Choose $R_0$ so large that the trial measurable event
\begin{align*}
 \mathcal G_{S,M}
 =\left\{
  \|Z^{1,S}\|_{C([0,t_M];C^\gamma)}\le R_0
  \right\}
\end{align*}
satisfies
\begin{align*}
 \mathbf P(\mathcal G_{S,M}\mid\mathcal F_S)\ge\frac34
\end{align*}
whenever $t_M\le1$.

Condition additionally on the residual path.  The radial increments $W^0$
remain fresh.  On $\mathcal G_{S,M}$ the deterministic background in
\Cref{lem:revision-radial-background-trial} is
$z=Z^{1,S}$.  This path is compatible almost surely by its Fourier
truncations, and on $\mathcal G_{S,M}$ it lies in the fixed $R_0$ ball.
Moreover, the quantity $U_0$ in that lemma is
\begin{align*}
 U_0=1+\|u_S\|_{C^\gamma}.
\end{align*}
Thus, after increasing $K_{\rm tr},r_{\rm tr}$, the conditional radial
success probability is at least $\alpha_0$ for every residual path in
$\mathcal G_{S,M}$.  Integrating first in $W^0$ and then in $W^1$ gives
\eqref{eq:revision-conditional-success} with
$\alpha_{\rm tr}=3\alpha_0/4$.  The substitution of a random
$\mathcal F_S$ measurable integer follows by disintegration.
\end{proof}

The remaining spectral descent will use the base stochastic estimates
collected in \Cref{lem:base-stochastic-estimates}.  They depend only on the
upper and lower power law bounds for the real Fourier phase amplitudes and
not on the radial subnoise decomposition.

\subsection{Fixed ratio spectral descent and terminal smoothing}

Put
\begin{align*}
 m(\rho)=M^{(1)}(\rho),
 \qquad m_2(\rho)=M^{(2)}(\rho),
 \qquad A_*=\frac{\alpha_{\rm H}}2+4=\lambda_0-1.
\end{align*}
Choose once and for all
\begin{align*}
 \vartheta=\left(1-\frac1{2\lambda_0}\right)^{1/2}\in(0,1),
 \qquad
 \lambda_0\vartheta^2=A_*+\frac12.
\end{align*}
Choose a deterministic integer floor $M_\bullet$ so large that
$\lceil\vartheta M\rceil<M$ for every integer $M>M_\bullet$.  For
$M\ge M_\bullet$ set
\begin{align*}
 L_M=M_\bullet\vee\lceil\vartheta M\rceil,
 \quad h_M=\lambda_0M^{-2}\log M,
 \quad n_M=\lceil(\log M)^2\rceil,
 \quad H_M=n_Mh_M.
\end{align*}
After increasing $M_\bullet$ once more, we have, for every $M>M_\bullet$,
\begin{align*}
 \qquad \bar\vartheta=\frac{1+\vartheta}{2}<1,
 \qquad
 L_M=M_\bullet\quad\hbox{or}\quad L_M\le\bar\vartheta M,
 \qquad L_M<M.
\end{align*}
Thus every sequence $M_{i+1}=L_{M_i}$ reaches $M_\bullet$ after finitely
many steps and is geometrically decreasing until its last step.  In
particular,
\begin{align*}
 T_{\rm it}:=
 \sup_{M_0\ge M_\bullet}
 \sum_{i:M_i>M_\bullet}H_{M_i}+H_{M_\bullet}
 \le C_\vartheta H_{M_\bullet}<1
\end{align*}
after one further enlargement of $M_\bullet$.
Indeed, reading the geometric sequence backwards from its last scale and
using $H_M\asymp M^{-2}(\log M)^3$ bounds the sum by a convergent series
$\sum_{k\ge0}\bar\vartheta^{2k}(1+k)^3H_{M_\bullet}$.

A single successful trial only lowers the spectral scale by a fixed ratio.
The next lemma iterates the trials over geometrically decreasing scales and
adds a capped terminal step that supplies uniform $H^1$ moments.
\begin{lemma}
\label{lem:revision-conditional-iteration}
Let $\theta$ be an almost surely finite stopping time with
$u_\theta\in C^\gamma$, and let $M\ge M_\bullet$ be an
$\mathcal F_\theta$ measurable integer with $m(\rho_\theta)\le M$.  Define
\begin{align*}
 d_{\theta,M}
 &=\inf\{t\ge0:m(\rho_{\theta+t})\le L_M\},\notag\\
 e_{\theta,M}
 &=\inf\{t\ge0:m_2(\rho_{\theta+t})>M\},\notag\\
 c_{\theta,M}&=e_{\theta,M}\wedge H_M.
\end{align*}
For every $N<\infty$ there are $C_N,R_N<\infty$ such that
\begin{align}
 \mathbf P_\theta\{d_{\theta,M}>H_M\}
 +\mathbf P_\theta\{e_{\theta,M}\le H_M\}
 \le C_NM^{-N}U_\theta^{R_N},
 \qquad U_\theta=1+\|u_\theta\|_{C^\gamma}.
 \label{eq:revision-level-failure}
\end{align}
For every $q<\infty$ there are $C_q,R_q<\infty$ such that
\begin{align}
 \mathbf E_\theta
 \|\mathsf p_{\theta+c_{\theta,M}}\|_{H^1}^q
 \le C_qM^qU_\theta^{R_q}.
 \label{eq:revision-capped-level-H1}
\end{align}

Starting at an arbitrary finite stopping time $S$ with
$u_S\in C^\gamma$, there is an adapted geometrically descending
construction and a terminal time $\eta_S$ such that
\begin{align*}
 0<\eta_S-S\le T_{\rm it}
\end{align*}
and, for every $r<\infty$,
\begin{align}
 \mathbf E_S\|\mathsf p_{\eta_S}\|_{H^1}^{2r}
 \le C_rU_S^{R_r}.
 \label{eq:revision-stopped-H1}
\end{align}
The deterministic horizon is independent of $S$, of the initial fibre
direction, and of the requested moment order.
\end{lemma}

\begin{proof}
Let
\begin{align*}
 \upsilon_{\theta,M}
 =\inf\{t\ge0:U_{\theta+t}>M^\varepsilon U_\theta\},
\end{align*}
After enlarging $r_{\rm tr}$ if necessary, assume $r_{\rm tr}\ge1$ and
choose $\varepsilon>0$ so that $\varepsilon r_{\rm tr}<1/2$.  Suppose first that
\begin{align}
 M\ge K_1U_\theta^{2r_{\rm tr}}.
 \label{eq:revision-good-threshold}
\end{align}
For
\begin{align*}
 A_j=\{jh_M<d_{\theta,M}\wedge e_{\theta,M}
                         \wedge\upsilon_{\theta,M}\},
 \qquad 0\le j\le n_M,
\end{align*}
the absence of the upper quantile and velocity exits gives, on $A_j$,
\begin{align*}
 m_2(\rho_{\theta+jh_M})\le M,
 \qquad
 K_{\rm tr}U_{\theta+jh_M}^{r_{\rm tr}}\le M.
\end{align*}
Thus \Cref{prop:revision-conditional-radial-trial} applies at every trial
start with the same success probability $\alpha_{\rm tr}$.

We show that success on trial $j$ is incompatible with $A_{j+1}$.  On
$A_{j+1}$ the median remains above $L_M$ throughout the trial.  The
high mode energy calculation of \cite[Lemma~3.3]{HPSRY}, now made at the
cutoff $L_M$, gives
\begin{align}
 \|\Pi_{>L_M}\rho_{\theta+(j+1)h_M}\|_2
 &\le C\exp\left(
  -L_M^2h_M+CL_Mh_MM^\varepsilon U_\theta
 \right)
 \|\rho_{\theta+jh_M}\|_2\notag\\
 &\le CM^{-A_*-1/4}\|\rho_{\theta+jh_M}\|_2.
 \label{eq:revision-fixed-ratio-high-mode}
\end{align}
Indeed, under \eqref{eq:revision-good-threshold},
$U_\theta\le C M^{1/(2r_{\rm tr})}$, so the positive exponent is
\begin{align*}
 CL_Mh_MM^\varepsilon U_\theta
 \le C M^{-1+\varepsilon+1/(2r_{\rm tr})}\log M=o(1).
\end{align*}
Here we used $r_{\rm tr}\ge1$ and
$\varepsilon r_{\rm tr}<1/2$, which imply
$\varepsilon+1/(2r_{\rm tr})<1$.  On the other hand,
\begin{align*}
 L_M^2h_M
 \ge\lambda_0\vartheta^2\log M
 =\left(A_*+\frac12\right)\log M.
\end{align*}
Thus, after enlarging $M_\bullet$, the positive exponential is bounded by
$M^{1/4}$ and the net decay is at least $M^{-A_*-1/4}$, as claimed in
\eqref{eq:revision-fixed-ratio-high-mode}.  On the other hand,
$m_2(\rho_{\theta+jh_M})\le M$ gives
\begin{align*}
 \|\rho_{\theta+jh_M}\|_{H^{-(\alpha_{\rm H}/2+1)}}
 \ge cM^{-(\alpha_{\rm H}/2+1)}
       \|\rho_{\theta+jh_M}\|_2.
\end{align*}
The success event in
\eqref{eq:revision-conditional-success} therefore produces a fixed low mode
component at least
$cM^{-A_*}\|\rho_{\theta+jh_M}\|_2$.  It dominates
\eqref{eq:revision-fixed-ratio-high-mode}; hence the endpoint median is at
most $L_M$, contradicting $A_{j+1}$.  The strong Markov and tower properties
give
\begin{align*}
 \mathbf P_\theta(A_{n_M})
 \le(1-\alpha_{\rm tr})^{n_M}
 \le C_NM^{-N}
\end{align*}
for every $N$.

The two auxiliary exits are estimated by the same argument as in
\cite[Lemmas~6.1 and~7.5]{HPSRY}.  With
\begin{align*}
 R_M(t)=
 \frac{\|\Pi_{>M}\rho_{\theta+t}\|_2}
      {\|\Pi_{\le M}\rho_{\theta+t}\|_2},
\end{align*}
one has $R_M(0)\le1$, whereas $e_{\theta,M}<\infty$ requires $R_M$ to
reach $2$, and
\begin{align*}
 \frac{\dd}{\dd t}R_M(t)
 \le CMU_{\theta+t}(1+R_M(t)^2).
\end{align*}
Consequently, \eqref{eq:revision-conditional-path-moment} and Markov's
inequality imply, for every $q<\infty$,
\begin{align*}
 \mathbf P_\theta\{e_{\theta,M}\le H_M\}
 &\le C_q(MH_M)^qU_\theta^{R_q},\notag\\
 \mathbf P_\theta\{\upsilon_{\theta,M}\le H_M\}
 &\le C_qM^{-\varepsilon q}U_\theta^{R_q}.
\end{align*}
Because $H_M=n_Mh_M$,
\begin{align*}
 \{d_{\theta,M}>H_M\}
 \subset A_{n_M}\cup\{e_{\theta,M}\le H_M\}
                  \cup\{\upsilon_{\theta,M}\le H_M\}.
\end{align*}
Since $MH_M\le CM^{-1}(\log M)^3$, choosing $q$ after $N$ proves
\eqref{eq:revision-level-failure} under
\eqref{eq:revision-good-threshold}.  If that threshold fails, the trivial
probability bound is absorbed into $C_NM^{-N}U_\theta^{R_N}$.

For the capped estimate, we first record the corresponding short time
exit bound.  Before $e_{\theta,M}$ the ratio $R_M$ stays below $2$, and
therefore
\begin{align*}
 \frac{\dd}{\dd t}\arctan R_M(t)\le CMU_{\theta+t}.
\end{align*}
Since $R_M(0)\le1$ and $R_M(e_{\theta,M})=2$ whenever the exit is finite,
the event $\{Me_{\theta,M}\le s\}$, $0<s\le1$, implies
\begin{align*}
 \sup_{0\le t\le s/M}U_{\theta+t}\ge c s^{-1}.
\end{align*}
The conditional path moment estimate
\eqref{eq:revision-conditional-path-moment} and Markov's inequality
therefore give, for every $\ell<\infty$,
\begin{align*}
 \mathbf P_\theta\{Me_{\theta,M}\le s\}
 \le C_\ell s^\ell U_\theta^{R_\ell}.
\end{align*}
In particular the exit time is strictly positive almost surely, and
integration of this tail against $qs^{-q-1}\dd s$, choosing $\ell>q$,
gives
\begin{align}
 \mathbf E_\theta(Me_{\theta,M})^{-q}
 \le C_qU_\theta^{R_q}.
 \label{eq:revision-inverse-exit}
\end{align}  The deterministic
estimate underlying \cite[Proposition~6.3]{HPSRY} is
\begin{align*}
 \|\mathsf p_{\theta+c_{\theta,M}}\|_{H^1}
 \le C\bigl(M+c_{\theta,M}^{-1/2}\bigr)
   \left(1+\sup_{0\le t\le H_M}
                   \|u_{\theta+t}\|_{C^\gamma}\right)^C.
\end{align*}
Since
\begin{align*}
 c_{\theta,M}^{-1/2}
 \le e_{\theta,M}^{-1/2}+H_M^{-1/2}
 \le CM\left(1+(Me_{\theta,M})^{-1/2}\right),
\end{align*}
\eqref{eq:revision-inverse-exit} and
\eqref{eq:revision-conditional-path-moment} prove
\eqref{eq:revision-capped-level-H1}.

We finish with the full geometric construction.  Set
\begin{align*}
 M_0=m(\rho_S)\vee M_\bullet,
 \qquad
 M_{i+1}=L_{M_i}\quad\text{while }M_i>M_\bullet
\end{align*}
and let $J=\inf\{i:M_i=M_\bullet\}$.  Put
$\mathcal R_0=\Omega$ and $\theta_0=S$.  Suppose level $i<J$ has been
reached.  On $\mathcal R_i$ define
\begin{align*}
 d_i=d_{\theta_i,M_i},
 \qquad e_i=e_{\theta_i,M_i},
 \qquad c_i=e_i\wedge H_{M_i},
\end{align*}
and set
\begin{align*}
 \mathcal F_i&=\mathcal R_i\cap\{c_i<d_i\},
 &\mathcal R_{i+1}&=\mathcal R_i\cap\{d_i\le c_i\}.
\end{align*}
On $\mathcal R_{i+1}$ put $\theta_{i+1}=\theta_i+d_i$; off this event
extend $\theta_{i+1}$ by the already attained finite stopping time.  On
$\mathcal F_i$ stop the construction and put $\eta_S=\theta_i+c_i$.
If every geometric level is crossed, then on $\mathcal R_J$ perform the
final capped step
\begin{align*}
 \eta_S=\theta_J+c_{\theta_J,M_\bullet}.
\end{align*}
The standard pasting lemma makes these random index constructions honest
stopping times.  The disjoint family of reached failure events together
with $\mathcal R_J$ partitions the probability space.  Every continued
branch satisfies
$m(\rho_{\theta_{i+1}})\le M_{i+1}$, exactly the hypothesis needed at the
next level; no failed branch is restarted.  Moreover,
\begin{align*}
 0<\eta_S-S\le\sum_{i=0}^{J}H_{M_i}\le T_{\rm it}.
\end{align*}
Strict positivity follows from the final or failed capped step, since every
cap is positive almost surely.

On a reached nonterminal level,
\begin{align*}
 \{c_i<d_i\}
 \subset\{d_{\theta_i,M_i}>H_{M_i}\}
       \cup\{e_{\theta_i,M_i}\le H_{M_i}\}.
\end{align*}
Conditional Cauchy--Schwarz,
\eqref{eq:revision-level-failure}, and
\eqref{eq:revision-capped-level-H1} at order $4r$ give
\begin{align*}
 \mathbf E_{\theta_i}\left[
  \mathbf1_{\{c_i<d_i\}}
  \|\mathsf p_{\theta_i+c_i}\|_{H^1}^{2r}\right]
 \le C_{r,N}M_i^{2r-N/2}U_{\theta_i}^{R_{r,N}}.
\end{align*}
On $\mathcal R_J$, the capped estimate at the fixed floor gives
\begin{align*}
 \mathbf E_{\theta_J}
 \|\mathsf p_{\theta_J+c_{\theta_J,M_\bullet}}\|_{H^1}^{2r}
 \le C_rM_\bullet^{2r}U_{\theta_J}^{R_r}.
\end{align*}
The geometric decrease implies, for $N>4r+2$,
\begin{align*}
 \sup_{M_0}\sum_{i=0}^{J}M_i^{2r-N/2}<\infty.
\end{align*}
Every reached $\theta_i\in[S,S+T_{\rm it}]$.  Consequently
\eqref{eq:revision-conditional-path-moment} controls all random velocity
powers by a single conditional moment of
$1+\sup_{t\le T_{\rm it}}\|u_{S+t}\|_{C^\gamma}$.  Summing the disjoint
branches and using the tower property now yields
\eqref{eq:revision-stopped-H1}.
\end{proof}

The geometric descent can now be placed after a deterministic positive time
base regularisation and propagated to one common deterministic time.  This
yields the covariance robust reset and block energy hierarchy used
throughout the coupling argument.
\begin{theorem}
\label{lem:robust-HPSRY-hierarchy}
Assume \eqref{eq:q}, put $b=(a-3)/2$, and fix
$b<\beta<a/2-1$.  There is a deterministic time
$t_*>0$, depending only on $(a,\nu,c_Q,C_Q)$, such that the following
assertions hold.

For every $r<\infty$, there are an integer $R_r$, an admissible base weight
\begin{align}
 V_r(w)=
 (1+\|w\|_{H^b}^2)^{R_r}
 \exp(\eta_0\|w\|_2^2),
 \label{eq:revision-base-weight}
\end{align}
where one sufficiently small $\eta_0>0$ is fixed for the whole hierarchy,
and $C_r<\infty$ such that, uniformly over unit $\mathsf p\in H$,
\begin{align}
 \mathbf E_{w,\mathsf p}\|\mathsf p_{t_*}\|_{H^1}^{2r}
 &\le C_rV_r(w),
 \label{eq:robust-all-H1-moments}\\
 \mathbf E_w\|w_{t_*}\|_{C^\beta}^r
 &\le C_rV_r(w).
 \label{eq:robust-base-reset}
\end{align}
The same estimates hold after every almost surely finite stopping time $S$:
\begin{align}
 \mathbf E\left[
  \|\mathsf p_{S+t_*}\|_{H^1}^{2r}\mid\mathcal F_S\right]
 &\le C_rV_r(w_S),
 \label{eq:revision-cond-reset-p}\\
 \mathbf E\left[
  \|w_{S+t_*}\|_{C^\beta}^{r}\mid\mathcal F_S\right]
 &\le C_rV_r(w_S).
 \label{eq:revision-cond-reset-base}
\end{align}
The constants do not depend on $S$ or on the value of $\mathsf p_S$.

For every fixed $T\ge t_*$, every $r<\infty$, and every unit
$\mathsf p\in H^1$,
\begin{align}
 \mathbf E_{w,\mathsf p}
 \left(\int_0^T\|\mathsf p_s\|_{H^1}^2\,\dd s\right)^r
 \le C_{r,T}\bigl(V_r(w)+1+\|\mathsf p\|_{H^1}^{R_r}\bigr),
 \label{eq:robust-block-energy}
\end{align}
where the polynomial orders are independent of $T$.  Finally, all weights
may be chosen from one nested hierarchy.  Given any finite list of moment
orders and polynomial powers, one member $V_*$ dominates every corresponding
right hand side and satisfies
\begin{align}
 P_T^{\rm base}V_*
 \le C_{0,*}e^{-\gamma_*T}V_*+B_*,
 \qquad T\ge0.
 \label{eq:revision-base-Foster}
\end{align}
\end{theorem}

\begin{proof}
Work first in the unit viscosity clock.  From a deterministic initial state,
run the base equation for one fixed $t_{\rm v}>0$.  The positive time bound
\eqref{eq:revision-positive-time-Cgamma} gives every finite moment of
$U_{t_{\rm v}}=1+\|u_{t_{\rm v}}\|_{C^\gamma}$ with an
exponential polynomial weight of the initial base state.  Apply
\Cref{lem:revision-conditional-iteration} from that time.  Its terminal time
is bounded by the deterministic $T_{\rm it}$, independently of the moment
order.  On paths on which it finishes early, propagate the fibre to the
common endpoint $t_{\rm v}+T_{\rm it}$ with the deterministic estimate
\cite[Lemma~6.4]{HPSRY}.  More explicitly, put
$T_c=t_{\rm v}+T_{\rm it}$ and choose $2<\gamma_0<3$ equal to the Sobolev
exponent $\zeta$ used in
\eqref{eq:revision-lower-norm-exponential}.  Since the first
inequality in that lemma is pathwise, it applies at the stopping time
$\eta=\eta_{t_{\rm v}}$ and gives
\begin{align*}
 \|\mathsf p_{T_c}\|_{H^1}
 &\le C\exp\left(C\int_\eta^{T_c}
                     \|u_t\|_{H^{\gamma_0}}\,\dd t\right)\\
 &\quad\times\left(\|\mathsf p_\eta\|_{H^1}
 +(T_c-\eta)\sup_{\eta\le t\le T_c}
                    \|u_t\|_{H^{\gamma_0+1}}\right).
\end{align*}
Both base path factors are bounded by the corresponding quantities on the
deterministic interval $[t_{\rm v},T_c]$.  Apply
\eqref{eq:revision-stopped-H1} at order $4r$, then H\"older's inequality.
The exponential factor is controlled by
\eqref{eq:revision-lower-norm-exponential}, while the positive time
$H^{\gamma_0+1}$ supremum is controlled by
\eqref{eq:revision-lower-norm-exponential} (recall that
$b>11/2>\gamma_0+1$).  Thus
\begin{align*}
 \sup_{\|\mathsf p\|_2=1}
 \mathbf E_{u,\mathsf p}\|\mathsf p_{t_{\rm v}+T_{\rm it}}\|_{H^1}^{2r}
 \le C_rV_r(u).
\end{align*}
Returning to the original clock defines one common $t_*$ and proves
\eqref{eq:robust-all-H1-moments}.  For the base reset, apply the same
positive time velocity estimate with Hölder exponent $\beta+1$; this is
admissible because $\beta<a/2-1$.  Since $w=\operatorname{curl}u$, it gives
\eqref{eq:robust-base-reset}.  The strong Markov
identity
\begin{align*}
 \mathbf E[F(z_{S+t_*})\mid\mathcal F_S]=P_{t_*}F(z_S)
\end{align*}
then gives \eqref{eq:revision-cond-reset-p}--
\eqref{eq:revision-cond-reset-base} without a separate stopping time
argument.

For the block energy, the deterministic calculation of
\cite[Lemma~6.4]{HPSRY} gives, for $2<\gamma<3$,
\begin{align*}
 \sup_{s\le t\le T}\|\mathsf p_t\|_{H^1}
 &\le C\exp\left(C\int_s^T\|u_t\|_{H^\gamma}\,\dd t\right)\\
 &\quad\times
 \left(\|\mathsf p_s\|_{H^1}
 +(T-s)\sup_{s\le t\le T}\|u_t\|_{H^{\gamma+1}}\right).
\end{align*}
Raise this estimate to the required power, integrate in time, and combine it
with the standard base maximal estimates.  This proves
\eqref{eq:robust-block-energy}; its polynomial orders are determined by the
moment order, not by $T$.

Finally, the super Lyapunov estimate in
\cite[Lemma~3.7]{BBPS22a}, equivalently the second assertion of
\cite[Lemma~7.4]{HPSRY} after the parameter translation above, supplies one
sufficiently small $\eta_0>0$ that works for every polynomial order.  In
particular, for
\begin{align*}
 V^{(n)}(w)=(1+\|w\|_{H^b}^2)^n
                    e^{\eta_0\|w\|_2^2},
\end{align*}
one has
\begin{align*}
 P_T^{\rm base}V^{(n)}
 \le C_ne^{-\gamma_nT}V^{(n)}+B_n.
\end{align*}
Increasing $n$ accommodates any prescribed finite list and proves
\eqref{eq:revision-base-Foster}.
\end{proof}

Only finitely many moment orders are needed in the main proof.  The
following corollary packages them into one sampled block and one inf-compact
Lyapunov function.
\begin{corollary}
\label{cor:revision-common-block}
Fix any finite collection of moment orders and polynomial powers.  There
are one base weight $V_*$, an integer $M_*$, and $N_*<\infty$ such that,
for every $N\ge N_*$, the block $\tau=Nt_*$ admits a constant
$C_\tau<\infty$ for which, with
\begin{align*}
 G_*(w,[\mathsf p])=\|w\|_{C^\beta}^{M_*}+\|\mathsf p\|_{H^1}^{M_*},
\end{align*}
where both norms are extended by $+\infty$ off their natural domains,
\begin{align}
 P_\tau^{\rm base}V_*&\le\theta V_*+C_\tau,
 \qquad 0<\theta<1,\notag\\
 P_\tau G_*&\le C_\tau V_*.
 \label{eq:revision-common-reset}
\end{align}
Thus
\begin{align*}
 W_0=V_*+\varepsilon G_*:
 H^b\times\PP(H)\longrightarrow[1,\infty]
\end{align*}
is lower semicontinuous, finite on
$\mathsf D_0=C^\beta\times\PP(H^1)$, and has compact finite sublevel
sets in $H^b\times\PP(H)$.  For sufficiently small $\varepsilon>0$,
\begin{align}
 P_\tau W_0\le\varrho_0W_0+C_\tau,
 \qquad 0<\varrho_0<1,
 \label{eq:robust-joint-drift}
\end{align}
in the extended valued sense.  Moreover,
\begin{align}
 P_\tau W_0(w,[\mathsf p])
 \le C_\tau(1+V_*(w))<\infty
 \qquad\text{for every }(w,[\mathsf p])\in H^b\times\PP(H).
 \label{eq:revision-W0-pointwise-reset}
\end{align}
Consequently $P_\tau(z,\mathsf D_0)=1$ for every
$z\in H^b\times\PP(H)$.
\end{corollary}

\begin{proof}
Choose $N_*$ so that
$C_{0,*}e^{-\gamma_*N_*t_*}<1$ in
\eqref{eq:revision-base-Foster}.  For the joint reset, condition at time
$\tau-t_*$, apply
\eqref{eq:revision-cond-reset-p}--
\eqref{eq:revision-cond-reset-base} on the last reset interval, and use the
base Foster estimate on the preceding interval.  This proves
\eqref{eq:revision-common-reset}, uniformly in the initial projective
direction.  Choosing
$\theta+\varepsilon C_\tau<\varrho_0<1$ gives
\eqref{eq:robust-joint-drift}, and the same estimates give
\eqref{eq:revision-W0-pointwise-reset}.  A nonnegative extended valued
random variable with finite expectation is finite almost surely.  Finally,
lower semicontinuity and the compact embeddings
$C^\beta\Subset H^b$ and $H^1\Subset H$ give compactness of every finite
sublevel.
\end{proof}

\section{Base and linearisation estimates}
\label{sec:base-linearisation-estimates}

This appendix collects the base stochastic estimates and the fixed path
linearisation estimates used throughout the paper.  They are separated
from the projective Lyapunov construction in \Cref{sec:foundations} in order to keep
the latter focused on the positive time reset and inf-compactness mechanism.

\subsection{Base stochastic estimates}

The projective arguments repeatedly require base moments that are uniform
after stopping times and over positive time intervals.  We collect the
precise exponential, maximal, and positive time Schauder estimates here for
later reference.

\begin{lemma}
\label{lem:base-stochastic-estimates}
Let $u=Kw$ be the velocity corresponding to the vorticity solution, with the
unit viscosity rescaling used above.  There is a number $\eta_0>0$ such that,
for every $T<\infty$, $m<\infty$, $C_0<\infty$, and $2<\zeta<3$, there are
$R<\infty$ and $C<\infty$ for which
\begin{align}
 &\mathbf E_{u_0}\left[
  \exp\left(C_0\int_0^T\|u_t\|_{H^\zeta}\,\dd t\right)
  \sup_{0\le t\le T}(1+\|u_t\|_{H^{b+1}}^2)^m
  \right]                                                      \notag\\
 &\hspace{25mm}\le C(1+\|u_0\|_{H^{b+1}}^2)^R
                 \exp(\eta_0\|u_0\|_{H^1}^2).
 \label{eq:revision-lower-norm-exponential}
\end{align}
Equivalently, in the original vorticity variables, for every
$m,r\in[1,\infty)$,
\begin{align}
 \mathbf E_w\left[1+\|w\|_{C([0,T];H^b)}^m\right]^r
 \le C_{T,m,r}(1+\|w\|_{H^b}^2)^R
                    \exp(\eta_0\|w\|_2^2).
 \label{eq:revision-maximal-base-moment}
\end{align}
For $s_*<\gamma<\alpha_{\rm H}/2$ and every $q<\infty$,
\begin{align}
 \mathbf E\left[
  \sup_{0\le t\le1}(1+\|u_{S+t}\|_{C^\gamma})^q
  \;\middle|\;\mathcal F_S\right]
 \le C_q(1+\|u_S\|_{C^\gamma})^{R_q}
 \label{eq:revision-conditional-path-moment}
\end{align}
after every finite stopping time $S$ for which $u_S\in C^\gamma$.
Moreover, for every fixed $t_{\rm v}>0$,
\begin{align}
 \mathbf E\left[
  (1+\|u_{S+t_{\rm v}}\|_{C^\gamma})^q
  \;\middle|\;\mathcal F_S\right]
 \le C_{q,t_{\rm v}}
  (1+\|u_S\|_{H^{b+1}}^2)^{R_{q,t_{\rm v}}}
  e^{\eta_0\|u_S\|_{H^1}^2}.
 \label{eq:revision-positive-time-Cgamma}
\end{align}
The same estimate holds for the supremum over a compact interval bounded
away from $S$.  The maximal and exponential estimates also hold
conditionally after a finite stopping time, with the state at that time on
the right hand side.
\end{lemma}

\begin{proof}
In the real Fourier basis used in \cite{BBPS22a}, the velocity noise
amplitudes associated with \eqref{eq:q} satisfy
$|\widetilde q_{k,\iota}|\asymp |k|^{-(a+2)/2}$.  Thus
\cite[Assumption~1]{BBPS22a} holds with
$\alpha=(a+2)/2$, while $b+1=(a-1)/2$ lies in the admissible state space
range.  After the deterministic viscosity and clock rescaling, apply
\cite[Lemma~3.7]{BBPS22a} with its growth parameter $\gamma=0$ and with
the polynomial power in its super Lyapunov function chosen larger than the
prescribed $m$.  Since that function contains
\begin{align*}
 (1+\|u\|_{H^{b+1}}^2)^R\exp(\eta_0\|u\|_{H^1}^2),
\end{align*}
the resulting estimate controls simultaneously the exponential
time integral and the displayed polynomial supremum, and gives
\eqref{eq:revision-lower-norm-exponential}.  This is also the estimate
recorded in \cite[Lemma~7.4]{HPSRY}.  The argument is formulated for
independent real Fourier phases with individually comparable amplitudes, so
no radial or equal phase assumption is needed.  The vorticity estimate
\eqref{eq:revision-maximal-base-moment} follows by taking polynomial
moments in the same bound and using
$\|Kw\|_{H^{b+1}}\asymp\|w\|_{H^b}$.

The positive time estimates
\eqref{eq:revision-conditional-path-moment}--
\eqref{eq:revision-positive-time-Cgamma} are the velocity bounds in
\cite[Lemma~7.5]{HPSRY}.  Its proof uses the upper power law estimate for the stochastic
convolution together with the two dimensional enstrophy inequality and a
deterministic Schauder bootstrap.  The same argument therefore applies to
the phase dependent amplitudes in \eqref{eq:q}.  The stopping time formulations follow from the
strong Markov property and fresh Wiener increments.
\end{proof}

\subsection{Regularity estimates for the tangent evolution}

The tangent evolution is controlled by a bilinear Sobolev estimate for
$\cS$ together with parabolic regularity of the associated nonautonomous
evolution family.  These estimates are used in the dense range argument,
the prepared bridge, and the compactness statement in
\Cref{cor:rough-driver-state-compact}.
\begin{lemma}\label{lem:product}
Let $b>2$.  For $-1\le r\le b$, in the distributional sense at the
lower endpoint,
\begin{align*}
 \norm{\cS(w,u)}_{H^{r-1}}
 \le C_{b,r}\norm{w}_{H^b}\norm{u}_{H^r}.
\end{align*}
\end{lemma}

\begin{proof}
Since $\operatorname{div}Kf=0$,
$B(f,g)=\operatorname{div}((Kf)g)$.  At $r=0$, multiplication by the
$H^b$ coefficient and the mapping $K:H\to H^1$ give an $H^{-1}$
bound.  At $r=b$, the Sobolev algebra property in two dimensions and
$K:H^b\to H^{b+1}$ put the products in $H^b$, before the outer
derivative.  The $r=-1$ endpoint follows by duality and multiplication by
the $H^b$ coefficient.  Interpolation gives the intermediate exponents.
\end{proof}

The next step is the parabolic smoothing of the nonautonomous tangent
evolution.  The lemma gives both small fractional gains and their finite
iteration across a positive time gap.
\begin{lemma}
\label{lem:gain}
Let $b>2$, $I=[s_0,s_1]$, and
\begin{align*}
 w\in L^p(I;H^b),\qquad 2<p\le\infty,\qquad
 \kappa_p=1-\frac2p,
\end{align*}
where $\kappa_\infty=1$. Let $U_w(t,s)$ be the evolution family
of
\begin{align*}
 \partial_tz=\nu\Delta z+\cS(w_t,z).
\end{align*}
Then:
\begin{enumerate}[label=\textup{(\roman*)}]
\item $U_w(t,s)$ is bounded on $H^q$, $-1\le q\le b$.
On $\{\|w\|_{L^pH^b}\le M\}$ the bound is uniform for
$s_0\le s\le t\le s_1$.
\item If $0<\delta<\kappa_p$, then
\begin{align*}
 U_w(t,s):H^q\longrightarrow H^{q+\delta},
 \qquad -1\le q\le b,
\end{align*}
with norm locally uniform when $t-s$ is bounded away from zero. For
$q=b$, this endpoint estimate into $H^{b+\delta}$ does not require
the coefficient to belong to that stronger space.
\item If $-1\le q<r\le b$, then, for every $\varepsilon>0$,
\begin{align*}
 \sup_{\substack{s_0\le s<t\le s_1\\t-s\ge\varepsilon}}
 \|U_w(t,s)\|_{H^q\to H^r}
 \le C(\nu,b,p,q,r,I,M,\varepsilon).
\end{align*}
In particular, over a time gap bounded below,
\begin{align*}
 U_w(t,s):H^{-1}\to H^\delta,\qquad
 U_w(t,s):H\to H^b
\end{align*}
for every $0<\delta<b$.
\item The exact energy smoothing estimates
\begin{align*}
 \|U_w(t,s)\|_{H\to H^1}
 +\|U_w(t,s)\|_{H^{-1}\to H}
 \le C_{I,M}(t-s)^{-1/2}
\end{align*}
hold for $s<t$. The second map is the transposition extension. The same
conclusions hold for the backward adjoint evolution.
\end{enumerate}
\end{lemma}

\begin{proof}
For $-1\le q\le b$, divergence form, duality at the lower endpoint, the
$H^b$ multiplier theorem, and interpolation give
\begin{align*}
 \|\cS(w,z)\|_{H^{q-1}}
 \le C_{b,q}\|w\|_{H^b}\|z\|_{H^q}.
\end{align*}
On $J=[a,a+\ell]\subset I$, the mild equation gives
\begin{align*}
 \|z\|_{C(J;H^q)}
 &\le C\|z_a\|_{H^q}
 +C\sup_{t\in J}\int_a^t(t-r)^{-1/2}
       \|w_r\|_{H^b}\|z_r\|_{H^q}\,\dd r \\
 &\le C\|z_a\|_{H^q}
 +C\ell^{\frac12-\frac1p}
       \|w\|_{L^p(J;H^b)}\|z\|_{C(J;H^q)}.
\end{align*}
For each fixed $p>2$, divide $I$ into finitely many intervals on which the
last coefficient is at most $1/2$. This proves \textup{(i)}, uniformly on
the indicated bounded set. The number of intervals and the resulting
constant may deteriorate as $p\downarrow2$; no uniformity at that endpoint
is claimed.

For $0<\delta<\kappa_p$, use
\begin{align*}
 \|e^{\nu(t-r)\Delta}\|_{H^{q-1}\to H^{q+\delta}}
 \le C_\nu(t-r)^{-(1+\delta)/2}.
\end{align*}
Hölder applies precisely because
\begin{align*}
 \frac{1+\delta}{2}\frac p{p-1}<1
 \quad\Longleftrightarrow\quad \delta<1-\frac2p.
\end{align*}
Together with \textup{(i)} this proves \textup{(ii)}. Only the
$H^q$ bound of the solution and the $H^{q-1}$ source bound enter,
which explains the $q=b$ endpoint.

For \textup{(iii)}, choose $N$ and
$\delta_1,\ldots,\delta_N\in(0,\kappa_p)$ with
$\sum_j\delta_j=r-q$, and put $t_j=s+j(t-s)/N$. Then
\begin{align*}
 U_w(t,s)=U_w(t,t_{N-1})\cdots U_w(t_1,s):
 H^q\longrightarrow H^r
\end{align*}
by successive applications of \textup{(ii)}. If $t-s\ge\varepsilon$,
each factor acts over a gap at least $\varepsilon/N$, proving uniformity.

The $L^2$ energy inequalities for the evolution and its adjoint are
\begin{align*}
 \frac{\dd}{\dd t}\|z_t\|_2^2+\nu\|z_t\|_{H^1}^2
 \le C_\nu(1+\|w_t\|_{H^b}^2)\|z_t\|_2^2.
\end{align*}
The coefficient is integrable since $p>2$. Integration on the first half
of $[s,t]$ gives $r\in[s,(s+t)/2]$ with
\begin{align*}
 \|U_w(r,s)z\|_{H^1}
 \le C_{I,M}(t-s)^{-1/2}\|z\|_2.
\end{align*}
Propagate the $H^1$ norm from $r$ to $t$ by \textup{(i)}. Applying
the same argument to the backward adjoint and dualizing gives the
$H^{-1}\to H$ estimate.
\end{proof}

The dense range proof needs the forward and adjoint fibres to be
differentiable in time on interior intervals.  This follows directly from
the smoothing lemma and the product estimate.
\begin{corollary}
\label{cor:interior-forward-adjoint}
Let $w\in C([0,T];H^b)$, let $2<s<b$, and let $I\Subset(0,T)$.
If $\rho$ solves the forward tangent equation with initial datum in $H$
and $\beta$ solves its backward adjoint equation with terminal datum in
$H$, then
\begin{align}\label{eq:interior-forward-adjoint}
 \rho,\beta\in C(I;H^s)\cap AC(I;H^{s-2}),
 \qquad \rho',\beta'\in L^1(I;H^{s-2}).
\end{align}
\end{corollary}

\begin{proof}
Choose $I\Subset I'\Subset(0,T)$. The finite iteration in
\Cref{lem:gain}\textup{(iii)}, across the positive gaps between the
endpoints of $I$ and $I'$, gives the $C(I;H^s)$ bounds for the forward
evolution and, using the last sentence of \Cref{lem:gain}, for the
backward adjoint. The equations and \Cref{lem:product} give
\begin{align*}
 \|\nu\Delta r+\cS(w,r)\|_{H^{s-2}}
 \le C(1+\|w\|_{H^b})\|r\|_{H^s},
\end{align*}
and the same transposed estimate for the adjoint. The right hand sides are
continuous, hence integrable, on $I$, which proves absolute continuity and
\eqref{eq:interior-forward-adjoint}.
\end{proof}

\section{Two jet and chart radius bounds}
\label{subsec:global-two-jet-revision}

This appendix collects the deterministic rough driver input used in the
compact dense coupling argument.  For $T>0$, let $\mathsf Y_T$ be the Hilbert driver space and $Y_T$ the
corresponding Gaussian driver:
\begin{align}\label{e.Tdriver}
 \mathsf Y_T=H^\sigma(0,T;H^b)\oplus H^b,
 \qquad
 Y_T=(Z|_{[0,T]},Z_T).
\end{align}
The subscript is suppressed when $T=\tau$, the fixed coupling block length. Fix once and for all $1/4<\sigma<1/2$, and put
\begin{align*}
 p_\sigma=\frac2{1-2\sigma}>4,
 \qquad
 q_\sigma=\left(\frac12-\frac1{p_\sigma}\right)^{-1}<4.
\end{align*}

The first task is to make the base equation globally well posed in the
Hilbert rough driver topology used for differentiability and Ramer shifts.
The next lemma gives the required pathwise bound; see
\cite[Propositions~2.4.5 and~2.4.7]{KS12} for the corresponding $L^2$
energy space formulation.

\begin{lemma}
\label{lem:compatible-driver-global}
Let $T<\infty$, $b>2$, and $p>4$. For every $w_0\in H^b$ and every
$y^\circ\in L^p(0,T;H^b)$, the equation
\begin{align}\label{eq:compatible-global-mild}
 v(t)=e^{\nu t\Delta}w_0-
 \int_0^t e^{\nu(t-s)\Delta}
 B(v(s)+y_s^\circ,v(s)+y_s^\circ)\,\dd s
\end{align}
has a unique solution $v\in C([0,T];H^b)$. Moreover, for every
$R<\infty$,
\begin{align}\label{eq:rough-driver-global-bound}
 \sup_{\substack{\|w_0\|_{H^b}
 +\|y^\circ\|_{L^p(0,T;H^b)}\le R}}
 \|v\|_{C([0,T];H^b)}<\infty.
\end{align}
Consequently, since
$H^\sigma(0,T;H^b)\hookrightarrow L^{p_\sigma}(0,T;H^b)$ and
$p_\sigma>4$, the equation is globally well posed for every
$(w_0,y)\in H^b\times\mathsf Y_T$.
\end{lemma}

\begin{proof}
The product estimate
\begin{align*}
 \|B(f,g)\|_{H^{b-1}}
 \le C_b\|f\|_{H^b}\|g\|_{H^b}
\end{align*}
and the heat estimates
\begin{align*}
 \left\|\int_s^t e^{\nu(t-r)\Delta}F_r\,\dd r\right\|_{C_IH^b}
 &\le C\ell^{1/2-1/p}\|F\|_{L^p(I;H^{b-1})},\\
 \left\|\int_s^t e^{\nu(t-r)\Delta}F_r\,\dd r\right\|_{C_IH^b}
 &\le C\ell^{1/2-2/p}\|F\|_{L^{p/2}(I;H^{b-1})}
\end{align*}
on an interval $I$ of length $\ell$ give local mild well posedness in
$C([0,T_0];H^b)$. Since $p>4$, both time powers are positive, and
$T_0$ is uniform on bounded subsets of
$H^b\times L^p(0,T;H^b)$.

The cancellations
\begin{align*}
 \langle B(v,v),v\rangle_2=0,
 \qquad
 \langle B(y^\circ,v),v\rangle_2=0
\end{align*}
and the Biot--Savart estimate give
\begin{align*}
 \frac12\frac{\dd}{\dd t}\|v\|_2^2
 +\nu\|v\|_{H^1}^2
 \le C\bigl(1+\|y_t^\circ\|_{H^b}\bigr)\|v\|_2^2
      +C\|y_t^\circ\|_{H^b}^4.
\end{align*}
Since $p>4$, the coefficients on the right belong to $L^1(0,T)$.
Hence
\begin{align}\label{eq:rough-driver-L2-energy}
 \sup_{t\le T}\|v_t\|_2^2
 +\nu\int_0^T\|v_t\|_{H^1}^2\,\dd t
 \le C_{T,R}
\end{align}
whenever
$\|w_0\|_{H^b}+\|y^\circ\|_{L^p(0,T;H^b)}\le R$.

At the $H^1$ level, the standard two dimensional vorticity estimate
has the form
\begin{align*}
 \frac{\dd}{\dd t}\|v\|_{H^1}^2
 +\frac\nu2\|v\|_{H^2}^2
 \le C\|v\|_2\|v\|_{H^1}^3
 +C\bigl(1+\|y_t^\circ\|_{H^b}^4\bigr)
      \bigl(1+\|v\|_{H^1}^2\bigr).
\end{align*}
After division by $1+\|v\|_{H^1}^2$, the first term is integrable by
\eqref{eq:rough-driver-L2-energy}, since
$\int_0^T\|v_t\|_{H^1}\,\dd t<\infty$. Thus Gronwall's inequality gives
an $L^\infty(0,T;H^1)\cap L^2(0,T;H^2)$ bound depending only on
$T$ and $R$.

For $1<r\le b$, put $r_-:=\max\{1,r-1\}$. The usual commutator and
product estimates yield, for some $m_r<\infty$,
\begin{align*}
 \frac{\dd}{\dd t}\|v\|_{H^r}^2
 +\frac\nu2\|v\|_{H^{r+1}}^2
 \le C_{\nu,r,b}
 \bigl(1+\|y_t^\circ\|_{H^b}^4
       +\|v\|_{H^{r_-}}^{m_r}\bigr)
 \bigl(1+\|v\|_{H^r}^2\bigr).
\end{align*}
Starting with the $H^1$ bound and iterating over finitely many Sobolev
levels gives
\begin{align*}
 \sup_{t\le T}\|v_t\|_{H^b}^2
 +\nu\int_0^T\|v_t\|_{H^{b+1}}^2\,\dd t
 \le C_{T,R,b}.
\end{align*}
The uniform local existence time can therefore be restarted finitely many
times to cover $[0,T]$, proving global existence and
\eqref{eq:rough-driver-global-bound}. The same difference estimate gives
uniqueness. 
\end{proof}

For a reference trajectory $\tilde{w}$, set
$L_{\tilde w}(t)g=\nu\Delta g+\cS(\tilde w_t,g)$, and measure the
coefficient size by
\begin{align}\label{e.031601}
 \Lambda_T(\tilde w)
 =T+T^{1-1/p}\|\tilde w\|_{L^p(0,T;H^b)}.
\end{align}

To control projective two jets we first need energy estimates for the
homogeneous, first variation, and second variation fibre equations.  The
next lemma keeps the relevant time integrability exponents explicit.
\begin{lemma}\label{lem:linearised-variation-energy}
Let $T<\infty$, $b>2$, $p>4$, and
$\tilde w\in L^p(0,T;H^b)$. 
\begin{itemize}
\item If $\partial_t\rho=L_{\tilde w}(t)\rho$, then
\begin{align}\label{eq:revision-raw-homogeneous-energy}
 \mathbb A_T(\rho):=
 \sup_{t\le T}\|\rho_t\|_2^2
 +\nu\int_0^T\|\rho_t\|_{H^1}^2\,\dd t
 \le e^{C\Lambda_T(\tilde w)}\|\rho_0\|_2^2.
\end{align}
\item Let $f_i\in L^p(0,T;H^b), i=1,2$, 
$f_{12}\in L^{p/2}(0,T;H^b)$, and let
\begin{align*}
 \partial_t\zeta^i
 &=L_{\tilde w}(t)\zeta^i+\cS(f_i,\rho),
 &\zeta^i(0)&=\eta_i,\\
 \partial_t\zeta^{12}
 &=L_{\tilde w}(t)\zeta^{12}+\cS(f_{12},\rho)
   +\cS(f_1,\zeta^2)+\cS(f_2,\zeta^1),
 &\zeta^{12}(0)&=\eta_{12}.
\end{align*}
Then
\begin{align}
 \mathbb A_T(\zeta^i)
 &\le Ce^{C\Lambda_T(\tilde w)}
 \left(\|\eta_i\|_2^2
 +T^{1-2/p}\|f_i\|_{L^p(0,T;H^b)}^2\|\rho_0\|_2^2\right),
 \label{eq:revision-raw-first-energy}\\
 \mathbb A_T(\zeta^{12})
 &\le Ce^{C\Lambda_T(\tilde w)}\bigg(
 \|\eta_{12}\|_2^2
 +T^{1-4/p}\|f_{12}\|_{L^{p/2}(0,T;H^b)}^2\|\rho_0\|_2^2
 \notag\\
 &\hspace{26mm}
 +T^{1-2/p}\sum_{i=1}^2
       \|f_i\|_{L^p(0,T;H^b)}^2\mathbb A_T(\zeta^{3-i})
 \bigg).
 \label{eq:revision-raw-second-energy}
\end{align}
\end{itemize}
\end{lemma}

\begin{proof}
Since $b>2$, 
\begin{align}\label{e.030201}
\|\mathcal{S}(f, g)\|_{H^{-1}} \leq C\|f\|_{H^b}\|g\|_2.
\end{align}
Therefore, every solution of
$\partial_tg=L_{\tilde w}(t)g+F$ satisfies
\begin{align}\label{eq:revision-raw-basic-energy}
 \frac{\dd}{\dd t}\|g\|_2^2+\nu\|g\|_{H^1}^2
 \le C(1+\|\tilde w\|_{H^b})\|g\|_2^2
      +C_\nu\|F\|_{H^{-1}}^2.
\end{align}
Gronwall, H\"older's inequality, and
$\|\tilde w\|_{L^1(0,T;H^b)}\le
T^{1-1/p}\|\tilde w\|_{L^p(0,T;H^b)}$ prove
\eqref{eq:revision-raw-homogeneous-energy}. Insert successively
$F=\cS(f_i,\rho)$ and
\begin{align*}
 F=\cS(f_{12},\rho)+\cS(f_1,\zeta^2)
    +\cS(f_2,\zeta^1)
\end{align*}
in \eqref{eq:revision-raw-basic-energy}. Estimating the squared sources
with \eqref{e.030201}, followed by
$L^p\hookrightarrow L^2$ and $L^{p/2}\hookrightarrow L^2$ on $[0,T]$,
gives \eqref{eq:revision-raw-first-energy}--
\eqref{eq:revision-raw-second-energy}. 
\end{proof}

For every $(x,\tilde y)\in H^b\times\mathsf Y_T$, let
$v^{x,\tilde y}\in C([0,T];H^b)$ denote the solution of
\eqref{eq:compatible-global-mild} with initial value $x$ and path
component $\tilde y^\circ$. Set
\begin{align*}
 w^{x,\tilde y}=v^{x,\tilde y}+\tilde y^\circ,
 \qquad
 w_T^{x,\tilde y}=v^{x,\tilde y}(T)+\tilde y^T.
\end{align*}
Use
\begin{align*}
 \mathbb W_T(w_0,y)=1+\|w^{w_0,y}\|_{L^{p_\sigma}(0,T;H^b)}
\end{align*}
as the size of the central base coefficient.

The base solution map must be twice differentiable in the Hilbert driver
topology with quantitative local radii.  The following lemma supplies this
$C^2$ structure together with the pathwise bounds used to define the
projective chart radius.
\begin{lemma}
\label{lem:revision-base-two-jet}
Fix $T<\infty$. There is a deterministic constant
$C_{{\rm loc},T}\ge1$ such that, if
\begin{align*}
 \mathcal L_T(w_0,y)
 &=\exp\left(C_{{\rm loc},T}
       (1+\mathbb W_T(w_0,y)^{q_\sigma})\right),\\
 \delta_T(w_0,y)
 &=\min\left\{1,\frac1{8\mathcal L_T(w_0,y)^2}\right\},
\end{align*}
then the following assertions hold for every $(w_0,y)\in H^b\times\mathsf Y_T$.
\begin{enumerate}[label=\textup{(\roman*)}]
\item
On the open $H^b\times\mathsf Y_T$ ball
\begin{align*}
 B_{H^b\times\mathsf Y_T}
 \left((w_0,y),\frac{\delta_T(w_0,y)}2\right),
\end{align*}
the map $(x,\tilde y)\longmapsto
\bigl(v^{x,\tilde y},w_T^{x,\tilde y}\bigr)$ 
is twice Fr\'echet differentiable from $H^b\times\mathsf Y_T$ into
$C([0,T];H^b)\times H^b$, with continuous operator valued first and
second derivatives. In particular, for every
$\|(h,k)\|_{H^b\times\mathsf Y_T}<\delta_T(w_0,y)/2$,
\begin{align}\label{eq:revision-base-coefficient-tube}
 \|w^{w_0+h,y+k}\|_{L^{p_\sigma}(0,T;H^b)}
 \le \mathbb W_T(w_0,y)+C_T.
\end{align}

\item
For $\mathsf p_0\in H$, put
\begin{align*}
 w^{\circ,x,\tilde y}
 &=v^{x,\tilde y}+\tilde y^\circ,
 \qquad
 \rho_t^{x,\mathsf p_0,\tilde y}
 =U_{w^{\circ,x,\tilde y}}(t,0)\mathsf p_0.
\end{align*}
Then
\begin{align*}
 (x,\mathsf p_0,\tilde y)\longmapsto
 \bigl(w_T^{x,\tilde y},\rho_T^{x,\mathsf p_0,\tilde y}\bigr)
\end{align*}
is locally twice Fr\'echet differentiable into $H^b\times H$, again
with continuous operator valued first and second derivatives. On the
open set where $\rho_T\ne0$, projective normalisation gives an
operator norm $C^2$ projective endpoint.

\item
Let
\begin{align*}
 \mathcal V_T:H^b\times\mathsf Y_T\longrightarrow C([0,T];H^b),
 \qquad
 \mathcal V_T(x,\tilde y)=v^{x,\tilde y}.
\end{align*}
For
$\|(h,k)\|_{H^b\times\mathsf Y_T}<\delta_T(w_0,y)/2$, set
\begin{align*}
 \Xi_1^{h,k}
 &=
 \bigl\|D\mathcal V_T(w_0+h,y+k)\bigr\|_
 {\mathcal L(H^b\times\mathsf Y_T,C([0,T];H^b))},\\
 \Xi_2^{h,k}
 &=
 \bigl\|D^2\mathcal V_T(w_0+h,y+k)\bigr\|_
 {\mathcal L^{(2)}(H^b\times\mathsf Y_T,C([0,T];H^b))},
\end{align*}
and put
\begin{align*}
 \mathcal X_T(w_0,y)
 =
 1+\sup_{\|(h,k)\|_{H^b\times\mathsf Y_T}<\delta_T(w_0,y)/2}
 \bigl(\Xi_1^{h,k}+\Xi_2^{h,k}\bigr).
\end{align*}
Then
\begin{align}\label{eq:revision-base-two-jet-pathwise}
 \log\delta_T(w_0,y)^{-1}
 +\log^+\mathcal X_T(w_0,y)
 \le C_T\bigl(1+\mathbb W_T(w_0,y)^{q_\sigma}\bigr).
\end{align}
\end{enumerate}
\end{lemma}

\begin{proof}
By \Cref{lem:compatible-driver-global} and $H^\sigma(0,T;H^b)\hookrightarrow L^{p_\sigma}(0,T;H^b)$, the rough base equation is globally wellposed for every $(w_0,y)\in H^b\times\mathsf Y_T$.

We repeatedly use the following linear estimate. Let $\bar w\in L^{p_\sigma}(0,T;H^b)$ and suppose
\begin{align*}
 \partial_tz=\nu\Delta z-B(z,\bar w)-B(\bar w,z)+F,\qquad z(0)=h.
\end{align*}
Since $\|B(f,g)\|_{H^{b-1}}\le C\|f\|_{H^b}\|g\|_{H^b}$, the $H^b$ energy estimate gives
\begin{align*}
 \frac{\dd}{\dd t}\|z\|_{H^b}^2+\nu\|z\|_{H^{b+1}}^2
 \le C\|\bar w\|_{H^b}^2\|z\|_{H^b}^2+C\|F\|_{H^{b-1}}^2.
\end{align*}
Hence, since $p_\sigma>4$ and therefore $L^{p_\sigma/2}(0,T)\hookrightarrow L^2(0,T)$,
\begin{align}\label{eq:revision-global-linear-estimate}
 \|z\|_{C_TH^b}
 \le C_T\exp\left(C_T\|\bar w\|_{L^{p_\sigma}(0,T;H^b)}^2\right)
 \left(\|h\|_{H^b}+\|F\|_{L^{p_\sigma/2}(0,T;H^{b-1})}\right).
\end{align}

Fix $(w_0,y)\in H^b\times\mathsf Y_T$. For $a_i=(h_i,k_i)\in H^b\times\mathsf Y_T$, where $k_i=(k_i^\circ,k_i^T)$, let the first and second base variations be defined by
\begin{align}\label{rev:eq-base-first-variation}
 \delta_iv_t=e^{\nu t\Delta}h_i-\int_0^t e^{\nu(t-s)\Delta}
 \left[B(\delta_iw_s,w_s)+B(w_s,\delta_iw_s)\right]\,\dd s,
 \qquad \delta_iw=\delta_iv+k_i^\circ,
\end{align}
and
\begin{align}\label{rev:eq-base-second-variation}
 \delta_{12}v_t=-\int_0^t e^{\nu(t-s)\Delta}\bigl[
 B(\delta_{12}v_s,w_s)+B(w_s,\delta_{12}v_s)
 +B(\delta_1w_s,\delta_2w_s)+B(\delta_2w_s,\delta_1w_s)\bigr]\,\dd s.
\end{align}
At the endpoint,
\begin{align*}
 \delta_iw_T=\delta_iv_T+k_i^T,\qquad \delta_{12}w_T=\delta_{12}v_T.
\end{align*}

In \eqref{rev:eq-base-first-variation}, the terms containing $k_i^\circ$ are regarded as forcing. Since $\mathsf Y_T\hookrightarrow L^{p_\sigma}(0,T;H^b)$,
\begin{align*}
 \|B(k_i^\circ,w)+B(w,k_i^\circ)\|_{L^{p_\sigma/2}(0,T;H^{b-1})}
 \le C_T\|k_i\|_{\mathsf Y_T}\|w\|_{L^{p_\sigma}(0,T;H^b)}.
\end{align*}
Applying \eqref{eq:revision-global-linear-estimate} first to \eqref{rev:eq-base-first-variation} and then to \eqref{rev:eq-base-second-variation} gives
\begin{align}\label{eq:revision-base-variation-bound}
\begin{split}
 \|\delta_iv\|_{C_TH^b}+\|\delta_iw\|_{L^{p_\sigma}(0,T;H^b)}
 &\le \exp\left(C_T(1+\mathbb W_T(w_0,y)^{q_\sigma})\right)\|a_i\|,\\
 \|\delta_{12}v\|_{C_TH^b}+\|\delta_{12}v\|_{L^{p_\sigma/2}(0,T;H^b)}
 &\le \exp\left(C_T(1+\mathbb W_T(w_0,y)^{q_\sigma})\right)\|a_1\|\|a_2\|.
\end{split}
\end{align}
Here we used $\mathbb W_T(w_0,y)\ge1$ and $q_\sigma>2$ to absorb the quadratic exponent arising from \eqref{eq:revision-global-linear-estimate}.

We next control perturbations of the base equation. Let $v^{w_0+h,y+k}$ be the solution corresponding to the perturbed data and set $\bar v=v^{w_0+h,y+k}-v^{w_0,y}$. Subtracting the two equations and applying \eqref{eq:revision-global-linear-estimate} with coefficient $w^{w_0,y}$ gives
\begin{align*}
 \|\bar v\|_{C_TH^b}
 \le \mathcal L_T(w_0,y)\|(h,k)\|
 +\mathcal L_T(w_0,y)\left(\|\bar v\|_{C_TH^b}+\|(h,k)\|\right)^2.
\end{align*}
By the definition of $\delta_T(w_0,y)$, a continuity argument yields
\begin{align}\label{eq:revision-base-Lipschitz-tube}
 \|\bar v\|_{C_TH^b}\le2\mathcal L_T(w_0,y)\|(h,k)\|
\end{align}
whenever $\|(h,k)\|_{H^b\times\mathsf Y_T}<\delta_T(w_0,y)/2$. Consequently,
\begin{align*}
 \|w^{w_0+h,y+k}\|_{L^{p_\sigma}(0,T;H^b)}
 \le \mathbb W_T(w_0,y)+C_T
\end{align*}
throughout this ball.

To identify the first variation, let $\delta v[h,k]$ be the solution of \eqref{rev:eq-base-first-variation} and set
\begin{align*}
 R=v^{w_0+h,y+k}-v^{w_0,y}-\delta v[h,k].
\end{align*}
Subtracting the nonlinear difference equation and the first variation equation shows that $R$ satisfies a linear equation of the form used in \eqref{eq:revision-global-linear-estimate}, whose forcing is quadratic in $v^{w_0+h,y+k}-v^{w_0,y}+k^\circ$. Hence \eqref{eq:revision-base-Lipschitz-tube} and \eqref{eq:revision-global-linear-estimate} give
\begin{align}\label{eq:revision-base-first-remainder}
 \|R\|_{C_TH^b}
 \le C_T\exp\left(C_T(1+\mathbb W_T(w_0,y)^{q_\sigma})\right)\|(h,k)\|^2.
\end{align}
Thus the first variation is the Fr\'echet derivative of the base solution map. Subtracting the first variation equations at two nearby base points and applying the same estimate gives
\begin{align*}
 \bigl\|D\mathcal V_T(w_0+h_2,y+k_2)[a_1]-D\mathcal V_T(w_0,y)[a_1]-\delta_{12}v\bigr\|_{C_TH^b}
 \le C\|a_1\|\|a_2\|^2.
\end{align*}
Therefore $\delta_{12}v=D^2\mathcal V_T(w_0,y)[a_1,a_2]$. The same estimates at nearby base points prove continuity in operator norm of the first two derivatives. Since the terminal coordinate $y^T$ enters affinely, this proves assertion \textup{(i)}.

Moreover, \eqref{eq:revision-base-coefficient-tube} allows \eqref{eq:revision-base-variation-bound} to be applied uniformly at every point of the $\delta_T(w_0,y)/2$ ball. Hence
\begin{align*}
 \mathcal X_T(w_0,y)
 \le \exp\left(C_T(1+\mathbb W_T(w_0,y)^{q_\sigma})\right).
\end{align*}
Together with the definition of $\delta_T(w_0,y)$, this yields \eqref{eq:revision-base-two-jet-pathwise}.

It remains to treat the raw fibre. Let $\alpha_i=(a_i,\eta_i)$, where $a_i=(h_i,k_i)\in H^b\times\mathsf Y_T$ and $\eta_i\in H$. The first two raw fibre variations satisfy
\begin{align*}
 \partial_t\zeta^i
 &=L_{w^{w_0,y}}(t)\zeta^i+\cS(\delta_iw,\rho),
 &\zeta^i(0)&=\eta_i,\notag\\
 \partial_t\zeta^{12}
 &=L_{w^{w_0,y}}(t)\zeta^{12}+\cS(\delta_{12}v,\rho)
   +\cS(\delta_1w,\zeta^2)+\cS(\delta_2w,\zeta^1),
 &\zeta^{12}(0)&=0.
\end{align*}
Combining \eqref{eq:revision-base-variation-bound} with \Cref{lem:linearised-variation-energy} gives
\begin{align}\label{eq:revision-fibre-variation-bound}
 \|\zeta^i\|_{C_TH}
 &\le C\|\alpha_i\|,\notag\\
 \|\zeta^{12}\|_{C_TH}
 &\le C\|\alpha_1\|\|\alpha_2\|
\end{align}
locally uniformly in the base point.

The Fr\'echet identification follows by the same subtraction argument. If $(\tilde w,\tilde\rho)$ corresponds to a perturbed input and $\zeta$ denotes the first variation, then
\begin{align*}
 R=\tilde\rho-\rho-\zeta
\end{align*}
satisfies a linear equation whose source consists of the quadratic base remainder and products of first variations. By \eqref{eq:revision-base-first-remainder}, \eqref{eq:revision-fibre-variation-bound}, and \Cref{lem:linearised-variation-energy},
\begin{align*}
 \|R\|_{C_TH}\le C\|\alpha\|^2.
\end{align*}
Repeating the same argument for the difference of first variation equations gives the second derivative and continuity of both operator valued derivatives. Thus the raw joint endpoint
\begin{align*}
 (x,\mathsf p_0,\tilde y)\longmapsto\bigl(w_T^{x,\tilde y},\rho_T^{x,\mathsf p_0,\tilde y}\bigr)
\end{align*}
is locally $C^2$ into $H^b\times H$.

Finally, $N(r)=r/\|r\|_2$ is $C^\infty$ on $H\setminus\{0\}$ and
\begin{align*}
 \|DN(r)\|\le\|r\|_2^{-1},\qquad \|D^2N(r)\|\le3\|r\|_2^{-2}.
\end{align*}
Composing with $N$ and the smooth orthogonal bundle frames \eqref{eq:chart-rotation} proves assertion \textup{(ii)} on the open set where the raw fibre endpoint is nonzero.
\end{proof}

Parabolic regularity turns the state derivative into a compact operator in
the rough driver topology.  The next corollary gives the stronger space
factorisation uniformly on bounded coefficient sets with a nonvanishing
endpoint.
\begin{corollary}
\label{cor:rough-driver-state-compact}
Fix
$0<\delta_{\sigma}<\min\{1-2/p_\sigma,b-1\}$.
At every nonvanishing endpoint in
\Cref{lem:revision-base-two-jet}, the state derivative 
\[J(z, y): H^b \oplus \mathsf p_0^{\perp} \longrightarrow H^b \oplus \mathsf p_T^{\perp}\]
is compact. More precisely, on
every set on which
\begin{align*}
 \|w^\circ\|_{L^{p_\sigma}(0,T;H^b)}\le M,
 \qquad
 \|\rho_T\|_2\ge d>0,
\end{align*}
$J(z, y)$ is uniformly bounded from
$H^b\oplus \mathsf p_0^\perp$ into
$H^{b+\delta_{\sigma}}\oplus H^{\delta_{\sigma}}$. 
\end{corollary}

\begin{proof}
Let $0<\delta_{\sigma}<\min\{1-2/p_\sigma,b-1\}$ and
write $\widehat w$ for a coefficient satisfying
$\|\widehat w\|_{L^{p_\sigma}(0,T;H^b)}\le M$. For a unit initial
fibre $\mathsf p_0$, put
\begin{align*}
 \rho_s=U_{\widehat w}(s,0)\mathsf p_0.
\end{align*}
For a state perturbation $(h,\eta)\in H^b\oplus \mathsf p_0^\perp$, the raw
state derivative at time $T$ is
\begin{align*}
 (h,\eta)\longmapsto
 \left(U_{\widehat w}(T,0)h,
 U_{\widehat w}(T,0)\eta+Rh\right),
\end{align*}
where
\begin{align*}
 Rh=\int_0^T U_{\widehat w}(T,s)
 \cS\bigl(\rho_s,U_{\widehat w}(s,0)h\bigr)\,\dd s.
\end{align*}
By \Cref{lem:gain},
\begin{align*}
 U_{\widehat w}(T,0):H^b\longrightarrow H^{b+\delta_{\sigma}},
 \qquad
 U_{\widehat w}(T,0):H\longrightarrow H^{\delta_{\sigma}},
\end{align*}
with operator norms bounded uniformly on the indicated
$L^{p_\sigma}(0,T;H^b)$ bounded set. It remains to control $R$.

For $0\le s\le T/2$, \Cref{lem:product} and boundedness of the
evolution on $H$ and $H^b$ give
\begin{align*}
 \bigl\|\cS(\rho_s,U_{\widehat w}(s,0)h)\bigr\|_{H^{-1}}
 \le C_M\|h\|_{H^b}.
\end{align*}
Since $T-s\ge T/2$, choose finitely many gains, each smaller than
$1-2/p_\sigma$, whose sum is $1+\delta_{\sigma}$. Iterating
\Cref{lem:gain} over a fixed subdivision of $[s,T]$ gives
\begin{align*}
 \sup_{0\le s\le T/2}
 \|U_{\widehat w}(T,s)\|_{H^{-1}\to H^{\delta_{\sigma}}}
 \le C_M.
\end{align*}
Hence the contribution of $[0,T/2]$ to $Rh$ is bounded by
$C_M\|h\|_{H^b}$ in $H^{\delta_{\sigma}}$.

For $T/2\le s\le T$, positive time smoothing from the initial time
gives
\begin{align*}
 \sup_{T/2\le s\le T}\|\rho_s\|_{H^b}\le C_M,
 \qquad
 \sup_{0\le s\le T}
 \|U_{\widehat w}(s,0)h\|_{H^b}\le C_M\|h\|_{H^b}.
\end{align*}
Thus \Cref{lem:product} gives
\begin{align*}
 \bigl\|\cS(\rho_s,U_{\widehat w}(s,0)h)\bigr\|_{H^{b-1}}
 \le C_M\|h\|_{H^b},
 \qquad T/2\le s\le T.
\end{align*}
The evolution is bounded on $H^{b-1}$ and
$H^{b-1}\hookrightarrow H^{\delta_{\sigma}}$, so the second half of
$Rh$ satisfies the same bound. Consequently $R:H^b\longrightarrow H^{\delta_{\sigma}}$
is bounded, uniformly under the stated coefficient bound. The raw state
derivative therefore factors boundedly through
\begin{align*}
 H^{b+\delta_{\sigma}}\oplus H^{\delta_{\sigma}}.
\end{align*}
Since
$H^{b+\delta_{\sigma}}\Subset H^b$ and
$H^{\delta_{\sigma}}\Subset H$, the raw state derivative is compact
into $H^b\oplus H$. If also $\|\rho_T\|_2\ge d>0$, the derivative of
projective normalisation is bounded by $d^{-1}$; composing with it
preserves compactness and gives the uniform normalised statement.

\end{proof}

For the compatible Gaussian driver $Y_T$, use the corresponding base path
$w^{w_0,Y_T}$. Since $Y_T$ is a centred Radon Gaussian variable
in $\mathsf Y_T$, Fernique's theorem gives every finite
$\mathsf Y_T$ moment. Hence there is an integer $m_0\ge q_\sigma$ such
that, with
\begin{align*}
 \mathcal A_T
 =1+\|Y_T\|_{\mathsf Y_T}^{m_0}
   +\|w^{w_0,Y_T}\|_{L^{p_\sigma}(0,T;H^b)}^{m_0},
\end{align*}
for every $r<\infty$ there is $V_r$ of the base estimates in
\eqref{eq:revision-base-weight} such that 
\begin{align}\label{eq:revision-base-path-moments}
 \mathbf E_{w_0}\mathcal A_T^r\le C_{T,r}V_r(w_0).
\end{align}

Use the affine state coordinate $\Delta_z$ introduced in
\eqref{eq:state-difference-chart}.  All tangent and operator norms are the
quotient metric norms fixed above, and continuity as the central projective
point varies is read in the orthogonal frames
\eqref{eq:chart-rotation}.  On the finitely many subcharts used below, the
projective denominators are uniformly separated from zero; the first two
frame transition derivatives are therefore bounded and are included in
$C_{\rm ch}$.

Recall that for $z=(w_0,[\mathsf p])$ with $\mathsf p$ an $L^2$ unit representative, the input coordinate
and its inverse are
\begin{align*}
 \Delta_z(w_0',[q])
 &=\left(w_0'-w_0,\chi_{\mathsf p}([q])\right),\\
 \Delta_z^{-1}(h,\eta)
 &=\left(w_0+h,[\mathsf p+\eta]\right).
\end{align*}
Let $\mathcal E_T$ be the unnormalised base raw fibre endpoint in
this chart and in the Hilbert driver coordinate $\mathsf Y_T$ of
\eqref{e.Tdriver}. At the central raw endpoint
$\rho_T=U(T,0)\mathsf p$, put
$d_T=\|\rho_T\|_2$. If $d_T>0$, set
$\mathsf p_T=\rho_T/d_T$ and use the output chart
\begin{align*}
 \chi_{\mathsf p_T}([r])=
 \frac{\Pi_{\mathsf p_T^\perp}r}{\langle r,\mathsf p_T\rangle},
 \qquad \langle r,\mathsf p_T\rangle\ne0.
\end{align*}
Define the local two jet size by
\begin{align*}
 \begin{aligned}
 M_T(z,y)=1+\sup_{\|(h,\eta,k)\|<\delta_T(w_0,y)/2}\bigl(&
   \|D\mathcal E_T(w_0+h,[\mathsf p+\eta],y+k)\|\\
   &+\|D^2\mathcal E_T(w_0+h,[\mathsf p+\eta],y+k)\|\bigr),
 \end{aligned}
\end{align*}
where $h\in H^b$, $\eta\in\mathsf p^\perp$, and $k\in\mathsf Y_T$, and the
derivatives are taken in the unscaled native coordinates.  Together,
\Cref{lem:linearised-variation-energy,lem:revision-base-two-jet} show that
this supremum is finite at every compatible central point.  Define the
admissible chart radius by
\begin{align}\label{eq:revision-chart-radius}
 \mathsf r_T(z,y)=
 \frac18\min\left\{\delta_T(w_0,y),
                         \frac{d_T(z,y)}{M_T(z,y)}\right\}.
\end{align}
On the ball of radius $\mathsf r_T$, the raw endpoint differs from its central
value by at most $d_T/8$, so the output chart is valid.  Fix a universal
chart constant $C_{\rm ch}$ and define the corresponding two jet bound
\begin{align}\label{eq:revision-C-def}
 \mathsf C_T(z,y)=C_{\rm ch}\left(
 1+M_T+d_T^{-1}M_T+d_T^{-1}M_T^2
                 +d_T^{-2}M_T^2\right).
\end{align}
At points where the central equation or output chart is undefined, set
$\mathsf r_T=0$ and $\mathsf C_T=+\infty$.

For
\begin{align*}
 e=(h,\eta,k)\in H^b\oplus\mathsf p^\perp\oplus\mathsf Y_T,
 \qquad \|e\|<\mathsf r_T,
\end{align*}
write
\begin{align*}
 \mathcal E_T(w_0+h,[\mathsf p+\eta],y+k)
 =\bigl(w_T^e,\rho_T^e\bigr).
\end{align*}
Then
\begin{align}\label{eq:projective-chart-denominator}
 \|\rho_T^e-\rho_T\|_2\le\frac{d_T}{8},
 \qquad
 \langle\rho_T^e,\mathsf p_T\rangle\ge\frac{7d_T}{8}.
\end{align}
In particular, the fixed output chart $\chi_{\mathsf p_T}([r])
=\frac{\Pi_{\mathsf p_T^\perp}r}{\langle r,\mathsf p_T\rangle}$ 
is valid throughout the $\mathsf r_T$ ball. The actual projective endpoint is
\begin{align*}
 \Phi_T(w_0+h,[\mathsf p+\eta],y+k)
 =\bigl(w_T^e,[\rho_T^e]\bigr).
\end{align*}
Its difference from the central endpoint in the unified output coordinate is
\begin{align*}
 \widehat\Phi_T(e)
 &:=\Delta_{\Phi_T(z,y)}\left(
       \Phi_T\bigl(\Delta_z^{-1}(h,\eta),y+k\bigr)
     \right)\notag\\
 &=\left(
 w_T^e-w_T,
 \frac{\Pi_{\mathsf p_T^\perp}\rho_T^e}
      {\langle\rho_T^e,\mathsf p_T\rangle}
 \right)
 \in H^b\oplus \mathsf p_T^\perp.
\end{align*}
Thus $\widehat\Phi_T(0)=0$.  
After enlarging the universal chart constant $C_{\rm ch}$ if necessary,
a direct computation shows that throughout the $\mathsf r_T$ ball
\begin{align}\label{eq:projective-chart-two-jet-bound}
 \|D\widehat\Phi_T(e)\|
 +\|D^2\widehat\Phi_T(e)\|
 \le \mathsf C_T.
\end{align}
At the central point, the restriction of
$D\widehat\Phi_T(0)$ to state directions
$H^b\oplus\mathsf p^\perp$ is the coordinate representative of
$J(z,y)=D_z\Phi(z,y)$. Consequently,
\begin{align*}
 \|J(z,y)\|
 \le C_{\rm ch}\bigl(M_T+d_T^{-1}M_T\bigr)
 \le\mathsf C_T.
\end{align*}

The continuity of the solution map on $H^b\times\mathsf Y_T$ implies that $\mathbb W_T$ and $\delta_T$ are continuous. In each of the fixed trivialisations \eqref{eq:chart-rotation}, the first and second endpoint jets depend continuously on the base point and the perturbation. Moreover, if $\|e\|<\delta_T(w_0,y)/2$, then the same strict inequality holds at all sufficiently nearby base points. Taking the supremum over admissible perturbations therefore shows that $M_T$ is lower semicontinuous.

The raw endpoint norm $d_T$ is continuous. Hence $\mathsf r_T^{-1}$ and $\mathsf C_T$ are lower semicontinuous on $\{d_T>0\}$. We extend both quantities by $+\infty$ on $\{d_T=0\}$. Since $M_T\ge1$ and
\begin{align*}
\mathsf r_T^{-1}\ge 8M_T d_T^{-1},
\end{align*}
these extensions remain lower semicontinuous. Since $\|J\|$ is continuous on $\{d_T>0\}$, the function
\begin{align*}
(z,y)\longmapsto \log^+\left(1+\|J(z,y)\|+\mathsf C_T(z,y)+\mathsf r_T(z,y)^{-1}\right),
\end{align*}
extended by $+\infty$ on $\{d_T=0\}$, is lower semicontinuous and hence Borel measurable. Indeed, if $d_T(z_n,y_n)\to0$, then $\mathsf r_T(z_n,y_n)^{-1}\to+\infty$ by the preceding bound.

Finally, for the compatible driver $y=Y$, the standard backward
uniqueness theorem gives $d_T(z,Y)>0$ at the central endpoint.  The
raw endpoint closeness in \eqref{eq:revision-chart-radius} then guarantees
nonvanishing throughout the corresponding chart ball.

Recall the base weight 
\[V_r(w_0) = (1+\|w_0\|_{H^b}^2)^{R_r}\exp(\eta_0\|w_0\|_2^2) \]
from \eqref{eq:revision-base-weight}  and the reset time $t_*$ from \Cref{lem:robust-HPSRY-hierarchy}. 
The deterministic chart bounds must finally be integrated under the
stochastic driver.  Combining the raw two jet estimates, the positive time
reset, and the base moments gives the logarithmic moment estimate used in
\Cref{sec:random-stable-ball}.
\begin{theorem}
\label{lem:projective-log-moment}
Fix $T\ge t_*$.  For every $r<\infty$, there are integers
$d_r,R_r<\infty$, and a constant
$C_{T,r}$ such that, for every $w_0\in H^b$ and every unit
$\mathsf p\in H^1$,
\begin{align}
 &\mathbf{E}_{w_0,\mathsf p}
 \left[\log^+\left(
  1+\|J(z,Y_T)\|+\mathsf C_T(z,Y_T)+\mathsf r_T(z,Y_T)^{-1}
 \right)\right]^r \le
 C_{T,r}\left(V_r(w_0)+1+\|\mathsf p\|_{H^1}^{d_r}\right).
 \label{eq:revision-global-two-jet-moment}
\end{align}
The orders in $V_r$ and $d_r$ are independent of $T$.  The
estimate is interpreted as $+\infty$ when $[\mathsf p]\notin\PP(H^1)$.
It holds with the same form conditionally after a stopping time.
\end{theorem}

\begin{proof}
Fix a unit representative $\mathsf p\in H^1$, and let $\rho_t=U(t,0)\mathsf p$,  $d_t=\|\rho_t\|_2$,  $ \mathsf p_t=\frac{\rho_t}{d_t}$. 
Backward uniqueness gives $d_t>0$ for $t>0$. Since
\[\partial_t\rho_t=\nu\Delta \rho_t +\cS(w_t,\rho_t),\] 
a direct computation shows 
\begin{align*}
 \frac{\dd}{\dd t}\log d_t^{-1}=-\langle \mathsf p_t, \nu\Delta \mathsf p_t +\cS(w_t,\mathsf p_t)\rangle\leq \nu \|\mathsf p_t\|_{H^1}^2 + C\|w_t\|_{H^b}
\end{align*}
as $\|\mathsf p_t\|_2=1$ and $b>2$. Because $d_0=1$, integration gives
\begin{align}\label{eq:revision-d-inverse}
 \log^+d_T^{-1}
 \le
 \nu\int_0^T\|\mathsf p_t\|_{H^1}^2\,\dd t
 +C\int_0^T\|w_t\|_{H^b}\,\dd t.
\end{align}

We next control the raw endpoint two jet. Fix the compatible central pair
$(w_0,Y_T)$ and write $w^{w_0+h,Y_T+k}$ as $w$ for simplicity. Consider an arbitrary point of the adaptive tube,
\begin{align*}
 \|(h,k)\|_{H^b\times\mathsf Y_T}<\delta_T(w_0,Y_T)/2,
\end{align*}
By
\Cref{lem:revision-base-two-jet}, uniformly over this tube,
\begin{align}\label{eq:revision-base-variation-tube-bound}
\begin{split}
\|w\|_{L^{p_\sigma}(0,T;H^b)}
 &\le \mathbb W_T(w_0,Y_T)+C_T,\\
 \|\delta_iw\|_{L^{p_\sigma}(0,T;H^b)}
 &\le
 \exp\bigl(C_T(1+\mathbb W_T(w_0,Y_T)^{q_\sigma})\bigr)
 \|a_i\|,\\
 \|\delta_{12}w\|_{L^{p_\sigma/2}(0,T;H^b)}
 &\le
 \exp\bigl(C_T(1+\mathbb W_T(w_0,Y_T)^{q_\sigma})\bigr)
 \|a_1\|\,\|a_2\|,
\end{split}
\end{align}
where $a_i=(h_i,k_i)\in H^b\times\mathsf Y_T$ are base driver
directions. The bounds hold as operator and bilinear operator estimates,
respectively.

Let the affine fibre input at this tube point be $\mathsf p+\bar\eta$, where
$\|\bar\eta\|<\delta_T(w_0,Y_T)/2$. Since $\delta_T\le1$, one has
$\|\mathsf p+\bar\eta\|_2\le3/2$. Let $\rho$ be the corresponding raw fibre
solution, and let $\zeta^i$ and $\zeta^{12}$ be its first and second
variations in full directions
\begin{align*}
 \alpha_i=(a_i,\eta_i),
 \qquad \eta_i\in H,
\end{align*}
that satisfying
\begin{align*}
 \partial_t\zeta^i
 &=L_{w}(t)\zeta^i+\cS(\delta_iw,\rho),
 &\zeta^i(0)&=\eta_i,\\
 \partial_t\zeta^{12}
 &=L_{w}(t)\zeta^{12}
   +\cS(\delta_{12}w,\rho)
   +\cS(\delta_1w,\zeta^2)
   +\cS(\delta_2w,\zeta^1),
 &\zeta^{12}(0)&=0.
\end{align*}

Since $q_\sigma>1$, \eqref{e.031601} and  the first line of
\eqref{eq:revision-base-variation-tube-bound} implies, after enlarging
$C_T$,
\begin{align*}
 \exp\bigl(C\Lambda_T(w)\bigr)
 \le \exp\bigl(C_T(1+\mathbb W_T(w_0,Y_T)^{q_\sigma})\bigr).
\end{align*}
The homogeneous estimate
\eqref{eq:revision-raw-homogeneous-energy} and initial $\|\mathsf p+\bar\eta\|_2\le3/2$ therefore give
\begin{align}\label{eq:revision-rho-energy-tube}
 \mathbb A_T(\rho)
 \le \exp\bigl(C_T(1+\mathbb W_T(w_0,Y_T)^{q_\sigma})\bigr).
\end{align}

Applying \eqref{eq:revision-raw-first-energy} and using
\eqref{eq:revision-base-variation-tube-bound}--\eqref{eq:revision-rho-energy-tube}, we obtain
\begin{align*}
 \mathbb A_T(\zeta^i)
 &\le C\exp\bigl(C_T(1+\mathbb W_T(w_0,Y_T)^{q_\sigma})\bigr)
 \left(
  \|\eta_i\|_2^2
  +T^{1-2/p_\sigma}
   \|\delta_iw\|_{L^{p_\sigma}(0,T;H^b)}^2
   \|\mathsf p+\bar\eta\|_2^2
 \right)\\
 &\le
 \exp\bigl(C_T(1+\mathbb W_T(w_0,Y_T)^{q_\sigma})\bigr)
 \|\alpha_i\|^2,
\end{align*}
where repeated exponential factors have been absorbed by increasing the
deterministic constant $C_T$. Similarly,
\eqref{eq:revision-raw-second-energy} and the first variation estimate give
\begin{align*}
 \mathbb A_T(\zeta^{12})
 &\le C\exp\bigl(C_T(1+\mathbb W_T(w_0,Y_T)^{q_\sigma})\bigr)
 \bigg(
 T^{1-4/p_\sigma}
 \|\delta_{12}w\|_{L^{p_\sigma/2}(0,T;H^b)}^2
 \|\mathsf p+\bar\eta\|_2^2\\
 &\hspace{28mm}
 +T^{1-2/p_\sigma}
 \sum_{i=1}^2
 \|\delta_iw\|_{L^{p_\sigma}(0,T;H^b)}^2
 \mathbb A_T(\zeta^{3-i})
 \bigg)\\
 &\le
 \exp\bigl(C_T(1+\mathbb W_T(w_0,Y_T)^{q_\sigma})\bigr)
 \|\alpha_1\|^2\|\alpha_2\|^2.
\end{align*}

Let $\mathcal R_1$ and $\mathcal R_2$ denote the suprema, over the
adaptive tube and over unit full directions, of the $C([0,T];H)$ norms
of the first and second raw fibre variations. Since
$\|\zeta\|_{C_TH}^2\le\mathbb A_T(\zeta)$, the preceding estimates yield
\begin{align}\label{eq:revision-raw-two-jet}
 \log^+(1+\mathcal R_1+\mathcal R_2)
 \le C_T(1+\mathbb W_T(w_0,Y_T)^{q_\sigma}).
\end{align}

The raw endpoint $\mathcal E_T$ consists of the base endpoint and the raw
fibre endpoint. Therefore
\eqref{eq:revision-base-two-jet-pathwise} and
\eqref{eq:revision-raw-two-jet} imply
\begin{align*}
 M_T
 \le C_T\bigl(1+\mathcal X_T+\mathcal R_1+\mathcal R_2\bigr),
\end{align*}
and hence
\begin{align}\label{eq:revision-M-global}
 \log\delta_T^{-1}+\log^+M_T
 \le C_T(1+\mathbb W_T(w_0,Y_T)^{q_\sigma}).
\end{align}
For the Gaussian compatible driver $Y_T$, the integer $m_0$ in
$\mathcal A_T$ was chosen with $m_0\ge q_\sigma$ in \eqref{eq:revision-base-path-moments}. Thus
\begin{align*}
 1+\mathbb W_T(w_0,Y_T)^{q_\sigma}\le C_T\mathcal A_T,
\end{align*}
and \eqref{eq:revision-M-global} becomes
\begin{align}\label{eq:revision-M-global-random}
 \log\delta_T^{-1}+\log^+M_T
 \le C_T\mathcal A_T.
\end{align}

The definitions \eqref{eq:revision-chart-radius} and
\eqref{eq:revision-C-def} give
\begin{align*}
 \mathsf r_T^{-1}
 &\le 8\left(\delta_T^{-1}+M_Td_T^{-1}\right),\\
 \mathsf C_T
 &=C_{\rm ch}\left(
 1+M_T+d_T^{-1}M_T+d_T^{-1}M_T^2+d_T^{-2}M_T^2
 \right).
\end{align*}
Moreover, $\mathsf C_T$ bounds the first two derivatives of the
normalised projective endpoint on the $\mathsf r_T$ ball, and in
particular
\begin{align*}
 \|J(z,Y_T)\|\le \mathsf C_T.
\end{align*}
Elementary logarithmic inequalities therefore give
\begin{align*}
 &\log^+\left(
 1+\|J(z,Y_T)\|+\mathsf C_T+\mathsf r_T^{-1}
 \right)\\
 &\qquad\le
 C_T\left(
 1+\log\delta_T^{-1}
 +\log^+M_T+\log^+d_T^{-1}
 \right).
\end{align*}
Combining this with
\eqref{eq:revision-d-inverse} and
\eqref{eq:revision-M-global-random}, and using
\begin{align*}
 \int_0^T\|w_t\|_{H^b}\,\dd t
 \le
 T^{1-1/p_\sigma}
 \|w\|_{L^{p_\sigma}(0,T;H^b)}
 \le C_T\mathcal A_T,
\end{align*}
gives the pathwise estimate
\begin{align}\label{eq:C2-log-bound}
 \log^+\left(
 1+\|J(z,Y_T)\|+\mathsf C_T+\mathsf r_T^{-1}
 \right)
 \le
 C_T\left(
 \mathcal A_T
 +1+\int_0^T\|\mathsf p_t\|_{H^1}^2\,\dd t
 \right).
\end{align}

Finally take the $r$-th moment. By
\eqref{eq:revision-base-path-moments},
\begin{align*}
 \mathbf E_{w_0}\mathcal A_T^r
 \le C_{T,r}V_r(w_0),
\end{align*}
after choosing a sufficiently large member of the nested base weight
hierarchy. By \eqref{eq:robust-block-energy},
\begin{align*}
 \mathbf E_{w_0,\mathsf p}
 \left(
 \int_0^T\|\mathsf p_t\|_{H^1}^2\,\dd t
 \right)^r
 \le
 C_{T,r}\left(
 V_r(w_0)+1+\|\mathsf p\|_{H^1}^{d_r}
 \right)
\end{align*}
after enlarging $d_r$ and, if necessary, the polynomial exponent
$R_r$ in $V_r$. Applying the inequality for the $r$-th
power of a finite sum to \eqref{eq:C2-log-bound} proves
\eqref{eq:revision-global-two-jet-moment}.

The same deterministic estimates apply after a finite stopping time.
Conditioning on the state at that time and using the strong Markov
property gives the conditional version with the same form.
\end{proof}

\section*{Statements and declarations}

\subsection*{Competing interests.}
The authors declare that they have no competing interests.

\subsection*{Use of generative artificial intelligence}
OpenAI ChatGPT  was used to assist with proof auditing and manuscript preparation. The authors take full responsibility for the mathematical content.

\end{document}